\documentclass{amsart}

\usepackage{amsmath}
\usepackage{amssymb}
\usepackage[all]{xy}
\usepackage{mathrsfs}
\usepackage{upgreek}
\usepackage{stmaryrd}
\usepackage[colorlinks = true]{hyperref}
\usepackage{rotating}
\usepackage{tikz}
\usetikzlibrary{cd}

\DeclareFontFamily{U}{mathx}{\hyphenchar\font45}
\DeclareFontShape{U}{mathx}{m}{n}{
      <5> <6> <7> <8> <9> <10>
      <10.95> <12> <14.4> <17.28> <20.74> <24.88>
      mathx10
      }{}
\DeclareSymbolFont{mathx}{U}{mathx}{m}{n}
\DeclareFontSubstitution{U}{mathx}{m}{n}
\DeclareMathAccent{\widecheck}{0}{mathx}{"71}

\newcommand{\bk}{\Bbbk}
\newcommand{\bO}{\mathbb{O}}
\newcommand{\bK}{\mathbb{K}}
\newcommand{\bF}{\mathbb{F}}
\newcommand{\bE}{\mathbb{E}}
\newcommand{\Q}{\mathbb{Q}}

\newcommand{\Z}{\mathbb{Z}}

\DeclareFontFamily{U}{mathx}{\hyphenchar\font45}
\DeclareFontShape{U}{mathx}{m}{n}{
      <5> <6> <7> <8> <9> <10>
      <10.95> <12> <14.4> <17.28> <20.74> <24.88>
      mathx10
      }{}
\DeclareSymbolFont{mathx}{U}{mathx}{m}{n}
\DeclareFontSubstitution{U}{mathx}{m}{n}
\DeclareMathAccent{\widecheck}{0}{mathx}{"71}

\newcommand{\bG}{\widecheck{\mathbf{G}}}
\newcommand{\bB}{\widecheck{\mathbf{B}}}
\newcommand{\bT}{\widecheck{\mathbf{T}}}

\newcommand{\Ga}{\mathbb{G}_{\mathrm{a}}}

\newcommand{\bY}{\mathbf{Y}}
\newcommand{\Fr}{\mathrm{Fr}}
\newcommand{\Ind}{\mathrm{Ind}}
\newcommand{\weyl}{\mathsf{M}}
\newcommand{\coweyl}{\mathsf{N}}
\newcommand{\cL}{\mathcal{L}}
\newcommand{\Satake}{\mathsf{Sat}}

\newcommand{\fR}{\mathfrak{R}}
\newcommand{\aff}{{\mathrm{aff}}}
\newcommand{\ext}{{\mathrm{ext}}}
\newcommand{\res}{\mathrm{res}}

\newcommand{\scO}{\mathscr{O}}

\newcommand{\IC}{\mathcal{IC}}
\newcommand{\cI}{\mathcal{I}}
\newcommand{\cR}{\mathcal{R}}
\newcommand{\cS}{\mathcal{S}}
\newcommand{\cT}{\mathcal{T}}
\newcommand{\cX}{\mathcal{X}}
\newcommand{\cC}{\mathcal{C}}
\newcommand{\cH}{\mathcal{H}}
\newcommand{\cM}{\mathcal{M}}
\newcommand{\cN}{\mathcal{N}}
\newcommand{\cP}{\mathcal{P}}
\newcommand{\cF}{\mathcal{F}}
\newcommand{\cG}{\mathcal{G}}
\newcommand{\cK}{\mathcal{K}}
\newcommand{\cQ}{\mathcal{Q}}

\newcommand{\DFl}{\mathscr{D}}
\newcommand{\NFl}{\mathscr{N}}
\newcommand{\LFl}{\mathscr{L}}
\newcommand{\TFl}{\mathscr{T}}
\newcommand{\DGr}{\Delta}
\newcommand{\NGr}{\nabla}
\newcommand{\LGr}{\mathsf{L}}
\newcommand{\TGr}{\mathsf{T}}
\newcommand{\PGr}{\mathsf{P}}
\newcommand{\fr}{{\mathrm{fr}}}

\newcommand{\LGT}{\widehat{\mathcal{L}}}
\newcommand{\cbV}{\widehat{\mathcal{Z}}'}
\newcommand{\bV}{\widehat{\mathcal{Z}}}
\newcommand{\injh}{\widehat{\mathcal{Q}}}

\newcommand{\D}{\mathbb{D}}
\newcommand{\Av}{\mathrm{Av}}
\newcommand{\AS}{{\mathrm{AS}}}
\newcommand{\unip}{{\mathrm{u}}}

\newcommand{\uH}{\underline{H}}
\newcommand{\uM}{\underline{M}}

\newcommand{\Gr}{\mathrm{Gr}}
\newcommand{\Fl}{\mathrm{Fl}}

\newcommand{\Rep}{\mathsf{Rep}}
\newcommand{\Coh}{\mathsf{Coh}}
\newcommand{\Perv}{\mathsf{Perv}}
\newcommand{\Mod}{\mathrm{Mod}}
\newcommand{\flen}{{\mathrm{flen}}}
\newcommand{\modf}{\mathrm{mod}}
\newcommand{\Db}{D^{\mathrm{b}}}

\newcommand{\simto}{\overset{\sim}{\to}}
\newcommand{\la}{\langle}
\newcommand{\ra}{\rangle}
\newcommand{\id}{\mathrm{id}}
\newcommand{\Hom}{\mathrm{Hom}}
\newcommand{\End}{\mathrm{End}}
\newcommand{\Ext}{\mathrm{Ext}}
\newcommand{\For}{\mathrm{For}}

\newcommand{\F}{\mathbb{F}}
\newcommand{\bX}{\mathbf{X}}
\newcommand{\fRs}{\mathfrak{R}_{\mathrm{s}}}
\newcommand{\Afund}{\mathfrak{A}_{\mathrm{fund}}}

\newcommand{\pH}{{}^p \hspace{-1pt} \mathcal{H}}

\makeatletter
\def\lotimes{\@ifnextchar_{\@lotimessub}{\@lotimesnosub}}
\def\@lotimessub_#1{\mathchoice{\mathbin{\mathop{\otimes}^L}_{#1}}%
  {\otimes^L_{#1}}{\otimes^L_{#1}}{\otimes^L_{#1}}}
\def\@lotimesnosub{\mathbin{\mathop{\otimes}^L}}
\makeatother

\makeatletter
\def\lboxtimes{\@ifnextchar_{\@lboxtimessub}{\@lboxtimesnosub}}
\def\@lboxtimessub_#1{\mathchoice{\mathbin{\mathop{\boxtimes}^L}_{#1}}%
  {\boxtimes^L_{#1}}{\boxtimes^L_{#1}}{\boxtimes^L_{#1}}}
\def\@lboxtimesnosub{\mathbin{\mathop{\boxtimes}^L}}
\makeatother

\numberwithin{equation}{section}
\newtheorem{thm}{Theorem}[section]
\newtheorem{lem}[thm]{Lemma}
\newtheorem{prop}[thm]{Proposition}
\newtheorem{cor}[thm]{Corollary}

\theoremstyle{definition}

\theoremstyle{remark}
\newtheorem{rmk}[thm]{Remark}

\title[A geometric model for blocks of Frobenius kernels]{A geometric model for blocks of Frobenius kernels}

\author{Pramod N. Achar}
\address{Department of Mathematics\\
  Louisiana State University\\
  Baton Rouge, LA 70803\\
  U.S.A.}
\email{pramod@math.lsu.edu}

\author{Simon Riche}
\address{Universit\'e Clermont Auvergne, CNRS, LMBP, F-63000 Clermont-Ferrand, France.
}
\email{simon.riche@uca.fr}
 
\thanks{P.A. was supported by NSF Grant No.~DMS-1802241.  This project has received
funding from the European Research Council (ERC) under the European Union's Horizon 2020
research and innovation programme (grant agreement No 101002592).
}

\begin{document}

\begin{abstract}
Building on a geometric counterpart of Steinberg's tensor product formula for simple representations of a connected reductive algebraic group $\bG$ over a field of positive characteristic proved in~\cite{projGr1}, and following an idea of Arkhipov--Bezrukavnikov--Braverman--Gaitsgory--Mirkovi{\'c}, we define and initiate the study of some categories of perverse sheaves on the affine Grassmannian of the Langlands dual group to $\bG$ that should provide geometric models for blocks of representations of the Frobenius kernel $\bG_1$ of $\bG$. In particular, we show that these categories admit enough projective and injective objects, which are closely related to some tilting perverse sheaves, and that they are highest weight categories in an appropriate generalized sense.
\end{abstract}

\maketitle

\setcounter{tocdepth}{1}
\tableofcontents

\section{Introduction}
\label{sec:intro}

\subsection{Overview}
\label{ss:intro-overview}

Let $\bk$ be an algebraically closed field of characteristic $\ell > 0$, and let $\bG$ be a connected reductive algebraic group over $\bk$. We also let $G$ be a connected reductive algebraic group over an algebraically closed field of characteristic $p \neq \ell$ such that the Langlands dual group $G^\vee_\bk$ over $\bk$ is the Frobenius twist $\bG^{(1)}$ of $\bG$. The \emph{Finkelberg--Mirkovi\'c conjecture}~\cite{fm} 
predicts\footnote{When this paper first appeared in preprint form, the Finkelberg--Mirkovi\'c conjecture was open, but a proof has recently been obtained (under assumptions slightly stronger than $\ell \geq h$) by Bezrukavnikov and the second author~\cite{br2}.} that when $\ell$ is larger than the Coxeter number $h$ for $\bG$, the category of (\'etale) Iwahori-constructible perverse $\bk$-sheaves on
the affine Grassmannian $\Gr$ of $G$, denoted by $\Perv_{I_\unip}(\Gr,\bk)$, should be equivalent to the (extended) principal block of $\bG$.  See~\cite[\S1.2]{projGr1} for a precise statement and further discussion.
In anticipation of this conjecture, one might look for ``representation-theoretic phenomena'' in $\Perv_{I_\unip}(\Gr,\bk)$,
which might hold with milder assumptions on $\ell$.  The present paper and its companion paper~\cite{projGr1} (where we established a formula for convolution products of certain simple perverse sheaves, modeled on the Steinberg tensor product formula for group representations, without any restriction on $\ell$) both pursue this idea.  

More specifically, the motivation for the present paper is as follows.  Let $\bG_1$ denote the first Frobenius kernel of $\bG$.  The representation theory of $\bG_1$ is, of course, closely related to that of $\bG$: indeed, as illustrated in, say,~\cite[Chap.~II.3]{jantzen}, it is essential to study $\bG_1$-modules even if one is primarily interested in proving results about $\bG$. It is often convenient to also study $\bG_1\bT$-modules, where $\bT \subset \bG$ is a maximal torus.
Our goal in this paper is to construct and study two new abelian categories, to be denoted by $\modf_{I_\unip}(\cR)$ and $\modf_{I_\unip}^\bY(\cR)$, that are related to $\Perv_{I_\unip}(\Gr,\bk)$ in the same way that $\bG_1$-representations and $\bG_1\bT$-representations, respectively, are related to $\bG$-representations.   

\subsection{Main results}
\label{ss:intro-main}


For most of the paper, we can work with more general fields than those considered in~\S\ref{ss:intro-overview}.  Let $p$ and $\ell$ be distinct primes, and let $\bk$ be a field of one of the following kinds: 
\begin{enumerate}
\item a finite extension of $\Q_\ell$ containing a nontrivial $p$-th root of unity;
\item an algebraic closure of $\Q_\ell$;
\item a finite extension of $\F_\ell$ containing a nontrivial $p$-th root of unity;
\item an algebraic closure of $\F_\ell$.
\end{enumerate}
We will fix a maximal torus $T$ and a Borel subgroup $B$ in $G$ such that $T \subset B$. For technical reasons, we will assume that the quotient of the character lattice of $T$ by the root lattice is torsion-free.
(This condition holds, for instance, if $G$ is semisimple of adjoint type.)
Our constructions will make use of ind-objects in categories; we recall in~\S\ref{ss:ind} the basic facts on this construction that we will use, with references to~\cite{ks}.

Let $\bY$ be the cocharacter lattice of $T$.  The category $\modf_{I_\unip}^\bY(\cR)$ consists of certain $\bY$-graded ind-objects in $\Perv_{I_\unip}(\Gr,\bk)$ equipped with an action of an algebra ind-object denoted by $\cR$ (on the right). Here $\cR$ is the ind-object in the Satake category $\Perv_{\cL^+G}(\Gr,\bk)$ of perverse sheaves on $\Gr$ which are equivariant with respect to the action of the positive loop group $\cL^+G$ that corresponds to the regular representation $\scO(\bG^{(1)})$ under the geometric Satake equivalence discussed in~\S\ref{ss:intro-FM-G1T} below; this category acts on $\Perv_{I_\unip}(\Gr,\bk)$ by convolution on the right, so that it makes sense to consider $\cR$-modules in ind-objects in this category. The precise definition will be given in Section~\ref{sec:mod-R}.  For now, we remark that there is an easy way to take an ordinary perverse sheaf $\cF \in \Perv_{I_\unip}(\Gr,\bk)$ and produce an ind-perverse sheaf $\Phi(\cF) \in \modf_{I_\unip}^\bY(\cR)$ by taking the ``free $\cR$-module on $\cF$.'' This construction yields a functor
\[
\Phi: \Perv_{I_\unip}(\Gr,\bk) \to \modf_{I_\unip}^\bY(\cR).
\]

Let $W$ be the Weyl group of $G$, and let $W_\ext := W \ltimes \bY$ be the extended affine Weyl group.  The following statement gathers some of the main results of this paper (see Theorems~\ref{thm:R-simple} and~\ref{thm:proj-Hecke-Whit} and Propositions~\ref{prop:bV-injective},~\ref{prop:bV-projective},~\ref{prop:inj-Verma-filt} and~\ref{prop:reciprocity}).  (The partial order $\preceq$ appearing in~\eqref{it:intro-thm-3} is Lusztig's ``periodic order'' on $W_\ext$, whose definition is recalled in~\S\ref{ss:periodic-order}. In~\eqref{it:intro-thm-2}, an element of $W_\ext$ is called \emph{restricted} if it sends the fundamental alcove to an alcove in the restricted region; see~\cite[\S 2.4]{projGr1} for a precise definition.)

\begin{thm}
\phantomsection
\label{thm:main-intro}
\begin{enumerate}
\item 
\label{it:intro-thm-1}
The category $\modf_{I_\unip}^\bY(\cR)$ is a finite-length abelian category.
\item 
\label{it:intro-thm-2}
For each $w \in W_\ext$, there is a simple object $\LGT_w \in \modf_{I_\unip}^\bY(\cR)$, and the assignment $w \mapsto \LGT_w$ yields a bijection
\[
W_\ext \simto \{\text{isomorphism classes of simple objects in $\modf_{I_\unip}^\bY(\cR)$} \}.
\]
If $w$ is restricted then $\LGT_w$ is the image under $\Phi$ of the simple object in $\Perv_{I_\unip}(\Gr,\bk)$ labeled by $w$.
\item 
\label{it:intro-thm-3}
In the Serre subcategory of $\modf_{I_\unip}^\bY(\cR)$ generated by the objects $\LGT_y$ with $y \not\succ w$, $\LGT_w$ admits a projective cover $\bV_w$ and an injective hull $\cbV_w$.
\item 
\label{it:intro-thm-4}
For each $w \in W_\ext$, there is an object $\injh_w$ in $\modf_{I_\unip}^\bY(\cR)$ that is both the injective hull and projective cover of $\LGT_w$.  Moreover, $\injh_w$ admits a filtration with subquotients of the form $\bV_y$, and a filtration with subquotients of the form $\cbV_y$.
\end{enumerate}
\end{thm}

The objects $\bV_w$ and $\cbV_w$ appearing in~\eqref{it:intro-thm-3} are geometric incarnations of baby Verma and baby co-Verma $\bG_1\bT$-modules.
All of the properties in Theorem~\ref{thm:main-intro} are geometric counterparts of standard results on $\bG_1\bT$-modules using the dictionary explained in~\S\ref{ss:intro-FM-G1T} below. For instance, see~\cite[Proposition~II.9.6]{jantzen} for~\eqref{it:intro-thm-2} and~\cite[Proposition~II.11.4 and~\S II.11.5]{jantzen} for~\eqref{it:intro-thm-4}. (Note that we do not impose any restriction on $\ell$ in this theorem, although the conjectural translation to Representation Theory requires $\ell$ to be larger than the Coxeter number.)  

We will also prove similar results for the category $\modf_{I_\unip}(\cR)$ of \emph{non graded} $\cR$-modules that we do not state here, see Section~\ref{sec:ungrad-R-mod}.

\subsection{Some comments on Theorem~\ref{thm:main-intro}}

In this subsection we make a few comments on Theorem~\ref{thm:main-intro}, that are not used in the rest of the introduction.

The broad form of Theorem~\ref{thm:main-intro} is motivated by the following observation: if we replace ``finitely generated $\bY$-graded $\cR$-modules in ind-perverse sheaves on $\Gr$'' by ``finitely generated $\bY$-graded modules over the coordinate ring of $\bG^{(1)}$ in ind-finite-dimensional $\bG$-representations,'' we arrive at a category that is equivalent to the category of finite-dimensional $\bG_1\bT$-representations: see Section~\ref{sec:mod-R-rep} for a precise statement.

In the body of the paper we will define \emph{two} versions of the category $\modf_{I_\unip}^\bY(\cR)$, based on two different approaches to interpreting the phrase ``finitely generated'' in geometric terms.  Proving that these two definitions in fact give rise to the same category (see Theorem~\ref{thm:proj-Hecke-Whit}\eqref{it:proj-enough}) is what will require most of our efforts. The construction and study of projective and injective objects is essential to our approach to this question. This study also proves in passing that $\modf_{I_\unip}^\bY(\cR)$ is a highest weight category in a generalized sense recently formulated by Brundan--Stroppel~\cite{bs}, which implies the second sentence in Theorem~\ref{thm:main-intro}\eqref{it:intro-thm-4}.

We will actually prove a more general version of Theorem~\ref{thm:main-intro} that accommodates ``Whittaker perverse sheaves'' on $\Gr$, rather than just Iwahori-constructible perverse sheaves.  The Whittaker versions of this theorem are not merely generalizations for their own sake: our proof of Theorem~\ref{thm:main-intro}, even in the simplified case stated above, makes crucial use of functors that allow us to pass to and from various Whittaker versions.

Another key tool in the proof of part~\eqref{it:intro-thm-4} is the ``Iwahori--Whittaker model for the Satake category'' of~\cite{bgmrr}.  The counterpart of $\modf^{\bY}_{I_\unip}(\cR)$ in the setting of~\cite{bgmrr} turns out to be equivalent to the category of finite-dimensional representations of the torus $\bT$: in particular, it is a semisimple category, and thus has a rich supply of projective and injective objects, which give rise to projective and injective objects in our other categories via appropriate averaging functors.

\begin{rmk}
\begin{enumerate}
\item
In the case $\bk = \Q_\ell$, parts~\eqref{it:intro-thm-1}--\eqref{it:intro-thm-3} of Theorem~\ref{thm:main-intro} were previously obtained by Arkhipov--Bezrukavnikov--Braverman--Gaitsgory--Mir\-kovi{\'c}~\cite{abbgm}.
(Indeed, our definitions of the category $\modf^{\bY}_{I_\unip}(\cR)$ and of the objects $\LGT_w$, $\cbV_w$ and $\bV_w$ are essentially copied from~\cite{abbgm}.)  Thus, for $\bk = \Q_\ell$, the main new contribution of the present paper is the study of projective/injective objects in $\modf^\bY_{I_\unip}(\cR)$.
\item
Let us mention in passing that, in case $\bk$ has characteristic $0$, our methods also provide a geometric proof that the abelian category $\Perv_{I_\unip}(\Gr,\bk)$ has enough projectives and injectives, and that these two classes of objects coincide: see~\S\ref{ss:proj-covers}.
(In contrast, for $\bk$ of positive characteristic, $\Perv_{I_\unip}(\Gr,\bk)$ has no nonzero projective or injective objects unless $G$ is a torus.) This fact was previously known: it can be deduced from a representation-theoretic result of Andersen--Polo--Wen~\cite{apw} via intermediaries discussed in Remark~\ref{rmk:intro-quantum} below.  The problem of finding  a geometric proof of this property was in fact the starting point of this work. It finally allows us to answer a question that we asked ourselves (and a few colleagues) more than ten years ago.
\end{enumerate}
\end{rmk}

\subsection{A Finkelberg--Mirkovi{\'c} conjecture for the Frobenius kernel}
\label{ss:intro-FM-G1T}

In this subsection, we make precise the expected relationship between $\modf_{I_\unip}^\bY(\cR)$ and the category of $\bG_1\bT$-representations. (These considerations have obvious analogues relating the ``Whittaker'' variants of $\modf_{I_\unip}^\bY(\cR)$ considered in the body of the paper to singular blocks of $\bG_1\bT$-representations; we leave it to the reader to formulate these variants.) This subsection is for motivation only; it does not play any logical role in the rest of the paper.

We assume in this subsection that $\bk$ is an algebraic closure of $\F_\ell$.
Let $\Rep(\bG)$, resp.~$\Rep(\bG_1\bT)$, be the category of finite-dimensional rational $\bG$-, resp.~$\bG_1\bT$-, representations.  Then there is a forgetful functor
\[
\For: \Rep(\bG) \to \Rep(\bG_1\bT).
\]
Here is a brief review of the representation theory of $\bG_1\bT$, following, for instance~\cite[Chap.~II.9--11]{jantzen}. For each $\lambda \in \bY$, there is a \emph{baby Verma module} $\widehat{\mathsf{Z}}(\lambda)$ and a \emph{baby co-Verma module} $\widehat{\mathsf{Z}}'(\lambda)$, both with highest weight $\lambda$.  The socle of $\widehat{\mathsf{Z}}'(\lambda)$ can be identified with the head of $\widehat{\mathsf{Z}}(\lambda)$, and this irreducible module is denoted by $\widehat{\mathsf{L}}(\lambda)$.  Every simple $\bG_1\bT$-module is of this form (for a unique $\lambda \in \bY$).  The module $\widehat{\mathsf{L}}(\lambda)$ admits an injective hull $\widehat{\mathsf{Q}}(\lambda)$ that is also its projective cover.
Moreover, $\widehat{\mathsf{Q}}(\lambda)$ admits filtrations by both baby Verma modules and baby co-Verma modules.

Assume now that $\ell \geq h$, where $h$ is the Coxeter number for $\bG$, and that the quotient of the cocharacter lattice of $T$ by the coroot lattice has no $\ell$-torsion. (See~\cite[Remark~1.2(2)]{projGr1} for a discussion of this assumption.)
Let $\Rep_{[0]}(\bG) \subset \Rep(\bG)$ and $\Rep_{[0]}(\bG_1\bT) \subset \Rep(\bG_1\bT)$ be the \emph{extended principal blocks} of $\bG$ and of $\bG_1\bT$, respectively, i.e.~the Serre subcategories generated by simple modules whose highest weights lie in $W_\ext \cdot_\ell 0$, where 
$\cdot_\ell$ denotes the ``$\ell$-dilated dot action.'' 
The Steinberg tensor product formula implies that the forgetful functor $\For : \Rep(\bG) \to \Rep(\bG_1\bT)$ restricts to a functor
\begin{equation*}
\For_0 : \Rep_{[0]}(\bG) \to \Rep_{[0]}(\bG_1\bT).
\end{equation*}
Tensor product with the pullback of representations along the Frobenius morphism $\Fr: \bG \to \bG^{(1)}$ induces actions of the category $\Rep(\bG^{(1)})$ on both of these categories, and $\For_0$ commutes with these actions.

Consider now the Satake category $\Perv_{\cL^+G}(\Gr,\bk)$, its convolution product $\star^{\cL^+G}$, and the geometric Satake equivalence
\[
 \Satake : \Perv_{\cL^+G}(\Gr,\bk) \simto \Rep(\bG^{(1)})
\]
of~\cite{mv}. As mentioned in~\S\ref{ss:intro-main} there is a natural action of the monoidal category $\Perv_{\cL^+G}(\Gr,\bk)$ on $\Perv_{I_\unip}(\Gr,\bk)$ by convolution on the right; the corresponding bifunctor will also be denoted $\star^{\cL^+G}$.
Recall that the \emph{Finkelberg--Mirkovi\'c conjecture} asserts the existence of an equivalence of categories
\[
\mathsf{FM}: \Perv_{I_\unip}(\Gr,\bk) \simto \Rep_{[0]}(\bG)
\]
together with a natural isomorphism
\[
\mathsf{FM}(\cF \star^{\cL^+G} \cG) \cong \mathsf{FM}(\cF) \otimes \Fr^*  \bigl( \Satake(\mathrm{sw}( \cG)) \bigr)
\]
for $\cF \in \Perv_{I_\unip}(\Gr,\bk)$ and $\cG \in \Perv_{\cL^+G}(\Gr,\bk)$, where 
$\mathrm{sw}$ is a certain autoequivalence of $\Perv_{\cL^+G}(\Gr,\bk)$ induced by the inversion map in the loop group. We refer the reader to~\cite[\S1.2]{projGr1} for more precise definitions of the notation, and further discussion of this statement.

Since the convolution product in the Satake category is commutative, convolution on the right induces an action of the category $\Perv_{\cL^+G}(\Gr,\bk)$ on $\modf^\bY_{I_\unip}(\cR)$; the corresponding bifunctor will once again be denoted $\star^{\cL^+G}$.
The following statement is a consequence of the results of this paper.

\begin{prop}
\label{prop:frobenius-fm}
Assume $\bk$ is an algebraic closure of $\F_\ell$, and that the Finkelberg--Mirkovi\'c conjecture holds for $\bG$. Then there exists an equivalence of categories
\[
\mathsf{FM}_{\mathrm{Frob}} : \modf^\bY_{I_\unip}(\cR) \simto \Rep_{[0]}(\bG_1\bT)
\]
such that
\[
\mathsf{FM}_{\mathrm{Frob}} \circ \Phi \cong \mathsf{FM} \circ \For_0,
\]
which satisfies
\begin{gather*}
\mathsf{FM}_{\mathrm{Frob}}(\LGT_w) \cong \widehat{\mathsf{L}}(w^{-1} \cdot_\ell 0), \quad \mathsf{FM}_{\mathrm{Frob}}(\injh_w) \cong \widehat{\mathsf{Q}}(w^{-1} \cdot_\ell 0), \\
\mathsf{FM}_{\mathrm{Frob}}(\bV_w) \cong \widehat{\mathsf{Z}}(w^{-1} \cdot_\ell 0), \quad \mathsf{FM}_{\mathrm{Frob}}(\cbV_w) \cong \widehat{\mathsf{Z}}'(w^{-1} \cdot_\ell 0)
\end{gather*}
for any $w \in W_\ext$, and such that
\[
\mathsf{FM}_{\mathrm{Frob}}(\cF \star^{\cL^+G} \cG) \cong \mathsf{FM}_{\mathrm{Frob}}(\cF) \otimes \Fr^*  \bigl( \Satake(\mathrm{sw}^* \cG) \bigr)
\]
functorially for any $\cF$ in $\modf^{\bY}_{I_\unip}(\cR)$ and $\cG$ in $\Perv_{\cL^+G}(\Gr,\bk)$.
\end{prop}

A brief sketch of the proof is as follows: the equivalence $\Satake$ extends to an equivalence on ind-objects, and this extension sends $\cR$ to the $\bG^{(1)}$-representation $\scO(\bG^{(1)})$ by construction.  In Section~\ref{sec:mod-R-rep}, we will review how to describe the category $\Rep_{[0]}(\bG_1\bT)$ in terms ``finitely generated $\bY$-graded $\scO(\bG^{(1)})$-modules on ind-objects in $\Rep_{[0]}(\bG_1)$.''  The transport of this description across $\mathsf{FM}$ is precisely the definition of $ \modf^\bY_{I_\unip}(\cR)$.

\begin{rmk}
\label{rmk:intro-quantum}
For $\bk$ of characteristic $0$, the categories $\Perv_{I_\unip}(\Gr,\bk)$ and $\modf^\bY_{I_\unip}(\cR)$ are \emph{not} related to representations of $\bG$ or $\bG_1$, but rather to their quantum analogues.  Specifically, the quantum counterpart of the Finkelberg--Mirkovi\'c conjecture is a theorem of Arkhipov--Bezrukavnikov--Ginzburg~\cite{abg} relating $\Perv_{I_\unip}(\Gr,\bk)$ to the principal block of a quantum group $\mathsf{U}_\zeta(\widecheck{\mathfrak{g}})$ at a root of unity.
As observed in~\cite{abbgm}, a quantum analogue of Proposition~\ref{prop:frobenius-fm} holds in this setting: for $\bk = \Q_\ell$, the category $\modf^\bY_{I_\unip}(\cR)$ is equivalent to the principal block of graded representations of the \emph{small} quantum group $\mathsf{u}_\zeta(\widecheck{\mathfrak{g}})$.

Note that in~\cite{abbgm} the authors provide a third incarnation of the same category, in terms of perverse sheaves on a semi-infinite affine flag variety. It is likely that a similar description can be obtained in our setting of positive-characteristic coefficients; this question will be the subject of future work.
\end{rmk}

\subsection{Tilting \texorpdfstring{$\bG$}{G}-modules and projective \texorpdfstring{$\bG_1\bT$}{G1T}-modules}
\label{ss:intro-TMC}

In view of the discussion in~\S\ref{ss:intro-FM-G1T},
the second sentence in Theorem~\ref{thm:main-intro}\eqref{it:intro-thm-2} is a geometric counterpart of the fact that simple $\bG$-modules with restricted highest weight remain simple as $\bG_1\bT$-modules. There is another class of $\bG$-modules whose behaviour upon restriction to $\bG_1\bT$ is remarkable, namely the \emph{tilting} modules. The geometric counterpart of tilting modules in the principal block of $\bG$ are the tilting perverse sheaves in $\Perv_{I_\unip}(\Gr,\bk)$. 

In Propositions~\ref{prop:tilt-proj-inj-R} and~\ref{prop:injective-tilting} we determine which indecomposable tilting perverse sheaves are sent to projective/injective $\cR$-modules by the functor $\Phi$, providing a geometric counterpart of~\cite[Lemma~E.8]{jantzen}.
Another very interesting property of this operation (long known to hold if $p \geq 2h-2$, recently improved to $p \geq 2h-4$~\cite{bnps2}, and conjecturally in greater generality)
 is that
  some of these indecomposable tilting modules remain indecomposable as $\bG_1\bT$-modules. It is one of our motivations for developing this theory to provide new tools to understand this property. We study this question in~\S\ref{ss:Donkin-conj}, but we must admit that our progress towards solving this question is very modest so far.

\subsection{Contents of the paper}

In Section~\ref{sec:combinatorics} we prove a number of combinatorial results on the affine Weyl group attached to a connected reductive algebraic group $G$. In Section~\ref{sec:perv} we collect some facts on various categories of perverse sheaves on the affine Grassmannian $\Gr$ and the affine flag variety $\Fl$ of $G$. Most of these results are known to some extent, but many proofs are not available in the literature in the generality we require, so that we discuss these proofs to some extend. In Section~\ref{sec:mod-R-rep} we explain some constructions that allow us to describe modules over the Frobenius kernel (or some variants) of a connected reductive algebraic group $\bG$ in terms of representations of the whole group. These constructions will serve as guiding principles for many constructions in the rest of the paper, and justify Proposition~\ref{prop:frobenius-fm}.

In Section~\ref{sec:mod-R} we introduce our main object of study, the category $\modf_{I_\unip}^\bY(\cR)$, together with a variant $\Mod^\bY_{I_\unip}(\cR)^\flen$. We show that the second of these categories is a finite-length abelian category, in which we classify the simple objects, and define some geometric incarnations of baby co-Verma modules. (We also treat some ``Whittaker-type'' analogues in parallel.) In Section~\ref{sec:averaging}, exploiting results from~\cite{bgmrr} we study some perverse sheaves on $\Gr$ arising from the ``big tilting perverse sheaf'' on the flag variety of $G$, and derive some first applications to the study of (geometric) baby co-Verma modules. In Section~\ref{sec:proj-R-mod} we prove that the categories $\modf_{I_\unip}^\bY(\cR)$ and $\Mod^\bY_{I_\unip}(\cR)^\flen$ coincide, that these categories have enough injectives and enough projectives, and also that these classes of objects coincide and are closely related to tilting perverse sheaves on $\Gr$. (Again, all of these results are proved also in the Whittaker setting). In the course of the proof of these results, we show that these categories satisfy the ``generalized highest weight'' formalism of Brundan--Stroppel~\cite{bs}.

In Section~\ref{sec:ungrad-R-mod} we study a variant of our formalism that omits part of the structure. Finally, in Section~\ref{sec:baby-verma} we define a duality functor on $\modf_{I_\unip}^\bY(\cR)$, and use this functor to define geometric counterparts of baby Verma modules. We also prove a number of results regarding the combinatorics of the category $\modf_{I_\unip}^\bY(\cR)$ that are analogues of known results on representations of Frobenius kernels.

\subsection{Conventions on fields}

The notation $\bk$ is used in this paper as follows:
\begin{itemize}
\item In sections devoted to motivation from or analogies with algebraic groups,  $\bk$ is an algebraically closed field of characteristic $\ell$.  This applies to \S\ref{ss:intro-overview}, \S\S\ref{ss:intro-FM-G1T}--\ref{ss:intro-TMC}, and all of Section~\ref{sec:mod-R-rep}.
\item In all other parts of the paper,  $\bk$ may be any field satisfying the assumptions from~\S\ref{ss:intro-main}.  
\end{itemize}
We will denote by $\kappa$ an algebraically closed field that will be the ground field over which we define the affine Grassmannian and the affine flag variety.    Starting from~\S\ref{ss:Whittaker}, $\kappa$ is assumed to have positive characteristic.  Finally, we will write $\bK$ and $\bF$ for parts of an $\ell$-modular system in~\S\ref{ss:integral} and part of~\S\ref{ss:more-on-inj}.

\section{Combinatorics of the affine Weyl group}
\label{sec:combinatorics}

\subsection{The extended affine affine Weyl group}
\label{ss:Wext}

Let $\kappa$ be an algebraically closed field, and $G$ be a connected reductive algebraic group over $\kappa$. We fix a Borel subgroup $B \subset G$ and a maximal torus $T \subset B$. We will denote by $\bX:=X^*(T)$ the character lattice of $T$, by $\fR \subset \bX$ the root system of $(G,T)$, by $\bY:=X_*(T)$ the coweight lattice, and by $\fR^\vee \subset \bY$ the coroot system; the natural bijection from $\fR$ to $\fR^\vee$ will be denoted $\alpha \mapsto \alpha^\vee$ as usual. 

We will denote by $\fR_+ \subset \fR$ the system of positive roots consisting of the $T$-weights in $\mathrm{Lie}(G)/\mathrm{Lie}(B)$, and by $\fRs$ the associated basis of $\fR$. 
The corresponding sets of dominant coweights and strictly dominant coweights will be denoted $\bY_+$ and $\bY_{++}$ respectively. We will denote by $W$ the Weyl group of $(G,T)$. If we denote by $S \subset W$ the subset consisting of the reflections $s_{\alpha^\vee}$ for $\alpha \in \fRs$, then it is well known that $(W,S)$ is a Coxeter system. The longest element in this group will be denoted $w_\circ$.

We will assume that $\bX/\Z\fR$ has no torsion, or in other words that the restriction morphism
\[
\bY \to \Hom_\Z(\Z\fR, \Z)
\]
is surjective.  (This is equivalent to requiring that the scheme-theoretic center of $G$ be a torus.)
In particular, this condition ensures that there exists $\varsigma \in \bY$ such that $\langle \alpha,\varsigma \rangle = 1$ for all $\alpha \in \fR_{\mathrm{s}}$; we fix such an element once and for all.

The affine Weyl group associated with $G$
is the semidirect product
\[
W_\aff := W \ltimes \Z\fR^\vee,
\]
where $\Z\fR^\vee \subset \bY$ is the lattice generated by $\fR^\vee$. For $\lambda \in \Z\fR^\vee$, we will write $t_\lambda$ for the corresponding element of $W_\aff$. It is a standard fact that if we denote by $S_\aff \subset W_\aff$ the subset consisting of $S$ together with the elements $t_{\beta^\vee} s_{\beta^\vee}$ where $\beta^\vee \in \fR^\vee$ is a maximal short coroot, then the pair $(W_\aff,S_\aff)$ is a Coxeter system. Moreover, classical results of Iwahori--Matsumoto~\cite{im} show that the associated length function on $W_\aff$ can be described by the following formula for $w \in W$ and $\lambda \in \Z\fR^\vee$:
\begin{equation}
\label{eqn:formula-length}
\ell(w t_\lambda) = \sum_{\substack{\alpha \in \fR_+ \\ w(\alpha) \in \fR_+}} |\langle \lambda,\alpha \rangle| + \sum_{\substack{\alpha \in \fR_+ \\ w(\alpha) \in -\fR_+}} |1+\langle \lambda,\alpha \rangle|.
\end{equation}

The formula on the right-hand side of~\eqref{eqn:formula-length} makes sense more generally for $\lambda \in \bY$, which lets one to extend the function $\ell$ to the larger group
\[
W_\ext := W \ltimes \bY,
\]
in such a way that $\ell(ww') \leq \ell(w)+\ell(w')$ for any $w,w' \in W_\ext$.
The subgroup $W_\aff \subset W_\ext$ is normal, and if we set
\[
\Omega:=\{w \in W_\ext \mid \ell(w)=0\}
\]
then $\Omega$ is a finitely generated abelian group acting on $W_\aff$ (via conjugation) by Coxeter group automorphisms. Multiplication induces a group isomorphism
\[
\Omega \ltimes W_\aff \simto W_\ext,
\]
and $\ell(\omega w)=\ell(w \omega)=\ell(w)$ for any $w \in W_\ext$ and $\omega \in \Omega$. We can also ``extend'' the Bruhat order $\leq$ on $W_\aff$ to $W_\ext$ by declaring that for $\omega,\omega' \in \Omega$ and $w,w' \in W_\aff$ we have $\omega w \leq \omega' w'$ iff $\omega=\omega'$ and $w \leq w'$. (The same rule will then also apply when switching the order of $\omega$ and $w$.) We define a \emph{reduced expression} for an element $w \in W_\ext$ to be an expression of the form $w=s_1 \cdots s_r \omega$ or $w=\omega s_1 \cdots s_r$ with $\omega \in \Omega$, $s_i \in S_\aff$ for any $i \in \{1, \ldots, r\}$, and $r=\ell(w)$.

Given a subset $A \subset S_\aff$ , we will denote by $W_A$ the subgroup of $W_\aff$ generated by $A$. We will say that $A$ is \emph{finitary} if $W_A$ is finite; in this case 
we will denote by $w_A$ the longest element in $W_A$. If $A$ is finitary, the theory of Coxeter systems guarantees that for any $w \in W_\ext$, the coset $W_A w$, resp.~$wW_A$, admits a unique maximal, resp.~minimal, element with respect to the Bruhat order. In particular, for $A=S$, we will denote by $W_\ext^S \subset W_\ext$ the subset consisting of elements $w$ which are minimal in $wW$. The basic properties of minimal coset representatives recalled above guarantee that the composition $W_\ext^S \hookrightarrow W_\ext \to W_\ext / W$ is a bijection.

\subsection{Geometry of alcoves and restricted elements}
\label{ss:alcoves}

Consider the vector space $V:=\bY \otimes_\Z \mathbb{R}$, and the action of $W_\ext$ given by
\[
(t_\lambda w) \cdot v = w(v) + \lambda
\]
for $w \in W$ and $\lambda \in \bY$, where $W$ acts on $V$ via its natural action on $\bY$. In $V$ we have the affine hyperplanes defined by
\[
H_{\beta,n} := \{v \in V \mid \langle \beta,v \rangle = n\}
\]
for $\beta \in \fR$ and $n \in \Z$, which are permuted by the action of $W_\ext$. The connected components of the complement of the union of these hyperplanes are called alcoves; if we set
\[
\Afund := \{v \in V \mid \forall \beta \in \fR_+, \, 0 < \langle \beta,v \rangle < 1\},
\]
then $\Afund$ is an alcove (called the \emph{fundamental alcove}), and the assignment $w \mapsto w(\Afund)$ induces a bijection
from $W_\ext / \Omega$
(where $\Omega$ is as in~\S\ref{ss:Wext}) to the set of alcoves. If
\[
\mathcal{C} = \{v \in V \mid \forall \beta \in \fR_+, \, \langle \beta,v \rangle > 0\},
\]
then it is a standard fact that
\begin{equation}
\label{eqn:min-Wext-Afund}
W_\ext^S = \{w \in W_\ext \mid w^{-1}(\Afund) \subset \mathcal{C}\}.
\end{equation}

For $\mu \in \bY$ we set
\[
\Pi_\mu := \{v \in V \mid \forall \alpha \in \fR_{\mathrm{s}}, \, \langle \alpha,\mu \rangle -1 < \langle \alpha, v \rangle < \langle \alpha,\mu \rangle \};
\]
our assumption on $\bX/\Z\fR$ ensures that each alcove is contained in a subset of this form (sometimes called a \emph{box}). Of particular importance is the set
\[
\Pi_\varsigma = \{v \in V \mid \forall \alpha \in \fR_{\mathrm{s}}, \, 0 < \langle \alpha, v \rangle < 1 \}.
\]
This set (which is evidently independent of the choice of $\varsigma$) contains $\Afund$ and is sometimes called the \emph{fundamental box}.  We define the subset of \emph{restricted elements} in $W_\ext$ by setting
\[
W_\ext^\res := \{w \in W_\ext \mid w^{-1}(\Afund) \subset \Pi_\varsigma\}.
\]
Since any alcove belongs to a subset $\Pi_\mu$, any element $w$ of $W_\ext$ can be written as a product $w=yt_\lambda$ with $y \in W_\ext^\res$ and $\lambda \in \bY$. It is easy to see that
\begin{equation}
\label{eqn:WS-Wres}
W_\ext^S = \{x t_{\lambda} : x \in W_\ext^\res, \, \lambda \in -\bY_+\},
\end{equation}
see~\cite[\S 2.4]{projGr1} for details.

Let us also record the following property, proved in~\cite[Lemma~2.7]{projGr1}, which shows in particular that lengths always add in a decomposition given by~\eqref{eqn:WS-Wres}.

\begin{lem}
\label{lem:lengths-add-Wres}
For any $w \in W_\ext^S$ and $\lambda \in -\bY_+$ we have $\ell(wt_\lambda)=\ell(w)+\ell(t_\lambda)$.
\end{lem}

As explained above, given $w \in W_\ext$ there exists $\mu \in \bY$ such that $w^{-1}(\Afund) \subset \Pi_\mu$. We then set
\[
w^{\triangle} = w t_\mu w_\circ t_{-\mu}.
\]
(Here $\mu$ is unique only up to addition of a coweight $\nu$ orthogonal to all roots; however the product $t_\mu w_\circ t_{-\mu}$ is independent of the choice of $\mu$, so that this definition makes sense.) 
This definition is chosen in such a way that $w^\triangle(\Afund) = \widehat{w(\Afund)}$, where the operation on alcoves $A \mapsto \widehat{A}$ is as in~\cite[p.~98]{soergel} (see also~\cite[\S 2.2]{rw-simple}).
It is easily seen from the definition that if $w \in W_\ext^S$ then $w^\triangle \in  W_\ext^S$, and that for $w \in W_\ext$ and $\lambda \in \bY$ we have
\begin{equation}
\label{eqn:triangle-translation}
(w t_\lambda)^\triangle = w^\triangle t_\lambda.
\end{equation}

Note that $(w^\triangle)^{-1}(\Afund) \subset \Pi_{\mu + \varsigma}$.  Using this observation, we can write down the inverse of the map $w \mapsto w^\triangle$:  it is given by
\begin{equation}\label{eqn:triangle-inverse}
v \mapsto v t_{\nu - \varsigma} w_\circ t_{\varsigma - \nu}
\qquad\text{where $v^{-1}(\Afund) \subset \Pi_\nu$.}
\end{equation}

\subsection{A length computation}

\begin{lem}
\label{lem:res-complement}
If $x \in W_\ext^\res$ and $y= t_{\varsigma} w_\circ x^{-1}$, then we have
\[
\ell(x)+\ell(y)=\ell(t_{\varsigma} w_\circ).
\]
\end{lem}

\begin{proof}
Using~~\eqref{eqn:formula-length} we see that $\ell(t_{\varsigma} w_\circ) = \langle 2\rho, \varsigma \rangle - \ell(w_\circ)$. On the other hand, write $x= w t_\lambda$ with $\lambda \in \bY$ and $w \in W$. Then we have
\[
\ell(x) = \sum_{\substack{\alpha \in \fR_+ \\ w(\alpha) \in \fR_+}} |\langle \alpha,\lambda \rangle| + \sum_{\substack{\alpha \in \fR_+ \\ w(\alpha) \in -\fR_+}} |1+\langle \alpha, \lambda \rangle|.
\]
By \cite[Lemma~2.6]{projGr1},
on the right-hand side we have $\langle \alpha, \lambda \rangle \leq 0$ for any $\alpha \in \fR_+$.  Moreover, if $w(\alpha) \in -\fR_+$, then at least one simple root $\gamma$ appearing in the decomposition of $\alpha$ as a sum of simple roots must satisfy $w(\gamma) \in -\fR_+$; we therefore have $\langle \alpha, \lambda \rangle \leq -1$ in this case. We deduce that
\[
\ell(x) = -\langle 2\rho, \lambda \rangle - \ell(w).
\]
Similarly we have $y= t_{\varsigma} w_\circ x^{-1}=t_{\varsigma-w_\circ(\lambda)} w_\circ w^{-1}=(w w_\circ t_{w_\circ(\lambda)-\varsigma})^{-1}$, and hence
\[
\ell(y)=\sum_{\substack{\alpha \in \fR_+ \\ w w_\circ (\alpha) \in \fR_+}} |\langle \alpha,w_\circ(\lambda)-\varsigma\rangle| + \sum_{\substack{\alpha \in \fR_+ \\ w w_\circ(\alpha) \in -\fR_+}} |1+\langle \alpha,w_\circ(\lambda)-\varsigma \rangle|.
\]
Setting $\beta=-w_\circ(\alpha)$ we obtain that
\[
\ell(y)=\sum_{\substack{\beta \in \fR_+ \\ w (\beta) \in -\fR_+}} |\langle \beta,\lambda-w_\circ(\varsigma)\rangle| + \sum_{\substack{\beta \in \fR_+ \\ w (\beta) \in \fR_+}} |1-\langle \beta,\lambda-w_\circ(\varsigma) \rangle|.
\]
For the same reason as before, 
we have $\langle \beta,\lambda-w_\circ(\varsigma) \rangle \geq 0$ for any $\beta \in \fR_+$, and $\langle \beta,\lambda-w_\circ(\varsigma) \rangle \geq 1$ if $w(\beta) \in \fR_+$. It follows that
\begin{multline*}
\ell(y)= \langle 2\rho, \varsigma \rangle + \langle 2\rho,\lambda \rangle - \#\{\beta \in \fR_+ \mid w(\beta) \in \fR_+\}  = \\
\langle 2\rho, \varsigma \rangle + \langle 2\rho,\lambda \rangle - \ell(w_\circ w) = \ell( t_{\varsigma} w_\circ) - \ell(x),
\end{multline*}
as stated in the lemma.
\end{proof}

\subsection{Coset representatives}
\label{ss:coset-representatives}

Let $A \subset S_\aff$ be a finitary subset. We will denote by ${}^A W_\ext^S \subset W_\ext$ the subset consisting of the elements $w$ such that $\ell(w_A w w_\circ)=\ell(w_A)+\ell(w)+\ell(w_\circ)$. Other characterizations of these elements are given in~\cite[Lemma~2.4]{projGr1}; in particular we have
\begin{equation}
\label{eqn:AWextS-characterization}
w \in {}^A W_\ext^S \quad \Leftrightarrow \quad \text{$w$ is minimal in $W_A w$ and $vw \in W_\ext^S$ for any $v \in W_A$.}
\end{equation}
(Of course, this shows that ${}^A W_\ext^S \subset W_\ext^S$.)
We set
\[
{}^A W_\ext^\res := {}^A W_\ext^S \cap W_\ext^\res.
\]
Then as explained in~\cite[\S 2.5]{projGr1} we have
\begin{equation}
\label{eqn:WS-Wres-Whit}
{}^A W^S_\ext = \{wt_\lambda : w \in {}^A W^\res_\ext, \, \lambda \in -\bY_+ \}.
\end{equation}

We now set
\[
{}^A W_\ext = \{wt_\lambda : w \in {}^A W^\res_\ext, \, \lambda \in \bY \}.
\]
We emphasize that ${}^A W_\ext$ is \emph{not} the set of elements $w$ which are minimal in their coset $W_A w$. However, this subset is also a set of representatives for the quotient $W_A \backslash W_\ext$, as stated in the following lemma.

\begin{lem}
\label{lem:AWext-representatives}
The composition
\[
{}^A W_\ext \to W_\ext \to W_A \backslash W_\ext
\]
is a bijection.
\end{lem}

\begin{proof}
We first prove surjectivity. Let $w \in W_\ext$. Then there exists $\lambda \in \bY_+$ such that $t_\lambda w^{-1} v(\mathfrak{A}_{\mathrm{fund}}) \subset \cC$ for any $v \in W_A$. If we fix such a $\lambda$, by definition all the elements $vwt_{-\lambda}$ ($v \in W_A$) belong to $W_\ext^S$. If $v \in W_A$ is such that $vwt_{-\lambda}$ is minimal in $W_A w t_{-\lambda}$, then $vwt_{-\lambda}$ belongs to ${}^A W^S_\ext$ by~\eqref{eqn:AWextS-characterization}. By~\eqref{eqn:WS-Wres-Whit}, there exist $y \in {}^A W^\res_\ext$ and $\mu \in -\bY_+$ such that $vwt_{-\lambda} = y t_{\mu}$. Then $w = v^{-1} y t_{\lambda+\mu}$, proving surjectivity.

As for injectivity, we consider $y,y' \in {}^A W^\res_\ext$ and $\lambda,\lambda' \in \bY$ such that
\[
W_A yt_\lambda = W_A y' t_{\lambda'}.
\]
Multiplying on the right by an antidominant element we can assume that $\lambda,\lambda' \in -\bY_+$. Then $yt_\lambda$ and $y' t_{\lambda'}$ belong to ${}^A W^S_\ext$ by~\eqref{eqn:WS-Wres-Whit}; in particular these elements are minimal in their respective cosets $W_A yt_\lambda$ and $W_A y' t_{\lambda'}$, see~\eqref{eqn:AWextS-characterization}. Since these cosets coincide, this implies that $yt_\lambda=y' t_{\lambda'}$, as desired.
\end{proof}

\subsection{The periodic order}
\label{ss:periodic-order}

In this subsection we introduce an order on $W_\ext$ which is different from the Bruhat order, and which will play a crucial role in our constructions.
%
%
Recall that any $w \in W_\ext$ can be written as $yt_\mu$ for some $y \in W_\ext^\res$ and $\mu \in \bY$, see~\S\ref{ss:alcoves}; in particular, in view of~\eqref{eqn:WS-Wres}, there exists $\lambda \in \bY$ such that $wt_\lambda \in W_\ext^S$. More generally, given any finite collection $w_1, \ldots, w_r$ of elements of $W_\ext$, there exists $\lambda \in \bY$ such that $w_i t_\lambda$ belongs to $W_\ext^S$ for any $i \in \{1, \ldots, r\}$.


\begin{lem}
\label{lem:periodic-order-cond}
Let $y,y' \in W_\ext$. The following conditions are equivalent:
\begin{enumerate}
\item
\label{it:periodic-order-cond-1}
there exists $\lambda \in \bY$ such that $yt_\lambda$ and $y't_\lambda$ belong to $W_\ext^S$ and $yt_\lambda \leq y't_\lambda$ in the Bruhat order;
\item
\label{it:periodic-order-cond-2}
for any $\lambda \in \bY$ such that $yt_\lambda$ and $y't_\lambda$ belong to $W_\ext^S$ we have $yt_\lambda \leq y't_\lambda$ in the Bruhat order.
\end{enumerate}
\end{lem}

\begin{proof}
Of course~\eqref{it:periodic-order-cond-2} implies~\eqref{it:periodic-order-cond-1}, since as explained above there exists $\lambda \in \bY$ such that $yt_\lambda$ and $y't_\lambda$ belong to $W_\ext^S$. Conversely, suppose that~\eqref{it:periodic-order-cond-1} holds, and fix some $\lambda \in \bY$ which satisfies this condition. Let $\mu \in \bY$ be such that $yt_\mu$ and $y't_\mu$ belong to $W_\ext^S$, and choose $\nu \in \bY$ such that $\nu-\lambda$ and $\nu-\mu$ are antidominant. Using a standard compatibility property of the Bruhat order with multiplication when lengths add (see~\cite[Lemma~2.1]{projGr1}) and Lemma~\ref{lem:lengths-add-Wres}, from the fact that $yt_\lambda \leq y't_\lambda$, we deduce that $yt_\nu \leq y't_\nu$, and then that $yt_\mu \leq y't_\mu$, as desired.
\end{proof}

We define the periodic order $\preceq$ on $W_\ext$ by saying that $y \preceq y'$ iff $y$ and $y'$ satisfy the equivalent conditions of Lemma~\ref{lem:periodic-order-cond}. In other words, if $\lambda \in \bY$ is any element such that $yt_\lambda$ and $y' t_\lambda$ belong to $W_\ext^S$, we have $y \preceq y'$ iff $yt_\lambda \leq y't_\lambda$.

The following lemma gathers some easy properties of the periodic order.

\begin{lem}
\phantomsection
\label{lem:per-order-properties}
\begin{enumerate}
\item
\label{it:per-order-1}
If $w \in W_\ext$ and $s \in S_\aff$, then we have either $sw \preceq w$ or $w \preceq sw$.
\item 
\label{it:per-order-2}
If $y,y' \in W_\ext$ and $\mu \in \bY$ we have $y \preceq y'$ iff $yt_\mu \preceq y' t_\mu$.
\item 
\label{it:per-order-3}
If $y,y' \in W_\ext^S$ we have $y \preceq y'$ iff $y \leq y'$.
\item
\label{it:per-order-4}
If $y,y' \in W_\ext$ satisfy $y \preceq y'$, and if $s \in S_\aff$ satisfies $sy \preceq y$, then we have $sy \preceq y'$ and $sy \preceq sy'$.
\item
\label{it:per-order-5}
If $y,y' \in W_\ext$ satisfy $y \preceq y'$, and if $s \in S_\aff$ satisfies $y' \preceq sy'$, then we have $y \preceq sy'$ and $sy \preceq sy'$.
\end{enumerate}
\end{lem}

\begin{proof}
\eqref{it:per-order-1} 
Fix $\lambda \in \bY$ such that $yt_\lambda$ and $syt_\lambda$ belong to $W_\ext^S$.
Then we have either $yt_\lambda \leq sy t_\lambda$ or $syt_\lambda \leq y t_\lambda$. In the first case we have $y \preceq sy$, and in the second case we have $sy \preceq y$.

\eqref{it:per-order-2}
If $\lambda \in \bY$ is such that $yt_{\lambda}$ and $y' t_{\lambda}$ belong to $W_\ext^S$, then by definition we have
$y \preceq y'$ iff $y t_{\lambda} \leq y' t_{\lambda}$. On the other hand we have $(yt_\mu) t_{\lambda-\mu}=y t_{\lambda}$ and similarly for $y'$, so that this condition holds iff $yt_\mu \preceq y't_\mu$.

\eqref{it:per-order-3}
This property is obvious from the definition (taking $\lambda=0$).

\eqref{it:per-order-4}
Of course we have $sy \preceq y \preceq y'$, and if $y' \preceq sy'$ then $sy \preceq y \preceq y' \preceq sy'$. Now assume $y' \succeq sy'$, and choose $\lambda \in \bY$ such that $yt_\lambda$, $y't_\lambda$, $syt_\lambda$ and $sy' t_\lambda$ all belong to $W_\ext^S$. Then we have $syt_\lambda \leq yt_\lambda \leq y' t_\lambda$ and $sy't_\lambda \leq y' t_\lambda$. By the last inequality, there exists a reduced expression $y't_\lambda = s s_1 \cdots s_r \omega$ with each $s_i$ in $S_\aff$ and $\omega \in \Omega$. Then $yt_\lambda$ admits a reduced expression obtained by omitting some of the simple reflections in this expression. If $s$ is not among the omitted simple reflections, then clearly $syt_\lambda \leq s_1 \cdots s_r \omega = sy' t_\lambda$. If $s$ \emph{is} omitted then we also have $syt_\lambda \leq sy't_\lambda$ by the exchange condition. Hence $sy \preceq sy'$ in all cases.

\eqref{it:per-order-5}
Of course we have $y \preceq y' \preceq sy'$, and if $sy \preceq y$ then $sy \preceq y \preceq y' \preceq sy'$. Now assume $sy \succeq y$, and choose $\lambda \in \bY$ such that $yt_\lambda$, $y't_\lambda$, $syt_\lambda$ and $sy' t_\lambda$ all belong to $W_\ext^S$. Then we have $yt_\lambda \leq y't_\lambda$, $y t_\lambda \leq syt_\lambda$ and $y' t_\lambda \leq sy't_\lambda$. Fixing a reduced expression for $y't_\lambda$, an expression for $yt_\lambda$ can be obtained by omitting some reflections. Adding $s$ on the left we obtain a reduced expression for $sy't_\lambda$, from which an expression for $syt_\lambda$ can be obtained by omitting the same reflections.
This shows that $syt_\lambda \leq sy't_\lambda$, and hence that $sy \preceq sy'$, as desired.
\end{proof}

\begin{rmk}
\label{rmk:min-preceq}
Now that we have introduced the order $\preceq$, we can reinterpret Lem\-ma~\ref{lem:AWext-representatives} (and its proof) as saying that ${}^A W_\ext$ consists of the elements $w \in W_\ext$ which are minimal for the order $\preceq$ in the coset $W_A w$ (and that each such coset contains a unique minimal element).
\end{rmk}

\begin{lem}
\label{lem:per-order-coset}
Let $y, y' \in {}^A W_\ext$.  Then $y \preceq y'$ if and only if $w_A y \preceq w_A y'$.
\end{lem}

\begin{proof}
In view of~\eqref{eqn:WS-Wres-Whit}, there exists $\lambda \in \bY$ such that $yt_\lambda$ and $y' t_\lambda$ belong to ${}^A W_\ext^S$. Then by definition $y \preceq y'$ if and only if $y t_\lambda \leq y't_\lambda$. By~\eqref{eqn:AWextS-characterization} the elements $y t_\lambda$ and $y't_\lambda$ are minimal in their respective cosets $W_A yt_\lambda$ and $W_A y' t_\lambda$, which implies that
\[
\ell(w_A y t_\lambda) = \ell(w_A)+\ell(yt_\lambda), \qquad \ell(w_A y' t_\lambda) = \ell(w_A)+\ell(y't_\lambda).
\]
By the compatibility property of the Bruhat order with multiplication when lengths add (as used in the proof of Lemma~\ref{lem:periodic-order-cond}; cf.~\cite[Lemma~2.1]{projGr1}), this implies that $y t_\lambda \leq y't_\lambda$ if and only if $w_A y t_\lambda \leq w_A y't_\lambda$. Finally, by~\eqref{eqn:AWextS-characterization} the elements $w_A y t_\lambda$ and $w_A y't_\lambda$ belong to $W_\ext^S$; we therefore have $w_A y t_\lambda \leq w_A y't_\lambda$ if and only if $w_A y \preceq w_A y'$, and the lemma follows.
\end{proof}

\subsection{The Hecke algebra and the left spherical module}

Let $v$ be an indeterminate, and let $\cH_\ext$ be the 
Hecke algebra of $W_\ext$ over $\Z[v,v^{-1}]$.  Recall that this is a $\Z[v,v^{-1}]$-algebra that is free as a $\Z[v,v^{-1}]$-module, with a basis (called the \emph{standard basis}) $( H_w : w \in W_\ext )$, and with multiplication determined by the following rules:
\begin{align*}
H_xH_y &= H_{xy} &&\text{if $\ell(xy) = \ell(x) + \ell(y)$,} \\
H_s^2 &= H_e + (v^{-1} - v)H_s &&\text{for all $s \in S_\aff$.}
\end{align*}
(Here and below we will follow the notational conventions of~\cite{soergel}.)
The algebra $\cH_\ext$ is also equipped with a \emph{canonical basis}~\cite{kl}, denoted by
\[
(\uH_w : w \in W_\ext),
\]
and uniquely characterized as follows: $\uH_w$ is fixed by a certain involution of $\cH_\ext$ (called the \emph{bar involution}), and 
\[
\uH_x \in H_x + \sum_{y < x} v\Z[v] H_y.
\]
Let us write each of these basis elements in terms of the standard basis:
\[
\uH_x = \sum_{y \in W_\ext} h_{y,x} H_y;
\]
the polynomials $h_{y,x} \in \Z[v]$ are then known as the \emph{Kazhdan--Lusztig polynomials}.  
Their geometric interpretation (in terms of perverse sheaves) will be recalled in~\S\ref{ss:kl} below.

We will denote by $\cM$ 
the left $\cH_\ext$-module obtained by taking the quotient
\[
\cM = \cH_\ext/\cH_\ext \cdot \{ H_s - v^{-1} : s \in S \}.
\]
This module is known as the \emph{left spherical module}.
(Note that much of the relevant literature, including~\cite{soergel,rw-simple}, treats a similarly defined \emph{right} $\cH_\ext$-module instead; however, the left version is better suited to the purposes of this paper. These two modules can e.g.~be related using the anti-involution of $W_\ext$ given by $w \mapsto w^{-1}$.)  This module remains free over $\Z[v,v^{-1}]$; specifically, if for $w \in W^S_\ext$ we let $M_w$
denote the image of $H_w$ in $\cM$, then
\[
( M_w : w \in W^S_\ext)
\]
is a basis for $\cM$.  This module also admits a canonical basis
\[
( \uM_w : w \in W^S_\ext)
\]
characterized similarly to the canonical basis of $\cH_\ext$. In fact, the map $h \mapsto h \uH_{w_\circ}$ factors through a morphism of left $\cH_\ext$-modules
\[
\zeta : \cM \to \cH_\ext,
\]
and this module satisfies
\[
\zeta(\uM_w)=\uH_{w w_\circ}
\]
for any $w \in W^S_\ext$, see~\cite{soergel}. Equivalently, if we define the polynomials $m_{y,w}$ by setting
\[
\uM_w = \sum_{y \in W^S_\ext} m_{y,w} M_y,
\]
then for $y,w \in W^S_\ext$ we have
\begin{equation}
\label{eqn:kl-sph}
m_{y,w} = h_{y', ww_\circ} \quad \text{for any $y' \in yW$.}
\end{equation}


We also introduce notation for the ``inverse matrix'' of $(m_{y,w})_{y,w \in W_\ext^S}$, again following~\cite{soergel}. Namely, we define the polynomials $(m^{y,w} : y,w \in W_\ext^S)$ by the condition that
\begin{equation}\label{eqn:kl-inv}
M_x = \sum_{y \in W^S_\ext} (-1)^{\ell(y)+\ell(x)} m^{x,y}\uM_y \quad \text{for any $x \in W_\ext^S$.}
\end{equation}

The only property of Kazhdan--Lusztig polynomials that will be used below and for which we do not have a geometric proof is the following.

\begin{lem}
\label{lem:m-triangle}
For any $w \in W_\ext^S$ we have 
$m^{w^\triangle,w}=v^{\ell(w_\circ)}$.
\end{lem}

\begin{proof}
This equality can be obtained by translating~\cite[Theorem~5.1]{soergel} (in the special case $B=\widehat{A}$) in our present conventions.
\end{proof}

\section{Perverse sheaves on affine Grassmannians}
\label{sec:perv}

\subsection{The affine Grassmannian and the affine flag variety}

We now denote by $z$ an indeterminate, and consider the functor $\cL G$, resp.~$\cL^+ G$, from $\kappa$-algebras to groups, which sends $R$ to $G(R( \hspace{-1pt} (z) \hspace{-1pt} ))$, resp.~$G(R[ \hspace{-1pt} [z] \hspace{-1pt} ])$. It is well known that $\cL G$ is represented by a group ind-scheme over $\kappa$, and that $\cL^+ G$ is represented by a group scheme over $\kappa$. Moreover, the fppf quotient $(\cL G/\cL^+G)_{\mathrm{fppf}}$ is represented by an ind-projective ind-scheme, which is denoted $\Gr$ and called the \emph{affine Grassmannian} of $G$.

There is an obvious morphism of group schemes $\cL^+G \to G$ induced by the assignment $z \mapsto 0$.  Let $I \subset \cL^+G$ and $I_\unip \subset I$ be the preimages of the Borel subgroup $B \subset G$ and its unipotent radical $U \subset B$, respectively, under this map. These are both subgroup schemes of $\cL^+G$.  The group $I$ is known as an \emph{Iwahori subgroup}, and $I_\unip$ as its pro-unipotent radical. 

We will consider also the affine flag variety $\Fl$ of $G$, defined as the fppf quotient $(\cL G/I)_{\mathrm{fppf}}$. Again $\Fl$ is represented by an ind-projective ind-scheme, and the natural morphism $\pi : \Fl \to \Gr$ is a Zariski locally trivial fibration with fibers isomorphic to $G/B$. 

Let $\mathrm{N}_G(T)$ be the normalizer of the maximal torus $T \subset G$, so that $\mathrm{N}_G(T)/T = W$.  For each $w \in W$, choose a representative $\dot w \in \mathrm{N}_G(T)$.  More generally, if $w \in W_\ext$, say $w = vt_\lambda$ with $v \in W$ and $\lambda \in \bY$, we set
\[
\dot w = \dot v z^\lambda \quad \in \cL G(\bk).
\]

For $w \in W_\ext$ we will denote by $\Fl_w$ the $I$-orbit of the image of $\dot{w}$; then it is well known that $\Fl_w$ is also the $I_\unip$-orbit of $\dot{w}$, and is isomorphic to an affine space of dimension $\ell(w)$. Moreover we have
\[
\Fl_{\mathrm{red}} = \bigsqcup_{w \in W_\ext} \Fl_w,
\quad \text{and} \quad
\bigl( \overline{\Fl_w} \subset \overline{\Fl_y} \quad \Leftrightarrow \quad w \leq y \bigr).
\]
Similarly,
for $w \in W_\ext^S$ we will denote by $\Gr_w$ the $I$-orbit of the image of $\dot{w}$ in $\Gr$. It is well known that $\Gr_w$ is also the $I_\unip$-orbit of the image of $\dot{w}$, that it is isomorphic to an affine space of dimension $\ell(w)$, and that we have
\[
\Gr_{\mathrm{red}} = \bigsqcup_{w \in W_\ext} \Gr_w
\quad \text{and} \quad
\bigl( \overline{\Gr_w} \subset \overline{\Gr_y} \quad \Leftrightarrow \quad w \leq y \bigr).
\]

\subsection{\texorpdfstring{$I_\unip$}{Iu}-equivariant perverse sheaves}
\label{ss:Iu-eq-PS}

Let $\bk$ be a field that is of one of the four kinds mentioned in~\S\ref{ss:intro-main}.  For such a field, 
we can consider the $I_\unip$-equivariant derived categories $\Db_{I_\unip}(\Fl,\bk)$ and $\Db_{I_\unip}(\Gr,\bk)$ of \'etale $\bk$-sheaves on $\Fl$ and $\Gr$; see~\cite[\S 3.2]{projGr1} for details. These categories have natural perverse t-structures, whose hearts will be denoted $\Perv_{I_\unip}(\Gr,\bk)$ and $\Perv_{I_\unip}(\Fl,\bk)$ respectively.

For any $w \in W_\ext$ we have a ``standard perverse sheaf'' $\DFl_w$ in $\Perv_{I_\unip}(\Fl,\bk)$, defined as the $!$-pushforward of the complex $\underline{\bk}_{\Fl_w}[\ell(w)]$ under the embedding $\Fl_w \to \Fl$, and a ``costandard perverse sheaf'' $\NFl_w$ in $\Perv_{I_\unip}(\Fl,\bk)$, defined as the $*$-pushforward of the complex $\underline{\bk}_{\Fl_w}[\ell(w)]$ under the embedding $\Fl_w \to \Fl$. (These complexes are indeed perverse sheaves since this embedding is affine.) The image of the unique (up to scalar) nonzero morphism $\DFl_w \to \NFl_w$ is simple, and will be denoted $\LFl_w$; it is the intersection cohomology complex associated with the constant local system on ${\Fl_w}$. Then the objects $(\LFl_w : w \in W_\ext)$ are representatives for the isomorphism classes of simple objects in the abelian category $\Perv_{I_\unip}(\Fl,\bk)$.

Similarly, for $w \in W_\ext^S$ we have a ``standard perverse sheaf'' $\DGr_w$ in $\Perv_{I_\unip}(\Gr,\bk)$, defined as the $!$-pushforward of the complex $\underline{\bk}_{\Gr_w}[\ell(w)]$ under the embedding $\Gr_w \to \Gr$, and a ``costandard perverse sheaf'' $\NGr_w$ in $\Perv_{I_\unip}(\Gr,\bk)$, defined as the $*$-pushforward of the complex $\underline{\bk}_{\Gr_w}[\ell(w)]$ under the embedding $\Gr_w \to \Gr$. (Once again these complexes are indeed perverse sheaves since the embedding $\Gr_w \to \Gr$ is affine.)
The image of the unique (up to scalar) nonzero morphism $\DGr_w \to \NGr_w$ is simple, and will be denoted $\LGr_w$; it is the intersection cohomology complex associated with the constant local system on ${\Gr_w}$. Then the objects $(\LGr_w : w \in W^S_\ext)$ are representatives for the isomorphism classes of simple objects in the abelian category $\Perv_{I_\unip}(\Gr,\bk)$.

Since the morphism $\pi : \Fl \to \Gr$ is smooth with connected fibers, the functor
\[
\pi^\dag:=\pi^*[\dim(G/B)] \cong \pi^![-\dim(G/B)] : \Db_{I_\unip}(\Gr,\bk) \to \Db_{I_\unip}(\Fl,\bk)
\]
is t-exact for the perverse t-structures, its restriction to perverse sheaves is fully faithful, and it sends simple perverse sheaves to simple perverse sheaves, see~\cite[Proposition~4.2.5]{bbd}; more explicitly, in this case we have
\begin{equation}
\label{eqn:pidag-IC}
\pi^\dag \LGr_w \cong \LFl_{ww_\circ}
\end{equation}
for any $w \in W_\ext^S$. 

The results of~\cite[\S 3.3]{bgs} show that the category $\Perv_{I_\unip}(\Gr,\bk)$ admits a natural structure of a highest weight category (in the sense of~\cite[\S 7]{riche-hab}) with weight poset $(W^S_\ext, \leq)$; the standard objects are the standard perverse sheaves $(\DGr_w : w \in W_\ext^S)$ and the costandard objects are the costandard perverse sheaves $(\NGr_w : w \in W^S_\ext)$. In particular, it makes sense to consider the \emph{tilting} objects in this category, i.e.~the objects which admit both a filtration with standard subquotients and a filtration with costandard subquotients. The indecomposable tilting objects are parametrized by $W_\ext^S$; the indecomposable object associated with $w$ will be denoted $\TGr_w$. Similar comments apply to the category the category $\Perv_{I_\unip}(\Fl,\bk)$ (for the weight poset $(W_\ext, \leq)$). The indecomposable tilting object attached to $w$ will be denoted $\TFl_w$.

We will also occasionally consider the $I$-equivariant derived categories $\Db_I(\Fl,\bk)$ and $\Db_I(\Gr,\bk)$. We have forgetful functors
\[
\For^I_{I_\unip} : \Db_I(\Fl,\bk) \to \Db_{I_\unip}(\Fl,\bk), \quad \For^I_{I_\unip} : \Db_I(\Gr,\bk) \to \Db_{I_\unip}(\Gr,\bk),
\]
and the objects $\DFl_w, \NFl_w$ and $\DGr_w,\NGr_w$ naturally ``lift" to objects of $\Db_I(\Fl,\bk)$ and $\Db_I(\Gr,\bk)$ respectively (and will be denoted by the same symbol in the equivariant context). We also have ``convolution'' bifunctors
\begin{gather*}
\Db_I(\Fl,\bk) \times \Db_I(\Fl,\bk) \to \Db_I(\Fl,\bk), \quad \Db_I(\Fl,\bk) \times \Db_I(\Gr,\bk) \to \Db_I(\Gr,\bk), \\
\Db_{I_\unip}(\Fl,\bk) \times \Db_I(\Fl,\bk) \to \Db_{I_\unip}(\Fl,\bk), \quad \Db_{I_\unip}(\Fl,\bk) \times \Db_I(\Gr,\bk) \to \Db_{I_\unip}(\Gr,\bk),
\end{gather*}
which will all be denoted $\star^I$, and are compatible with one another in the expected ways.

The following lemma gathers some standard properties of convolutions of standard and costandard objects (see e.g.~\cite[\S 8.2]{abg}).

\begin{lem}
\phantomsection
\label{lem:conv-D-N}
\begin{enumerate}
\item
For $w,y \in W_\ext$ such that $\ell(wy)=\ell(w)+\ell(y)$, there exist canonical isomorphisms
\[
\DFl_w \star^I \DFl_y \simto \DFl_{wy}, \qquad \NFl_w \star^I \NFl_y \simto \NFl_{wy}.
\]
\item
For $w \in W_\ext$, there exist canonical isomorphisms
\[
\DFl_w \star^I \NFl_{w^{-1}} \cong \DFl_e \cong \NFl_{w^{-1}} \star^I \DFl_w.
\]
\item
For $w,y \in W_\ext$, the objects $\NFl_w \star^I \DFl_y$ and $\DFl_w \star^I \NFl_y$ are perverse.
\item
\label{it:conv-DNGr}
For $w,y \in W_\ext$ such that $\ell(wy)=\ell(w)+\ell(y)$ and both $wy$ and $y$ belong to $W_\ext^S$, there exist canonical isomorphisms
\[
\DFl_w \star^I \DGr_y \simto \DGr_{wy}, \qquad \NFl_w \star^I \NGr_y \simto \NGr_{wy}.
\]
\item
For $w \in W_\ext$ and $y \in W_\ext^S$, the objects $\NFl_w \star^I \DGr_y$ and $\DFl_w \star^I \NGr_y$ are perverse.
\end{enumerate}
\end{lem}

\subsection{Relation with the Satake category}
\label{ss:Satake-category}

Below we will also consider the $\cL^+G$-equivariant derived category $\Db_{\cL^+ G}(\Gr,\bk)$. Once again this category has a natural perverse t-structure, whose heart will be denoted $\Perv_{\cL^+ G}(\Gr,\bk)$. For $\lambda \in \bY_+$ we will denote by $L_\lambda$ the image of $z^\lambda$ in $\Gr$, and by $\Gr^\lambda$ its $\cL^+G$-orbit; then $\Gr^\lambda$ is the union of the $I_\unip$-orbits labeled by the minimal representatives of the elements $(t_\mu : \mu \in W(\lambda))$,
and
\[
\Gr_{\mathrm{red}} = \bigsqcup_{\lambda \in \bY_+} \Gr^\lambda.
\]
We will consider these orbits in particular when $\lambda \in \bY_{++}$. It is a classical fact that, in this case, there exists a smooth $\cL^+G$-equivariant morphism
\[
p_\lambda : \Gr^\lambda \to G/B
\]
sending $L_\lambda$ to the base point $B/B$, where $\cL^+G$ acts on $G/B$ through the natural morphism $\cL^+G \to G$. Here $G/B$ has the Bruhat stratification by orbits of $B$, parametrized by $W$; its pullback to $\Gr^\lambda$ identifies with the decomposition into the $I$-orbits given by
\[
\Gr^\lambda = \bigsqcup_{w \in W} \Gr_{w t_\lambda w_\circ}.
\]
In particular, 
the unique open $I$-orbit in $\Gr^\lambda$ is $\Gr_{t_{w_\circ(\lambda)}}$.

The simple objects in the category $\Perv_{\cL^+ G}(\Gr,\bk)$ are in natural bijection with $\bY_+$, via the operation sending $\lambda$ to the intersection cohomology complex $\IC^\lambda$ associated with the constant local system on $\Gr^\lambda$. The forgetful functor
\[
\For^{\cL^+G}_{I_\unip} : \Db_{\cL^+ G}(\Gr,\bk) \to \Db_{I_\unip}(\Gr,\bk)
\]
is t-exact, and restricts to a fully faithful functor on perverse sheaves; moreover we have
\[
\For^{\cL^+G}_{I_\unip}(\IC^\lambda)=\LGr_{t_{w_\circ(\lambda)}}
\]
for any $\lambda \in \bY_+$.

To each $\lambda \in \bY_+$ one can also associate the ``standard'' and ``costandard'' objects defined respectively by
\[
\cI_!^\mu = {}^{\mathrm{p}}\hspace{-1pt} \tau^{\geq 0}(j^\mu_! \underline{\bk}_{\Gr^\mu}[\langle 2\rho,\mu \rangle]), \quad \cI_*^\mu = {}^{\mathrm{p}}\hspace{-1pt} \tau^{\leq 0}(j^\mu_* \underline{\bk}_{\Gr^\mu}[\langle 2\rho,\mu \rangle]), 
\]
where $j^\mu: \Gr^\mu \hookrightarrow \Gr$ is the inclusion and ${}^{\mathrm{p}} \hspace{-1pt} \tau^{\geq 0}, {}^{\mathrm{p}} \hspace{-1pt} \tau^{\leq 0}$ are the perverse truncation functors. (Note that $j^\mu$ is \emph{not} an affine morphism in general, so unlike in the $I_\unip$-equivariant case, if we omit the perverse truncation functors, the resulting objects are not in general perverse.) With this notation there exists (up to scalar) a unique nonzero morphism $\cI_!^\mu \to \cI_*^\mu$, and its image is $\IC^\mu$. Once again the category $\Perv_{\cL^+ G}(\Gr,\bk)$ has a natural highest weight structure with standard objects the perverse sheaves $(\cI_!^\mu : \mu \in \bY_+)$ and costandard objects the perverse sheaves $(\cI_*^\mu : \mu \in \bY_+)$, see~\cite[Proposition~1.12.4]{bar}. In particular one can consider the tilting objects in this category, and the indecomposable such objects are parametrized by $\bY_+$. For any $\lambda \in \bY_+$ we will denote by $\cT^\lambda$ the corresponding indecomposable tilting object.

As in the $I$-equivariant setting (see~\S\ref{ss:Iu-eq-PS}), we also have a canonical convolution product
\begin{equation}
\label{eqn:convol-L+G}
\star^{\cL^+G} : \Db_{\cL^+ G}(\Gr,\bk) \times \Db_{\cL^+ G}(\Gr,\bk) \to \Db_{\cL^+ G}(\Gr,\bk)
\end{equation}
which equips $\Db_{\cL^+ G}(\Gr,\bk)$ with the structure of a monoidal category. In this case it is known that this product is t-exact (in the sense that a product of perverse sheaves is perverse), and hence induces a monoidal structure on the abelian category $\Perv_{\cL^+G}(\Gr,\bk)$; see~\cite[\S 1.6.3]{bar} for details. The \emph{geometric Satake equivalence} describes the monoidal category $(\Perv_{\cL^+G}(\Gr,\bk),\star^{\cL^+G})$ in representation-theoretic terms: more explicitly, in~\cite{mv} the authors construct a canonical affine $\bk$-group scheme $G^\vee_\bk$ equipped with a split maximal torus $T^\vee_\bk$ whose group of characters is $\bY$ and a canonical equivalence of monoidal categories
\[
\Satake: (\Perv_{\cL^+G}(\Gr,\bk), \star^{\cL^+G}) \simto ( \Rep(G^\vee_\bk), \otimes).
\]
They also show that $G^\vee_\bk$ is a split connected reductive group, and that the root datum of $(G^\vee_\bk, T^\vee_\bk)$ is dual to that of $(G,T)$. Under this equivalence, $\cI_!^\mu$ corresponds to the Weyl module of highest weight $\mu$, and $\cI^*_\mu$ to the induced module of highest weight $\mu$.

Below we will use the fact that the monoidal category
\[
(\Perv_{\cL^+ G}(\Gr,\bk), \star^{\cL^+G})
\]
is rigid: every object $\cF$ has a left and right dual $\cF^\vee$. (This fact can either be checked directly or deduced from the geometric Satake equivalence.) We will not need an explicit description of this operation, but only that for $\mu \in \bY_+$ we have
\begin{equation}
\label{eqn:duals}
(\cI_!^\mu)^\vee \cong \cI_*^{-w_\circ(\mu)}, \quad (\cI_*^\mu)^\vee \cong \cI_!^{-w_\circ(\mu)}, \quad (\IC^\mu)^\vee \cong \IC_*^{-w_\circ(\mu)}.
\end{equation}

Using the geometry of spherical orbits we prove the following property of the periodic order, which will be required later.

\begin{lem}
\label{lem:per-order-weights}
Let $\mu,\nu \in \bY$ be such that $\mu-\nu$ is a sum of positive roots, and let $y \in W_\ext^\res$. Then $y t_{w_\circ(\nu)} \preceq y t_{w_\circ(\mu)}$.
\end{lem}

\begin{proof}
Choose $\eta \in \bY$ such that $\eta+\nu$ and $\eta + \mu$ are strictly dominant. Then it is well known that our assumption implies that $\overline{\Gr^{\eta+\nu}} \subset \overline{\Gr^{\eta+\mu}}$, i.e.~that $\overline{\Gr_{t_{w_\circ(\eta+\nu)}}} \subset \overline{\Gr_{t_{w_\circ(\eta+\mu)}}}$, and hence that $t_{w_\circ(\eta+\nu)} \leq t_{w_\circ(\eta+\mu)}$ in the Bruhat order. By~\cite[Lemmas~2.1 and~2.7]{projGr1}, this implies that $y t_{w_\circ(\eta+\nu)} \leq y t_{w_\circ(\eta+\mu)}$. Since these elements belong to $W_\ext^S$, by Lemma~\ref{lem:per-order-properties}\eqref{it:per-order-3} this implies that $y t_{w_\circ(\eta+\nu)} \preceq y t_{w_\circ(\eta+\mu)}$. Using Lemma~\ref{lem:per-order-properties}\eqref{it:per-order-2} we deduce that $y t_{w_\circ(\nu)} \preceq y t_{w_\circ(\mu)}$, as desired.
\end{proof}

\subsection{Whittaker categories}
\label{ss:Whittaker}

From now on we
assume that $\kappa$ has characteristic $p > 0$.  Recall that by our assumptions from~\S\ref{ss:intro-main}, the field $\bk$ contains
a nontrivial $p$-th root of unity (which we fix once and for all).  Let $A \subset S_\aff$ be a finitary subset, and let $I_\unip^A:=\dot{w}_A I_\unip \dot{w}_A^{-1}$. In~\cite[\S 3.4]{projGr1} we have explained the construction of a ``generic" character $\psi_A : I_\unip^A \to \Ga$; as in~\cite[\S 3.5]{projGr1}
one can then consider the categories
\[
\Db_{(I_\unip^A,\cX_A)}(\Fl,\bk) \quad \text{and} \quad \Db_{(I_\unip^A,\cX_A)}(\Gr,\bk)
\]
of $(I_\unip^A,\cX_A)$-equivariant $\bk$-sheaves on $\Fl$ and $\Gr$,
where $\cX_A$ is the pullback along $\psi_A$ of an Artin--Schreier local system on $\Ga$. In the case where $A=\varnothing$, these categories are just the ordinary $I_\unip$-equivariant derived categories considered in~\S\ref{ss:Iu-eq-PS}; in this case we will omit ``$A$'' from all notations introduced below.

These categories have natural perverse t-structures, whose hearts are denoted $\Perv_{(I_\unip^A,\cX_A)}(\Fl,\bk)$ and $\Perv_{(I_\unip^A,\cX_A)}(\Gr,\bk)$ respectively. The isomorphism classes of simple objects in $\Perv_{(I_\unip^A,\cX_A)}(\Fl,\bk)$ are in canonical bijection with the subset of $W_\ext$ consisting of elements $w$ minimal in $W_A w$; for such $w$ the corresponding simple object is denoted by $\LFl^A_w$. For any $w \in W_\ext$ minimal in $W_A w$ we have a ``standard" perverse sheaf $\DFl^A_w$ and a ``costandard" perverse sheaf $\NFl^A_w$ in $\Perv_{(I_\unip^A,\cX_A)}(\Fl,\bk)$, obtained by $!$- and $*$-pushforward respectively of a local system on the $I^A_\unip$-orbit of the image of $\dot{w}$ in $\Fl$. There exists a unique (up to scalar) nonzero morphism $\DFl^A_w \to \NFl^A_w$, whose image is $\LFl^A_w$, and these objects equip $\Perv_{(I_\unip^A,\cX_A)}(\Fl,\bk)$ with the structure of a highest weight category with weight poset $\{w \in W_\ext \mid \text{$w$ minimal in $W_A w$}\}$, endowed with the restriction of the Bruhat order.

Similarly, the isomorphism classes of simple objects in $\Perv_{(I_\unip^A,\cX_A)}(\Gr,\bk)$ are in canonical bijection with ${}^A W^S_\ext$; for $w \in {}^A W^S_\ext$ the corresponding simple object is denoted by $\LGr^A_w$. For any $w \in {}^A W_\ext^S$ we have a ``standard" perverse sheaf $\DGr^A_w$ and a ``costandard" perverse sheaf $\NGr^A_w$ in $\Perv_{(I_\unip^A,\cX_A)}(\Gr,\bk)$, obtained by $!$- and $*$-pushforward respectively of a local system on the $I^A_\unip$-orbit of the image of $\dot{w}$ in $\Gr$. There exists a unique (up to scalar) nonzero morphism $\DGr^A_w \to \NGr^A_w$, whose image is $\LGr^A_w$, and these objects equip $\Perv_{(I_\unip^A,\cX_A)}(\Gr,\bk)$ with the structure of a highest weight category with weight poset ${}^A W_\ext^S$, endowed with the restriction of the Bruhat order. The indecomposable tilting objects in this category are then also parametrized by ${}^A W_\ext^S$; the object associated with $w$ will be denoted $\TGr^A_w$.
Below we will also use the fact that
for $w \in {}^A W_\ext^S$ we have
\begin{equation}
\label{eqn:pi*-D-N}
\pi_* \DFl^A_w \cong \DGr^A_w, \qquad \pi_* \NFl^A_w \cong \NGr^A_w,
\end{equation}
see e.g.~\cite[Lemma~A.1]{acr} for similar considerations.

The same construction as in~\eqref{eqn:convol-L+G} yields a canonical bifunctor
\[
\star^{\cL^+G} : \Db_{(I_\unip^A,\cX_A)}(\Gr,\bk) \times \Db_{\cL^+ G}(\Gr,\bk) \to \Db_{(I_\unip^A,\cX_A)}(\Gr,\bk)
\]
which defines a right action of the monoidal category $(\Db_{\cL^+ G}(\Gr,\bk),\star^{\cL^+G})$ on $\Db_{(I_\unip^A,\cX_A)}(\Gr,\bk)$. This bifunctor is t-exact in the sense that if $\cF \in \Perv_{(I_\unip^A,\cX_A)}(\Gr,\bk)$ and $\cG \in \Perv_{\cL^+ G}(\Gr,\bk)$ then $\cF \star^{\cL^+G} \cG$ is perverse; see~\cite[\S 4.1]{projGr1} for references.

\begin{rmk}
\label{rmk:Verdier-duality}
In order to use Verdier duality arguments, we will also have to consider the local system $\cX_A^{-1}$ on $I_\unip^A$ (i.e.~the local system defined by the inverse of the $p$-th root of unity fixed above); namely, Verdier duality induces anti-equivalences of categories 
\[
\Db_{(I_\unip^A,\cX_A)}(\Fl,\bk) \simto \Db_{(I_\unip^A,\cX^{-1}_A)}(\Fl,\bk) \quad \text{and} \quad \Db_{(I_\unip^A,\cX_A)}(\Gr,\bk) \simto \Db_{(I_\unip^A,\cX^{-1}_A)}(\Gr,\bk),
\]
which will both be denoted $\D$. None of our considerations below will depend on the choice of root of unity; in particular, they are equally valid in both the $(I_\unip^A, \cX_A)$- and $(I_\unip^A, \cX_A^{-1}$)-equivariant settings. For this reason, we may write $\DGr^A_w$ to indicate either the $(I_\unip^A, \cX_A)$-equivariant standard perverse sheaf, or the $(I_\unip^A, \cX_A^{-1}$)-equivariant one: this minor abuse of notation should not lead to any confusion.

With the comments above in mind, the behavior of $\D$ on the various objects introduced above can be summarized as follows:
for any $w \in {}^A W^S_\ext$ we have
\[
\D(\DGr^A_w) \cong \NGr^A_w, \quad \D(\NGr^A_w) \cong \DGr^A_w, \quad \D(\LGr^A_w) \cong \LGr^A_w, \quad \D(\TGr^A_w) \cong \TGr^A_w,
\]
and similarly for the corresponding objects on $\Fl$.
\end{rmk}

\subsection{Combinatorics of perverse sheaves}
\label{ss:kl}

As explained in~\S\ref{ss:Whittaker} the category $\Perv_{(I_\unip^A,\cX_A)}(\Gr,\bk)$ has a natural structure of a highest weight category. There are two kinds of numerical quantities one can compute in this setting. First, given a perverse sheaf $\cF$ in $\Perv_{(I_\unip^A,\cX_A)}(\Gr,\bk)$, one can consider the multiplicity of a simple object $\LGr^A_w$ as a composition factor of $\cF$; this number is denoted
\[
[\cF : \LGr^A_w].
\]
Next, if we assume that $\cF$ admits a standard filtration (i.e.~a filtration whose subquotients are standard perverse sheaves), one can consider the number of occurrences of a given standard object $\DGr^A_w$, which is denoted
\[
(\cF : \DGr^A_w).
\]
It is a standard fact that this number is well defined (i.e.~does not depend on the choice of filtration) and additive with respect to direct sums; in fact we have
\[
(\cF : \DGr^A_w)=\dim_\bk \Hom(\cF,\NGr^A_w).
\]
Similar comments apply to the multiplicity of a given costandard object in a costandard filtration of an object $\cF$ (assumed to admit such a filtration), which will be denoted
\[
(\cF : \NGr^A_w).
\]

Let us now consider a triple $(\mathbb{K},\mathbb{O},\bk)$ where $\mathbb{K}$ is a finite extension of $\mathbb{Q}_\ell$ containing a nontrivial $p$-th root of unity, $\mathbb{O}$ is its ring of integers, and $\bk$ is the residue field of $\mathbb{O}$. In this setting we can consider the categories $\Perv_{(I_\unip^A,\cX_A)}(\Gr,\bk)$ and $\Perv_{(I_\unip^A,\cX_A)}(\Gr,\mathbb{K})$. In both of these categories the indecomposable tilting perverse sheaves are parametrized by $W_\ext^S$; to distinguish the two cases the objects associated with $w$ will be denoted $\TGr_w^{A,\bk}$ and $\TGr_w^{A,\mathbb{K}}$ respectively. We will use similar conventions for standard objects.

\begin{lem}
\label{lem:mult-tilt-modred}
For any $w,y \in {}^A W_\ext^S$ we have
\[
(\TGr_w^{A,\bk} : \DGr_y^{A,\bk}) \geq (\TGr_w^{A,\mathbb{K}} : \DGr_y^{A,\mathbb{K}}).
\]
\end{lem}

\begin{proof}
In addition to the categories $\Perv_{(I_\unip^A,\cX_A)}(\Gr,\bk)$ and $\Perv_{(I_\unip^A,\cX_A)}(\Gr,\mathbb{K})$, we can also consider the category $\Perv_{(I_\unip^A,\cX_A)}(\Gr,\mathbb{O})$ of $(I_\unip^A,\cX_A)$-equivariant $\mathbb{O}$-perver\-se sheaves on $\Gr$. In this category we also have standard and costandard objects, denoted $\DGr_y^{A,\mathbb{O}}$ and $\NGr_y^{A,\mathbb{O}}$, respectively, and we can speak of tilting perverse sheaves. As explained in~\cite[Appendix~B]{modrap1}, the indecomposable tilting objects are again classified by ${}^A W_\ext^S$. More specifically, given $w \in {}^A W_\ext^S$ there exists an indecomposable tilting perverse sheaf $\TGr_w^{A,\mathbb{O}}$ such that
\[
\bk \lotimes_{\mathbb{O}} \TGr_w^{A,\mathbb{O}} \cong \TGr_w^{A,\bk},
\]
where
\[
\bk \lotimes_{\mathbb{O}} (-) : \Db_{(I_\unip^A,\cX_A)}(\Gr,\mathbb{O}) \to \Db_{(I_\unip^A,\cX_A)}(\Gr,\bk)
\]
is the ``modular reduction'' functor. In particular, for any $y \in {}^A W_\ext^S$ the multiplicity of $\DGr_y^{A,\mathbb{O}}$ in a standard filtration of $\TGr_w^{A,\mathbb{O}}$ is $(\TGr_w^{A,\bk} : \DGr_y^{A,\bk})$. We can also consider the tensor product functor
\[
\mathbb{K} \lotimes_{\mathbb{O}} (-) : \Db_{(I_\unip^A,\cX_A)}(\Gr,\mathbb{O}) \to \Db_{(I_\unip^A,\cX_A)}(\Gr,\mathbb{K});
\]
the perverse sheaf $\mathbb{K} \lotimes_{\mathbb{O}} \TGr_w^{A,\mathbb{O}}$ is tilting, supported on $\overline{\Gr_w}$, and satisfies $(\mathbb{K} \lotimes_{\mathbb{O}} \TGr_w^{A,\mathbb{O}} : \DGr_w^{A,\mathbb{K}}) = 1$; it therefore admits $\TGr_w^{A,\mathbb{K}}$ as a direct summand. We deduce that $(\TGr_w^{A,\mathbb{K}} : \DGr_y^{A,\mathbb{K}})$ is at most the multiplicity of $\DGr_y^{A,\mathbb{O}}$ in a standard filtration of $\TGr_w^{A,\mathbb{O}}$, which proves the desired inequality.
\end{proof}

For the rest of this subsection we assume that $A=\varnothing$ and $\bk$ has characteristic $0$.  It is a classical fact going back to~\cite{kl2} that the Kazhdan--Lusztig polynomials $(h_{y,x} : y \in W_\ext)$ encode the dimensions of stalks of the simple perverse sheaf $\LFl_x$.  Explicitly, we have
\begin{equation}\label{eqn:kl-ic-fl}
h_{y,x} = \sum_{i \in \Z} \mathrm{rk} \left( \mathscr{H}^{-\ell(y) - i}(\LFl_x{}_{|\Fl_y}) \right) \cdot v^i = \sum_{i \in \Z} \dim \Hom(\LFl_x,\NFl_y[i]) \cdot v^i.
\end{equation}
Similarly, the spherical Kazhdan--Lusztig polynomials describe stalks of simple perverse sheaves on $\Gr$: we have
\[
m_{y,x} = \sum_{i \in \Z} \mathrm{rk} \left( \mathscr{H}^{-\ell(y) - i}(\LGr_x{}_{|\Gr_y}) \right) \cdot v^i = \sum_{i \in \Z} \dim \Hom(\LGr_x,\NGr_y[i]) \cdot v^i.
\]
(In fact, this equality easily follows from~\eqref{eqn:kl-ic-fl}, comparing~\eqref{eqn:kl-sph} and~\eqref{eqn:pidag-IC}.)

Let us now work in the Grothendieck group $[\Db_{I_\unip}(\Gr,\bk)]$ of the triangulated category $\Db_{I_\unip}(\Gr,\bk)$. Since the basis $([\DGr_w] : w \in W_\ext^S)$ is dual to the basis $([\NGr_w] : w \in W_\ext^S)$ for the natural Euler pairing, the preceding equality means that
\[
[\LGr_w] = \sum_y m_{y,w}{}_{|v=-1} [\DGr_w]
\]
for any $w \in W_\ext^S$.
As a consequence, the polynomials $(m^{y,w} : y,w \in W_\ext^S)$ from~\eqref{eqn:kl-inv} have the following interpretation:
\begin{equation*}
[\DGr_w] = \sum (-1)^{\ell(y)+\ell(w)} m^{w,y}{}_{|v=-1} [\LGr_y] \quad \text{for any $w \in W_\ext^S$;}
\end{equation*}
in other words, for any $y,w \in W_\ext^S$ we have
\[
[\DGr_w : \LGr_y] = (-1)^{\ell(y)+\ell(w)} m^{w,y}{}_{|v=-1}.
\]
In particular, from Lemma~\ref{lem:m-triangle} we deduce that for any $w,y \in W_\ext$ we have
\begin{equation}
\label{eqn:mult-triangle}
[\DGr_{w^\triangle} : \LGr_w] = 1.
\end{equation}
(Here we know a priori that the left-hand side is nonnegative; but the fact that $\ell(w)+\ell(w^\triangle)+\ell(w_\circ)$ is even can also be seen directly from the computations in~\S\ref{ss:wall-crossing} below.)

\subsection{Averaging functors}
\label{ss:Av-Iw}

%


We continue to consider a finitary subset $A \subset S_\aff$.
As explained in~\cite[\S 3.6--3.8]{projGr1},
there is a t-exact ``averaging functor''
\[
\Av^A_\psi: \Db_{I_\unip}(\Gr,\bk) \to \Db_{(I_\unip^A,\cX_A)}(\Gr,\bk)
\]
with t-exact left and right adjoints, denoted by
\[
\Av^A_!, \Av^A_*: \Db_{(I_\unip^A,\cX_A)}(\Gr,\bk) \to \Db_{I_\unip}(\Gr,\bk),
\]
respectively.
By~\cite[Lemma~3.3(3)]{projGr1} the functor $\Av^A_\psi$ sends each standard object either to $0$ or to a standard object, and each costandard object either to $0$ or to a costandard object. As a consequence, it sends objects admitting a standard filtration, resp.~a costandard filtration, to objects admitting a standard filtration, resp.~a costandard filtration.

\begin{rmk}
\label{rmk:Av-notkill}
The functors $\Av^A_!$ and $\Av^A_*$ do not kill any nonzero perverse sheaf. In fact, if $\cF$ is a nonzero object in $\Perv_{(I^A_\unip, \cX_A)}(\Gr,\bk)$ and $w \in {}^A W_\ext^S$ is such that $\LGr^A_w$ is a quotient of $\cF$, then by adjunction and~\cite[Lemma~3.3(4)]{projGr1} the object $\Av^A_!(\cF)$ admits a nonzero morphism to $\LGr_w$, hence is nonzero. A dual argument applies to $\Av^A_*$.
\end{rmk}

\begin{lem}
\phantomsection
\label{lem:Av-Iw}
\begin{enumerate}
\item
\label{it:Av-Iw-2}
The functor
\[
\Av^A_! : \Perv_{(I^A_\unip, \cX_A)}(\Gr,\bk) \to \Perv_{I_\unip}(\Gr,\bk)
\]
sends objects admitting a standard filtration to objects admitting a standard filtration. More explicitly, for any $w \in {}^A W_\ext^S$ and $y \in W_\ext^S$ we have
\[
(\Av^A_!(\DGr^A_w) : \DGr_y) = \begin{cases}
1 & \text{if $y \in W_A w$;} \\
0 & \text{otherwise.}
\end{cases}
\]
\item
\label{it:Av-Iw-3}
The functor
\[
\Av^A_* : \Perv_{(I^A_\unip, \cX_A)}(\Gr,\bk) \to \Perv_{I_\unip}(\Gr,\bk)
\]
sends objects admitting a costandard filtration to objects admitting a costandard filtration. More explicitly, for any $w \in {}^A W_\ext^S$ and $y \in W_\ext^S$ we have
\[
(\Av^A_!(\NGr^A_w) : \NGr_y) = \begin{cases}
1 & \text{if $y \in W_A w$;} \\
0 & \text{otherwise.}
\end{cases}
\]
\end{enumerate}
\end{lem}

\begin{proof}
%
\eqref{it:Av-Iw-2} 
Recall that an object $X$ in a highest-weight category has a standard filtration if and only if $\Ext^1(X,{-})$ vanishes on all costandard objects, see~\cite[Proposition~7.12]{riche-hab}.  Let us apply this criterion to the categories
$\Perv_{(I^A_\unip, \cX_A)}(\Gr,\bk)$
and $\Perv_{I_\unip}(\Gr,\bk)$.
Since $\Av^A_\psi$ is exact and sends costandard objects to objects admitting a costandard filtration, its exact left adjoint $\Av^A_!$ preserves the property of having a standard filtration. The description of multiplicities also follows from adjunction, together with the explicit description of images under $\Av^A_\psi$ of the standard objects in~\cite[Lemma~3.3(3)]{projGr1}.



\eqref{it:Av-Iw-3} The proof is similar to that of~\eqref{it:Av-Iw-2}.
\end{proof}

%
%
%

\begin{rmk}
\label{rmk:Av-isom}
The functor $\Av^A_\psi$ also has $*$- and $!$-versions, which turn out to be canonically isomorphic, see~\cite[Lemma~3.2]{projGr1}. (By definition, $\Av^A_\psi$ is identified with these two versions.) Similarly there exists a canonical morphism of functors
\[
\Av^A_! \to \Av^A_*,
\]
but this map is \emph{not} an isomorphism in general.  For instance, when $A$ is a singleton $\{s\}$, it is known that both $\Av^A_!(\DFl^A_e)$ and $\Av^A_*(\DFl^A_e)$ are 
isomorphic to $\TFl_s$, but one can check by direct calculation that the image of $\Av^A_!(\DFl^A_e) \to \Av^A_*(\DFl_e)$ is isomorphic to $\LFl_e$.

Nevertheless, the philosophy of Koszul duality suggests that there should exist \emph{some} natural isomorphism $\Av^A_! \cong \Av^A_*$.  In more detail, the functor $\Av^A_\psi$ is the counterpart under Koszul duality of push-forward along the projection $\pi_A: \Fl \to \Fl_A$, where $\Fl_A$ is a partial affine flag variety (depending on $A$).  This map is proper and smooth, so the left and right adjoints of $(\pi_A)_*$---namely, $\pi_A^*$ and $\pi_A^!$---are isomorphic up to a shift.  Similarly, $\Av^A_!$ and $\Av^A_*$ should be isomorphic (up to a Tate twist in the setting of mixed sheaves).  However, we were unable to find a direct proof of this claim. (See Lemma~\ref{lem:isom-xi} and Remark~\ref{rmk:Av-Groth} below for related results.)
\end{rmk}


\subsection{Wall-crossing functors}
\label{ss:wc-functors}

For a finitary subset $A \subset S_\aff$, we consider the functors
\[
\xi_A^!, \xi^*_A :
\Db_{I_\unip}(\Gr,\bk) \to \Db_{I_\unip}(\Gr,\bk)
\qquad
\text{defined by}
\qquad
\begin{aligned}
\xi_A^! &= \Av^A_! \circ \Av^A_\psi, \\
\xi_A^* &= \Av^A_* \circ \Av^A_\psi.
\end{aligned}
\]
The results recalled in~\S\ref{ss:Av-Iw} imply that:
\begin{itemize}
\item
each of these functors is t-exact with respect to the perverse t-structure; 
\item
$\xi_A^!$ sends perverse sheaves admitting a standard filtration to perverse sheaves admitting a standard filtration;
\item
$\xi_A^*$ sends perverse sheaves admitting a costandard filtration to perverse sheaves admitting a costandard filtration;
\item
$\xi_A^!$ is left-adjoint to $\xi^*_A$. 
\end{itemize}

\begin{lem}
\label{lem:isom-xi}
There exists an isomorphism of functors $\xi_A^! \cong \xi^*_A$.
As a consequence, these functors send objects admitting a standard filtration to objects admitting a standard filtration, and objects admitting a costandard filtration to objects admitting a costandard filtration; in particular, they send tilting objects to tilting objects.
\end{lem}

\begin{proof}[Proof sketch]
The proof requires a different realization of the functors $\xi_A^!$ and $\xi_A^*$. Namely, following Yun (see~\cite[Appendix~A]{by}; see also~\cite{br} for a review of this construction, which explicitly allows more general coefficients) one can consider the ``free-monodromic completion'' $D^\wedge$ of the $I_\unip$-equivariant derived category of sheaves on $\cL G/I_\unip$, constructible with respect to the stratification by $I$-orbits. In this category we have a perverse t-structure, and a notion of tilting perverse sheaves, see~\cite[\S A.7]{by} or~\cite[\S 5.4]{br}; the indecomposable tilting objects are parametrized (in terms of their support) by $W_\ext$. This category also has a monoidal structure, and this monoidal category acts in a natural way on the category $\Db_{I_\unip}(\Gr,\bk)$.
 Now it follows from~\cite[Lemma~10.1]{br} that both $\xi_A^!$ and $\xi^*_A$ are both isomorphic to the functor given by convolution with the indecomposable tilting object associated with the element $w_A$. These functors are therefore isomorphic.

The second claim of the lemma is a consequence of this isomorphism and the properties of $\xi_A^!$ and $\xi_A^*$ listed above.
\end{proof}

\begin{rmk}
\label{rmk:xi-convolution}
The considerations in the proof of Lemma~\ref{lem:isom-xi} simplify in case we apply the functors $\xi_A^*$ and $\xi_A^!$ to an object of the form $\For^I_{I_\unip}(\cF)$ for some $\cF$ in $\Db_I(\Gr,\bk)$.
In this case we have
\[
\xi_A^!(\For^I_{I_\unip}(\cF)) \cong \TFl_{w_A} \star^I \cF \cong \xi_A^*(\For^I_{I_\unip}(\cF))
\]
where the convolution bifunctor $\star^I$ is as in~\S\ref{ss:Iu-eq-PS}.
\end{rmk}

Thanks to Lemma~\ref{lem:isom-xi}, we will write $\xi_A$ for $\xi_A^! \cong \xi^*_A$ when the choice among these functors does not matter. 
In case $A=\{s\}$ for some $s \in S_\aff$, we will simplify this notation even further and write $\xi_s$ for $\xi_{\{s\}}$.

\begin{cor}
\label{cor:Av-tilting}
The functors 
\[
\Av^A_!, \Av^A_* : \Perv_{(I^A_\unip, \cX_A)}(\Gr,\bk) \to \Perv_{I_\unip}(\Gr,\bk)
\]
send objects admitting a standard filtration to objects admitting a standard filtration, and similarly for costandard filtrations. In particular, these functors send tilting objects to tilting objects.
\end{cor}

\begin{proof}
We write the proof for $\Av^A_!$;
the other case is similar. What we have to prove is that $\Av^A_!(\DGr^A_w)$ admits a standard filtration and $\Av^A_!(\NGr^A_w)$ admits a costandard filtration for any $w \in {}^A W^S_\ext$. The case of standard filtrations has already been proved in Lemma~\ref{lem:Av-Iw}. For costandard filtrations, we observe that
\[
\Av^A_!(\NGr^A_w) \cong \Av^A_! \Av^A_\psi(\NGr_w) = \xi_A^!(\NGr_w) \cong \xi_A^*(\NGr_w) \cong \Av^A_* \Av^A_\psi(\NGr_w) =\Av^A_*(\NGr^A_w)
\]
by~\cite[Lemma~3.3(3)]{projGr1} and Lemma~\ref{lem:isom-xi}. The right-hand side admits a costandard filtration by Lemma~\ref{lem:Av-Iw}, so we are done.
\end{proof}

\begin{rmk}
\label{rmk:Av-Groth}
If we denote by $[\Perv_{(I^A_\unip, \cX_A)}(\Gr,\bk)]$ and $[\Perv_{I_\unip}(\Gr,\bk)]$ the Grothen\-dieck groups of the categories $\Perv_{(I^A_\unip, \cX_A)}(\Gr,\bk)$ and $\Perv_{I_\unip}(\Gr,\bk)$ respectively, and by $[\Av^A_!]$ and $[\Av^A_*]$ the morphisms induced by $\Av^A_!$ and $\Av^A_*$ on Grothendieck groups, then we have
\[
[\Av^A_!] = [\Av^A_*].
\]
Indeed, this follows from the observation that for any $w \in {}^A W_\ext^S$ we have
\[
\Av^A_!(\LGr^A_w) \cong \Av^A_*(\LGr^A_w),
\]
by the same considerations as in the proof of Corollary~\ref{cor:Av-tilting}.
\end{rmk}

We can in fact be more precise regarding the effect of the functors $\Av^A_!$ and $\Av^A_*$ on indecomposable tilting perverse sheaves by adapting the proof of a ``Koszul dual'' statement in~\cite[Proposition~3.5]{williamson}, as follows.

\begin{prop}
\label{prop:Av-tilting-indec}
For any $w \in {}^A W_\ext^S$ we have
\begin{gather*}
\Av^A_!(\TGr^A_w) \cong \TGr_{w_A w} \cong \Av^A_*(\TGr^A_w),\\ \Av^A_\psi(\TGr_{w_A w}) \cong (\TGr^A_w)^{\oplus \# W_A}.
\end{gather*}
\end{prop}

\begin{proof}
We will prove the first isomorphism on the first line above, and the isomorphism on the second line; the second isomorphism on the first line can be obtained by similar arguments, or deduced using Verdier duality.

First, we note that
\begin{equation}
\label{eqn:Avpsi-Av!}
\Av^A_\psi(\Av^A_!(\TGr^A_w)) \cong (\TGr^A_w)^{\oplus \# W_A}.
\end{equation}
Indeed, by the comments in~\S\ref{ss:Av-Iw} and Corollary~\ref{cor:Av-tilting} the left-hand side is tilting. A closer look at~\cite[Lemma~3.3(3)]{projGr1} and Lemma~\ref{lem:Av-Iw}\eqref{it:Av-Iw-2} shows that for any $y \in {}^A W_\ext^S$ we have
\[
(\Av^A_\psi(\Av^A_!(\TGr^A_w)) : \DGr_y^A) = (\# W_A) \times ( \TGr^A_w : \DGr_y^A ) = \bigl( (\TGr^A_w)^{\oplus \# W_A} : \DGr_y^A \bigr).
\]
Since a tilting object is determined (up to isomorphism) by its standard multiplicities, this implies the desired isomorphism.

Now, as justified above $\Av^A_!(\TGr^A_w)$ is tilting. From the description of multiplicities in Lemma~\ref{lem:Av-Iw} one sees that $w_A w$ is maximal among the labels of standard objects appearing in a standard filtration of this object; it follows that $\TGr_{w_A w}$ is a direct summand in it. Let us write
\[
\Av^A_!(\TGr^A_w) = \TGr_{w_A w} \oplus \cT.
\]
Then $\cT$ is tilting, and the standard objects appearing in a standard filtration of this object are of the form $xy$ with $x \in W_A$ and $y \in {}^A W_\ext^S$ (because this property holds true for $\Av^A_!(\TGr^A_w)$). We will show that $\Av^A_\psi(\cT)=0$, which will imply that $\cT=0$ (by exactness of $\Av^A_\psi$ and its effect on standard objects as described in~\cite[Lemma~3.3(3)]{projGr1}), and thereby conclude the proof. 

Using~\eqref{eqn:Avpsi-Av!} we obtain that
\[
(\TGr^A_w)^{\oplus \# W_A} \cong \Av^A_\psi(\TGr_{w_A w}) \oplus \Av^A_\psi(\cT).
\]
By the Krull--Schmidt property, this implies that $\Av^A_\psi(\TGr_{w_A w})$ and $\Av^A_\psi(\cT)$ are both direct sums of copies of $\TGr^A_w$. To determine the number of copies in $\Av^A_\psi(\TGr_{w_A w})$ it suffices to compute $(\Av^A_\psi(\TGr_{w_A w}) : \DGr^A_w)$; we will show that this number is at least $\# W_A$, which will imply that $\Av^A_\psi(\cT)=0$, as desired. For that, using once again the exactness of $\Av^A_\psi$ and~\cite[Lemma~3.3(3)]{projGr1}, it suffices to prove that for any $x \in W_A$ we have
\begin{equation}
\label{eqn:multiplicities-w_Aw}
(\TGr_{w_A w} : \DGr_{xw}) \geq 1.
\end{equation}
However, the union
\[
\bigsqcup_{x \in W_A} \Gr_{xw}
\]
is open in $\overline{\Gr_{w_A w}}$, and is an affine space bundle over $M_A / B_A$ where $M_A$ is the reductive group attached to $A$ as in~\cite[\S 3.4]{projGr1} and $B_A$ is its natural (negative) Borel subgroup. The restriction of $\TGr_{w_A w}$ to this union is again tilting, and it must admit the indecomposable tilting object with label $w_A w$ (i.e~the pullback of the indecomposable tilting object on $M_A/B_A$ attached to $w_A$) as a direct summand. It is a standard fact that the standard object with label $xw$ appears with multiplicity $1$ in the latter object for any $x$ (see~e.g.~\cite{br}), which implies~\eqref{eqn:multiplicities-w_Aw} and finishes the proof.
\end{proof}

\begin{rmk}
\begin{enumerate}
\item
From the philosophy of the Finkelberg--Mirkovi{\'c} conjecture and its ``singular'' variants, Proposition~\ref{prop:Av-tilting-indec} can be considered a geometric counterpart to~\cite[Proposition~E.11]{jantzen}.
\item
Alternatively, to prove~\eqref{eqn:multiplicities-w_Aw} one can reduce the proof to the characteristic-$0$ setting using Lemma~\ref{lem:mult-tilt-modred}, and then use the formula for multiplicities of tilting perverse sheaves in terms of Kazhdan--Lusztig polynomials in this case proved in~\cite{yun}.
\end{enumerate}
\end{rmk}


Below we will also consider some easier ``wall-crossing functors'' associated with elements of $\Omega$. Note that
if $\omega \in \Omega$ we have $\dot{\omega} I_\unip \dot{\omega}^{-1} = I_\unip$; as a consequence, left multiplication by $\dot{\omega}$ induces an autoequivalence
\[
\xi_\omega : \Perv_{I_\unip}(\Gr,\bk) \simto \Perv_{I_\unip}(\Gr,\bk)
\]
which satisfies
\begin{equation}
\label{eqn:xi-D-N}
\xi_\omega(\DGr_y) \cong \DGr_{\omega y}, \qquad \xi_\omega(\NGr_y) \cong \NGr_{\omega y}
\end{equation}
for any $y \in W_\ext^S$. From this we deduce that we also have
\begin{equation}
\label{eqn:xi-L-T}
\xi_\omega(\LGr_y) \cong \LGr_{\omega y}, \qquad \xi_\omega(\TGr_y) \cong \TGr_{\omega y},
\end{equation}
again for any $y \in W_\ext^S$.
 
\subsection{A support computation}

The following two lemmas, which describe the effect of wall-crossing functors on the support of perverse sheaves, will be needed in~\S\ref{ss:big-tilting} below.

\begin{lem}
\label{lem:xi-simples}
Let $y \in W_\ext^S$ and $s \in S_\aff$.
\begin{enumerate}
\item
\label{it:xi-simples-1}
If $sy<y$, 
or if $sy > y$ and $s y \notin W_\ext^S$, then $\xi_{s}(\LGr_y)=0$.
\item
\label{it:xi-simples-2}
If $sy > y$ and $s y \in W_\ext^S$ then $\xi_{s}(\LGr_y)$ is supported on $\overline{\Gr_{s y}}$, and admits $\LGr_{sy}$ as a composition factor with multiplicity $1$.
\end{enumerate}
\end{lem}

\begin{proof}
\eqref{it:xi-simples-1}
According to~\eqref{eqn:AWextS-characterization}, the condition that $sy<y$ implies that $y \notin {}^{\{s\}} W_\ext^S$.  Similarly, if $sy > y$ and $s y \notin W_\ext^S$, then again we have $y \notin {}^{\{s\}} W_\ext^S$.  In view of~\cite[Lemma~3.3(4)]{projGr1}, either of these conditions therefore implies that $\Av_\psi^{\{s\}}(\LGr_y)=0$, and hence a fortiori that $\xi_{s}(\LGr_y)=0$.

\eqref{it:xi-simples-2}
By Remark~\ref{rmk:xi-convolution}, there exists a perverse sheaf $\cF$ in $\Perv_{I_\unip}(\Fl,\bk)$, supported on $\overline{\Fl_s}$ and satisfying $\cF_{|\Fl_s} \cong \underline{\bk}_{\Fl_s}[1]$, such that $\xi_{s}(\LGr_y) \cong \cF \star^I \LGr_y$. It is easily seen that the right-hand side is supported on $\overline{\Gr_{s y}}$, and that its restriction to $\Gr_{sy}$ is $\underline{\bk}[\ell(sy)]$; the multiplicity of $\LGr_{sy}$ in this perverse sheaf is therefore $1$.
\end{proof}

\begin{lem}
\label{lem:support}
Let $w \in W_\ext$, $\omega \in \Omega$ and $s_1, \ldots, s_r \in S_\aff$ be such that
\[
\ell(\omega s_1 \cdots s_r w)=\ell(w)+r
\]
and $\omega s_1 \cdots s_r w$ belongs to $W_\ext^S$. Then:
\begin{enumerate}
\item
\label{it:support-1}
 $w$ belongs to $W_\ext^S$;
\item
\label{it:support-3}
for any perverse sheaf $\cF$ supported on $\overline{\Gr_w}$, $\xi_\omega \xi_{s_1} \cdots \xi_{s_r}(\cF)$ is supported on $\overline{\Gr_{\omega s_1 \cdots s_r w}}$; moreover we have
\[
[\xi_\omega \xi_{s_1} \cdots \xi_{s_r}(\cF) : \LGr_{\omega s_1 \cdots s_r w}] = [\cF : \LGr_w].
\]
\end{enumerate}
\end{lem}

\begin{proof}
In view of~\eqref{eqn:xi-L-T}, and since $W_\ext^S$ is stable under left multiplication by elements of $\Omega$, we can assume that $\omega=e$. Then we proceed by induction on $r$, the case $r=0$ being obvious. If $r \geq 1$, then~\cite[Lemma~2.2]{projGr1} ensures that $s_2 \cdots s_r w \in W_\ext^S$, so that by induction $w \in W_\ext^S$ (establishing~\eqref{it:support-1}).

Now let $\cF$ be a perverse sheaf supported on $\overline{\Gr_w}$. By induction, $\xi_{s_2} \cdots \xi_{s_r}(\cF)$ is supported on $\overline{\Gr_{s_2 \cdots s_r w}}$, so that its composition factors are of the form $\LGr_y$ with $y \in W_\ext^S$ satisfying $y \leq s_2 \cdots s_r w$, and in case $y=s_2 \cdots s_r w$ we have
\[
[\xi_{s_2} \cdots \xi_{s_r}(\cF) : \LGr_{s_2 \cdots s_r w}] = [\cF : \LGr_w].
\]
By exactness of $\xi_{s_1}$ (see~\S\ref{ss:wc-functors}) the object $\xi_{s_1} \cdots \xi_{s_r}(\cF)$ is then an extension of perverse sheaves $\xi_{s_1}(\LGr_y)$ for such $y$'s. If $s_1y<y$ or if $s_1y > y$ and $s_1 y \notin W_\ext^S$, then $\xi_{s_1}(\LGr_y)=0$ by Lemma~\ref{lem:xi-simples}\eqref{it:xi-simples-1}.
If $s_1y > y$ and $s_1 y \in W_\ext^S$ with $y \neq s_2 \cdots s_r w$ then by Lemma~\ref{lem:xi-simples}\eqref{it:xi-simples-2} $\xi_{s_1}(\LGr_y)$ is supported on $\overline{\Gr_{s_1 y}}$; since $s_1 y < s_1 \cdots s_r w$ this perverse sheaf is therefore supported on $\overline{\Gr_{s_1 \cdots s_r w}}$ but does not admit $\LGr_{s_1 \cdots s_r w}$ as a composition factor. Finally, for $y=s_2 \cdots s_r w$, again by Lemma~\ref{lem:xi-simples}\eqref{it:xi-simples-2} the perverse sheaf $\xi_{s_1}(\LGr_{s_2 \cdots s_r w})$ is supported on $\overline{\Gr_{s_1 \cdots s_r w}}$, and
\[
[\xi_{s_1}(\LGr_{s_2 \cdots s_r w}) : \LGr_{s_1 s_2 \cdots s_r w}] = 1.
\]
We deduce statement~\eqref{it:support-3}, which finishes the induction.
\end{proof}

\subsection{The geometric Steinberg formula}

To finish this section we state the ``geometric Steinberg formula'' proved in~\cite[Theorems~4.1 and~4.3]{projGr1}. This statement will be the starting point for all the main constructions of the present paper.

\begin{thm}
\label{thm:geometric-Steinberg}
Let $A \subset S_\aff$ be a finitary subset.
\begin{enumerate}
\item
\label{it:thm-Steinberg}
For any $w \in {}^A W_\ext^\res$ and any $\mu \in \bY_+$ we have
\[
\mathsf{L}^A_w \star^{\cL^+ G} \IC^\mu \cong \mathsf{L}^A_{wt_{w_\circ(\mu)}}.
\]
\item
\label{it:thm-Steinberg-ff}
For any $w \in {}^A W_\ext^\res$, the functor
\[
\mathsf{L}^A_w \star^{\cL^+ G} (-) : \Perv_{\cL^+G}(\Gr,\bk) \to \Perv_{(I_\unip^A, \cX_A)}(\Gr,\bk)
\]
is fully faithful.
\end{enumerate}
\end{thm}


Later we will also need the following corollary of Theorem~\ref{thm:geometric-Steinberg}.

\begin{cor}
\label{cor:Hom-vanishing-Steinberg}
Let $y,y' \in {}^A W_\ext^\res$, $\mu \in \bY_+$ and $\nu \in -\bY_+$. 

\begin{enumerate}
\item
\label{it:Hom-vanishing-Steinberg-1}
If
$\Hom(\LGr^A_y \star^{\cL^+G} \cI_!^\mu, \NGr^A_{y't_\nu}) \neq 0$,
then there exists $\lambda \in \bY$ orthogonal to all roots such that $y=y't_\lambda$.
\item
\label{it:Hom-vanishing-Steinberg-2}
If $\nu \neq w_\circ(\mu)$, then $\Hom(\LGr^A_y \star^{\cL^+G} \cI_!^\mu, \NGr^A_{yt_\nu}) = 0$.
\item
\label{it:Hom-vanishing-Steinberg-3}
The space $\Hom(\LGr^A_y \star^{\cL^+G} \cI_!^\mu, \NGr^A_{yt_{w_\circ(\mu)}})$ is $1$-dimensional, and spanned by the composition
\[
\LGr^A_{y} \star^{\cL^+G} \cI_!^{\mu} \twoheadrightarrow \LGr^A_{y} \star^{\cL^+G} \IC^{\mu} \cong \LGr^A_{y t_{w_\circ(\mu)}} \hookrightarrow \NGr^A_{yt_{w_\circ(\mu)}},
\]
where the surjection is induced by the surjection $\cI_!^{\mu} \twoheadrightarrow \IC^{\mu}$ and the isomorphism is given by Theorem~\ref{thm:geometric-Steinberg}\eqref{it:thm-Steinberg}.
\end{enumerate}
\end{cor}

\begin{proof}
\eqref{it:Hom-vanishing-Steinberg-1}
By Theorem~\ref{thm:geometric-Steinberg} and exactness of the bifunctor $\star^{\cL^+G}$, taking a composition series of $\cI_!^\mu$ we obtain a composition series of $\LGr^A_y \star^{\cL^+G} \cI_!^\mu$, all of whose subquotients are of the form $\LGr^A_{yt_{w_\circ(\eta)}}$ for $\eta \in \bY_+$. If our $\Hom$-space is nonzero, there exists such an $\eta$ which satisfies
\[
\Hom(\LGr^A_{yt_{w_\circ(\eta)}},\NGr^A_{y't_\nu}) \neq 0,
\]
i.e.~such that $yt_{w_\circ(\eta)}=y't_\nu$. By the definition of $W_\ext^\res$, this implies that $\nu - w_\circ(\eta)$ is orthogonal to all roots, and then the desired claim.

\eqref{it:Hom-vanishing-Steinberg-2}
By~\cite[Lemma~3.3(4)]{projGr1} we have $\Av^A_\psi(\LGr_y) \cong \LGr^A_y$, so
\[
\Av^A_\psi(\LGr_y \star^{\cL^+G} \cI_!^\mu) \cong \LGr^A_y \star^{\cL^+G} \cI_!^\mu.
\]
Using adjunction, we deduce that
\[
\Hom(\LGr^A_y \star^{\cL^+G} \cI_!^\mu, \NGr^A_{yt_\nu}) \cong \Hom(\LGr_y \star^{\cL^+G} \cI_!^\mu, \Av^A_*(\NGr^A_{yt_\nu})).
\]
Now by~\cite[Lemma~3.3(3)]{projGr1} and Remark~\ref{rmk:xi-convolution} we have
\[
\Av^A_*(\NGr^A_{yt_\nu}) \cong \TFl_{w_A} \star^I \NGr_{yt_\nu}.
\]
It is a standard fact that there exists an embedding $\NFl_e \hookrightarrow \TFl_{w_A}$ whose cokernel admits a costandard filtration with subquotients of the form $\NFl_x$ with $x \in W_A \smallsetminus \{e\}$, each appearing once. Since $yt_\nu$ is minimal in $W_A yt_\nu$ (see~\S\ref{ss:coset-representatives}), using Lemma~\ref{lem:conv-D-N}\eqref{it:conv-DNGr} we deduce an embedding $\NGr_{yt_\nu} \hookrightarrow \TFl_{w_A} \star^I \NGr_{yt_\nu}$ whose cokernel admits a costandard filtration with subquotients $\NGr_{xyt_\nu}$ with $x \in W_A \smallsetminus \{e\}$. Since all the composition factors of $\LGr_y \star^{\cL^+G} \cI_!^\mu$ have their label in ${}^A W_\ext^S$ (see the proof of~\eqref{it:Hom-vanishing-Steinberg-1} and~\eqref{eqn:WS-Wres-Whit}), for any such $x$ we have
\[
\Hom(\LGr_y \star^{\cL^+G} \cI_!^\mu, \NGr_{xyt_\nu})=0.
\]
We deduce that $\Hom(\LGr_y \star^{\cL^+G} \cI_!^\mu, \NGr_{yt_\nu}) \cong \Hom(\LGr_y \star^{\cL^+G} \cI_!^\mu, \TFl_{w_A} \star^I \NGr_{yt_\nu})$, and hence
\[
\Hom(\LGr^A_y \star^{\cL^+G} \cI_!^\mu, \NGr^A_{yt_\nu}) \cong \Hom(\LGr_y \star^{\cL^+G} \cI_!^\mu, \NGr_{yt_\nu}).
\]
Finally we have a surjection $\DGr_y \twoheadrightarrow \LGr_y$, which induces an embedding
\[
\Hom(\LGr_y \star^{\cL^+G} \cI_!^\mu, \NGr_{yt_\nu}) \hookrightarrow \Hom(\DGr_y \star^{\cL^+G} \cI_!^\mu, \NGr_{yt_\nu}).
\]
By~\cite[Lemma~4.9]{projGr1} the right-hand side vanishes if $\nu \neq w_\circ(\mu)$, which implies the desired claim.

\eqref{it:Hom-vanishing-Steinberg-3}
As in~\eqref{it:Hom-vanishing-Steinberg-2} we have an embedding
\[
\Hom(\LGr^A_y \star^{\cL^+G} \cI_!^\mu, \NGr^A_{yt_{w_\circ(\mu)}})
\hookrightarrow \Hom(\DGr_y \star^{\cL^+G} \cI_!^\mu, \NGr_{yt_{w_\circ(\mu)}}).
\]
By~\cite[Remark~4.10]{projGr1} the right-hand side is $1$-dimensional, so that the left-hand side has dimension $0$ or $1$. The nonzero morphism exhibited in the statement shows that this space is nonzero; it must therefore be $1$-dimensional.
\end{proof}

\section{Background from representation theory}
\label{sec:mod-R-rep}

In the next section we will introduce our main object of study, a certain category of ind-perverse sheaves which should be thought of as a geometric counterpart of the category of $\bG_1\bT$-modules for a reductive group $\bG$ (with maximal torus $\bT$) such that $\bG^{(1)}$ is Langlands dual to $G$. In order to motivate this construction, and to justify our conventions, we explain in this section the representation-theoretic version of this construction. 

We fix an algebraically closed field $\bk$ of characteristic $p>0$ and,
for any affine group scheme $H$ over $\bk$, denote by $\Rep(H)$ the category of finite-dimensional $H$-modules.

\subsection{Representations of Frobenius kernels as representations of reductive groups with additional structure}
\label{ss:Rep-G1T}

We fix a connected reductive algebraic group $\bG$ over $\bk$, with a Borel subgroup $\bB$ and a maximal torus $\bT \subset \bB$. (We use this notation because, in practice, we want to use the results below in the case $G$ and $\bG$ are related as in~\S\ref{ss:intro-overview}. But in the present section $\bG$ can be arbitrary.) We will denote by $\bY$ the lattice of characters of $\bT$, and by $\bY_+ \subset \bY$ the subset of dominant weights, with the convention that the nonzero $\bT$-weights in $\mathrm{Lie}(\bB)$ are the \emph{negative} roots. (Of course, in general these sets might differ from those denoted in the same way in~\S\ref{ss:Wext}, but they will coincide in the case we are interested in.) We have the Frobenius morphism $\Fr : \bG \to \bG^{(1)}$, which restricts to the Frobenius morphisms of $\bB$ and $\bT$; we will identify the character lattice of $\bT^{(1)}$ with $\bY$ in such a way that the morphism $\Fr^* : X^*(\bT^{(1)}) \to \bY$ identifies with $\lambda \mapsto p\lambda$.

We will be mostly interested in representations of $\bG_1\bT := \Fr^{-1}(\bT^{(1)})$. Following the point of view of~\cite{ag}, we consider the composition of equivalences
\begin{equation}
\label{eqn:equiv-G1T-Coh}
 \Rep(\bG_1\bT) \cong \Coh^{\bG_1\bT}(\mathrm{pt}) \cong \Coh^{\bG \times \bT^{(1)}}\bigl( (\bG \times \bT^{(1)}) / \bG_1\bT \bigr),
\end{equation}
where $\bG_1\bT$ is seen as a subgroup in $\bG \times \bT^{(1)}$ via $g \mapsto (g,\Fr(g))$. Now the map $(g,t) \mapsto \Fr(g)t^{-1}$ induces an isomorphism
\[
 (\bG \times \bT^{(1)}) / \bG_1\bT \simto \bG^{(1)},
\]
so that the category on the right-hand side of~\eqref{eqn:equiv-G1T-Coh} identifies with the category of $\bG \times \bT^{(1)}$-equivariant coherent sheaves on $\bG^{(1)}$, or in other words with $\bG$-equivariant $\bY$-graded $\scO(\bG^{(1)})$-modules which are finitely generated over $\scO(\bG^{(1)})$. Here $\scO(\bG^{(1)})$ is endowed with the left regular representation structure, and considered as $\bY$-graded with
\[
 \scO(\bG^{(1)})_\lambda = \Ind_{\bT^{(1)}}^{\bG^{(1)}}(-\lambda),
\]
and the equivalence sends a $\bG_1\bT$-module $M$ to the $\bY$-graded module whose degree-$\lambda$ component is
\[
 \Ind_{\bG_1\bT}^{\bG} \bigl( M \otimes \Fr^*(\bk_{\bT^{(1)}}(-\lambda)) \bigr).
\]
In particular, in view of the tensor identity (see~\cite[Proposition~I.3.6]{jantzen}), under this equivalence the restriction functor $\Rep(\bG) \to \Rep(\bG_1\bT)$ identifies with the functor sending $M$ to the $\bY$-graded module whose degree-$\lambda$ component is
\[
 M \otimes \Ind_{\bG_1\bT}^{\bG} \bigl( \Fr^*(\bk_{\bT^{(1)}}(-\lambda)) \bigr)  = M \otimes \Ind_{\bT^{(1)}}^{\bG^{(1)}}(-\lambda) = M \otimes \scO(\bG^{(1)})_{\lambda}.
\]

The $\bG$-modules considered above are typically infinite-dimensional; however the category of all (possibly infinite-dimensional) algebraic representations of an algebraic group identifies with the category of ind-objects in the category of finite-dimensional representations. (See~\S\ref{ss:ind} below for some comments and references on ind-objects. The statement above can be deduced from~\cite[Corollary~6.3.5]{ks}.) From this point of view, we therefore obtain an equivalence of categories between $\Rep(\bG_1\bT)$ and the category of $\bY$-graded ind-objects in $\Rep(\bG)$ endowed with a graded action of $\scO(\bG^{(1)})$ (seen as an algebra object in the category of $\bY$-graded ind-objects in $\Rep(\bG)$) and which are finitely generated with respect to this action (i.e.~isomorphic to a quotient of a finite direct sum of grading shifts of objects $M \otimes \scO(\bG^{(1)})$ with $M \in \Rep(\bG)$).

\begin{rmk}
\phantomsection
\label{rmk:AG-G^1}
\begin{enumerate}
\item
\label{it:AG-G^1}
The same considerations show that the category of $\bY$-graded ind-objects in $\Rep(\bG^{(1)})$ endowed with a graded action of $\scO(\bG^{(1)})$ identifies with the category of $\bY$-graded $\bk$-vector spaces (i.e.~algebraic---but not necessarily finite dimensional---representations of $\bT^{(1)}$).
\item
One can also omit the torus $\bT$ in this construction, i.e.~consider the Frobenius kernel $\bG_1$ instead of $\bG_1\bT$, and deduce an equivalence of categories between $\Rep(\bG_1)$ and the category of ind-objects in $\Rep(\bG)$ endowed with an action of $\scO(\bG^{(1)})$ and which are finitely generated with respect to this action.
\end{enumerate}
\end{rmk}

\subsection{The left regular representation as an ind-object}
\label{ss:reg-ind}

In this subsection we explain how the object $\Ind_{\bT^{(1)}}^{\bG^{(1)}}(-\lambda)$ can be explicitly represented as an ind-object in
$\Rep(\bG^{(1)})$.  (See the discussion in~\S\ref{ss:ind} below for generalities on ind-objects.)

For $\lambda \in \bY_+$, we will denote by $\weyl^{(1)}(\lambda)$, resp.~ $\coweyl^{(1)}(\lambda)$, the Weyl module, resp.~induced module, for $\bG^{(1)}$ of highest weight $\lambda$; by definition we have $\weyl^{(1)}(\lambda)=\coweyl^{(1)}(-w_\circ(\lambda))^*$.
It is a standard fact that for $\lambda,\mu \in \bY_+$ and $n \in \Z$ we have
\begin{equation}
\label{eqn:Ext-vanishing-Rep}
\Ext^n_{\Rep(\bG^{(1)})}(\weyl^{(1)}(\lambda), \coweyl^{(1)}(\mu)) \cong \begin{cases}
\bk & \text{if $\lambda=\mu$ and $n=0$;} \\
0 & \text{otherwise;}
\end{cases}
\end{equation}
see~\cite[Proposition~II.4.13]{jantzen}.
Given $\lambda,\lambda' \in \bY_+$ there exists a unique morphism of $\bG^{(1)}$-modules
\begin{equation}
\label{eqn:can-morph-rep-2}
\coweyl^{(1)}(\lambda) \otimes \coweyl^{(1)}(\lambda') \to \coweyl^{(1)}(\lambda+\lambda')
\end{equation}
sending the tensor product of the highest weight vectors in the left-hand side to the highest weight vector in the right-hand side. By duality, we deduce for any $\lambda,\lambda' \in \bY_+$ a canonical morphism
\begin{equation}
\label{eqn:can-morph-rep-2-M}
 \weyl^{(1)}(\lambda+\lambda') \to \weyl^{(1)}(\lambda) \otimes \weyl^{(1)}(\lambda').
\end{equation}

We will assume we are given, for any $\lambda \in \bY_+$, a nonzero morphism of $\bG^{(1)}$-modules
\[
 \varphi_\lambda : \weyl^{(1)}(\lambda) \to \coweyl^{(1)}(\lambda),
\]
such that for $\lambda,\lambda' \in \bY_+$ the composition
\[
\weyl^{(1)}(\lambda+\lambda') \xrightarrow{\eqref{eqn:can-morph-rep-2-M}} \weyl^{(1)}(\lambda) \otimes \weyl^{(1)}(\lambda') \xrightarrow{\varphi_\lambda \otimes \varphi_{\lambda'}} \coweyl^{(1)}(\lambda) \otimes \coweyl^{(1)}(\lambda') \xrightarrow{\eqref{eqn:can-morph-rep-2}} \coweyl^{(1)}(\lambda+\lambda')
\]
coincides with $\varphi_{\lambda+\lambda'}$. (See Remark~\ref{rmk:morphisms-lift} below for a discussion of this condition.)
By adjunction, $\varphi_\lambda$ determines a canonical morphism
 \begin{equation}
\label{eqn:can-morph-rep-1}
\bk \to \coweyl^{(1)}(-w_\circ(\lambda)) \otimes \coweyl^{(1)}(\lambda).
\end{equation}

Below we will consider various (formal) inductive limits parametrized by some subsets of $\bY_+$; in each case, this subset is endowed with the restriction of the partial order on $\bY_+$ such that $\lambda$ is smaller than $\lambda'$ iff $\lambda'-\lambda$ is dominant.
Given a weight $\mu \in \bY$, we consider the ind-object
\[
 \underset{\lambda \in \bY_+ \cap (w_\circ(\mu) + \bY_+)}{``\varinjlim{}"} \coweyl^{(1)}(\mu-w_\circ(\lambda)) \otimes \coweyl^{(1)}(\lambda)
\]
where given $\eta \in \bY_+$ the transition morphism
\[
 \coweyl^{(1)}(\mu-w_\circ(\lambda)) \otimes \coweyl^{(1)}(\lambda) \to \coweyl^{(1)}(\mu-w_\circ(\lambda+\eta)) \otimes \coweyl^{(1)}(\lambda+\eta)
\]
is the composition
\begin{multline*}
 \coweyl^{(1)}(\mu-w_\circ(\lambda)) \otimes \coweyl^{(1)}(\lambda) \to \coweyl^{(1)}(\mu-w_\circ(\lambda)) \otimes \coweyl^{(1)}(-w_\circ(\eta)) \otimes \coweyl^{(1)}(\eta) \otimes \coweyl^{(1)}(\lambda) \\
 \to \coweyl^{(1)}(\mu-w_\circ(\lambda+\eta)) \otimes \coweyl^{(1)}(\lambda+\eta)
\end{multline*}
where the first map is induced by~\eqref{eqn:can-morph-rep-1}, and the second one by~\eqref{eqn:can-morph-rep-2} (applied in the first two and last two factors).

\begin{lem}
\label{lem:reg-ind}
For any $\mu \in \bY$,
 the functor represented by the ind-object
 \[
 \underset{\lambda \in \bY_+ \cap (w_\circ(\mu) + \bY_+)}{``\varinjlim{}"} \coweyl^{(1)}(\mu-w_\circ(\lambda)) \otimes \coweyl^{(1)}(\lambda)
\]
is given by $V \mapsto \Hom_{\bG^{(1)}}(V,\Ind_{\bT^{(1)}}^{\bG^{(1)}}(\mu))$.
\end{lem}

\begin{proof}
 By Frobenius reciprocity~\cite[Proposition~II.3.4]{jantzen}, for any $V$ in $\Rep(\bG^{(1)})$ we have
 \[
 \Hom_{\bG^{(1)}}(V,\Ind_{\bT^{(1)}}^{\bG^{(1)}}(\mu)) = (V_\mu)^*,
 \]
where $V_\mu$ is the $\mu$-weight space of $V$. Now we have a canonical morphism of $\bB^{(1)}$-modules $\coweyl^{(1)}(\mu-w_\circ(\lambda)) \to \bk_{\bB^{(1)}}(\mu-w_\circ(\lambda))$, and
the morphism $\varphi_{-w_\circ(\lambda)}$ determines a morphism of $\bB^{(1)}$-modules $\weyl^{(1)}(-w_\circ(\lambda)) \to \bk_{\bB^{(1)}}(-w_\circ(\lambda))$. In turn, this morphism defines a highest weight vector in $\weyl^{(1)}(-w_\circ(\lambda))$, hence a lowest weight vector in $\coweyl^{(1)}(\lambda)$, which determines a morphism of $\bT^{(1)}$-modules $\coweyl^{(1)}(\lambda) \to \bk_{\bT^{(1)}}(w_\circ(\lambda))$. Tensoring these morphisms we obtain a morphism of $\bT^{(1)}$-modules $\coweyl^{(1)}(\mu-w_\circ(\lambda)) \otimes \coweyl^{(1)}(\lambda) \to \bk_{\bT^{(1)}}(\mu)$, hence a morphism of $\bG^{(1)}$-modules
\[
\coweyl^{(1)}(\mu-w_\circ(\lambda)) \otimes \coweyl^{(1)}(\lambda) \to \Ind_{\bT^{(1)}}^{\bG^{(1)}}(\mu),
\]
and finally a morphism of functors
\[
\Hom_{\bG^{(1)}}(-,\coweyl^{(1)}(\mu-w_\circ(\lambda)) \otimes \coweyl^{(1)}(\lambda)) \to \Hom_{\bG^{(1)}}(-, \Ind_{\bT^{(1)}}^{\bG^{(1)}}(\mu)).
\]
We deduce a morphism of functors
\[
\varinjlim_\lambda \Hom_{\bG^{(1)}}(-,\coweyl^{(1)}(\mu-w_\circ(\lambda)) \otimes \coweyl^{(1)}(\lambda)) \to \Hom_{\bG^{(1)}}(-, \Ind_{\bT^{(1)}}^{\bG^{(1)}}(\mu)),
\]
and to conclude it suffices to prove that this morphism is an isomorphism.

On the other hand, for $\lambda \in \bY_+ \cap (w_\circ(\mu) + \bY_+)$ we have
 \begin{multline}
 \label{eqn:reg-ind-proof}
  \Hom_{\bG^{(1)}} \bigl( V, \coweyl^{(1)}(\mu-w_\circ(\lambda)) \otimes \coweyl^{(1)}(\lambda) \bigr) \\
  \cong \Hom_{\bG^{(1)}} \bigl( \weyl^{(1)}(-w_\circ(\lambda)), V^* \otimes \coweyl^{(1)}(\mu-w_\circ(\lambda)) \bigr).
 \end{multline}
Assume now that $\lambda$ is large enough that $\nu + \mu-w_\circ(\lambda)$ is dominant for any weight $\nu$ of $V^*$.  Then, by the tensor identity~\cite[Proposition~I.3.6]{jantzen} and Kempf's vanishing theorem~\cite[Proposition~II.4.5]{jantzen}, the module $V^* \otimes \coweyl^{(1)}(\mu-w_\circ(\lambda))$ admits a finite filtration with associated graded
\[
 \bigoplus_\nu (V^*)_\nu \otimes \coweyl^{(1)}(\mu+\nu-w_\circ(\lambda)).
\]
In this case, in view of~\eqref{eqn:Ext-vanishing-Rep} the space in~\eqref{eqn:reg-ind-proof} identifies with
\[
 (V^*)_{-\mu} \otimes \Hom_{\bG^{(1)}} \bigl( \weyl^{(1)}(-w_\circ(\lambda)), \coweyl^{(1)}(-w_\circ(\lambda)) \bigr) \cong (V_\mu)^*,
\]
which concludes the proof.
\end{proof}

\begin{rmk}
\phantomsection
\label{rmk:Ind-coweyl}
\begin{enumerate}
\item
\label{it:reg-ind-B}
Below we will also require a variant of Lemma~\ref{lem:reg-ind}, which follows from similar arguments (using also Frobenius reciprocity for the induction from $\bB^{(1)}$ to $\bG^{(1)}$). Consider, for a given $\mu \in \bY$, the ind-object
 \[
 \underset{\lambda \in \bY_+ \cap (w_\circ(\mu) + \bY_+)}{``\varinjlim{}"} \bk_{\bB^{(1)}}(\mu-w_\circ(\lambda)) \otimes \coweyl^{(1)}(\lambda)
\]
in $\Rep(\bB^{(1)})$, where for $\eta \in \bY_+$ the transition morphism
\[
 \bk_{\bB^{(1)}}(\mu-w_\circ(\lambda)) \otimes \coweyl^{(1)}(\lambda) \to \bk_{\bB^{(1)}}(\mu-w_\circ(\lambda+\eta)) \otimes \coweyl^{(1)}(\lambda+\eta)
\]
is the composition
\begin{multline*}
 \bk_{\bB^{(1)}}(\mu-w_\circ(\lambda)) \otimes \coweyl^{(1)}(\lambda) \to \bk_{\bB^{(1)}}(\mu-w_\circ(\lambda)) \otimes \coweyl^{(1)}(-w_\circ(\eta)) \otimes \coweyl^{(1)}(\eta) \otimes \coweyl^{(1)}(\lambda) \\
 \to \bk_{\bB^{(1)}}(\mu-w_\circ(\lambda)) \otimes \bk_{\bB^{(1)}}(-w_\circ(\eta)) \otimes \coweyl^{(1)}(\eta) \otimes \coweyl^{(1)}(\lambda) \\
 \to \bk_{\bB^{(1)}}(\mu-w_\circ(\lambda+\eta)) \otimes \coweyl^{(1)}(\lambda+\eta)
\end{multline*}
where the first map is induced by~\eqref{eqn:can-morph-rep-1}, the second one by the natural morphism $\coweyl^{(1)}(-w_\circ(\eta)) \to \bk_{\bB^{(1)}}(-w_\circ(\eta))$, and the third one by~\eqref{eqn:can-morph-rep-2}. Then the functor represented by this ind-object is given by
\[
V \mapsto \Hom_{\bB^{(1)}}(V, \Ind_{\bT^{(1)}}^{\bB^{(1)}}(\mu)).
\]
\item
\label{it:Ind-coweyl-stab}
Given a property depending on a coweight $\lambda$ living in a subset $\Lambda \subset \bY$, we will say that this property holds \emph{when $\lambda$ is large enough} if there exists $\nu \in \bY$ such that the property holds for any $\lambda \in \Lambda \cap (\nu + \bY_+)$.
The proof of Lemma~\ref{lem:reg-ind} shows that given $V,V'$ in $\Rep(\bG^{(1)})$, the vector space
\[
\Hom_{\bG^{(1)}}(V, V' \otimes \coweyl^{(1)}(\mu-w_\circ(\lambda)) \otimes \coweyl^{(1)}(\lambda))
\]
does not depend on $\lambda$ (up to canonical isomorphism) if $\lambda$ is large enough.
\end{enumerate}
\end{rmk}

Since $\scO(\bG^{(1)})$ is an algebra, we have multiplication morphisms
\[
 \Ind_{\bT^{(1)}}^{\bG^{(1)}}(\mu) \otimes \Ind_{\bT^{(1)}}^{\bG^{(1)}}(\nu) \to \Ind_{\bT^{(1)}}^{\bG^{(1)}}(\mu+\nu)
\]
for any $\mu,\nu \in \bY$. Via the identification of Lemma~\ref{lem:reg-ind}, this morphism is induced by the collection of natural morphisms
\begin{multline*}
 \coweyl^{(1)}(\mu-w_\circ(\lambda)) \otimes \coweyl^{(1)}(\lambda) \otimes \coweyl^{(1)}(\nu-w_\circ(\lambda')) \otimes \coweyl^{(1)}(\lambda') \\
 \to \coweyl^{(1)}(\mu+\nu-w_\circ(\lambda+\lambda')) \otimes \coweyl^{(1)}(\lambda+\lambda')
\end{multline*}
(for $\lambda,\lambda'$ dominant and sufficiently large) induced by~\eqref{eqn:can-morph-rep-2}.

\begin{rmk}
\label{rmk:morphisms-lift}
The datum of a collection of morphisms $(\varphi_\lambda : \lambda \in \bY_+)$ as above is equivalent to the datum of a lift of the longest element $w_\circ$ of the Weyl group of $(\bG^{(1)},\bT^{(1)})$ to $\mathrm{N}_{\bG^{(1)}}(\bT^{(1)})$. 

Indeed, assume we are given a collection of morphisms as above.
For any $\mu \in \bY$, setting $\lambda'=\mu-w_\circ(\lambda)$ we obtain an isomorphism
\begin{multline*}
 \underset{\lambda \in \bY_+ \cap (w_\circ(\mu) + \bY_+)}{``\varinjlim{}"} \coweyl^{(1)}(\mu-w_\circ(\lambda)) \otimes \coweyl^{(1)}(\lambda) \\
 = \underset{\lambda' \in \bY_+ \cap (\mu + \bY_+)}{``\varinjlim{}"} \coweyl^{(1)}(w_\circ(\mu)-w_\circ(\lambda')) \otimes \coweyl^{(1)}(\lambda').
\end{multline*}
By Lemma~\ref{lem:reg-ind}, this identification provides, for any $V \in \Rep(\bG^{(1)})$, an isomorphism $V_\mu \simto V_{w_\circ(\mu)}$. One can check that these isomorphisms provide a tensor automorphism of the forgetful functor from $\Rep(\bG^{(1)})$ to finite-dimensional $\bk$-vector spaces hence, by Tannakian formalism, an element in $\bG^{(1)}$. The behaviour of this element with respect to weight spaces shows that this element is a lift of $w_\circ$ to $\mathrm{N}_{\bG^{(1)}}(\bT^{(1)})$. 

Conversely, recall that by construction the module $\coweyl^{(1)}(\lambda)$ comes with a canonical vector of weight $\lambda$. Given a lift of $w_\circ$ to $\mathrm{N}_{\bG^{(1)}}(\bT^{(1)})$, we obtain a canonical vector of weight $w_\circ(\lambda)$ in each $\coweyl^{(1)}(\lambda)$, and then a canonical vector of weight $\lambda$ in each $\weyl^{(1)}(\lambda)$. There exists then a unique morphism of $\bG^{(1)}$-modules from $\weyl^{(1)}(\lambda)$ to $\coweyl^{(1)}(\lambda)$ sending the highest weight vector of the former to the highest weight vector of the latter, which provides a construction of a morphism $\varphi_\lambda$ as above.
\end{rmk}

\subsection{Baby co-Verma modules as ind-\texorpdfstring{$\bG$}{G}-modules}
\label{ss:baby-coV-Rep}

We now consider the preimage $\bB_1\bT$ of $\bT^{(1)}$ under the Frobenius morphism $\bB \to \bB^{(1)}$.
Following the conventions of~\cite[\S II.9.1]{jantzen}, for $\nu \in \bY$ we consider the baby co-Verma module
\[
 \widehat{\mathsf{Z}}'(\nu) = \Ind_{\bB_1\bT}^{\bG_1\bT}(\nu),
\]
where on the right-hand side $\nu$ is seen as a character of $\bB$, and hence of $\bB_1 \bT$ by restriction.
In order to describe the image of this $\bG_1\bT$-module under the equivalence of~\S\ref{ss:Rep-G1T}, we need to describe, for any $\lambda \in \bY$, the representation
\[
 \Ind_{\bG_1\bT}^{\bG}(\widehat{\mathsf{Z}}'(\nu) \otimes \bk_{\bT^{(1)}}(-\lambda)) = \Ind_{\bG_1\bT}^{\bG}(\widehat{\mathsf{Z}}'(\nu-p\lambda)).
\]

For any $\mu \in \bY_+$, we will denote by $\coweyl(\mu)$ the induced $\bG$-module with highest weight $\mu$. Note that given $\mu,\mu' \in \bY_+$, there exists a canonical morphism
\begin{equation}
\label{eqn:can-morph-rep-3}
  \coweyl(\mu') \otimes \bigl( \Fr^* \coweyl^{(1)}(\mu) \bigr) \to \coweyl(p\mu+\mu')
\end{equation}
sending the tensor product of the canonical highest weight vectors on the left-hand side to the canonical highest weight vector on the right-hand side.
Given $\mu \in \bY$, we consider the ind-object
\[
 \underset{\lambda \in \bY_+ \cap \frac{1}{p}(w_\circ(\mu)+\bY_+)}{``\varinjlim{}"} \coweyl(\mu-pw_\circ(\lambda))  \otimes \Fr^*(\coweyl^{(1)}(\lambda))
\]
in $\Rep(\bG)$, where the transition morphisms are given by the compositions
\begin{multline*}
  \coweyl(\mu-pw_\circ(\lambda)) \otimes \Fr^*(\coweyl^{(1)}(\lambda)) \\
 \to \coweyl(\mu-pw_\circ(\lambda)) \otimes \Fr^*(\coweyl^{(1)}(-w_\circ(\eta))) \otimes \Fr^*(\coweyl^{(1)}(\eta)) \otimes \Fr^*(\coweyl^{(1)}(\lambda)) \\
 \to  \coweyl(\mu-pw_\circ(\lambda+\eta)) \otimes \Fr^*(\coweyl^{(1)}(\lambda+\eta))
\end{multline*}
for $\eta \in \bY_+$, where the first map is induced by~\eqref{eqn:can-morph-rep-1} and the second one by~\eqref{eqn:can-morph-rep-2} and~\eqref{eqn:can-morph-rep-3}.

\begin{lem}
 For any $\mu \in \bY$, the functor represented by the ind-object
 \[
 \underset{\lambda \in \bY_+ \cap \frac{1}{p}(w_\circ(\mu)+\bY_+)}{``\varinjlim{}"} \coweyl(\mu-pw_\circ(\lambda)) \otimes \Fr^*(\coweyl^{(1)}(\lambda))
\]
is given by $V \mapsto \Hom_{\bG}(V,\Ind_{\bG_1\bT}^{\bG}(\widehat{\mathsf{Z}}'(\mu)))$.
\end{lem}

\begin{proof}
We will consider the functor $(-)_{\bB_1}$ of ``coinvariants'' for finite-dimensional representations of the Frobenius kernel $\bB_1$ of $\bB$, given by
\[
V_{\bB_1} := ((V^*)^{\bB_1})^*.
\]
It is easily seen that this functor is left adjoint to the inclusion functor from finite-dimensional $\bk$-vector spaces to $\Rep(\bB_1)$; it also induces a functor $\Rep(\bB) \to \Rep(\bB^{(1)})$ which is left adjoint to the functor $\Fr^* : \Rep(\bB^{(1)}) \to \Rep(\bB)$.

For $V$ in $\Rep(\bG)$, we observe that by the tensor identity and Frobenius reciprocity we have
\[
\Hom_{\bG} \bigl( V,\coweyl(\mu-pw_\circ(\lambda)) \otimes \Fr^*(\coweyl^{(1)}(\lambda)) \bigr) \cong \Hom_{\bB} \bigl( V, \bk_{\bB}(\mu-pw_\circ(\lambda)) \otimes \Fr^*(\coweyl^{(1)}(\lambda)) \bigr).
\]
We deduce that this space identifies with
\[
\Hom_{\bB^{(1)}} \bigl( (V \otimes \bk_{\bB}(-\mu))_{\bB_1}, \bk_{\bB^{(1)}}(-w_\circ(\lambda)) \otimes \coweyl^{(1)}(\lambda) \bigr).
\]
Using Remark~\ref{rmk:Ind-coweyl}\eqref{it:reg-ind-B}, we deduce an isomorphism
\begin{multline*}
\varinjlim_{\lambda} \Hom_{\bG} \bigl( V,\coweyl(\mu-pw_\circ(\lambda)) \otimes \Fr^*(\coweyl^{(1)}(\lambda)) \bigr) \cong \\
\Hom_{\bB^{(1)}} \bigl( (V \otimes \bk_{\bB}(-\mu))_{\bB_1}, \Ind_{\bT^{(1)}}^{\bB^{(1)}}(\bk) \bigr).
\end{multline*}
Now we have
\begin{multline*}
\Hom_{\bB^{(1)}} \bigl( (V \otimes \bk_{\bB}(-\mu))_{\bB_1}, \Ind_{\bT^{(1)}}^{\bB^{(1)}}(\bk) \bigr) \cong \Hom_{\bB} \bigl( V \otimes \bk_{\bB}(-\mu), \Fr^*(\Ind_{\bT^{(1)}}^{\bB^{(1)}}(\bk)) \bigr) \\
\cong \Hom_{\bB} \bigl( V, \Fr^*(\Ind_{\bT^{(1)}}^{\bB^{(1)}}(\bk)) \otimes \bk_{\bB}(\mu) \bigr),
\end{multline*}
and
\[
\Fr^*(\Ind_{\bT^{(1)}}^{\bB^{(1)}}(\bk)) \otimes \bk_{\bB}(\mu) \cong \Ind_{\bB_1\bT}^{\bB}(\bk) \otimes \bk_{\bB}(\mu) \cong \Ind_{\bB_1\bT}^{\bB}(\mu)
\]
 by the tensor identity; using Frobenius reciprocity and transitivity of induction (see~\cite[\S II.3.5]{jantzen}) we deduce an isomorphism
 \[
\varinjlim_{\lambda} \Hom_{\bG} \bigl( V,\coweyl(\mu-pw_\circ(\lambda)) \otimes \Fr^*(\coweyl^{(1)}(\lambda)) \bigr) \cong
\Hom_{\bG}(V, \Ind_{\bB_1\bT}^{\bG}(\mu)).
\]
Transitivity of induction also implies that $\Ind_{\bG_1\bT}^{\bG}(\widehat{\mathsf{Z}}'(\mu)) \cong \Ind_{\bB_1\bT}^{\bG}(\mu)$, so that this provides the desired isomorphism.
\end{proof}

This lemma shows that for $\nu \in \bY$ the image of $\widehat{\mathsf{Z}}'(\nu)$ under the equivalence of~\S\ref{ss:Rep-G1T} is the $\bY$-graded ind-object with degree-$\lambda$ component given by
 \[
 \underset{\mu}{``\varinjlim{}"} \coweyl(\nu-p\lambda-pw_\circ(\mu)) \otimes \Fr^*(\coweyl^{(1)}(\mu)),
\]
where the transition morphisms are as above.
In these terms, and using the description of Lemma~\ref{lem:reg-ind}, the action of $\scO(\bG^{(1)})$ is induced by the morphisms
\begin{multline*}
 \coweyl(\nu-p\lambda-pw_\circ(\mu)) \otimes \coweyl^{(1)}(\mu) \otimes \coweyl^{(1)}(-\lambda'-w_\circ(\mu')) \otimes \coweyl^{(1)}(\mu') \\
 \to \coweyl(\nu-p(\lambda+\lambda')-pw_\circ(\mu+\mu')) \otimes \coweyl^{(1)}(\mu+\mu')
\end{multline*}
induced by~\eqref{eqn:can-morph-rep-2} and~\eqref{eqn:can-morph-rep-3} for $\lambda,\lambda' \in \bY$ and $\mu,\mu' \in \bY_+$ large enough (where we omit the functor $\Fr^*$ to lighten notation).

\section{Modules over the regular perverse sheaf}
\label{sec:mod-R}

We now come back to the setting of Sections~\ref{sec:combinatorics}--\ref{sec:perv}.

In this section, for any finitary subset $A \subset S_\aff$ we will define and initiate the study of a certain category of ind-objects in $\Perv_{(I^A_\unip, \cX_A)}(\Gr,\bk)$ equipped with additional structures: namely, a grading indexed by $\bY$, and the structure of a right module over a certain algebra ind-object $\cR$ in $\Perv_{\cL^+G}(\Gr,\bk)$. After some preliminaries in Section~\ref{sec:averaging}, in Section~\ref{sec:proj-R-mod} we will see that this category is a finite-length abelian category with enough injectives and projectives, and that it satisfies properties similar to those of the category $\Rep(\bG_1\bT)$ where $\bG$ is a connected reductive algebraic group over $\bk$ (with maximal torus $\bT$) whose Frobenius twist is $G^\vee_\bk$. One can also omit the $\bY$-grading, and obtain a similar theory that is analogous to that of $\bG_1$-modules. This theory will be reviewed in Section~\ref{sec:ungrad-R-mod}.

\subsection{Ind-objects}
\label{ss:ind}

Our constructions will make use of ind-objects in (locally small) categories; for the generalities on this construction we refer to~\cite[Chap.~6]{ks}.
We will repeatedly use the fact that any functor $F : \mathsf{C} \to \mathsf{D}$ extends in a canonical way to a functor $\mathrm{Ind}(\mathsf{C}) \to \mathrm{Ind}(\mathsf{D})$, which will be denoted by the same symbol, see~\cite[Proposition~6.1.9]{ks}. In view of~\cite[Proposition~6.1.12]{ks}, a similar comment applies to bifunctors. Recall also that the category of ind-objects in a locally small abelian category is locally small and abelian, see~\cite[Lemma~6.1.2 and Theorem~8.6.5]{ks}, and that the functor on ind-objects induced by an exact functor is exact, see~\cite[Corollary~8.6.8]{ks}.

Given a category $\mathsf{A}$ and a set $X$, by an $X$-graded object in $\mathsf{A}$ we will mean a collection $A=(A_x : x \in X)$ of objects in $\mathsf{A}$. We will write informally
\[
A = \bigoplus_{x \in X} A_x,
\]
but the symbol ``$\bigoplus$" has no formal meaning here. In particular,
we do \emph{not} assume that only finitely many objects $A_x$ are nonzero, nor that the (possibly infinite) direct sum exists in $\mathsf{A}$.

\subsection{The regular perverse sheaf}
\label{ss:reg-perv-sheaf}

Recall the objects $\cI^\lambda_*$ and $\cI^\lambda_!$ ($\lambda \in \bY_+$) introduced in~\S\ref{ss:Satake-category}. For any $\lambda \in \bY_+$, the natural morphisms for sheaf functors provide a canonical morphism
\[
 \cI^\lambda_! \to \cI^\lambda_*.
\]
Since $\cI_!^\lambda$ has rigid dual $\cI^{-w_\circ(\lambda)}_*$ (see~\eqref{eqn:duals}), this morphism induces a canonical morphism
\begin{equation}\label{eqn:costd-unit}
\IC^0 \to \cI^{\lambda}_* \star^{\cL^+G} \cI^{-w_\circ(\lambda)}_*.
\end{equation}
Next, for $\lambda,\mu \in \bY_+$, since the perverse sheaf $\cI^\lambda_* \star^{\cL^+G} \cI^\mu_*$ is supported on the closure of $\Gr^{\lambda+\mu}$ and has restriction to $\Gr^{\lambda+\mu}$ equal to $\underline{\bk}_{\Gr^{\lambda+\mu}}[\dim(\Gr^{\lambda+\mu})]$, there exists a canonical morphism
\begin{equation}\label{eqn:costd-conv-sph}
\cI^\lambda_* \star^{\cL^+G} \cI^\mu_* \to \cI^{\lambda+\mu}_*.
\end{equation}

Let us endow $\bY_+$ (and any of its subsets) with the preorder such that $\lambda$ is less or equal to $\lambda'$ iff $\lambda'-\lambda$ is dominant. For $\mu \in \bY$ we consider the ind-object
\begin{equation}
\label{eqn:description-Rmu}
\cR_\mu = \underset{\lambda \in \bY_+ \cap (-w_\circ(\mu) + \bY_+)}{``\varinjlim"} \cI^{w_\circ(\mu) + \lambda}_* \star^{\cL^+G} \cI^{-w_\circ(\lambda)}_*
\end{equation}
in $\Perv_{\cL^+ G}(\Gr,\bk)$,
where the transition maps are given by the morphisms
\begin{multline*}
 \cI_*^{w_\circ(\mu)+\lambda} \star^{\cL^+G} \cI_*^{-w_\circ(\lambda)} \to \cI_*^{w_\circ(\mu)+\lambda} \star^{\cL^+G} \cI_*^{\nu} \star^{\cL^+G} \cI_*^{-w_\circ(\nu)} \star^{\cL^+G} \cI_*^{-w_\circ(\lambda)} \\
\to \cI_*^{w_\circ(\mu)+\lambda+\nu} \star^{\cL^+G} \cI_*^{-w_\circ(\lambda+\nu)}
\end{multline*}
for $\nu \in \bY_+$.  Here, the first map comes from~\eqref{eqn:costd-unit}, and the second one from~\eqref{eqn:costd-conv-sph} (applied to the first two and last two factors).

We have an obvious ``unit map''
\begin{equation}
\label{eqn:reg-unit}
\eta: \IC^0 \to \cR_0,
\end{equation}
and for $\mu,\mu' \in \bY$ we have a ``multiplication map'' 
\begin{equation}
\label{eqn:reg-mult}
\cR_\mu \star^{\cL^+G} \cR_{\mu'} \to \cR_{\mu + \mu'}
\end{equation}
obtained
as the limit (over suitable $\lambda, \lambda'$)  of the maps 
\begin{multline*}
\bigl( \cI_*^{w_\circ(\mu)+\lambda} \star^{\cL^+G} \cI_*^{-w_\circ(\lambda)} \bigr) \star^{\cL^+G} \bigl( \cI_*^{w_\circ(\mu')+\lambda'} \star^{\cL^+G} \cI_*^{-w_\circ(\lambda')} \bigr) \\
\to \cI_*^{w_\circ(\mu+\mu')+\lambda+\lambda'} \star^{\cL^+G} \cI_*^{-w_\circ(\lambda+\lambda')}
\end{multline*}
provided by~\eqref{eqn:costd-conv-sph} and the commutativity constraint on the monoidal category $\Perv_{\cL^+G}(\Gr,\bk)$.  
%
The map~\eqref{eqn:reg-mult} satisfies an obvious associativity property, as well as an appropriate compatibility property with~\eqref{eqn:reg-unit}.  Therefore,~\eqref{eqn:reg-unit} and~\eqref{eqn:reg-mult} make the $\bY$-graded ind-object
\[
\cR := \bigoplus_{\mu \in \bY} \cR_\mu
\]
into an algebra object in the category of $\bY$-graded ind-perverse sheaves.  We call $\cR$ the \emph{regular perverse sheaf}.


Recall the autoequivalence $\mathrm{sw}$ of the category $\Perv_{\cL^+G}(\Gr,\bk)$ considered in~\cite[\S 1.2]{projGr1}. Then for any $\lambda \in \bY_+$ we have a canonical isomorphism $\mathrm{sw}(\cI_*^\lambda) \cong \cI_*^{-w_\circ(\lambda)}$. 
These isomorphisms and Lemma~\ref{lem:reg-ind}\footnote{In Section~\ref{sec:mod-R-rep} we have assumed that $\bk$ has positive characteristic; however Lemma~\ref{lem:reg-ind} also holds in characteristic $0$, if $\bG^{(1)}$ is interpreted as an abstract reductive group, without reference to another group $\bG$.} show that $\mathrm{sw}(\cR_\mu)$ corresponds, under (the extension to ind-objects of) the equivalence $\Satake$, to $\Ind_{T^{\vee}_\bk}^{G^\vee_\bk}(-\mu)$, seen as an ind-object in $\Rep(G^\vee_\bk)$. (In this case, we choose as morphism $\varphi_\lambda$ from~\S\ref{ss:reg-ind} the one induced by the canonical morphism $\cI^\lambda_! \to \cI^\lambda_*$.) This justifies our choice of convention for the definition of $\cR$, in view of the 
formulation of the Finkelberg--Mirkovi{\'c} conjecture in~\cite[Conjecture~1.1]{projGr1}.

\subsection{Splitting the unit map in characteristic $0$}
\label{ss:splitting-unit}

For later use (in~\S\ref{ss:more-on-inj}),
in this subsection we show that when $\bk$ has characteristic $0$, the unit map~\eqref{eqn:reg-unit} admits a left inverse, and hence that $\IC^0$ is a direct summand of $\cR_0$.  Recall that in this case the category $\Perv_{\cL^+G}(\Gr,\bk)$ is semisimple, and that the canonical morphism $\cI^\mu_! \to \cI^\mu_*$ is an isomorphism (and both objects can be identified with $\IC^\mu$).  We can therefore rewrite the definition of $\cR_0$ from~\eqref{eqn:description-Rmu} as
\begin{equation}\label{eqn:description-R0-char0}
\cR_0 = \underset{\lambda \in \bY_+}{``\varinjlim"} \, \IC^{\lambda} \star^{\cL^+G} \IC^{-w_\circ(\lambda)}.
\end{equation}
Next, dual to~\eqref{eqn:costd-unit}, we have a ``counit map'' $\epsilon_\lambda: \IC^\lambda \star^{\cL^+G} \IC^{-w_\circ(\lambda)} \to \IC^0$.  These maps are \emph{not} compatible with the transition maps in~\eqref{eqn:description-R0-char0}, so they do not define a map $\cR_0 \to \IC_0$.

However, in the present setting that $\mathrm{char}(\bk)=0$ we can correct this failure of compatibility by introducing the maps
\[
\textstyle
\bar\epsilon_\lambda = \frac{1}{\dim \Satake(\IC^\lambda)} \epsilon_\lambda = \frac{1}{\dim \Satake(\IC^{-w_\circ(\lambda)})} \epsilon_\lambda: \IC^\lambda \star^{\cL^+G} \IC^{-w_\circ(\lambda)} \to \IC^0.
\]
Let $\eta_\lambda: \IC^0 \to \IC^\lambda \star^{\cL^+G} \IC^{-w_\circ(\lambda)}$ be the map defined in~\eqref{eqn:costd-unit}.  The composition $\epsilon_\lambda \eta_\lambda: \IC^0 \to \IC^0$ is given by multiplication by $\dim \Satake(\IC^\lambda)$ (this can easily be seen by considering the analogous unit and counit maps in the category $\Rep(G^\vee_\bk)$), so $\bar\epsilon_\lambda \eta_\lambda = \id$.

We claim that the maps $\bar\epsilon_\lambda$ are compatible with the transition maps in~\eqref{eqn:description-R0-char0}, i.e., that the bottom square in the following diagram commutes:
\[
\begin{tikzcd}
\IC^0 \ar[d, "\eta_\lambda"']  \ar[dr, "\eta_{\lambda+\nu}"] \\
\IC^\lambda \star^{\cL^+G} \IC^{-w_\circ(\lambda)} \ar[d, "\bar\epsilon_\lambda"'] \ar[r] &
  \IC^{\lambda+\nu} \star^{\cL^+G} \IC^{-w_\circ(\lambda+\nu)} \ar[d, "\bar\epsilon_{\lambda+\nu}"]  \\
\IC^0 \ar[r, equal] & \IC^0
\end{tikzcd}
\]
Since $\dim \Hom(\IC^\lambda \star^{\cL^+G} \IC^{-w_\circ(\lambda)}, \IC^0) = 1$, the commutativity of the bottom square can be checked after composition with the unit maps in the top part of the diagram.  Commutativity follows from the observation that $\bar\epsilon_\lambda \eta_\lambda = \bar\epsilon_{\lambda+\nu} \eta_{\lambda+\nu} = \id$.  Together, the collection of maps $\bar\epsilon_\lambda$ define a map of ind-perverse sheaves
\[
\bar\epsilon: \cR_0 \to \IC^0
\]
satisfying $\bar\epsilon \circ \eta = \id$.  

\subsection{Graded \texorpdfstring{$\cR$}{R}-modules}
\label{ss:graded-R-modules}

A \emph{$\bY$-graded right $\cR$-module} is, by definition, a $\bY$-graded ind-object
\[
\cF = \bigoplus_{\lambda \in \bY} \cF_\lambda
\]
in $\Perv_{(I^A_\unip, \cX_A)}(\Gr,\bk)$,
along with a collection of maps
\[
\cF_\lambda \star^{\cL^+G} \cR_\mu \to \cF_{\lambda+\mu}
\]
for $\lambda,\mu \in \bY$, satisfying obvious unit and associativity axioms. 
Let
\[
\Mod^\bY_{(I^A_\unip, \cX_A)}(\cR)
\]
denote the category of $\bY$-graded right $\cR$-modules.  This is an abelian category.  In the special case where $A = \varnothing$, we simplify this notation to $\Mod^\bY_{I_\unip}(\cR)$.


For any $\nu \in \bY$, there is a \emph{shift-of-grading} functor on $\Mod^\bY_{(I^A_\unip, \cX_A)}(\cR)$, denoted by $\cF \mapsto \cF \langle \nu \rangle$ and defined by
\[
(\cF \langle \nu \rangle)_\mu = \cF_{\mu-\nu},
\]
with the $\cR$-module structure unchanged. Of course we have $\langle \nu \rangle \circ \langle \nu' \rangle = \langle \nu+\nu' \rangle$; in particular, $\langle \nu \rangle$ is an autoequivalence with inverse $\langle -\nu \rangle$.

Given a perverse sheaf $\cF \in \Perv_{(I^A_\unip, \cX_A)}(\Gr,\bk)$, we can construct a graded $\cR$-module by the formula
\[
\Phi^A(\cF) 
= \bigoplus_{\mu \in \bY} \cF \star^{\cL^+G} \cR_\mu,
\]
called the \emph{free $\cR$-module on $\cF$}.  This construction defines an exact functor
\[
\Phi^A: \Perv_{(I^A_\unip, \cX_A)}(\Gr,\bk) \to \Mod^{\bY}_{{(I^A_\unip, \cX_A)}}(\cR).
\]
In the case where $A = \varnothing$, we usually omit it from the notation and write
\[
\Phi: \Perv_{I_\unip}(\Gr,\bk) \to \Mod^\bY_{I_\unip}(\cR).
\]

More generally, a \emph{free graded $\cR$-module of finite type} is, by definition, a finite direct sum of objects of the form $\Phi^A(\cF)\la \nu\ra$, where $\cF \in \Perv_{(I^A_\unip, \cX_A)}(\Gr,\bk)$ and $\nu \in \bY$. Note that $\Phi^A$ is faithful; in fact by exactness this follows from the fact that it kills no nonzero object, which itself follows from the observation that the morphism $\cG \to \Phi^A(\cG)_0$ induced by the unit morphism $\eta$ (see~\eqref{eqn:reg-unit}) is injective, since $\eta$ is injective and the functor $\cG \star^{\cL^+G} (-)$ is exact.

Morphisms from free modules can be easily computed using the following lemma.

\begin{lem}
\label{lem:Rfree-hom}
For $\cF \in \Perv_{(I^A_\unip, \cX_A)}(\Gr,\bk)$
and $\cM \in \Mod^\bY_{(I^A_\unip, \cX_A)}(\cR)$, there is a natural isomorphism
\[
\Hom_{\Mod^\bY_{(I^A_\unip, \cX_A)}(\cR)}(\Phi^A(\cF), \cM) \cong \Hom_{\Perv_{(I^A_\unip, \cX_A)}(\Gr,\bk)}(\cF, \cM_{0}).
\]
\end{lem}

\begin{proof}
Consider the unit map $\eta: \IC^0 \to \cR_0$.  Composition with
\[
\id \star^{\cL^+G} \eta: \cF \to \cF \star^{\cL^+G} \cR_0 = \Phi(\cF)_0
\]
defines a map
\[
\Hom_{\Mod^\bY_{(I^A_\unip, \cX_A)}(\cR)}(\Phi(\cF), \cM) \to \Hom_{\Perv_{(I^A_\unip, \cX_A)}(\Gr,\bk)}(\cF, \cM_{0}).
\]
On the other hand, given a map $\phi: \cF \to \cM_{0}$ of (ind-)perverse sheaves, one can consider for any $\mu$ the following composition, which defines a map of graded $\cR$-modules:
\[
(\Phi^A(\cF))_\mu = \cF \star^{\cL^+G} \cR_{\mu} \xrightarrow{\phi \star^{\cL^+G} \id} \cM_{0} \star^{\cL^+G} \cR_{\mu} \to \cM_\mu.
\]
(Here, the second map comes from the $\cR$-module structure on $\cM$.)  It is straightforward to check that these two constructions are inverse to each other.
\end{proof}

An object of $\Mod^\bY_{(I^A_\unip, \cX_A)}(\cR)$ is said to be \emph{finitely generated} if it is a quotient of a free graded $\cR$-module of finite type.  Let
\[
\modf^\bY_{(I^A_\unip, \cX_A)}(\cR) =
\begin{array}{c}
\text{the full subcategory of $\Mod^\bY_{(I^A_\unip, \cX_A)}(\cR)$ consisting} \\
\text{of finitely generated modules.}
\end{array}
\]
At the moment, it is not clear that $\modf^\bY_{(I^A_\unip, \cX_A)}(\cR)$ is an abelian category.  This will be established later: see Theorem~\ref{thm:proj-Hecke-Whit}.

For any $G^\vee_\bk$-module $V$ and $\lambda \in \bY=X^*(T^\vee_\bk)$, we will denote by $V_\lambda$ the $\lambda$-weight space in $V$.

\begin{lem}
\label{lem:frobtwist-free}
For any $\cF \in \Perv_{(I^A_\unip, \cX_A)}(\Gr,\bk)$ and $\cG \in \Perv_{\cL^+G}(\Gr,\bk)$, the object $\Phi^A(\cF \star^{\cL^+G} \cG)$ is isomorphic to a finite direct sum of objects of the form $\Phi^A(\cF)\la \nu\ra$.  More specifically, we have a canonical isomorphism
\[
\Phi^A(\cF \star^{\cL^+G} \cG) \cong \bigoplus_{\nu \in \bY} \Satake(\cG)_{w_\circ(\nu)} \otimes \Phi^A(\cF)\la -\nu\ra.
\]
\end{lem}


\begin{proof}
To prove the lemma it suffices to provide canonical isomorphisms
\[
 \cG \star^{\cL^+G} \cR_\mu \cong \bigoplus_\nu \Satake(\cG)_{w_\circ(\nu)} \otimes \cR_{\mu+\nu}
\]
for any $\mu \in \bY$. Now if $V=\Satake(\mathrm{sw}(\cG))$, we have 
\[
 V \otimes \mathrm{Ind}_{T^\vee_\bk}^{G^\vee_\bk}(-\mu) \cong \mathrm{Ind}_{T^\vee_\bk}^{G^\vee_\bk}(V \otimes \bk_{T^\vee_\bk}(-\mu)) \cong \bigoplus_\nu V_\nu \otimes \mathrm{Ind}_{T^\vee_\bk}^{G^\vee_\bk}(-\mu+\nu),
\]
where on the right-hand side the action of $G^\vee_\bk$ is on each $\mathrm{Ind}_{T^\vee_\bk}^{G^\vee_\bk}(-\mu+\nu)$. Applying $\mathrm{sw} \circ \Satake^{-1}$, we deduce an isomorphism
\[
\cG \star^{\cL^+G} \cR_\mu \cong \bigoplus_\nu V_\nu \otimes \cR_{\mu-\nu}.
\]
Finally, by~\cite[Proposition~VI.12.1]{fs} we have $V_\nu = \Satake(\cG)_{-w_\circ(\nu)}$ for any $\nu \in \bY$, which finishes the proof.
%
%
\end{proof}

\subsection{Simple \texorpdfstring{$\cR$}{R}-modules}
\label{ss:simple-R-modules}

Recall the subset ${}^A W_\ext^\res \subset {}^A W_\ext$ introduced in~\S\ref{ss:coset-representatives}, and that (by definition)
for any $w \in {}^A W_\ext$ there exist $y \in {}^A W_\ext^\res$ and $\lambda \in \bY$ such that $w=yt_\lambda$. Choosing such $y$ and $\lambda$ we define a graded $\cR$-module $\LGT^A_w \in \Mod^\bY_{(I^A_\unip, \cX_A)}(\cR)$ by
\[
\LGT^A_w := \Phi^A(\LGr^A_y)\la -\lambda\ra.
\]
Here the elements $y$ and $\lambda$ are not unique, but if $yt_\lambda = y' t_{\lambda'}$, then the pairs $(y,\lambda)$ and $(y',\lambda')$ are related by the relations $y' = yt_\nu$, $\lambda' = \lambda - \nu$, for some $\nu \in \bY$ orthogonal to all roots.  Using the ``componentwise'' description
\[
(\LGT^A_w)_\mu = \LGr^A_y \star^{\cL^+G} \cR_{\lambda+\mu}
\qquad\text{for any $\mu \in \bY$,}
\]
along with the fact that for $\nu$ as above $\IC^\nu$ is the sky-scraper sheaf at $L_\nu$, so that we have $\LGr^A_{yt_\nu} \cong \LGr^A_y \star^{\cL^+G} \IC^\nu$ and $\IC^\nu \star^{\cL^+G} \cR_\mu \cong \cR_{\nu+\mu}$, we see that $\LGT^A_w$ is well defined. To remedy this non-uniqueness issue, it is sometimes convenient to choose a subset $({}^A W_\ext^\res)' \subset {}^A W_\ext^\res$ of representatives for the (free) action of the subgroup of $\bY$ consisting of elements orthogonal to all roots; then any element of ${}^A W_\ext$ can be written \emph{uniquely} as a product $yt_\lambda$ with $y \in ({}^A W_\ext^\res)'$ and $\lambda \in \bY$. It is clear from the definition that for $w \in {}^A W_\ext$ and $\lambda \in \bY$ we have
\begin{equation}
\label{eqn:simple-translation}
\LGT^A_{wt_\lambda} = \LGT^A_w \langle -\lambda \rangle.
\end{equation}

We will see in Theorem~\ref{thm:R-simple} below that these are the simple objects in the abelian category $\Mod^\bY_{(I^A_\unip, \cX_A)}(\cR)$.  Anticipating this, we define the category of \emph{$\cR$-modules of finite length} to be
\[
\Mod^{\bY}_{(I^A_\unip, \cX_A)}(\cR)^\flen := \begin{array}{c}
\text{the full subcategory of $\Mod^\bY_{(I^A_\unip, \cX_A)}(\cR)$} \\
\text{consisting of objects that admit a finite} \\
\text{filtration with subquotients of the form $\LGT^A_w$.}
\end{array}
\]

\begin{lem}
\label{lem:Rfree-filt}
For any (ordinary) perverse sheaf $\cF \in \Perv_{(I^A_\unip, \cX_A)}(\Gr,\bk)$, the free module $\Phi^A(\cF)$ lies in $\Mod^\bY_{(I^A_\unip, \cX_A)}(\cR)^\flen$.
\end{lem}

\begin{proof}
Any composition series for $\cF$ gives rise to a filtration of $\Phi^A(\cF)$ whose subquotients are of the form $\Phi^A(\LGr^A_x)$ with $x \in {}^A W^S_\ext$.  By~\eqref{eqn:WS-Wres-Whit}, Theorem~\ref{thm:geometric-Steinberg}\eqref{it:thm-Steinberg} and Lemma~\ref{lem:frobtwist-free}, each such $\Phi^A(\LGr^A_x)$ is a finite direct sum of objects of the form $\Phi^A(\LGr^A_y)\la \nu\ra$ with $y \in {}^A W^\res_\ext$ and $\nu \in \bY$.
\end{proof}

\begin{lem}\label{lem:Rsimple-hom}
Let $w, v\in {}^A W_\ext$.  We have
\[
\dim \Hom_{\Mod^\bY_{(I^A_\unip, \cX_A)}(\cR)}(\LGT^A_w, \LGT^A_v)
=
\begin{cases}
1 & \text{if $w = v$;} \\
0 & \text{otherwise.}
\end{cases}
\]
In particular, any nonzero endomorphism of $\LGT^A_w$ is an automorphism.
\end{lem}

\begin{proof}
Let us choose a subset $({}^A W_\ext^\res)' \subset {}^A W_\ext^\res$ as above.
By unwinding the definitions and using Lemma~\ref{lem:Rfree-hom}, we see that this lemma is equivalent to the claim that for $x, y \in ({}^A W^\res_\ext)'$ and $\mu \in \bY$, we have
\[
\dim \Hom_{\Perv_{(I^A_\unip, \cX_A)}(\Gr,\bk)}(\LGr^A_x, \LGr^A_y \star^{\cL^+G} \cR_{\mu})
=
\begin{cases}
1 & \text{if $x = y$ and $\mu = 0$;} \\
0 & \text{otherwise.}
\end{cases}
\]
To compute this $\Hom$-group, we must study
\begin{multline*}
\varinjlim_\lambda \Hom(\LGr^A_x, \LGr^A_y \star^{\cL^+G} \cI^{w_\circ(\mu)+\lambda}_* \star^{\cL^+G} \cI^{-w_\circ(\lambda)}_*)
\cong \\
\varinjlim_\lambda \Hom(\LGr^A_x \star^{\cL^+G} \cI^{\lambda}_!, \LGr^A_y \star^{\cL^+G} \cI^{w_\circ(\mu)+\lambda}_*),
\end{multline*}
where the isomorphism uses~\eqref{eqn:duals}. Here, by Theorem~\ref{thm:geometric-Steinberg}\eqref{it:thm-Steinberg} and exactness of $\star^{\cL^+G}$, we can obtain a composition series of $\LGr^A_x \star^{\cL^+G} \cI^{\lambda}_!$ by choosing a composition series of $\cI^{\lambda}_!$ and then convolving with $\LGr^A_x$, and likewise for $\LGr^A_y \star^{\cL^+G} \cI^{w_\circ(\mu)+\lambda}_*$. From this description we see that
if $x \ne y$, the objects $\LGr^A_x \star^{\cL^+G} \cI^{\lambda}_!$ and $\LGr^A_y \star^{\cL^+G} \cI^{w_\circ(\mu)+\lambda}_*$ have no composition factor in common, so that
\[
\Hom(\LGr^A_x \star^{\cL^+G} \cI^{\lambda}_!, \LGr^A_y \star^{\cL^+G} \cI^{w_\circ(\mu)+\lambda}_*)=0
\]
for any $\lambda$.
Assume now that $x = y$.  Then, by Theorem~\ref{thm:geometric-Steinberg}\eqref{it:thm-Steinberg-ff}, the $\Hom$-groups above can be identified with
\[
\varinjlim_\lambda \Hom(\cI^{\lambda}_!, \cI^{w_\circ(\mu)+\lambda}_*).
\]
It is easily seen that this $\Hom$-group is $1$-dimensional if $\mu = 0$, and vanishes otherwise.
\end{proof}

\begin{cor}\label{cor:Rfree-surj}
Suppose $\cM \in \Mod^\bY_{(I^A_\unip, \cX_A)}(\cR)^\flen$. For any $w \in {}^A W_\ext$, any nonzero morphism $\cM \to \LGT^A_w$ is surjective.
\end{cor}

\begin{proof}
This follows from Lemma~\ref{lem:Rsimple-hom} by induction on the length of the given filtration of $\cM$.
\end{proof}

We can now prove the properties of the objects $\LGT^A_w$ announced above.

\begin{thm}
\label{thm:R-simple}
For $w \in {}^A W_\ext$, the object $\LGT^A_w$ is a simple object in the abelian category $\Mod^\bY_{(I^A_\unip, \cX_A)}(\cR)$.  Moreover, the assignment $w \mapsto \LGT^A_w$ gives a bijection
\[
{}^A W_\ext \simto
\left\{
\begin{array}{c}
\text{isomorphism classes of simple} \\
\text{objects in $\Mod^\bY_{(I^A_\unip, \cX_A)}(\cR)$}
\end{array}
\right\}.
\]
\end{thm}

\begin{proof}
We begin by showing that any object $\cM \in \Mod^\bY_{(I^A_\unip, \cX_A)}(\cR)$ admits a nonzero morphism from a free $\cR$-module.  Choose some nonzero graded component $\cM_\mu$ in $\cM$.  As an ind-perverse sheaf, $\cM_\mu$ is an inductive limit of ordinary perverse sheaves, say $\cM_\mu \cong ``\varinjlim_i{}" \cF_i$.  Choose some term $\cF_i$ in this limit such that the natural map $\cF_i \to \cM_\mu$ is nonzero.  Via Lemma~\ref{lem:Rfree-hom}, we obtain a nonzero map
\begin{equation}\label{eqn:free-map}
\phi: \Phi^A(\cF_i)\la \mu\ra \to \cM.
\end{equation}

We will now show that each $\LGT^A_w$ is simple.  If not, there is some nonzero proper subobject $\cM \subset \LGT^A_w$.  Composing with a nonzero map as in~\eqref{eqn:free-map}, we obtain a nonzero, nonsurjective map $\Phi^A(\cF_i)\la \mu\ra \to \LGT^A_w$.  In view of Lemma~\ref{lem:Rfree-filt}, this contradicts Corollary~\ref{cor:Rfree-surj}.


Next, we will show that every simple object in $\Mod^\bY_{(I^A_\unip, \cX_A)}(\cR)$ is isomorphic to some $\LGT_w$.  Let $\cM \in \Mod^\bY_{(I^A_\unip, \cX_A)}(\cR)$ be simple, and choose a nonzero map as in~\eqref{eqn:free-map}; this map is necessarily surjective.  But by Lemma~\ref{lem:Rfree-filt}, we already know that $\Phi^A(\cF_i)\la \mu\ra$ has a composition series whose terms are of the form $\LGT^A_w$ with $w \in {}^A W_\ext$, so $\cM$ must be isomorphic to one of these composition factors, proving the desired claim.

Finally, the fact that $\LGT^A_w \not\cong \LGT^A_y$ if $w \ne y$ is immediate from Lemma~\ref{lem:Rsimple-hom}.  
\end{proof}

As a consequence of Theorem~\ref{thm:R-simple}, we see that $\Mod^\bY_{(I^A_\unip, \cX_A)}(\cR)^\flen$ is stable under subquotients in $\Mod^\bY_{(I^A_\unip, \cX_A)}(\cR)$. In particular, this category is abelian, and by construction every object in this category has finite length. Using Lemma~\ref{lem:Rfree-filt} we also see that
\begin{equation}
\label{eqn:inclusion-mod}
\modf^\bY_{(I^A_\unip, \cX_A)}(\cR) \subset \Mod^\bY_{(I^A_\unip, \cX_A)}(\cR)^\flen,
\end{equation}
but we reiterate that for the moment, we do not yet know whether $\modf^\bY_{(I^A_\unip, \cX_A)}(\cR)$ is an abelian category, with one exception: when $A = S$, we have the following result.

\begin{prop}
\label{prop:IW-R-semisimple}
The category $\Mod^\bY_{(I^S_\unip, \cX_S)}(\cR)$ is canonically equivalent to the category of (all) algebraic $T^\vee_\bk$-modules; in particular, it is semisimple. The subcategories $\modf^\bY_{(I^S_\unip, \cX_S)}(\cR)$ and $\Mod^\bY_{(I^S_\unip, \cX_S)}(\cR)^\flen$ coincide, and are equivalent (via the equivalence for $\Mod^\bY_{(I^S_\unip, \cX_S)}(\cR)$) to the category of finite-dimensional algebraic $T^\vee_\bk$-modules, i.e.~to 
the category of finite-dimensional $\bY$-graded $\bk$-vector spaces.
\end{prop}

\begin{proof}
By the main result of~\cite{bgmrr}, the category $\Perv_{(I^S_\unip, \cX_S)}(\Gr,\bk)$ is equivalent to the Satake category $\Perv_{\cL^+ G}(\Gr,\bk)$, which is itself equivalent to the category $\Rep(G^\vee_\bk)$ via the geometric Satake equivalence, see~\S\ref{ss:Satake-category}. Via the latter equivalence, $\cR$ corresponds by definition to (the ind-object represented by) the algebra $\scO(G^\vee_\bk)$ (for the left regular representation structure), with the $\bY$-grading coming from the action of $T^\vee_\bk$ induced by multiplication on the right. In view of Remark~\ref{rmk:AG-G^1}\eqref{it:AG-G^1}, the category $\Mod^\bY_{(I^S_\unip, \cX_S)}(\cR)$ therefore identifies with the category of (all) algebraic $T^\vee_\bk$-modules, and both $\modf^\bY_{(I^S_\unip, \cX_S)}(\cR)$ and $\Mod^\bY_{(I^S_\unip, \cX_S)}(\cR)^\flen$ identify with the subcategory of finite-dimensional modules.
\end{proof}


We conclude this subsection with a few technical consequences of the results above that will be required later.

\begin{lem}\label{lem:Phi-order}
Let $w \in {}^A W^S_\ext$.  If $z \in {}^A W_\ext$ is such that $\LGT^A_z$ is a composition factor of $\Phi^A(\LGr^A_w)$, then $z \preceq w$.
\end{lem}

\begin{proof}
Write $w$ as $w = xt_{w_\circ(\mu)}$ with $x \in {}^A W^\res_\ext$ and $\mu \in \bY_+$.   By Theorem~\ref{thm:geometric-Steinberg}\eqref{it:thm-Steinberg}, Lemma~\ref{lem:frobtwist-free} and~\eqref{eqn:simple-translation}, $\Phi^A(\LGr^A_w)$ is a direct sum of objects $\LGT^A_{xt_\nu}$ where $\nu$ is a weight of $\Satake(\IC^\mu)$. These weights are such that $\mu-w_\circ(\nu)$ is a sum of positive roots, so that by Lemma~\ref{lem:per-order-weights} we have $xt_\nu \preceq xt_{w_\circ(\mu)}$, i.e.~$xt_\nu \preceq w$, as desired.
\end{proof}


\begin{lem}
\label{lem:R-findim}
For $\cF, \cG \in \Mod^\bY_{(I^A_\unip, \cX_A)}(\cR)^\flen$, we have
\[
\dim \Hom_{\Mod^\bY_{(I^A_\unip, \cX_A)}(\cR)}(\cF,\cG) < \infty.
\]
In particular, $\Mod^\bY_{(I^A_\unip, \cX_A)}(\cR)^\flen$ is Krull--Schmidt in the sense of~\cite[\S A.1]{cyz}.
\end{lem}

\begin{proof}
The first claim follows from Lemma~\ref{lem:Rsimple-hom}. The second claim follows by~\cite[Remark~A.2]{cyz}.
\end{proof}

In the following statement we use the terminology introduced in Remark~\ref{rmk:Ind-coweyl}\eqref{it:Ind-coweyl-stab}.

\begin{lem}
\label{lem:hom-stabilize}
Let $\cF, \cG \in \Perv_{(I^A_\unip, \cX_A)}(\Gr,\bk)$, and let $\mu \in \bY$.  If $\lambda \in \bY_+ \cap (-w_\circ(\mu) + \bY_+)$ is large enough, the natural map
\begin{multline}
\label{eqn:hom-stabilize}
\Hom_{\Perv_{(I^A_\unip, \cX_A)}(\Gr,\bk)}(\cF, \cG \star^{\cL^+G}  \cI^{w_\circ(\mu) + \lambda}_* \star^{\cL^+G} \cI^{-w_\circ(\lambda)}_*) \\ \to
\Hom_{\Mod^\bY_{(I^A_\unip, \cX_A)}(\cR)}(\Phi^A(\cF)\la \mu\ra, \Phi^A(\cG))
\end{multline}
is an isomorphism.
\end{lem}

The map in this lemma comes from the identification of the right-hand side with $\Hom(\cF, \cG \star^{\cL^+G} \cR_\mu)$ (see Lemma~\ref{lem:Rfree-hom}), which in turn is identified with
\[
\varinjlim_{\lambda \in \bY_+ \cap (-w_\circ(\mu) + \bY_+)} \Hom (\cF, \cG \star^{\cL^+G}  \cI^{w_\circ(\mu) + \lambda}_* \star^{\cL^+G} \cI^{-w_\circ(\lambda)}_*).
\]

\begin{proof}
Any element of $\Hom(\Phi^A(\cF)\la \mu\ra, \Phi^A(\cG))$ lies in the image of~\eqref{eqn:hom-stabilize} for sufficiently large $\lambda$.  Since $\Hom(\Phi^A(\cF)\la \mu\ra, \Phi^A(\cG))$ is finite-dimensional by Lemma~\ref{lem:R-findim}, we deduce that~\eqref{eqn:hom-stabilize} is surjective for sufficiently large $\lambda$ (depending on $\cF$ and $\cG$).  Suppose now that $0 \to \cF' \to \cF \to \cF'' \to 0$ is a short exact sequence in $\Perv_{(I^A_\unip, \cX_A)}(\Gr,\bk)$, and consider the diagram
\[
\hbox{\small$
\begin{tikzcd}[column sep=tiny]
0 \ar[r] &
  \begin{array}{l}\Hom(\cF'', \cG \star {} \\ \quad \cI^{w_\circ(\mu) + \lambda}_* \star \cI^{-w_\circ(\lambda)}_*)\end{array} \ar[r] \ar[d] &
  \begin{array}{l}\Hom(\cF, \cG \star {} \\ \quad \cI^{w_\circ(\mu) + \lambda}_* \star \cI^{-w_\circ(\lambda)}_*)\end{array} \ar[r] \ar[d] &
  \begin{array}{l}\Hom(\cF', \cG \star {} \\ \quad \cI^{w_\circ(\mu) + \lambda}_* \star \cI^{-w_\circ(\lambda)}_*)\end{array} \ar[d] \\
0 \ar[r] & 
  \Hom(\Phi^A(\cF'')\la \mu\ra, \Phi^A(\cG)) \ar[r] &
  \Hom(\Phi^A(\cF)\la \mu\ra, \Phi^A(\cG)) \ar[r] &
  \Hom(\Phi^A(\cF')\la \mu\ra, \Phi^A(\cG))
\end{tikzcd}$}
\]
If the first and last columns are isomorphisms and the middle column is surjective, then the four lemma implies that the middle column is in fact also an isomorphism.  Thus, by induction on the length of $\cF$, we may reduce to the case where $\cF$ is simple.  A similar argument applies to $\cG$. It therefore suffices to prove the claim in case $\cF$ and $\cG$ are simple, which we assume from now on.

Choose a subset $({}^A W_\ext^\res)' \subset {}^A W_\ext$ as in the discussion above~\eqref{eqn:simple-translation}.
Then we may assume that
\[
\cF = \LGr^A_{x_1 t_{w_\circ(\nu_1)}} \cong \LGr^A_{x_1} \star^{\cL^+G} \IC^{\nu_1},
\qquad
\cF = \LGr^A_{x_2 t_{w_\circ(\nu_2)}} \cong \LGr^A_{x_2} \star^{\cL^+G} \IC^{\nu_2},
\]
where $x_1, x_2 \in ({}^A W_\ext^\res)'$ and $\nu_1, \nu_2 \in \bY_+$.  As in the proof of Lemma~\ref{lem:Rsimple-hom}, if $x_1 \ne x_2$, then Theorem~\ref{thm:geometric-Steinberg}\eqref{it:thm-Steinberg} implies that $\cF$ has no composition factors in common with any $\cG \star^{\cL^+G} \cI^{w_\circ(\mu) + \lambda}_* \star^{\cL^+G} \cI^{-w_\circ(\lambda)}_*$, so the left-hand side of~\eqref{eqn:hom-stabilize} vanishes for all $\lambda$, and hence so does the right-hand side.

On the other hand, if $x_1 = x_2$, then Theorem~\ref{thm:geometric-Steinberg}\eqref{it:thm-Steinberg-ff} lets us identify the left-hand side of~\eqref{eqn:hom-stabilize} with
\[
\Hom_{\Perv_{\cL^+G}(\Gr,\bk)}(\IC^{\nu_1}, \IC^{\nu_2} \star^{\cL^+G}  \cI^{w_\circ(\mu) + \lambda}_* \star^{\cL^+G} \cI^{-w_\circ(\lambda)}_*).
\]
By the geometric Satake equivalence, this is isomorphic to 
\[
\Hom_{G^\vee_\bk}(\Satake(\IC^{\nu_1}), \Satake(\IC^{\nu_2}) \otimes \coweyl^{(1)}(w_\circ(\mu)+\lambda) \otimes \coweyl^{(1)}(-w_\circ(\lambda)))
\]
where $\coweyl^{(1)}(\nu)$ is the induced $G^\vee_\bk$-module of highest weight $\nu$.
As explained in Remark~\ref{rmk:Ind-coweyl}\eqref{it:Ind-coweyl-stab}, this group is independent of $\lambda$ for $\lambda$ large enough, as desired.
\end{proof}

\begin{rmk}
\label{rmk:hom-stabilize}
If $\lambda,\lambda' \in \bY$ are such that $\lambda'-\lambda \in \bY_+$, the morphism~\eqref{eqn:hom-stabilize} factors as a composition
\begin{multline*}
\Hom_{\Perv_{(I^A_\unip, \cX_A)}(\Gr,\bk)}(\cF, \cG \star^{\cL^+G}  \cI^{w_\circ(\mu) + \lambda}_* \star^{\cL^+G} \cI^{-w_\circ(\lambda)}_*) \\ 
\to \Hom_{\Perv_{(I^A_\unip, \cX_A)}(\Gr,\bk)}(\cF, \cG \star^{\cL^+G}  \cI^{w_\circ(\mu) + \lambda'}_* \star^{\cL^+G} \cI^{-w_\circ(\lambda')}_*) \\
\to \Hom_{\Mod^\bY_{(I^A_\unip, \cX_A)}(\cR)}(\Phi^A(\cF)\la \mu\ra, \Phi^A(\cG))
\end{multline*}
where the second morphism is the analogue of~\eqref{eqn:hom-stabilize} for $\lambda'$. If $\lambda$ is large enough, this composition and its second member are isomorphisms, hence so is its first member.
\end{rmk}

\subsection{Baby co-Verma modules: definition and first properties}
\label{ss:cbV}

We now introduce geometric counterparts of the objects studied in~\S\ref{ss:baby-coV-Rep}.

For any $\mu \in \bY_+$, since $\Gr_{t_{w_\circ(\mu)}}$ is open in $\Gr^\mu$, by adjunction there exists a canonical map
\begin{equation}\label{eqn:costd-embed}
\cI^\mu_* \to \NGr_{t_{w_\circ(\mu)}}.
\end{equation}
Now let $w \in {}^A W^S_\ext$.  Then $\ell(w) + \ell(t_{w_\circ(\mu)}) = \ell(wt_{w_\circ(\mu)})$ by Lemma~\ref{lem:lengths-add-Wres}, which 
by Lemma~\ref{lem:conv-D-N}\eqref{it:conv-DNGr} implies that we have
a canonical isomorphism $\NFl_w \star^I \NGr_{t_{w_\circ(\mu)}} \cong \NGr_{wt_{w_\circ(\mu)}}$.
Applying $\Av^A_\psi$ and using~\cite[Lemma~3.3(1)--(3)]{projGr1} we deduce that 
\begin{equation}\label{eqn:costd-convolve}
\NFl^A_w \star^I \NGr_{t_{w_\circ(\mu)}} \cong \NGr^A_{wt_{w_\circ(\mu)}}.
\end{equation}
Also, for any $\cF \in \Db_{\cL^+G}(\Gr,\bk)$, we have canonical isomorphisms
\[
\NFl^A_w \star^I \cF \cong (\pi_*\NFl^A_w) \star^{\cL^+G} \cF \cong \NGr^A_w \star^{\cL^+G} \cF,
\]
see~\cite[Lemma~2.5]{bgmrr} and~\eqref{eqn:pi*-D-N}.
Thus, applying $\NFl^A_w \star^I ({-})$ to~\eqref{eqn:costd-embed}, we obtain a canonical morphism
\begin{equation}\label{eqn:costd-conv2}
\NGr^A_w \star^{\cL^+G} \cI^\mu_* \to \NGr^A_{wt_{w_\circ(\mu)}}.
\end{equation}

For $w \in {}^A W_\ext$ and $\mu \in \bY$ we set
\[
(\cbV^A_w)_\mu = \underset{\lambda}{``\varinjlim"} \, \NGr^A_{wt_{\mu+w_\circ(\lambda)}} \star^{\cL^+G} \cI_*^{-w_\circ(\lambda)},
\]
where $\lambda$ runs over the elements of $\bY_+$ such that $wt_{\mu+w_\circ(\lambda)}$ belongs to $W_\ext^S$ (which is automatic if $\lambda$ is sufficiently large, see~\S\ref{ss:periodic-order}) and where the transition morphisms in the inductive limit are given by the compositions
\begin{multline*}
\NGr^A_{wt_{\mu+w_\circ(\lambda)}} \star^{\cL^+G} \cI_*^{-w_\circ(\lambda)} \to \NGr^A_{wt_{\mu+w_\circ(\lambda)}} \star^{\cL^+G} \cI_*^{\nu} \star^{\cL^+G} \cI_*^{-w_\circ(\nu)} \star^{\cL^+G} \cI_*^{-w_\circ(\lambda)} \\
\to \NGr^A_{wt_{\mu+w_\circ(\lambda+\nu)}} \star^{\cL^+G} \cI_*^{-w_\circ(\lambda+\nu)}
\end{multline*}
for $\nu \in \bY_+$,
where the first morphism is induced by~\eqref{eqn:costd-unit} and the second one is induced by~\eqref{eqn:costd-conv2} (applied to the first two factors) and~\eqref{eqn:costd-conv-sph} (applied to the last two factors). We endow the $\bY$-graded ind-perverse sheaf
\[
\cbV^A_w := \bigoplus_{\mu \in \bY} (\cbV^A_w)_\mu
\]
with the structure of a graded $\cR$-module 
by defining the action morphism
\[
(\cbV^A_w)_\mu \star^{\cL^+G} \cR_\nu \to (\cbV^A_w)_{\mu+\nu}
\]
(for $\mu,\nu \in \bY$) as induced by the morphisms
\begin{multline*}
\bigl( \NGr^A_{wt_{\mu+w_\circ(\lambda)}} \star^{\cL^+G} \cI_*^{-w_\circ(\lambda)} \bigr) \star^{\cL^+G} \bigl( \cI_*^{w_\circ(\nu)+\lambda'} \star^{\cL^+G} \cI_*^{-w_\circ(\lambda')}  \bigr) \\
\to \NGr^A_{wt_{\mu+\nu+w_\circ(\lambda+\lambda')}} \star^{\cL^+G} \cI_*^{-w_\circ(\lambda+\lambda')}
\end{multline*}
induced by~\eqref{eqn:costd-conv2} and~\eqref{eqn:costd-conv-sph} (after application of the commutativity constraint for $\star^{\cL^+G}$ to the second and third factors), for $\lambda,\lambda'$ sufficiently dominant. 

It is clear from definition that for any $w \in {}^A W_\ext$ and $\nu \in \bY$ we have
\begin{equation}
\label{eqn:cbV-translation}
\cbV^A_{wt_\nu} = \cbV^A_w \langle -\nu \rangle.
\end{equation}

\begin{lem}
\label{lem:Hom-simple-bV}
For $w,y \in {}^A W_\ext$, we have
\[
\dim \Hom_{\Mod^\bY_{(I^A_\unip, \cX_A)}(\cR)}(\LGT^A_{y}, \cbV^A_w) = \begin{cases}
1 & \text{if $w=y$;} \\
0 & \text{otherwise.}
\end{cases}
\]
\end{lem}

\begin{proof}
Write $y=zt_\nu$, $w=z't_{\nu'}$ with $z,z' \in {}^A W_\ext^\res$ and $\nu,\nu' \in \bY$. From the definition of $\LGT^A_{y}$ and Lemma~\ref{lem:Rfree-hom} we see that
\[
\Hom_{\Mod^\bY_{(I^A_\unip, \cX_A)}(\cR)}(\LGT^A_{y}, \cbV^A_w) \cong \Hom_{\Perv_{(I^A_\unip, \cX_A)}(\Gr,\bk)}(\LGr^A_{z'}, (\cbV^A_{z'})_{\nu'-\nu}).
\]
It follows that
\begin{multline*}
\Hom_{\Mod^\bY_{(I^A_\unip, \cX_A)}(\cR)}(\LGT^A_{y}, \cbV^A_w) \cong \varinjlim_{\lambda} \Hom(\LGr^A_{z}, \NGr^A_{z't_{\nu'-\nu +w_\circ(\lambda)}} \star^{\cL^+G} \cI_*^{-w_\circ(\lambda)}) \\
\cong \varinjlim_{\lambda} \Hom(\LGr^A_{z} \star^{\cL^+G} \cI_!^{\lambda}, \NGr^A_{z't_{\nu'-\nu+w_\circ(\lambda)}})
\end{multline*}
where the second step uses~\eqref{eqn:duals}. By Corollary~\ref{cor:Hom-vanishing-Steinberg}\eqref{it:Hom-vanishing-Steinberg-1} the rightmost term vanishes unless $z$ and $z'$ differ by multiplication by $t_\eta$ for some coweight $\eta$ orthogonal to all roots. In this case we can assume that $z=z'$; then by Corollary~\ref{cor:Hom-vanishing-Steinberg}\eqref{it:Hom-vanishing-Steinberg-2} the $\Hom$ spaces vanish unless $\nu=\nu'$. Finally, if $z=z'$ and $\nu=\nu'$, by Corollary~\ref{cor:Hom-vanishing-Steinberg}\eqref{it:Hom-vanishing-Steinberg-3} each space $\Hom(\LGr^A_{z} \star^{\cL^+G} \cI_!^{\lambda}, \NGr^A_{zt_{w_\circ(\lambda)}})$ is $1$-dimensional.
It is easily seen that all transition morphisms are isomorphisms, so that our inductive limit is isomorphic to $\bk$.
\end{proof}

\section{Averaging and wall-crossing functors}
\label{sec:averaging}

\subsection{Averaging functors for \texorpdfstring{$\cR$}{R}-modules}
\label{ss:av-functors}

The averaging and wall-crossing functors defined in~\S\S\ref{ss:Av-Iw}--\ref{ss:wc-functors} extend to exact functors on ind-perverse sheaves. Moreover, for graded $\cR$-modules, these functors respect the $\cR$-module structure, and the induced functors commute in the obvious way with the functors $\Phi$ and $\Phi^A$, and with the shift-of-grading functors.  The properties of the functors constructed in this way, which follow directly from the results of~\S\S\ref{ss:Av-Iw}--\ref{ss:wc-functors}, are recorded in the following lemma.

\begin{lem}
\label{lem:functors-Rmod}
The functors
\begin{align*}
\Av^A_!, \Av^A_* &: \Mod^\bY_{(I^A_\unip, \cX_A)}(\cR) \to \Mod^\bY_{I_\unip}(\cR), \\
\Av^A_\psi &: \Mod^\bY_{I_\unip}(\cR) \to \Mod^\bY_{(I^A_\unip, \cX_A)}(\cR), \\
\xi^!_A, \xi^*_A &:  \Mod^\bY_{I_\unip}(\cR) \to  \Mod^\bY_{I_\unip}(\cR) 
\end{align*}
are exact, and send finitely generated, resp.~finite-length, $\cR$-modules to finitely generated, resp.~finite-length, $\cR$-modules.  Moreover, we have adjoint pairs 
\[
(\Av^A_{\psi}, \Av^A_*), \quad (\Av^A_{!}, \Av^A_{\psi}), \quad \text{and} \quad (\xi^!_A, \xi^*_A),
\]
and an isomorphism $\xi_A^! \cong \xi_A^*$.
\end{lem}

In view of the last claim in this lemma, we will sometimes write $\xi_A$ for $\xi_A^!$ or $\xi_A^*$, and will write $\xi_s$ for $\xi_{\{s\}}$ ($s \in S_\aff$).

\begin{rmk}
\label{rmk:Av-Groth-G1T}
As in Remark~\ref{rmk:Av-isom}, it is likely that the functors $\Av^A_!$ and $\Av^A_*$ are isomorphic. At least, as in Remark~\ref{rmk:Av-Groth}, for any $w \in {}^A W_\ext$ we have
\[
\Av^A_!(\LGT^A_w) \cong \Av^A_*(\LGT^A_w).
\]
As a consequence, if we denote by $[\Mod^\bY_{(I^A_\unip, \cX_A)}(\cR)^\flen]$ and $[\Mod^\bY_{I_\unip}(\cR)^\flen]$ the Grothendieck groups of the (abelian, finite length) categories $\Mod^\bY_{(I^A_\unip, \cX_A)}(\cR)^\flen$ and $\Mod^\bY_{I_\unip}(\cR)^\flen$, and by
\[
[\Av^A_!], [\Av^A_*] : [\Mod^\bY_{(I^A_\unip, \cX_A)}(\cR)^\flen] \to [\Mod^\bY_{I_\unip}(\cR)^\flen]
\]
the maps induced by $\Av^A_!$ and $\Av^A_*$ on Grothendieck groups, we have $[\Av^A_!]= [\Av^A_*]$.
\end{rmk}

\begin{lem}
\phantomsection
\label{lem:bV-Av}
\begin{enumerate}
\item 
\label{it:simples-Av}
For any $w \in W_\ext$, we have
\[
\Av^A_{\psi}(\LGT_w) \cong \begin{cases}
\LGT^A_w & \text{if $w \in {}^A W_\ext$,} \\
0 & \text{otherwise.}
\end{cases}
\]
\item
\label{it:bV-Av-psi}
For any $w \in {}^A W_\ext$ and any $v \in W_A$, we have
\[
\Av^A_\psi(\cbV_{vw}) \cong \cbV^A_w.
\]
\item
\label{it:bV-Av-*}
Choose an enumeration $x_1, \ldots, x_r$ of $W_A$ which refines the Bruhat order (so that necessarily $x_1=e$ and $x_r=w_A$). For $w \in {}^A W_\ext$, the object $\Av^A_*(\cbV^A_w)$ admits a filtration with successive subquotients $\cbV_{x_1 w}, \cbV_{x_2 w}, \ldots, \cbV_{x_r w}$. 
\end{enumerate}
\end{lem}

\begin{proof}
The claims are consequences of the behavior of the functors $\Av^A_\psi$ and $\Av^A_*$ on simple and costandard perverse sheaves (see~\cite[Lemma~3.3]{projGr1} and the proof of Corollary~\ref{cor:Hom-vanishing-Steinberg}\eqref{it:Hom-vanishing-Steinberg-2}) and the fact that for any $\lambda$ large enough, $wt_{\mu-\lambda}$ is the unique minimal element (for the Bruhat order) in $W_A w t_{\mu-\lambda}$ (see Remark~\ref{rmk:min-preceq}).
\end{proof}

\subsection{Some perverse sheaves arising from the big tilting object on the finite-dimensional flag variety}
\label{ss:big-tilting}

The considerations in this subsection are closely related to those in~\cite[\S 4.1]{bgmrr}; however, for the reader's convenience we will repeat the required proofs.

The ``big tilting object,'' denoted by $\cS$, is defined to be the unique indecomposable tilting perverse sheaf in $\Perv_U(G/B)$ with full support.  (Here, ``$\cS$'' stands for Soergel, who studied a representation-theoretic incarnation of this object.)  We will review one approach to constructing $\cS$ (following~\cite{by} in a characteristic-$0$ setting, and~\cite{modrap1} or~\cite[Lemma~10.1]{br} for general coefficients) that shows that this object is both the projective cover and the injective hull of the skyscraper sheaf at $B/B \in G/B$.  Recall that 
$\psi_S$ factors as a composition
\[
I_\unip^S \to U^+ \xrightarrow{\psi_+} \Ga,
\]
where $U^+$ is the ``positive'' unipotent subgroup of $G$ (see~\cite[\S 3.4]{projGr1}).
We can then set $\cX_+=\psi_+^* \AS$, and consider the equivariant derived category $\Db_{(U^+,\cX_+)}(G/B,\bk)$. The $*$- and $!$-pushforwards of the unique $(U^+,\cX_+)$-equivariant rank-$1$ local system on the orbit $U^+ B/B \subset G/B$ are canonically isomorphic, and will be denoted $\Delta^+$. We then have functors
\begin{align*}
\Av^U_! :& \ \Db_{(U^+,\cX_+)}(G/B,\bk) \to \Db_U(G/B,\bk), \\
 \Av^U_* :& \ \Db_{(U^+,\cX_+)}(G/B,\bk) \to \Db_U(G/B,\bk)
\end{align*}
defined as for $\Av^S_!$ and $\Av^S_*$, and we have
\[
\cS \cong \Av^U_!(\Delta^+) \cong \Av^U_*(\Delta^+).
\]

Recall from~\cite[Lemma~2.5]{projGr1} that the elements $w \in W_\ext^S$ such that $ww_\circ$ has minimal length in $Www_\circ$ are those of the form $t_\lambda w_\circ$ with $\lambda \in \bY_{++}$. 
For such $\lambda$, we have described the geometry of $\Gr^\lambda$ in~\S\ref{ss:Satake-category}, and in particular considered a morphism $p_\lambda : \Gr^\lambda \to G/B$.
We set
\[
\cS_\lambda := p_\lambda^* \cS [\dim(\Gr^\lambda)-\dim(G/B)].
\]
This is an $I_\unip$-equivariant perverse sheaf on $\Gr^\lambda$.  The following proposition describes some calculations one can carry out with $\cS_\lambda$.  

\begin{prop}
\label{prop:avS-calc}
Let $\lambda \in \bY_{++}$.
\begin{enumerate}
\item 
\label{it:avS-!}
We have
\[
j^\lambda_!\cS_\lambda \cong \left( j^{\varsigma}_{!}\cS_{\varsigma} \right) \star^{\cL^+G} \cI_!^{\lambda - \varsigma}.
\]
Moreover, this object has a standard filtration and a simple head, isomorphic to $\LGr_{t_\lambda w_\circ}$.
\item 
\label{it:avS-*}
We have
\[
j^\lambda_*\cS_\lambda \cong \left( j^{\varsigma}_{*}\cS_{\varsigma} \right) \star^{\cL^+G} \cI_*^{\lambda - \varsigma}.
\]
Moreover, this object has a costandard filtration and a simple socle, isomorphic to $\LGr_{t_\lambda w_\circ}$.
\item 
\label{it:avS-tilt}
We have
\[
j^{\varsigma}_{!}\cS_{\varsigma} \cong j^{\varsigma}_{!*}\cS_{\varsigma} \cong j^{\varsigma}_{*}\cS_{\varsigma}.
\]
As a consequence, if $\cI_!^{\lambda - \varsigma} \cong \cI_*^{\lambda -\varsigma}$, then
\[
j^\lambda_!\cS_\lambda \cong j^\lambda_*\cS_\lambda,
\]
and this object is an indecomposable tilting object, isomorphic to $\TGr_{t_{w_\circ(\lambda)}}$.
\end{enumerate}
\end{prop}

In view of~\eqref{it:avS-tilt},
the isomorphisms in~\eqref{it:avS-!} and~\eqref{it:avS-*} can also be written as
\begin{equation}
\label{eqn:isoms-Slambda}
j^\lambda_!\cS_\lambda \cong \TGr_{t_{w_\circ(\varsigma)}} \star^{\cL^+G} \cI_!^{\lambda -\varsigma}
\qquad\text{and}\qquad
j^\lambda_*\cS_\lambda \cong \TGr_{t_{w_\circ(\varsigma)}} \star^{\cL^+G} \cI_*^{\lambda -\varsigma}.
\end{equation}
Note that if $\bk$ has characteristic $0$, then the condition in part~\eqref{it:avS-tilt} applies to all $\lambda \in \bY_{++}$ (by semisimplicity of the Satake category in this case).

\begin{proof}
\eqref{it:avS-!} The functors $\Av^U_!$ and $\Av^S_!$ have a counterpart for sheaves on $\Gr^\lambda$, which will also be denoted $\Av^S_!$; then we have
\[
\Av^S_! \circ j^\lambda_! \cong j^\lambda_! \circ \Av^S_!, \quad \Av^S_! \circ p_\lambda^* \cong p_\lambda^* \circ \Av^U_!.
\]
Now consider the object $\DGr^S_{t_\lambda w_\circ} \in \Perv_{(I_\unip^S,\cX_S)}(\Gr,\bk)$. From the definition we see that 
\[
\DGr^S_{t_\lambda w_\circ} = j^\lambda_! p_\lambda^* \Delta^+ [\dim(\Gr^\lambda)-\dim(G/B)];
\]
we deduce that
\begin{equation}
\label{eqn:j!S-AvD}
\Av^S_! (\DGr^S_{t_\lambda w_\circ}) \cong j^\lambda_! p_\lambda^* \Av^U_!(\Delta^+) [\dim(\Gr^\lambda)-\dim(G/B)] \cong j^\lambda_! \cS_\lambda.
\end{equation}
The claim that $j^\lambda_!\cS_\lambda$ admits a standard filtration is immediate from the fact that $\cS_\lambda$ admits a standard filtration.  Alternatively, it is a consequence of the isomorphism above and Lemma~\ref{lem:Av-Iw}\eqref{it:Av-Iw-2}.

Next, for any $\cF$ in $\Perv_{I_\unip}(\Gr,\bk)$, by adjunction we have
\begin{equation}
\label{eqn:avS-adj}
\Hom_{\Perv_{I_\unip}(\Gr,\bk)}(\Av^S_! (\DGr^S_{t_\lambda w_\circ}), \cF) \cong \Hom_{\Perv_{(I_\unip^S,\cX_S)}(\Gr,\bk)}(\DGr^S_{t_\lambda w_\circ}, \Av^S_\psi( \cF)).
\end{equation}
In case $\cF$ is simple, the explicit description of $\Av^S_\psi(\cF)$ given in~\cite[Lemma~3.3(4)]{projGr1} shows that
\[
\Hom_{\Perv_{I_\unip}(\Gr,\bk)}(\Av^S_! (\DGr^S_{t_\lambda w_\circ}), \cF) =0
\]
unless $\cF \cong \LGr_{t_\lambda w_\circ}$, in which case this space is $1$-dimensional. We deduce that $\Av^S_! (\DGr^S_{t_\lambda w_\circ}) \cong j^\lambda_! \cS_\lambda$ has a simple head, isomorphic to $\LGr_{t_\lambda w_\circ}$. 

Finally, recall from~\cite[p.~723]{bgmrr} that
we have
\[
\DGr^S_{t_\lambda w_\circ} \cong \DGr^S_{t_{\varsigma} w_\circ} \star^{\cL^+G} \cI^{\lambda -\varsigma}_!.
\]
Since $\Av^S_!$ commutes with convolution on the right by objects of $\Perv_{\cL^+G}(\Gr,\bk)$, we see that
\[
j^\lambda_! \cS_\lambda \cong \Av^S_!(\DGr^S_{t_{\varsigma} w_\circ}) \star^{\cL^+G} \cI^{\lambda - \varsigma}_! \cong j^{\varsigma}_!\cS_{\varsigma} \star^{\cL^+G} \cI^{\lambda -\varsigma}_!,
\]
which finishes the proof.

\eqref{it:avS-*} The proof is very similar and will be omitted.

\eqref{it:avS-tilt}
Since $t_{\varsigma} w_\circ$ is minimal in ${}^S W_\ext^S$, we have $\DGr^S_{t_{\varsigma} w_\circ} \cong \LGr^S_{t_{\varsigma} w_\circ}$. 
Using~\cite[Lemma~3.3(4)]{projGr1} again, we deduce that
\[
\DGr^S_{t_{\varsigma} w_\circ} \cong \Av^S_\psi(\LGr_{t_{\varsigma} w_\circ}).
\]
In view of~\eqref{eqn:j!S-AvD}, it follows that
\[
j^{\varsigma}_!\cS_{\varsigma} \cong \Av^S_!(\DGr^S_{t_{\varsigma} w_\circ}) \cong \xi_S(\LGr_{t_{\varsigma} w_\circ}).
\]
Similar (dual) considerations show that $j^{\varsigma}_*\cS_{\varsigma} \cong \Av^S_*(\NGr^S_{t_{\varsigma} w_\circ}) \cong \xi_S(\LGr_{t_{\varsigma} w_\circ})$, and hence that
\[
j^{\varsigma}_!\cS_{\varsigma} \cong j^{\varsigma}_*\cS_{\varsigma}.
\]
It is then clear that these objects are also isomorphic to $j^{\varsigma}_{!*}\cS_{\varsigma}$.

Finally, let us assume that $\cI_!^{\lambda -\varsigma} \cong \cI_*^{\lambda -\varsigma}$. Then from parts~\eqref{it:avS-!} and~\eqref{it:avS-*} we deduce that $j^\lambda_! \cS_\lambda \cong j^\lambda_*\cS_\lambda$, and that this object is tilting. Its support is clearly $\overline{\Gr^\lambda} = \overline{\Gr_{t_{w_\circ(\lambda)}}}$, and it is indecomposable because it has a simple head (and a simple socle).  It must therefore be isomorphic to $\TGr_{t_{w_\circ(\lambda)}}$.
\end{proof}

We extract the following observations from the calculations in the preceding proof.

\begin{prop}
\label{prop:projectives}
Let $\lambda \in \bY_{++}$.
\begin{enumerate}
\item 
\label{it:proj-char0}
If $\bk$ has characteristic $0$, then in $\Perv_{I_\unip}(\Gr,\bk)$, the object
\[
j^\lambda_!\cS_\lambda \cong j^\lambda_*\cS_\lambda \cong \TGr_{t_{w_\circ(\lambda)}}
\]
is both projective and injective.
\item 
\label{it:proj-R}
In $\Mod^\bY_{I_\unip}(\cR)$, the object $\Phi(j^\lambda_!\cS_\lambda)$ is projective, and the object $\Phi(j^\lambda_*\cS_\lambda)$ is injective.
\end{enumerate}
\end{prop}

\begin{proof}
\eqref{it:proj-char0} Under our assumption the category $\Perv_{(I_\unip^S,\cX_S)}(\Gr,\bk)$ is semisimple by~\cite[Corollary~3.6]{bgmrr}. Then, since $\Av^S_\psi$ is exact, so is the right-hand side of~\eqref{eqn:avS-adj} (as a functor of $\cF$). The left-hand side is therefore also exact, which shows that $\Av^S_! (\DGr^S_{t_\lambda w_\circ}) \cong j^\lambda_! \cS_\lambda$ is projective. Dual arguments show that this object is also injective.

\eqref{it:proj-R} The proof is similar to that of~\eqref{it:proj-char0}, using the following $\cR$-module analogue of~\eqref{eqn:avS-adj}:
\begin{multline*}
\Hom_{\Mod^\bY_{I_\unip}(\cR)}(\Phi(\Av^S_! (\DGr^S_{t_\lambda w_\circ})) , \cF) \cong \Hom_{\Mod^\bY_{I_\unip}(\cR)}(\Av^S_! (\Phi^S(\DGr^S_{t_\lambda w_\circ})) , \cF) \\
\cong \Hom_{\Mod^\bY_{(I_\unip^S,\cX_S)}(\cR)}(\Phi^S(\DGr^S_{t_\lambda w_\circ}), \Av^S_\psi( \cF)).
\end{multline*}
By Lemma~\ref{lem:functors-Rmod} and Proposition~\ref{prop:IW-R-semisimple} the right-hand side is an exact functor of $\cF$, so that the object $\Phi(\Av^S_! (\DGr^S_{t_\lambda w_\circ})) \cong \Phi(j^\lambda_!\cS_\lambda)$ is projective. Dual arguments apply to $\Phi(j^\lambda_*\cS_\lambda)$.
\end{proof}

\subsection{Wall-crossing functors and objects arising from \texorpdfstring{$\cS$}{S}}
\label{ss:wall-crossing}

In the statement of the following lemma we use the fact that any element in $W_\ext^S$ can be written as a product $xt_{-\mu}$ where $x \in W^\res_\ext$ and $\mu \in \bY_{+}$, see~\eqref{eqn:WS-Wres}.

\begin{lem}
\label{lem:proj-prep}
Let $w \in W^S_\ext$, and write $w = xt_{w_\circ(\mu-\varsigma)}$ with $x \in W^\res_\ext$ and $\mu \in \bY_{++}$.  Let $y = t_{\varsigma} w_\circ x^{-1}$, and choose a reduced expression $y = \omega s_1 \cdots s_r$ with $\omega \in \Omega$ and $s_1, \ldots, s_r \in S_\aff$.

\begin{enumerate}
\item
\label{it:prep-length}
We have $yw = t_{\mu} w_\circ$ and $yw^\triangle = t_{w_\circ(\mu)}$, and moreover
\[
\ell(yw)=\ell(y) + \ell(w),
\qquad
\ell(y w^\triangle) = \ell(w^\triangle) - \ell(y).
\]
\item
\label{it:prep-std}
The object $\xi_{s_r} \cdots \xi_{s_1}\xi_{\omega^{-1}}(j^{\mu}_!\cS_{\mu})$ is supported on $\overline{\Gr_{w^\triangle}}$, and admits $\LGr_{w^\triangle}$ as a composition factor with multiplicity $1$.  Moreover, this object admits a standard filtration in which $\DGr_w$ occurs with multiplicity $1$, and $\LGr_w$ is a direct summand of its head with multiplicity~$1$.
\item 
\label{it:prep-costd}
The object $\xi_{s_r} \cdots \xi_{s_1}\xi_{\omega^{-1}}(j^{\mu}_*\cS_{\mu})$ is supported on $\overline{\Gr_{w^\triangle}}$, and admits $\LGr_{w^\triangle}$ as a composition factor with multiplicity $1$.  Moreover, this object admits a costandard filtration in which $\NGr_w$ occurs with multiplicity $1$, and $\LGr_w$ is a direct summand of its socle with multiplicity $1$.
\item 
\label{it:prep-tilt}
If $j^{\mu}_!\cS_{\mu} \cong j^{\mu}_*\cS_{\mu}$, then $\xi_{s_r} \cdots \xi_{s_1}\xi_{\omega^{-1}}(j^{\mu}_!\cS_{\mu})$ is tilting, and contains $\TGr_{w^\triangle}$ as a direct summand with multiplicity $1$.
\end{enumerate}
\end{lem}

\begin{proof}
\eqref{it:prep-length}
The fact that $yw = t_{\mu} w_\circ$ is immediate from the definitions, and then
\[
y^{-1} t_{w_\circ(\mu)}=x w_\circ t_{-\varsigma} t_{w_\circ(\mu)}
= w t_{w_\circ(\varsigma-\mu)}w_\circ t_{-\varsigma} t_{w_\circ(\mu)} = w t_{-w_\circ(\mu)} w_\circ t_{w_\circ(\mu)} = w^\triangle.
\]
(Observe that $w^{-1}(\Afund) \subset \Pi_{-w_\circ(\mu)}$.)
Using~\eqref{eqn:formula-length} for the first two equalities, then Lemma~\ref{lem:res-complement}, and finally~\cite[Lemma~2.7]{projGr1} and the fact that $\ell(t_{\mu-\varsigma})=\ell(t_{w_\circ(\mu-\varsigma)})$, we see that
\[
\ell(yw) = \ell(t_\mu)-\ell(w_\circ)= \ell(t_{\mu-\varsigma})+\ell(t_{\varsigma} w_\circ) = \ell(t_{\mu-\varsigma})+\ell(x)+\ell(y) = \ell(w)+\ell(y).
\]
Finally, we have
\[
\ell(w^\triangle) = \ell(x w_\circ t_{-\varsigma} t_{w_\circ(\mu)}) = \ell(x w_\circ t_{-\varsigma}) + \ell(t_{w_\circ(\mu)}) = \ell(y) + \ell(t_{w_\circ(\mu)}),
\]
where we use~\cite[Lemma~2.7]{projGr1} for the second equality, after noticing that $x w_\circ t_{-\varsigma}  \in W_\ext^\res$ since $t_{\varsigma} w_\circ x^{-1}(A_{\mathrm{fund}}) \subset t_{\varsigma} w_\circ (\Pi_{\varsigma})=\Pi_{\varsigma}$.

\eqref{it:prep-std}
Note that $yw \in W_\ext^S$ (see e.g.~\cite[Lemma~2.5]{projGr1}).  By Lemma~\ref{lem:support}, the objects
\[
\xi_\omega \xi_{s_1} \cdots \xi_{s_r} (\LGr_w)
\qquad\text{and}\qquad
\xi_\omega \xi_{s_1} \cdots \xi_{s_r} (\NGr_w)
\]
are supported on $\overline{\Gr_{t_{\mu} w_\circ}}$ and have $\LGr_{t_{\mu} w_\circ}$ as a composition factor with multiplicity~$1$.  Now, $\Gr_{t_{\mu} w_\circ}$ in the unique closed $I$-orbit in $\Gr^{\mu}$; it follows that
\[
\xi_\omega \xi_{s_1} \cdots \xi_{s_r} (\LGr_w)_{|\Gr^{\mu}} \cong
\xi_\omega \xi_{s_1} \cdots \xi_{s_r} (\NGr_w)_{|\Gr^{\mu}} \cong \underline{\bk}_{\Gr_{t_{\mu} w_\circ}}[\ell(t_\mu w_\circ)].
\]
By definition of $\cS_\mu$ and full faithfulness of $p_\mu^* (-) [\dim(\Gr^\lambda)-\dim(G/B)]$ on perverse sheaves (since $p_\mu$ is smooth with connected fibers), we deduce that
\begin{multline}\label{eqn:prep-dim-calc}
\dim \Hom(j^{\mu}_!\cS_{\mu}, \xi_\omega \xi_{s_1} \cdots \xi_{s_r} (\LGr_w)) = \\
\dim \Hom(j^{\mu}_!\cS_{\mu}, \xi_\omega \xi_{s_1} \cdots \xi_{s_r} (\NGr_w)) = 1.
\end{multline}
By adjunction, it follows that
\begin{multline}\label{eqn:prep-factor}
\dim \Hom (\xi_{s_r} \cdots \xi_{s_1}\xi_{\omega^{-1}}(j^{\mu}_!\cS_{\mu}), \LGr_w) = \\
 \dim \Hom (\xi_{s_r} \cdots \xi_{s_1}\xi_{\omega^{-1}}(j^{\mu}_!\cS_{\mu}), \NGr_w) = 1.
\end{multline}
Thus, $\LGr_w$ occurs in the head of $\xi_{s_r} \cdots \xi_{s_1}\xi_{\omega^{-1}}(j^{\mu}_!\cS_{\mu})$ with multiplicity $1$.  By Proposition~\ref{prop:avS-calc}\eqref{it:avS-!}, $j^{\mu}_!\cS_{\mu}$ admits a standard filtration; in view of Lemma~\ref{lem:isom-xi} and~\eqref{eqn:xi-D-N} this implies that $\xi_{s_r} \cdots \xi_{s_1}\xi_{\omega^{-1}}(j^{\mu}_!\cS_{\mu})$ admits a standard filtration.  The dimension calculation above shows that $\DGr_w$ occurs in this standard filtration with multiplicity $1$.  Finally, invoke Lemma~\ref{lem:support} again to conclude that our perverse sheaf is supported on $\overline{\Gr_{w^\triangle}}$ and admits $\LGr_{w^\triangle}$ as a composition factor with multiplicity $1$.

\eqref{it:prep-costd} The proof is similar (more specifically, dual) to that of part~\eqref{it:prep-std}.  

\eqref{it:prep-tilt} If $j^{\mu}_!\cS_{\mu} \cong j^{\mu}_*\cS_{\mu}$, then~\eqref{it:prep-std} and~\eqref{it:prep-costd} imply that $\xi_{s_r} \cdots \xi_{s_1}\xi_{\omega^{-1}}(j^{\mu}_!\cS_{\mu})$ is tilting, and that this perverse sheaf is supported on $\overline{\Gr_{w^\triangle}}$ and admits $\LGr_{w^\triangle}$ as a composition factor with multiplicity $1$. It follows that $\xi_{s_r} \cdots \xi_{s_1}\xi_{\omega^{-1}}(j^{\mu}_!\cS_{\mu})$ admits $\TGr_{w^\triangle}$ as a direct summand, with multiplicity $1$.
\end{proof}

For later use we note the following corollary of Lemma~\ref{lem:proj-prep}.

\begin{lem}
\label{lem:twtri-standard}
Let $w \in W^\res_\ext$, and let $\mu \in \bY_+$.  The object $\TGr_{w^\triangle} \star^{\cL^+G} \cI^\mu_!$ admits a standard filtration, and the object $\TGr_{w^\triangle} \star^{\cL^+G} \cI^\mu_*$ admits a costandard filtration.
\end{lem}

\begin{proof}
We will prove the claim for $\cI^\mu_!$; the other case is similar. Let $y = t_{\varsigma} w_\circ w^{-1}$, and choose a reduced expression $y = \omega s_1 \cdots s_r$ as in Lemma~\ref{lem:proj-prep}. Let $\lambda = \mu + \varsigma \in \bY_{++}$.  By Proposition~\ref{prop:avS-calc}\eqref{it:avS-tilt} and Lemma~\ref{lem:proj-prep}\eqref{it:prep-tilt} (applied with $\mu=\varsigma$), $\TGr_{w^\triangle}$ is a direct summand of $\xi_{s_r} \cdots \xi_{s_1} \xi_{\omega^{-1}} (\TGr_{t_{w_\circ(\varsigma)}})$; hence to prove the claim it is enough to show that the object
\begin{multline*}
\xi_{s_r} \cdots \xi_{s_1} \xi_{\omega^{-1}} (\TGr_{t_{w_\circ(\varsigma)}}) \star^{\cL^+G} \cI_!^\mu
\cong 
\xi_{s_r} \cdots \xi_{s_1} \xi_{\omega^{-1}} (\TGr_{t_{w_\circ(\varsigma)}}  \star^{\cL^+G} \cI_!^{\lambda -\varsigma}) \\
\overset{\eqref{eqn:isoms-Slambda}}{\cong}
\xi_{s_r} \cdots \xi_{s_1} \xi_{\omega^{-1}} (j^\lambda_!\cS_\lambda).
\end{multline*}
admits a standard filtration. This claim holds by Proposition~\ref{prop:avS-calc}\eqref{it:avS-!} together with Lemma~\ref{lem:isom-xi} and~\eqref{eqn:xi-D-N}.
\end{proof}

\subsection{Baby co-Verma modules are finitely generated}
\label{ss:bV-fg}

Our next task is to prove that each $\cbV^A_w$ is finitely generated, as stated in the following proposition.

\begin{prop}
\label{prop:bV-fg}
The object $\cbV^A_w$ belongs to $\modf^\bY_{(I^A_\unip, \cX_A)}(\cR)$.
\end{prop}

\begin{proof}
By Lemma~\ref{lem:bV-Av}\eqref{it:bV-Av-psi} and Lemma~\ref{lem:functors-Rmod}, it is enough prove this proposition in the special case where $A = \varnothing$.  We assume this from now on.  Furthermore, in view of~\eqref{eqn:cbV-translation}, we may assume that $w \in W$. Using the formalism of the free-monodromic completion from~\cite{by,br} one can construct a canonical triangulated functor
\[
\DFl_w \, \widetilde{\star} \, (-) : \Db_{I_\unip}(\Gr,\bk) \to \Db_{I_\unip}(\Gr,\bk)
\]
such that the diagram
\[
\begin{tikzcd}[column sep=large]
\Db_I(\Gr,\bk) \ar[r, "\DFl_w \star^I ({-})"] \ar[d, "\For^I_{I_\unip}"'] &
  \Db_I(\Gr,\bk) \ar[d, "\For^I_{I_\unip}"] \\
\Db_{I_\unip}(\Gr,\bk) \ar[r, "\DFl_w \, \widetilde{\star} \, (-)"] & \Db_{I_\unip}(\Gr,\bk)
\end{tikzcd}
\]
commutes.
We consider the complex $\DFl_w \, \widetilde{\star} \, j^{\varsigma}_{!*} \cS_\varsigma$. Here $j^{\varsigma}_{!*} \cS_\varsigma$ admits a costandard filtration (see Proposition~\ref{prop:avS-calc}). By Lemma~\ref{lem:conv-D-N} the convolution of a standard and a costandard perverse sheaf is perverse; our complex is therefore a perverse sheaf. We will construct a surjection
\begin{equation}
\label{eqn:cbV-fg}
\Phi(\DFl_w \, \widetilde{\star} \, j^{\varsigma}_{!*} \cS_\varsigma) \langle w_\circ(\varsigma) \rangle \to \cbV_w,
\end{equation}
which will prove the proposition.

By Lemma~\ref{lem:Rfree-hom} we have
\[
\Hom_{\Mod^\bY_{I_\unip}(\cR)}(\Phi(\DFl_w \, \widetilde{\star} \, j^{\varsigma}_{!*} \cS_\varsigma) \langle w_\circ(\varsigma) \rangle, \cbV_w) \cong \Hom(\DFl_w \, \widetilde{\star} \, j^{\varsigma}_{!*} \cS_\varsigma, (\cbV_w)_{w_\circ(\varsigma)}).
\]
Now since $j^{\varsigma}_{!*} \cS_\varsigma$ is tilting and supported on $\overline{\Gr_{t_{w_\circ(\varsigma)}}}$ there is a canonical surjection
\[
j^{\varsigma}_{!*} \cS_\varsigma \to \NGr_{t_{w_\circ(\varsigma)}}
\]
whose kernel admits a costandard filtration, and which therefore induces a surjection
\[
\DFl_w \, \widetilde{\star} \, j^{\varsigma}_{!*} \cS_\varsigma \to \DFl_w \, \widetilde{\star} \, \NGr_{t_{w_\circ(\varsigma)}} \cong \DFl_w \star^I \NGr_{t_{w_\circ(\varsigma)}} \cong \NGr_{wt_{w_\circ(\varsigma)}}
\]
in view of Lemma~\ref{lem:conv-D-N}. (Note that $\ell(wt_{w_\circ(\varsigma)}) = \ell(t_{w_\circ(\varsigma)})-\ell(w)$.)
There is also a canonical morphism
\[
\NGr_{wt_{w_\circ(\varsigma)}} \to \underset{\lambda}{``\varinjlim"} \NGr_{wt_{w_\circ(\varsigma)+w_\circ(\lambda)}} \star^{\cL^+ G} \cI_*^{-w_\circ(\lambda)} = (\cbV_w)_{w_\circ(\varsigma)},
\]
which provides the desired map~\eqref{eqn:cbV-fg}.

Now for any $\mu \in \bY$ we have
\begin{multline*}
\bigl( \Phi(\DFl_w \, \widetilde{\star} \, j^{\varsigma}_{!*} \cS_\varsigma) \langle w_\circ(\varsigma) \rangle \bigr)_\mu = (\DFl_w \, \widetilde{\star} \, j^{\varsigma}_{!*} \cS_\varsigma) \star^{\cL^+G} \cR_{\mu-w_\circ(\varsigma)} \\
= \underset{\lambda}{``\varinjlim"} \DFl_w \, \widetilde{\star} \, \bigl( j^{\varsigma}_{!*} \cS_\varsigma \star^{\cL^+G} \cI_*^{w_\circ(\mu)+\lambda-\varsigma} \bigr) \star^{\cL^+G} \cI_*^{-w_\circ(\lambda)} \\
\cong \underset{\lambda}{``\varinjlim"} \DFl_w \, \widetilde{\star} \, j^{w_\circ(\mu)+\lambda}_{*} \cS_{w_\circ(\mu)+\lambda} \star^{\cL^+G} \cI_*^{-w_\circ(\lambda)}
\end{multline*}
by
Proposition~\ref{prop:avS-calc}. 
As in the case of $j^{\varsigma}_{!*} \cS_\varsigma$, for any $\lambda \in -w_\circ(\mu) + \bY_{++}$ we have a canonical surjection
\begin{equation}
\label{eqn:surj-Pi-lambda}
\DFl_w \, \widetilde{\star} \, j^{w_\circ(\mu)+\lambda}_{*} \cS_{w_\circ(\mu)+\lambda} \twoheadrightarrow \DFl_w \, \widetilde{\star} \, \NGr_{t_{\mu+w_\circ(\lambda)}} \cong \NGr_{wt_{\mu+w_\circ(\lambda)}},
\end{equation}
and these morphisms induce~\eqref{eqn:cbV-fg}.
It follows that this morphism is surjective, as desired.
\end{proof}

\begin{cor}
\label{cor:cbV-socle}
For any $w \in {}^A W_\ext$ the object $\cbV^A_w$ has finite length and a simple socle, isomorphic to $\LGT^A_w$.  In particular, $\cbV^A_w$ is indecomposable.
\end{cor}

\begin{proof}
The finite-length property is immediate from Proposition~\ref{prop:bV-fg} and~\eqref{eqn:inclusion-mod}.  The description of its socle follows from Lemma~\ref{lem:Hom-simple-bV}.
\end{proof}


\subsection{Baby co-Verma filtrations}

We will say that an object $\cF$ in $\modf^\bY_{(I^A_\unip, \cX_A)}(\cR)$ \emph{admits a baby co-Verma filtration} if it admits a finite filtration whose subquotients are isomorphic to baby co-Verma modules.

\begin{lem}
\label{lem:Pi-sigma-bV-filt}
The object $\Phi(j^{\varsigma}_{!*} \cS_\varsigma)$ admits a baby co-Verma filtration with $\cbV_{w_\circ t_{w_\circ(\varsigma)}}$ at the bottom, $\cbV_{t_{w_\circ(\varsigma)}}$ at the top, and other subquotients of the form
$\cbV_{w t_{w_\circ(\varsigma)}}$ with $w \in W \smallsetminus \{e,w_\circ\}$, each appearing once.
\end{lem}

\begin{proof}
As in the proof of Proposition~\ref{prop:bV-fg}, for any $\mu \in \bY$ we have
\begin{multline*}
\Phi(j^{\varsigma}_{!*} \cS_\varsigma)_\mu = \underset{\lambda}{``\varinjlim"} j^{\varsigma}_{!*} \cS_\varsigma \star^{\cL^+G} \cI_*^{w_\circ(\mu)+\lambda} \star^{\cL^+G} \cI_*^{-w_\circ(\lambda)} \\
\cong \underset{\lambda}{``\varinjlim"} j^{w_\circ(\mu)+\lambda+\varsigma}_{*} \cS_{w_\circ(\mu)+\lambda+\varsigma} \star^{\cL^+G} \cI_*^{-w_\circ(\lambda)}.
\end{multline*}
Now the perverse sheaf $j^{w_\circ(\mu)+\lambda+\varsigma}_{*} \cS_{w_\circ(\mu)+\lambda+\varsigma}$ admits a filtration with subquotients of the form $\NGr_{wt_{\mu+w_\circ(\lambda)+w_\circ(\varsigma)}}$ with $w \in W$, each appearing once, with the case $w=w_\circ$ (corresponding to the closed stratum on $G/B$) at the bottom and the case $w=e$ (corresponding to the open stratum on $G/B$) at the top. These filtrations are compatible in a natural way with the transition morphisms in the inductive system, and therefore provide a filtration of $\Phi(j^{\varsigma}_{!*} \cS_\varsigma)_\mu$ whose subquotients are isomorphic to
\[
\underset{\lambda}{``\varinjlim"} \NGr_{wt_{\mu+w_\circ(\lambda)+w_\circ(\varsigma)}} \star^{\cL^+G} \cI_*^{-w_\circ(\lambda)} = (\cbV_{w t_{w_\circ(\varsigma)}})_\mu
\]
(for $w \in W$).
One can check that this collection of filtrations is also compatible with the $\cR$-actions, and therefore provides a filtration of $\Phi(j^{\varsigma}_{!*} \cS_\varsigma)$ with subquotients $\cbV_{w t_{w_\circ(\varsigma)}}$ for $w \in W$; in particular, this object admits a baby co-Verma filtration.
\end{proof}

\begin{lem}
\label{lem:cbV-mult-basic}
For any $w \in W_\ext$, we have
\[
\dim \Hom_{\modf^{\bY}_{I_\unip}(\cR)}(\Phi(j^\varsigma_{!*} \cS_\varsigma), \cbV_w) = \begin{cases}
1 & \text{if $w=y t_\varsigma w_\circ$ for some $y \in W$;} \\
0 & \text{otherwise.}
\end{cases}
\]
\end{lem}
\begin{proof}
By Lemma~\ref{lem:Rfree-hom} we have
\[
\Hom_{\modf^{\bY}_{I_\unip}(\cR)}(\Phi(j^\varsigma_{!*} \cS_\varsigma), \cbV_w) \cong \varinjlim_{\lambda} \Hom_{\Perv_{I_\unip}(\Gr,\bk)}(j^\varsigma_{!*} \cS_\varsigma, \NGr_{wt_{w_\circ(\lambda)}} \star^{\cL^+G} \cI^{-w_\circ(\lambda)}_*).
\]
Now by~\eqref{eqn:duals} and Proposition~\ref{prop:avS-calc}, for any $\lambda \in \bY_+$ such that $wt_{w_\circ(\lambda)} \in W_\ext^S$ we have
\begin{multline*}
\Hom_{\Perv_{I_\unip}(\Gr,\bk)}(j^\varsigma_{!*} \cS_\varsigma, \NGr_{wt_{w_\circ(\lambda)}} \star^{\cL^+G} \cI^{-w_\circ(\lambda)}_*) \\
\cong \Hom_{\Perv_{I_\unip}(\Gr,\bk)} \bigl( (j^\varsigma_{!*} \cS_\varsigma) \star^{\cL^+G} \cI^{\lambda}_!, \NGr_{wt_{w_\circ(\lambda)}} \bigr) \\
\cong \Hom_{\Perv_{I_\unip}(\Gr,\bk)} ( j^{\varsigma+\lambda}_{!} \cS_{\varsigma+\lambda}, \NGr_{wt_{w_\circ(\lambda)}} ).
\end{multline*}
Here $ j^{\varsigma+\lambda}_{!} \cS_{\varsigma+\lambda}$ admits a standard filtration, with subquotients $\DGr_{yt_{\varsigma+\lambda}w_\circ}$ for $y \in W$ (each appearing once). Hence our space is $1$-dimensional if $w$ is of the form $y t_\varsigma w_\circ$ for some $y \in W$, and vanishes otherwise. All the transition morphisms in our inductive limit are isomorphisms.  The lemma follows.
\end{proof}

We now study the behavior of wall-crossing functors with respect to baby co-Verma filtrations.

\begin{lem}
\phantomsection
\label{lem:bV-xi}
\begin{enumerate}
\item
\label{it:bV-xi-1}
For any $s \in S_\aff$, the functor
\[
\xi_s : \modf^\bY_{I_\unip}(\cR) \to \modf^\bY_{I_\unip}(\cR)
\]
sends objects admitting baby co-Verma filtrations to objects admitting baby co-Verma filtrations. More specifically, if $w \in W_\ext$, then if $sw \preceq w$ we have an exact sequence
\[
\cbV_{sw} \hookrightarrow \xi_s(\cbV_w) \twoheadrightarrow \cbV_w,
\]
and if $w \preceq sw$ we have an exact sequence
\[
\cbV_{w} \hookrightarrow \xi_s(\cbV_w) \twoheadrightarrow \cbV_{sw}.
\]
\item
\label{it:bV-xi-2}
For any $\omega \in \Omega$, the functor
\[
\xi_\omega : \modf^\bY_{I_\unip}(\cR) \to \modf^\bY_{I_\unip}(\cR)
\]
sends objects admitting baby co-Verma filtrations to objects admitting baby co-Verma filtrations. More specifically, if $w \in W_\ext$, then
\[
\xi_\omega(\cbV_w) \cong \cbV_{\omega w}.
\]
\end{enumerate}
\end{lem}

\begin{proof}
\eqref{it:bV-xi-1}
By Lemma~\ref{lem:per-order-properties}\eqref{it:per-order-1} and exactness of $\xi_s$ (see Lemma~\ref{lem:functors-Rmod}) it suffices to prove the second claim, which follows from Lemma~\ref{lem:bV-Av}\eqref{it:bV-Av-psi}--\eqref{it:bV-Av-*}.


\eqref{it:bV-xi-2}
The proof is similar to that of~\eqref{it:bV-xi-1}, using~\eqref{eqn:xi-D-N}.
\end{proof}

\section{Projectives and injectives}
\label{sec:proj-R-mod}

In this section, we will prove that $\modf^\bY_{(I_\unip^A, \cX_A)}(\cR)$ is an abelian category with enough projectives and injectives, that the projective and injective objects coincide and arise as direct summands of free $\cR$-modules associated with certain tilting perverse sheaves, and finally that the embedding $\modf^\bY_{(I_\unip^A, \cX_A)}(\cR) \subset \Mod^\bY_{(I_\unip^A, \cX_A)}(\cR)^\flen$ is an equality.

In the special case where $\bk$ has characteristic $0$, we will see (by nearly identical arguments) that $\Perv_{(I_\unip^A, \cX_A)}(\Gr,\bk)$ also has enough projectives and injectives, and that these coincide with certain tilting objects.  

\subsection{Projective and injective perverse sheaves}
\label{ss:proj-covers}

We start with the case of $\Perv_{(I_\unip^A, \cX_A)}(\Gr,\bk)$ (for characteristic-$0$ coefficients).

\begin{thm}
\label{thm:proj-char0-Whit}
Assume that $\bk$ has characteristic $0$. The category $\Perv_{(I^A_\unip, \cX_A)}(\Gr,\bk)$ has enough projectives and enough injectives.  Specifically, for any $w \in {}^A W^S_\ext$, we have $w_A w^\triangle \in {}^A W_\ext^S$, and the projective cover and the injective hull of $\LGr^A_w$ in $\Perv_{(I^A_\unip, \cX_A)}(\Gr,\bk)$ are both isomorphic to the tilting object $\TGr^A_{w_A w^\triangle}$.
\end{thm}

Note that this statement implies that \emph{all} projective objects in $\Perv_{(I^A_\unip, \cX_A)}(\Gr,\bk)$ are also injective and tilting.  Dually, all injective objects are projective and tilting.  (However, there may be tilting objects that are not projective nor injective.)

\begin{proof}
We break the proof up into two cases as follows.

\textit{Case 1. $A = \varnothing$.}  We wish to show that for $w \in W^S_\ext$, the object $\TGr_{w^\triangle}$ is both the projective cover and the injective envelope of $\LGr_w$. Write $w = xt_{w_\circ(\mu-\varsigma)}$  as in Lemma~\ref{lem:proj-prep}.  By Proposition~\ref{prop:projectives}\eqref{it:proj-char0}, $\TGr_{t_{w_\circ(\mu)}}$ is both projective and injective. 
Since they have exact left and right adjoints,
the functors $\xi_{s_1}, \ldots, \xi_{s_r}, \xi_\omega$ send projectives to projectives, and injectives to injectives.  Thus, the perverse sheaf $\xi_{s_r} \cdots \xi_{s_1}\xi_{\omega^{-1}}(\TGr_{t_{w_\circ(\mu)}})$ is both projective and injective.  Lemma~\ref{lem:proj-prep} then implies that $\Perv_{I_\unip}(\Gr,\bk)$ has enough projectives and injectives.

More specifically, let $\PGr_w$ be the projective cover of $\LGr_w$.  By the construction described above, this object is a direct summand of $\xi_{s_r} \cdots \xi_{s_1}\xi_{\omega^{-1}}(\TGr_{t_{w_\circ(\mu)}})$. The latter object is tilting by Lemma~\ref{lem:isom-xi} and~\eqref{eqn:xi-L-T}, and its support is contained in $\overline{\Gr_{w^\triangle}}$ by Lemma~\ref{lem:proj-prep}. Hence $\PGr_w$ is isomorphic to $\TGr_y$ for some $y \in W_\ext^S$ such that $y \leq w^\triangle$, and to show that $\PGr_w \cong \TGr_{w^\triangle}$ it is enough to show that $(\PGr_w)_{|\Gr_{w^\triangle}} \ne 0$, which in turn is equivalent to the claim that
\[
\Hom(\PGr_w, \NGr_{w^\triangle}) \ne 0.
\]
However the left-hand side has dimension $[\NGr_{w^\triangle} : \LGr_w]=[\DGr_{w^\triangle} : \LGr_w]$, which is equal to $1$ by~\eqref{eqn:mult-triangle}; this proves that $\PGr_w \cong \TGr_{w^\triangle}$.

Finally, the injective hull of $\LGr_w$ is the Verdier dual of $\PGr_w$, which is isomorphic to $\TGr_{w^\triangle}$ since every tilting object is Verdier self-dual.

\textit{Case 2. $A \ne \varnothing$.}
For any $w \in {}^A W^S_\ext$ we have $\Av^A_\psi(\LGr_w) \cong \LGr^A_w$ by~\cite[Lemma~3.3(4)]{projGr1}, and the functor $\Av^A_\psi$ is exact (see~\S\ref{ss:Av-Iw}). Hence from the surjection  $\TGr_{w^\triangle} \twoheadrightarrow \LGr_w$ and the embedding $\LGr_w \hookrightarrow \TGr_{w^\triangle}$ obtained in Case~1 we deduce a surjection and an embedding
\[
\Av^A_\psi(\TGr_{w^\triangle}) \twoheadrightarrow \LGr_w^A, \qquad \LGr^A_w \hookrightarrow \Av^A_\psi(\TGr_{w^\triangle}).
\]
Since $\Av^A_\psi$ has exact left and right adjoints (see~\S\ref{ss:Av-Iw}), it sends projective, resp.~injective, objects to projective, resp.~injective, objects.  This already implies that $\Perv_{(I^A_\unip, \cX_A)}(\Gr,\bk)$ has enough projectives and enough injectives, and that an object is projective if and only if it is injective; these objects are therefore tilting. Using Verdier duality (see Remark~\ref{rmk:Verdier-duality}) we see that for any $w \in {}^A W_\ext^S$ the projective cover of $\LGr^A_w$ is also its injective envelope. We now need to determine the label of this object (as an indecomposable tilting object). Fix $w \in {}^A W_\ext^S$, and denote this label by $y$. 

By adjunction and~\cite[Lemma~3.3(4)]{projGr1}, the object $\Av^A_!(\TGr^A_y)$ is the projective cover of $\LGr_w$, so that we have
\[
\Av^A_!(\TGr^A_y) \cong \TGr_{w^\triangle}
\]
by Case~1.
Now, by Proposition~\ref{prop:Av-tilting-indec} the left-hand side is isomorphic to
$\TGr_{w_A y}$, so that $w_A y=w^\triangle$, which finishes the proof.
\end{proof}

\begin{rmk}
\phantomsection
\label{rmk:triangle-AWext}
\begin{enumerate}
\item
\label{it:triangle-AWext}
In the course of the proof of Theorem~\ref{thm:proj-char0-Whit} we have proved that for any $w \in {}^A W^S_\ext$, we have $w_A w^\triangle \in {}^A W_\ext^S$. By definition of ${}^A W_\ext$ (see~\S\ref{ss:coset-representatives}) and~\eqref{eqn:triangle-translation}, it then follows that for any $w \in {}^A W_\ext$ we have $w_A w^\triangle \in {}^A W_\ext$.
\item
In case $\mathrm{char}(\bk)>0$, and if $G$ is not a torus, the category $\Perv_{(I_\unip^A, \cX_A)}(\Gr,\bk)$ does not have any nonzero projective or injective object. In fact, using Verdier duality it suffices to prove this claim for projective objects, and using Remark~\ref{rmk:Av-notkill} one can assume $A=\varnothing$. In this case, if $\cP$ is a nonzero projective object and if $\LGr_w$ is a simple quotient of $\cP$, if $\omega, s_1, \ldots, s_r$ are as in Lemma~\ref{lem:proj-prep}, then $\cP':=\xi_\omega \xi_{s_1} \cdots \xi_{s_r}(\cP)$ is a projective object surjecting to $\xi_\omega \xi_{s_1} \cdots \xi_{s_r}(\LGr_w)$. Then $\Av^S_\psi(\cP')$ is a projective object in $\Perv_{(I_\unip^S,\cX_S)}(\Gr,\bk)$ surjecting to the object $\Av^S_\psi(\xi_\omega \xi_{s_1} \cdots \xi_{s_r}(\LGr_w))$, which is nonzero by Lemmas~\ref{lem:proj-prep}\eqref{it:prep-length} and~\ref{lem:support} combined with~\cite[Lemmas~2.5 and~3.3(4)]{projGr1}. Now as explained in the proof of Proposition~\ref{prop:IW-R-semisimple} we have $\Perv_{(I_\unip^S,\cX_S)}(\Gr,\bk) \cong \Rep(G^\vee_\bk)$; the category $\Rep(G^\vee_\bk)$ therefore possesses a nonzero projective object. Using~\cite[Lemma~I.4.17]{jantzen} one sees that this object is projective in the category of all algebraic $G^\vee_\bk$-modules, which is impossible by the main result of~\cite{donkin-proj}. 
\end{enumerate}
\end{rmk}

We now drop the assumption that $\bk$ has characteristic $0$, and come back to the setting of arbitrary coefficients.

\begin{cor}
\label{cor:tilt-mult}
For any $w \in {}^A W_\ext^\res$, the object $\LGr^A_w$ occurs in both the head and socle of $\TGr^A_{w_A w^\triangle}$ with multiplicity~$1$.
\end{cor}

\begin{proof}
When $\bk$ has characteristic $0$, this claim is already part of Theorem~\ref{thm:proj-char0-Whit}. We must now treat the case when $\bk$ is finite. (This case will imply the case when $\bk$ is an algebraic closure of a finite field.)

First, assume that $A=\varnothing$, and 
continue with the notation from the proof of Lemma~\ref{lem:proj-prep} (for $\mu=\varsigma$, so that $w=x$).  Recall that $\TGr_{w^\triangle}$ is a direct summand of $\xi_{s_r} \cdots \xi_{s_1}\xi_{\omega^{-1}}(\TGr_{t_{w_\circ(\varsigma)}})$.  The dimension calculation in~\eqref{eqn:prep-factor} shows that $\xi_{s_r} \cdots \xi_{s_1}\xi_{\omega^{-1}}(\TGr_{t_{w_\circ(\varsigma)}})$ admits exactly one direct summand admitting a nonzero map to $\nabla_w$, and that this map is unique up to scalar and factors through the socle $\LGr_w$ of $\nabla_w$. Hence to prove that $\LGr_w$ occurs in the head of $\TGr_{w^\triangle}$ it suffices to prove that
\[
\dim \Hom(\TGr_{w^\triangle}, \NGr_w) \geq 1;
\]
then we will know that the multiplicity of $\LGr_w$ in the head of $\TGr_{w^\triangle}$ is exactly one, and since $\TGr_{w^\triangle}$ is Verdier self-dual the claim about its socle will also follow.
Of course, $\dim \Hom(\TGr_{w^\triangle}, \NGr_w)$ is the multiplicity of $\DGr_w$ in any standard filtration of $\TGr_{w^\triangle}$, denoted by $(\TGr_{w^\triangle}: \DGr_w)$; the statement we wish to prove is therefore equivalent to the claim that
\[
(\TGr_{w^\triangle}: \DGr_w) \ge 1.
\]

Let us consider a finite extension $\mathbb{K}$ of $\mathbb{Q}_\ell$ whose ring of integers admits $\bk$ as residue field, and adopt the notation of Lemma~\ref{lem:mult-tilt-modred}. By that lemma we have
\[
(\TGr^{\bk}_{w^\triangle}: \DGr^\bk_w) \ge (\TGr^{\mathbb{K}}_{w^\triangle}: \DGr^{\mathbb{K}}_w).
\]
By the characteristic-$0$ version of the corollary (which, as explained above, is already known) the right-hand side is at least $1$, which proves the desired inequality.

For a general $A$, one observes that
\[
\Hom(\TGr^A_{w_A w^\triangle}, \LGr^A_w) \cong \Hom(\TGr^A_{w_A w^\triangle}, \Av^A_\psi(\LGr_w)) \cong \Hom(\Av^A_!(\TGr^A_{w_A w^\triangle}), \LGr_w)
\]
by~\cite[Lemma~3.3(4)]{projGr1} and adjunction. Using Proposition~\ref{prop:Av-tilting-indec} we deduce that
\[
\Hom(\TGr^A_{w_A w^\triangle}, \LGr^A_w) \cong \Hom(\TGr_{w^\triangle}, \LGr_w),
\]
and the right-hand side is of dimension $1$ by the case $A=\varnothing$. (Note that $w \in W_\ext^\res$.) One shows similarly that $\Hom(\LGr^A_w, \TGr^A_{w_A w^\triangle})$ is $1$-dimensional, which finishes the proof.
\end{proof}

\subsection{Projective and injective \texorpdfstring{$\cR$}{R}-modules}
\label{ss:proj-inj}

We now study projective and injective objects in $\modf^\bY_{(I_\unip^A, \cX_A)}(\cR)$. In this setting, the replacement of the property of admitting a costandard filtration will be the existence of a baby co-Verma filtration. (The replacement for standard filtrations will be introduced later, in Section~\ref{sec:baby-verma}.) We start by constructing an ``explicit'' family of projective and injective objects, in the special case $A=\varnothing$, based on Proposition~\ref{prop:projectives}\eqref{it:proj-R}.

\begin{prop}
\label{prop:inj-hull-bV}
For any $x \in W_\ext$, there exists an object in $\modf^{\bY}_{I_\unip}(\cR)$ with the following properties:
\begin{enumerate}
\item It is both projective and injective as an object of $\Mod^{\bY}_{I_\unip}(\cR)^\flen$.
\item It admits $\LGT_x$ as both a subobject and a quotient.
\item It admits a baby co-Verma filtration with $\cbV_x$ at the bottom, $\cbV_{x^\triangle}$ at the top, and all other subquotients of the form $\cbV_z$ with $x \prec z \prec x^\triangle$.
\end{enumerate}
\end{prop}

\begin{proof}
By periodicity (see in particular~\eqref{eqn:simple-translation} and~\eqref{eqn:cbV-translation}), it suffices to prove this proposition in the case where $x \in W^\res_\ext$.  We assume this from now on.  As in Lemma~\ref{lem:proj-prep}, set $y=t_\varsigma w_\circ x^{-1}$, and choose a reduced expression $y=\omega s_1 \cdots s_r$ (with each $s_i$ in $S_\aff$, and $\omega \in \Omega$).  Lemma~\ref{lem:proj-prep}\eqref{it:prep-length} implies that
\begin{equation}\label{eqn:res-chain}
x \leq s_r x \leq s_{r-1} s_r x \leq \cdots \leq s_1 \cdots s_r x = \omega^{-1} t_\varsigma w_\circ
\end{equation}
and
\begin{equation}\label{eqn:res-tri-chain}
x^\triangle \ge s_r x^\triangle \ge s_{r-1}s_r x^\triangle \ge \cdots \ge s_1 \cdots s_r x^\triangle = \omega^{-1} t_{w_\circ(\varsigma)}.
\end{equation}
All elements in these chains of inequalities belong to $W_\ext^S$ by~\cite[Lemma~2.2]{projGr1}, so by Lemma~\ref{lem:per-order-properties}\eqref{it:per-order-3} we also have the same chains of inequalities when the symbols $\leq$ and $\geq$ are replaced by $\preceq$ and $\succeq$, respectively. 

Since the functors $\xi_{s_1}, \ldots, \xi_{s_r}, \xi_{\omega^{-1}}$ are exact and have exact left and right adjoints (see Lemma~\ref{lem:functors-Rmod}), they send projectives to projectives and injectives to injectives.  Applying Proposition~\ref{prop:projectives}\eqref{it:proj-R} to
\[
j^\varsigma_{!*}\cS_\varsigma \cong \TGr_{t_{w_\circ(\varsigma)}},
\]
we see that the object
\[
\xi_{s_r} \cdots \xi_{s_1}\xi_{\omega^{-1}}(\Phi(\TGr_{t_{w_\circ(\varsigma)}})) \cong \Phi(\xi_{s_r} \cdots \xi_{s_1}\xi_{\omega^{-1}}(\TGr_{t_{w_\circ(\varsigma)}}))
\]
is both projective and injective. Referring to Lemma~\ref{lem:proj-prep} once again, we see that the perverse sheaf $\xi_{s_r} \cdots \xi_{s_1}\xi_{\omega^{-1}}(\TGr_{t_{w_\circ(\varsigma)}})$ contains $\LGr_x$ in its head and socle, so applying the (faithful) functor $\Phi$ (see~\S\ref{ss:graded-R-modules}) yields nonzero maps
\[
\LGT_x \to \Phi(\xi_{s_r} \cdots \xi_{s_1}\xi_{\omega^{-1}}(\TGr_{t_{w_\circ(\varsigma)}})),
\qquad
\Phi(\xi_{s_r} \cdots \xi_{s_1}\xi_{\omega^{-1}}(\TGr_{t_{w_\circ(\varsigma)}})) \to \LGT_x,
\]
which must be injective and surjective respectively by simplicity of $\LGT_x$.

Finally, by Lemma~\ref{lem:Pi-sigma-bV-filt}, the object $\Phi(j^{\varsigma}_{!*} \cS_\varsigma)$ admits a baby co-Verma filtration with $\cbV_{t_{\varsigma}w_\circ}$ at the bottom, $\cbV_{t_{w_\circ(\varsigma)}} = \cbV_{(t_\varsigma w_\circ)^\triangle}$ at the top, and the other subquotients of the form $\cbV_{w t_{w_\circ(\varsigma)}}$ with $w \in W \smallsetminus \{e, w_\circ\}$. All the elements $w t_{w_\circ(\varsigma)}$ belong to $W_\ext^S$, and they satisfy $t_{\varsigma} w_\circ \leq w t_{w_\circ(\varsigma)} \leq (t_\varsigma w_\circ)^\triangle$, and hence
\[
t_{\varsigma} w_\circ \preceq w t_{w_\circ(\varsigma)} \preceq (t_{\varsigma} w_\circ)^\triangle
\]
by Lemma~\ref{lem:per-order-properties}\eqref{it:per-order-3}. Combining these observations with Lemma~\ref{lem:bV-xi} and the chains of inequalities~\eqref{eqn:res-chain} and~\eqref{eqn:res-tri-chain}, we see that $\xi_{s_r} \cdots \xi_{s_1}\xi_{\omega^{-1}}(\Phi(\TGr_{t_{w_\circ(\varsigma)}}))$ admits a baby co-Verma filtration with $\cbV_{x}$ at the bottom, $\cbV_{x^\triangle}$ at the top, and with the other subquotients of the form $\cbV_{y' w t_{w_\circ(\varsigma)}}$ with $w \in W$ and $y' \in W_\ext$ such that $y' \leq y^{-1}$ and $y' w t_{w_\circ(\varsigma)} \notin \{x,x^\triangle\}$. All these elements satisfy $x \preceq y' w t_{w_\circ(\varsigma)}$ by Lemma~\ref{lem:per-order-properties}\eqref{it:per-order-4}, and
$y' w t_{w_\circ(\varsigma)} \preceq x^\triangle$
by Lemma~\ref{lem:per-order-properties}\eqref{it:per-order-5}, as desired.
\end{proof}


\begin{cor}
\label{cor:multiplicites-bV}
If $w,x \in {}^A W_\ext$ satisfy $[\cbV^A_w : \LGT^A_x] \neq 0$, we have $x \preceq w \preceq w_A x^\triangle$. Moreover, we have $[\cbV^A_x : \LGT^A_x]=1$ and $[\cbV^A_{w_A x^\triangle} : \LGT_x] \le 1$.
\end{cor}

\begin{proof}
We first treat the special case where $A = \varnothing$.  In this case, we assume (by periodicity) that $x \in W^\res_\ext$, and we retain the notation from the proof of Proposition~\ref{prop:inj-hull-bV}.  Since $\LGT_x$ is a quotient of the projective object $\xi_{s_r} \cdots \xi_{s_1}\xi_{\omega^{-1}}(\Phi(j^\varsigma_{!*} \cS_\varsigma))$, we have
\[
[\cbV_w : \LGT_x ] \le \dim \Hom( \xi_{s_r} \cdots \xi_{s_1}\xi_{\omega^{-1}}(\Phi(j^\varsigma_{!*} \cS_\varsigma)), \cbV_w).
\]
By adjunction we have
\[
\Hom(\xi_{s_r} \cdots \xi_{s_1} \xi_{\omega^{-1}} \Phi(j^\varsigma_{!*} \cS_\varsigma), \cbV_w) \cong \Hom( \Phi(j^\varsigma_{!*} \cS_\varsigma), \xi_{\omega} \xi_{s_1} \cdots \xi_{s_r} (\cbV_w)).
\]
By Lemma~\ref{lem:bV-xi}, the object $\xi_{\omega} \xi_{s_1} \cdots \xi_{s_r} (\cbV_w)$ has a baby co-Verma filtration, whose subquotients have the form $\cbV_{y'w}$ with $y' \leq y$; moreover $\cbV_{yw}$ appears once in this filtration. By projectivity of $\Phi(j^\varsigma_{!*} \cS_\varsigma)$ and Lemma~\ref{lem:cbV-mult-basic}, we deduce that the space $\Hom(\xi_{s_r} \cdots \xi_{s_1} \xi_{\omega^{-1}} \Phi(j^\varsigma_{!*} \cS_\varsigma), \cbV_w)$ vanishes unless
\begin{equation}\label{eqn:Hom-Q-crit}
y'w = zt_\varsigma w_\circ \quad \text{for some $z \in W$ and $y' \in W_\ext$ such that $y' \leq y$.}
\end{equation}
Here, we have $t_\varsigma w_\circ \preceq zt_\varsigma w_\circ \preceq t_{w_\circ(\varsigma)}$, and as in the proof of Proposition~\ref{prop:inj-hull-bV} we obtain using Lemma~\ref{lem:per-order-properties}\eqref{it:per-order-4} that $w = (y')^{-1} zt_\varsigma w_\circ$ satisfies
\[
x = y^{-1}t_\varsigma w_\circ \preceq w = (y')^{-1} zt_\varsigma w_\circ  \preceq y^{-1} t_{w_\circ(\varsigma)} = x^\triangle.
\]
We have shown that $[\cbV_w : \LGT_x] \ne 0$ implies that $x \preceq w \preceq x^\triangle$.

Now let us take $w = x$.  Lemma~\ref{lem:proj-prep}\eqref{it:prep-length} implies that $y'x \le yx = t_\varsigma w_\circ$ for any $y' \le y$, so condition~\eqref{eqn:Hom-Q-crit} is satisfied only for $y' = y$ and $z = e$.  It follows that $\dim \Hom(\xi_{s_r} \cdots \xi_{s_1} \xi_{\omega^{-1}} \Phi(j^\varsigma_{!*} \cS_\varsigma), \cbV_x) = 1$, and hence that $[\cbV_x : \LGT_x] \leq 1$.
Since we know that $\LGT_x$ is the socle of $\cbV_x$ (see Corollary~\ref{cor:cbV-socle}) this multiplicity is then equal to $1$.

Finally, take $w = x^\triangle$.  In this case, Lemma~\ref{lem:proj-prep}\eqref{it:prep-length} implies that $y'w \ge yx^\triangle = t_{w_\circ(\varsigma)} = w_\circ t_{\varsigma} w_\circ$ for any $y' \le y$, so condition~\eqref{eqn:Hom-Q-crit} is satisfied only for $y' = y$ and $z = w_\circ$.  As in the previous paragraph, we conclude that $[\cbV_{x^\triangle} : \LGT_x] \le 1$.  This concludes the proof in the case where $A = \varnothing$.

Now suppose $A \ne \varnothing$, and let $w,x \in {}^A W_\ext$.  By Lemma~\ref{lem:bV-Av}\eqref{it:simples-Av}--\eqref{it:bV-Av-psi}, we have
\[
[ \cbV^A_w : \LGT^A_x ] = [\cbV_w : \LGT_x ] = [\cbV_{w_Aw} : \LGT_x ].
\]
Using the $A \ne \varnothing$ case of the corollary, we see that $[ \cbV^A_w : \LGT^A_x ] \ne 0$ implies that $x \preceq w$ and $w_Aw \preceq x^\triangle$.  Since $w_A x^\triangle$ lies in ${}^A W_\ext$ (see Remark~\ref{rmk:triangle-AWext}\eqref{it:triangle-AWext}), Lemma~\ref{lem:per-order-coset} tells us that the latter condition is equivalent to $w \preceq w_Ax^\triangle$.  The claims that $[\cbV^A_x : \LGT^A_x] = 1$ and $[\cbV^A_{w_A x^\triangle} : \LGT^A_x] \le 1$ likewise follow from the $A = \varnothing$ case.
\end{proof}

\begin{rmk}
\phantomsection
\label{rmk:cbV-Groth-gp}
\begin{enumerate}
\item
We will see in Proposition~\ref{prop:cbV-head} below that, in fact, in the setting of Corollary~\ref{cor:multiplicites-bV} we always have $[\cbV^A_{w_A x^\triangle} : \LGT_x] = 1$.
\item
The information on composition factors in Corollary~\ref{cor:multiplicites-bV} implies that the family $([\cbV^A_w] : w \in {}^A W_\ext)$ in the Grothendieck group $[\Mod^\bY_{(I_\unip^A, \cX_A)}(\cR)^\flen]$ is linearly independent. (This family is \emph{not} a basis, however.) This implies that if $\cM \in \Mod^\bY_{(I_\unip^A, \cX_A)}(\cR)^\flen$ admits a baby co-Verma filtration, then the number $(\cM : \cbV^A_w)$ of occurrences of a given baby co-Verma module $\cbV_w$ in such a filtration is independent of the choice of filtration; in fact these numbers are determined by the equality
\[
[\cM] = \sum_{w \in W_\ext} (\cM : \cbV_w) \cdot [\cbV_w]
\]
in $\mathsf{K}^0(\Mod^\bY_{(I_\unip^A, \cX_A)}(\cR)^\flen)$.
(Later, after we prove Theorem~\ref{thm:proj-Hecke-Whit}, we will be able to apply these comments to $[\modf^\bY_{(I_\unip^A, \cX_A)}(\cR)]$ instead.)
\end{enumerate}
\end{rmk}

\begin{prop}
\label{prop:bV-injective}
For any $w \in {}^A W_\ext$, $\cbV^A_w$ is the injective hull of $\LGT^A_w$ in the Serre subcategory of $\Mod^\bY_{(I^A_\unip,\cX_A)}(\cR)^\flen$ generated by the simple objects of the form $\LGT^A_y$ with $y \not\succ w$.
\end{prop}

\begin{proof}
Recall from Corollary~\ref{cor:cbV-socle} and Corollary~\ref{cor:multiplicites-bV} that $\cbV^A_w$ is indecomposable; its socle is $\LGT^A_w$; and it belongs to the Serre subcategory described in the statement of the proposition.  It remains to show that it is injective as an object of this subcategory.  In other words, we must show that
\[
\Ext^1_{\Mod^\bY_{(I^A_\unip,\cX_A)}(\cR)^\flen}(\LGT^A_y, \cbV^A_w) = 0
\qquad\text{if $y \not\succ w$.}
\]

Let us first treat the special case $A = \varnothing$.  In this case, we can invoke Proposition~\ref{prop:inj-hull-bV} to find an injective object $\cM \in \Mod^\bY_{I_\unip}(\cR)^\flen$ such that there is an inclusion $\cbV_w \hookrightarrow \cM$ whose cokernel $\cM'$ admits a baby co-Verma filtration by various $\cbV_u$ with $u \succ w$.  Since $y \not\succ w$ by assumption, Lemma~\ref{lem:Hom-simple-bV} tells us that $\Hom(\LGT_y, \cbV_u) = 0$ for any $u \succ w$.  It follows that $\Hom(\LGT_y, \cM') = 0$.  The exact sequence
\[
\cdots \to \Hom(\LGT_y, \cM') \to \Ext^1(\LGT_y, \cbV_w) \to \Ext^1(\LGT_y, \cM) \to \cdots
\]
then shows that $\Ext^1_{\Mod^{\bY}_{I_\unip}(\cR)^\flen}(\LGT_y, \cbV_w) = 0$.

Now suppose that $A \ne \varnothing$.  By Lemma~\ref{lem:bV-Av}\eqref{it:simples-Av} and adjunction we have
\[
\Ext^1(\LGT^A_y, \cbV^A_w) \cong \Ext^1(\Av^A_\psi(\LGT_y), \cbV^A_w)
\cong \Ext^1(\LGT_y, \Av^A_*(\cbV^A_w)).
\]
On the right-hand side, by Lemma~\ref{lem:bV-Av}\eqref{it:bV-Av-*} the object $\Av^A_*(\cbV^A_w)$ admits a filtration with successive subquotients the objects $\cbV_{vw}$ with $v \in W_A$. Using the $A = \varnothing$ case proved above, we conclude that $\Ext^1_{\Mod^\bY_{I_\unip}(\cR)^\flen}(\LGT_y, \Av^A_*(\cbV^A_w))$ vanishes unless $y \succ vw$ for such a $v$. Now we have $w \preceq vw$ (see Remark~\ref{rmk:min-preceq}), so that this condition implies that $y \succ w$.
\end{proof}

\begin{rmk}
\label{rmk:order-cbV-filtration}
Combining the information in Corollary~\ref{cor:multiplicites-bV} and Proposition~\ref{prop:bV-injective}, we obtain that $\Ext^1(\cbV^A_y, \cbV^A_w)=0$ unless $w \prec y$. This implies that if an object $\cM$ admits a baby co-Verma filtration, and if we choose a numbering $w_1, \ldots, w_n$ of the elements $z$ such that $(\cM : \cbV_z) \neq 0$ (counted with multiplicities) such that $w_i \prec w_j \Rightarrow i < j$, then there exists a chain of embeddings
\[
0 = \cM_0 \subset \cM_1 \subset \cdots \subset \cM_{n-1} \subset \cM_n = \cM 
\]
such that $\cM_i / \cM_{i-1} \cong \cbV^A_{w_i}$ for any $i \in \{1, \ldots, n\}$.
\end{rmk}


We can finally state and prove the main result of this section.

\begin{thm}
\phantomsection
\label{thm:proj-Hecke-Whit}
\begin{enumerate}
\item
\label{it:proj-enough}
The categories $\modf^\bY_{(I_\unip^A, \cX_A)}(\cR)$ and $\Mod^\bY_{(I_\unip^A, \cX_A)}(\cR)^\flen$ coincide, and this abelian category has enough projectives and enough injectives; moreover, an object is injective iff it is projective.  
\item
\label{it:proj-filt}
For $w \in {}^A W_\ext$, let $\injh^A_w$ denote the injective hull of $\LGT^A_w$.  Then $\injh^A_w$ admits a baby co-Verma filtration with subquotients of the form $\cbV^A_y$ with $y \in {}^A W_\ext$ which satisfies $w \preceq y$.
\item
\label{it:Av-injh}
For $w \in {}^A W_\ext$, we have $\Av^A_*(\injh_w^A) \cong \injh_w$.
\end{enumerate}
\end{thm}

\begin{proof}
\eqref{it:proj-enough}
When $A = \varnothing$, Proposition~\ref{prop:inj-hull-bV} tells us that every simple object in $\Mod^\bY_{I_\unip}(\cR)^\flen$ embeds in a finitely generated injective $\cR$-module that is also projective, and is a quotient of a finitely generated projective $\cR$-module that is also injective.  For general $A$, because $\Av^A_\psi$ is exact and has exact left and right adjoints, it sends projectives to projectives and injectives to injectives.  Given $w \in {}^A W_\ext$, apply $\Av^A_\psi$ to a nonzero map $\LGT_w \hookrightarrow \cM$ or $\cM \twoheadrightarrow \LGT_w$, where $\cM$ is finitely generated, projective, and injective; we conclude that every $\LGT^A_w$
embeds in a finitely generated injective $\cR$-module that is also projective, and is a quotient of a finitely generated projective $\cR$-module that is also injective. 

As a consequence, $\Mod^\bY_{(I_\unip^A, \cX_A)}(\cR)^\flen$ has enough projectives and injectives, and these classes coincide and consist of finitely generated $\cR$-modules.  In particular, every object of $\Mod^\bY_{(I_\unip^A, \cX_A)}(\cR)^\flen$ is a quotient of a finitely generated module, which implies that
$\modf^\bY_{(I_\unip^A, \cX_A)}(\cR) = \Mod^\bY_{(I_\unip^A, \cX_A)}(\cR)^\flen$.

\eqref{it:proj-filt}
We will apply a kind of ``highest weight'' formalism developed in~\cite{bs}. (Note that $\modf^\bY_{(I_\unip^A, \cX_A)}(\cR)$ is \emph{not} a highest weight category in the sense considered in, for instance,~\cite{acr,modrap1,ar:dkf}, because the poset that governs it has no minimal element.)

All objects in $\modf^\bY_{(I_\unip^A, \cX_A)}(\cR)$ have finite length, and by Lemma~\ref{lem:R-findim} all morphism spaces in this category are finite-dimensional; by~\cite[Lemma~2.1]{bs}, this category is therefore a ``locally finite abelian category'' in their terminology. By part~\eqref{it:proj-enough}, this category has enough injective and projective objects; hence by~\cite[Corollary~2.20]{bs} it is an ``essentially finite abelian category.'' Next, we define a ``stratification'' on this category in the sense of~\cite[\S 3.1]{bs} with underlying poset $({}^A W_\ext,\preceq)$, and with the labeling of simple objects given by $w \mapsto \LGT^A_w$. (The function ``$\rho$'' of~\cite[Definition~3.1]{bs} is therefore the identity map for this stratification.) This stratification is ``essentially finite.'' Comparing Corollary~\ref{cor:multiplicites-bV} and Proposition~\ref{prop:bV-injective} with~\cite[Lemma~3.1]{bs} we see that for any $w \in {}^A W_\ext$ the baby co-Verma module $\cbV^A_w$ is isomorphic to the objects denoted $\nabla(w)$ and $\bar{\nabla}(w)$ in~\cite{bs}.  In view of~\cite[Lemma~3.4]{bs}, this implies that all the strata are ``simple'' in the terminology of~\cite{bs}.

Next, we claim that condition $(\widehat{I\nabla})$ of~\cite[Remark~3.6]{bs} holds.  Translated into the language of the present paper, this condition says that for any $w \in {}^A W_\ext$, there exists an injective object admitting a baby co-Verma filtration with $\cbV^A_w$ at the bottom, and all other subquotients of the form $\cbV^A_z$ with $z \succeq w$.  For $A = \varnothing$, this claim is part of Proposition~\ref{prop:inj-hull-bV}.  For general $A$, it follows from Proposition~\ref{prop:inj-hull-bV} by applying $\Av^A_\psi$ and using Lemma~\ref{lem:bV-Av}\eqref{it:bV-Av-psi}.

Applying~\cite[Theorem~3.5]{bs}, we see that $\modf^\bY_{(I_\unip^A, \cX_A)}(\cR)$ is an \emph{essentially finite highest weight category} in the sense of~\cite[Definition~3.7]{bs}. In more concrete terms, this means that the injective envelope $\injh^A_w$ of $\LGT^A_w$ admits a baby co-Verma filtration whose subquotients $\cbV^A_y$ satisfy $y \succeq w$.

\eqref{it:Av-injh}
Since $\Av^A_*$ has an exact left adjoint, it sends injectives to injectives, so $\Av^A_*(\injh^A_w)$ is injective.  To show that it is isomorphic to $\injh_w$, it is enough to show that its socle is isomorphic to $\LGT_w$, or in other words that
\[
\dim \Hom(\LGT_y, \Av^A_*(\injh^A_w)) =
\begin{cases}
1 & \text{if $y = w$,} \\
0 & \text{otherwise.}
\end{cases}
\]
This claim holds by adjunction and Lemma~\ref{lem:bV-Av}\eqref{it:simples-Av}.
\end{proof}

It is clear that for any $w \in {}^A W_\ext$ and $\lambda \in \bY$ we have
\begin{equation}
\label{eqn:injh-translation}
\injh^A_{wt_\lambda} = \injh^A_w \langle -\lambda \rangle;
\end{equation}
in particular, in order to understand all these objects it is enough to understand those whose label belongs to ${}^A W_\ext^\res$. 

\begin{rmk}
\phantomsection
\label{rmk:free-proj-inj}
\begin{enumerate}
\item
\label{it:free-proj-inj}
The proof of Theorem~\ref{thm:proj-Hecke-Whit} provides a slightly more precise statement than the mere existence of enough projective and injective objects: it implies that any object of $\modf^\bY_{(I_\unip^A, \cX_A)}(\cR)$ is a quotient (resp.~a subobject) of a direct sum of objects of the form $\Phi^A(\Av^A_\psi(\TGr_{x})) \langle \mu \rangle$ with $x \in W_\ext^S$ and $\mu \in \bY$, these objects being both projective and injective.
\item
The proof of Theorem~\ref{thm:proj-Hecke-Whit} also shows that each $\injh^A_{w}$ (hence each projective object in $\modf^\bY_{(I_\unip^A, \cX_A)}(\cR)$) remains projective in the larger category $\Mod^\bY_{(I_\unip^A, \cX_A)}(\cR)$.
\item
Theorem~\ref{thm:proj-Hecke-Whit} implies that the projective cover of $\LGT^A_w$ is also an indecomposable injective object.  Define a map 
\begin{equation}
\label{eqn:def-iota}
\iota_A : {}^A W_\ext \simto {}^A W_\ext
\end{equation}
by requiring that $\LGT^A_{\iota_A(w)}$ be the socle of the projective cover of $\LGT^A_w$ (equivalently, $\injh^A_{\iota_A(w)}$ is the projective cover of $\LGT^A_w$).  We will see later (see Proposition~\ref{prop:reciprocity}\eqref{it:recip-proj}) that in fact $\iota_A$ is the identity map; in other words, the projective cover and injective envelope of $\LGT^A_w$ coincide.
\end{enumerate}
\end{rmk}

\subsection{Injective \texorpdfstring{$\cR$}{R}-modules and tilting perverse sheaves}
\label{ss:injective-tilting}

In this subsection we study the relation between injective objects in $\modf^\bY_{(I_\unip^A, \cX_A)}(\cR)$ and tilting objects in $\Perv_{(I_\unip^A, \cX_A)}(\Gr,\bk)$.


\begin{prop}
\phantomsection
\label{prop:tilt-proj-inj-R}
\begin{enumerate}
\item
\label{it:Phi-tilting-injective}
For any $x \in {}^A W_\ext^S$ the object $\Phi^A(\TGr^A_{w_A x^\triangle})$ is both injective and projective.
\item
\label{it:Phi-tilting-injective-2}
If $x \in {}^A W_\ext^\res$, then $\Phi^A(\TGr^A_{w_A x^\triangle})$
 contains $\injh^A_x$, resp.~$\injh^A_{\iota_A(x)}$, as a direct summand with multiplicity $1$, and does not admit any direct summand of the form $\injh^A_{xt_\mu}$, resp.~$\injh^A_{\iota_A(xt_\mu)}$, with $\mu \in \bY \smallsetminus \{0\}$.
\end{enumerate}
\end{prop}

A converse to part~\eqref{it:Phi-tilting-injective} will be proved in Proposition~\ref{prop:injective-tilting}.

\begin{proof}
\eqref{it:Phi-tilting-injective}
First, assume that $A=\varnothing$ and $x \in W_\ext^\res$.
In the proof of Proposition~\ref{prop:inj-hull-bV} we have constructed a projective and injective object admitting $\LGT_x$ both as a subobject and as a quotient. By Lemma~\ref{lem:proj-prep}\eqref{it:prep-tilt} this object contains $\Phi(\TGr_{x^\triangle})$ as a direct summand; the latter object is therefore also projective and injective.

Now we continue to assume that $A=\varnothing$, but take a general $w \in W_\ext^S$. We can write $w=yt_\lambda$ for some $y \in W_\ext^\res$ and $\lambda \in -\bY_+$, see~\eqref{eqn:WS-Wres}, and then we have $w^\triangle = y^{\triangle} t_\lambda$, see~\eqref{eqn:triangle-translation}. By Lemma~\ref{lem:twtri-standard} the object $\TGr_{y^\triangle} \star^{\cL^+G} \cT^{w_\circ(\lambda)}$ is tilting, and support considerations show that it contains $\TGr_{y^\triangle t_\lambda}$ as a direct summand. On the other hand, using the formula in Lemma~\ref{lem:frobtwist-free} and the fact that $\Phi(\TGr_{y^\triangle})$ is projective and injective we see that $\Phi(\TGr_{y^\triangle} \star^{\cL^+G} \cT^{w_\circ(\lambda)})$ is also injective and projective. Hence so is $\Phi(\TGr_{w^\triangle})$.

Finally we consider a general subset $A$, and $x \in {}^A W_\ext^S$. By Theorem~\ref{thm:proj-char0-Whit} we have $w_A x^\triangle \in {}^A W_\ext^S$, and by Proposition~\ref{prop:Av-tilting-indec} we know that $\Av^A_!(\Phi^A(\TGr^A_{w_A x^\triangle})) \cong \Phi(\TGr_{x^\triangle})$. Since the functor $\Av^A_!$ has an exact right adjoint, and since $\Phi(\TGr_{x^\triangle})$ is projective (by the case already treated, and since $x \in W_\ext^S$), this shows that $\Phi^A(\TGr^A_{w_A x^\triangle})$ is projective. A similar argument using $\Av^A_*$ instead of $\Av^A_!$ shows that this object is also injective. (Alternatively, one can use the fact that projective objects are automatically injective, see Theorem~\ref{thm:proj-Hecke-Whit}.)

\eqref{it:Phi-tilting-injective-2}
First, assume that $A=\varnothing$. Corollary~\ref{cor:tilt-mult} implies that $\Phi(\TGr_{x^\triangle})$ admits $\LGT_x$ as both a subobject and a quotient. It follows that both $\injh_x$ and $\injh_{\iota(x)}$ are direct summands in $\Phi(\TGr_{x^\triangle})$.

To conclude the proof in this case, we will prove that the object constructed in the proof of Proposition~\ref{prop:inj-hull-bV} admits $\injh_{\iota(x)}$ as a direct summand with multiplicity $1$, and no other direct summand of the form $\injh_{\iota(xt_\mu)}$. (A similar argument will apply for the injective hulls; alternatively this case can be deduced using the Verdier duality constructed in~\S\ref{ss:verdier} below.) Set $y=t_\varsigma w_\circ x^{-1}$, and consider a reduced expression $y = \omega s_1 \cdots s_r$.
Using adjunction, our claim will follow if we prove that
\begin{equation}
\label{eqn:prep-dim-calc3}
\dim \Hom_{\modf^\bY_{I_\unip}(\cR)}(\Phi(\TGr_{w_\circ(\varsigma)}), \Phi(\xi_\omega \xi_{s_1} \cdots \xi_{s_r} (\LGr_x)) \langle \mu \rangle) = \begin{cases}
1 & \text{if $\mu=0$;} \\
0 & \text{otherwise}.
\end{cases}
\end{equation}

As a preparation, let us first prove that for $\nu \in \bY_+ \smallsetminus \{0\}$ we have
\begin{equation}
\label{eqn:prep-dim-calc4}
\Hom(\TGr_{t_{w_\circ(\varsigma)}}, \xi_\omega \xi_{s_1} \cdots \xi_{s_r} (\LGr_x) \star^{\cL^+G} \cI^\nu_*) = 0,
\end{equation}
or equivalently (by adjunction and~\eqref{eqn:duals})
\begin{equation}
\label{eqn:prep-dim-calc5}
\Hom(\TGr_{t_{w_\circ(\varsigma)}} \star^{\cL^+G} \cI^{-w_\circ(\nu)}_!, \xi_\omega \xi_{s_1} \cdots \xi_{s_r} (\LGr_x) ) = 0.
\end{equation}
Proposition~\ref{prop:avS-calc}\eqref{it:avS-!} tells us that $\TGr_{t_{w_\circ(\varsigma)}} \star^{\cL^+G} \cI^{-w_\circ(\nu)}_!$ has a unique simple quotient, isomorphic to $\LGr_{t_{\varsigma-w_\circ(\nu)} w_\circ}$. Now we have $yx=t_\varsigma w_\circ$, and $\ell(yx)=\ell(x)+r$ by Lemma~\ref{lem:res-complement}. By Lemma~\ref{lem:support}, it follows that $\xi_\omega \xi_{s_1} \cdots \xi_{s_r} (\LGr_x)$ is supported on $\overline{\Gr_{t_\varsigma w_\circ}}$. Now $t_{\varsigma} w_\circ$ belongs to $W_\ext^S$, so that by~\cite[Lemma~2.7]{projGr1} we have
\[
\ell(t_{\varsigma-w_\circ(\nu)} w_\circ)=\ell(t_{\varsigma} w_\circ t_{-\nu}) = \ell(t_{\varsigma} w_\circ) + \ell( t_{-\nu});
\]
if $\nu \neq 0$ the orbit $\Gr_{t_{\varsigma-w_\circ(\nu)} w_\circ}$ is therefore not contained in $\overline{\Gr_{t_\varsigma w_\circ}}$, which implies that $\LGr_{t_{\varsigma-w_\circ(\nu)} w_\circ}$ is not a composition factor of $\xi_\omega \xi_{s_1} \cdots \xi_{s_r} (\LGr_x)$. This proves~\eqref{eqn:prep-dim-calc5}.

Now, let us prove~\eqref{eqn:prep-dim-calc3} in case $\mu \neq 0$. By Lemma~\ref{lem:Rfree-hom}, the space under consideration equals
\begin{multline}
\label{eqn:dim-calc-lim}
\Hom(\TGr_{t_{w_\circ(\varsigma)}}, \xi_\omega \xi_{s_1} \cdots \xi_{s_r} (\LGr_x) \star^{\cL^+G} \cR_{-\mu}) =  \\
\varinjlim_\lambda \Hom(\TGr_{t_{w_\circ(\varsigma)}}, \xi_\omega \xi_{s_1} \cdots \xi_{s_r} (\LGr_x) \star^{\cL^+G} \cI^{-w_\circ(\mu)+\lambda}_* \star^{\cL^+G} \cI^{-w_\circ(\lambda)}_*).
\end{multline}
For any $\lambda$
the perverse sheaf $\cI^{-w_\circ(\mu)+\lambda}_* \star^{\cL^+G} \cI^{-w_\circ(\lambda)}_*$ admits a costandard filtration, see~\cite[Proposition~4.8]{projGr1}; moreover, the object $\IC^0=\cI^0_*$ does not occur in such a filtration since
\[
\Hom(\cI_!^0, \cI^{-w_\circ(\mu)+\lambda}_* \star^{\cL^+G} \cI^{-w_\circ(\lambda)}_*)=\Hom(\cI^{\lambda}_!, \cI^{-w_\circ(\mu)+\lambda}_*)=0.
\]
In view of~\eqref{eqn:prep-dim-calc4} this implies that
\[
\Hom(\TGr_{t_{w_\circ(\varsigma)}}, \xi_\omega \xi_{s_1} \cdots \xi_{s_r} (\LGr_x) \star^{\cL^+G} \cI^{-w_\circ(\mu)+\lambda}_* \star^{\cL^+G} \cI^{-w_\circ(\lambda)}_*)=0
\]
for any $\lambda$, which proves~\eqref{eqn:prep-dim-calc3} in this case.

Finally, assume $\mu=0$. In this case the space we have to consider is
\begin{multline}
\label{eqn:dim-calc-lim-2}
\Hom(\TGr_{t_{w_\circ(\varsigma)}}, \xi_\omega \xi_{s_1} \cdots \xi_{s_r} (\LGr_x) \star^{\cL^+G} \cR_{0}) =  \\
\varinjlim_\lambda \Hom(\TGr_{t_{w_\circ(\varsigma)}}, \xi_\omega \xi_{s_1} \cdots \xi_{s_r} (\LGr_x) \star^{\cL^+G} \cI^{\lambda}_* \star^{\cL^+G} \cI^{-w_\circ(\lambda)}_*).
\end{multline}
Here again $\cI^{\lambda}_* \star^{\cL^+G} \cI^{-w_\circ(\lambda)}_*$ admits a costandard filtration, and in this case we have an embedding $\IC^0 \to  \cI^{\lambda}_* \star^{\cL^+G} \cI^{-w_\circ(\lambda)}_*$ whose cokernel is an extension of objects of the form $\cI_*^\nu$ with $\nu \neq 0$. We have obtained in the course of 
the proof of Lemma~\ref{lem:proj-prep} that
\begin{equation}
\label{eqn:prep-dim-calc2}
\dim \Hom_{\Perv_{I_\unip}(\Gr,\bk)}(\TGr_{t_{w_\circ(\varsigma)}}, \xi_\omega \xi_{s_1} \cdots \xi_{s_r} (\LGr_x)) = 1,
\end{equation}
see~\eqref{eqn:prep-dim-calc}. By the same considerations as above this implies that for any $\lambda$ we have
\[
\dim \Hom(\TGr_{t_{w_\circ(\varsigma)}}, \xi_\omega \xi_{s_1} \cdots \xi_{s_r} (\LGr_x) \star^{\cL^+G} \cI^{\lambda}_* \star^{\cL^+G} \cI^{-w_\circ(\lambda)}_*)=1.
\]
One can easily check that the transition morphisms in our inductive system are nonzero, which finishes the proof of~\eqref{eqn:prep-dim-calc3}, and hence of the statement in case $A=\varnothing$.

To treat the case of a general subset $A$, one simply observes that for $x \in {}^A W_\ext^\res$ and $\mu \in \bY$ we have
\begin{multline*}
\Hom(\Phi^A(\TGr^A_{w_A x^\triangle}), \LGT^A_{xt_\mu}) \cong \Hom(\Phi^A(\TGr^A_{w_A x^\triangle}), \Av^A_\psi(\LGT_{xt_\mu})) \\
\cong \Hom(\Av^A_!(\Phi(\TGr^A_{w_A x^\triangle})), \LGT_{xt_\mu}) \cong \Hom(\Phi(\TGr_{x^\triangle}),\LGT_{xt_\mu})
\end{multline*}
by Lemma~\ref{lem:bV-Av}\eqref{it:simples-Av}, adjunction and then Proposition~\ref{prop:Av-tilting-indec}. Then the claim follows from the case $A=\varnothing$ since $x \in W_\ext^\res$.
%
\end{proof}


\section{Ungraded \texorpdfstring{$\cR$}{R}-modules}
\label{sec:ungrad-R-mod}

In this section we present a variant of the theory developed so far, which omits the $\bY$-grading. (From the representation-theoretic point of view, and using the notation of Section~\ref{sec:intro}, this means that we study a geometric model for blocks of $\bG_1$-modules rather than $\bG_1\bT$-modules.) The only part of this section that will be used later in the paper is the statement given in Remark~\ref{rmk:indecomposable}\eqref{it:indecomposable-char0}. This statement does not involve ungraded modules, so a reader who is willing to accept this claim can skip this section.

\subsection{Definitions}
\label{ss:indecomp-Rmod}

We fix a finitary subset $A \subset S_\aff$.
Up to now we have worked with $\bY$-graded ind-objects in the category $\Perv_{(I^A_\unip,\cX_A)}(\Gr,\bk)$; in particular, in~\S\ref{ss:reg-perv-sheaf} we have defined $\cR$ as a \emph{formal} direct sum of ind-objects $\cR_\mu$. But in view of~\cite[Theorem~8.6.5(v)]{ks} the category of ind-objects in $\Perv_{(I^A_\unip,\cX_A)}(\Gr,\bk)$ admits arbitrary direct sums; in particular, the ``true'' direct sum (i.e.~coproduct) $\bigoplus_\mu \cR_\mu$ in this category makes sense. For simplicity we will also denote this object $\cR$.

An \emph{ungraded $\cR$-module} is, by definition, an ind-object $\cF$ in $\Perv_{(I^A_\unip,\cX_A)}(\Gr,\bk)$, together with a map
\[
\cF \star^{\cL^+G} \cR \to \cF
\]
equipping it with the structure of a module over the algebra object $\cR$.  Let
\[
\Mod_{(I^A_\unip,\cX_A)}(\cR)
\]
denote the abelian category of ungraded $\cR$-modules.  The theory of ungraded $\cR$-modules is very similar to that of graded $\cR$-modules.  In this subsection, we summarize the main facts about them.  Most proofs are essentially identical to those in the graded case, and will be omitted.

In a minor abuse of notation, we define the functor
\[
\Phi^A: \Perv_{(I^A_\unip,\cX_A)}(\Gr,\bk) \to \Mod_{(I^A_\unip,\cX_A)}(\cR)
\]
by $\Phi^A(\cF) = \cF \star^{\cL^+G} \cR$, where $\cF \star^{\cL^+G} \cR$ is equipped with the obvious module structure. As in the graded case, $\Phi^A$ is exact and faithful. Objects in the image of this functor are called \emph{free (ungraded) $\cR$-modules of finite type}.  There is an ungraded analogue of Lemma~\ref{lem:Rfree-hom} that says that
\begin{equation}
\label{eqn:Rfree-hom-ungr}
\Hom_{\Mod_{(I^A_\unip,\cX_A)}(\cR)}(\Phi^A(\cF), \cM) \cong \Hom_{\Perv_{(I^A_\unip,\cX_A)}(\Gr,\bk)}(\cF, \cM)
\end{equation}
for any $\cF \in \Perv_{(I^A_\unip,\cX_A)}(\Gr,\bk)$ and $\cM \in \Mod_{(I^A_\unip,\cX_A)}(\cR)$.
There is also an ungraded analogue of Lemma~\ref{lem:frobtwist-free} that says that
for any $\cF \in \Perv_{(I^A_\unip,\cX_A)}(\Gr,\bk)$ and $\cG \in \Perv_{\cL^+G}(\Gr,\bk)$ we have
\begin{equation}
\label{eqn:frobtwist-free}
\Phi^A(\cF \star^{\cL^+G} \cG) \cong \Satake(\cG) \otimes \Phi^A(\cF).
\end{equation}

The classification of simple objects in $\Mod_{(I^A_\unip,\cX_A)}(\cR)$ is given in Proposition~\ref{prop:R-simple-ungr}. In this statement, we denote by $\sim$ the equivalence relation on ${}^A W^\res_\ext$ given by
\[
w \sim w' \qquad
\text{if there is a $\lambda \in \bY$ such that $w = w't_\lambda$.}
\]
(In this case, $\lambda$ is necessarily orthogonal to all roots; in particular, if $G$ is semisimple this equivalence relation is trivial.) For $w \in {}^A W_\ext^\res$ we will denote by $[w]$ its equivalence class.

\begin{prop}
\label{prop:R-simple-ungr}
For $w \in {}^A W^\res_\ext$, the object $\Phi^A(\LGr^A_w)$ is a simple object in the abelian category $\Mod_{(I^A_\unip,\cX_A)}(\cR)$.  Moreover, the assignment $w \mapsto \Phi^A(\LGr^A_w)$ induces a bijection
\[
{}^A W^\res_\ext/{\sim} \, \simto
\left\{
\begin{array}{c}
\text{isomorphism classes of simple} \\
\text{objects in $\Mod_{(I^A_\unip,\cX_A)}(\cR)$}
\end{array}
\right\}.
\]
\end{prop}

For $c \in {}^A W^\res_\ext/{\sim}$ we will denote by $\cL_c^A \in \Mod_{(I^A_\unip,\cX_A)}(\cR)$ the corresponding simple object; for any $w \in {}^A W_\ext^\res$ we therefore have $\cL_{[w]}^A \cong \Phi^A(\LGr^A_w)$.

The definitions of the categories
\[
\modf_{(I^A_\unip,\cX_A)}(\cR) \subset \Mod_{(I^A_\unip,\cX_A)}(\cR)^\flen \subset \Mod_{(I^A_\unip,\cX_A)}(\cR)
\]
are analogous to their graded counterparts.  As in Lemma~\ref{lem:R-findim} we have
\begin{equation}\label{eqn:R-findim-ungr}
\dim \Hom_{\Mod_{(I^A_\unip,\cX_A)}(\cR)}(\cF,\cG) < \infty
\end{equation}
for all $\cF, \cG \in \Mod_{(I^A_\unip,\cX_A)}(\cR)^\flen$, and $\Mod_{(I^A_\unip,\cX_A)}(\cR)^\flen$ is a Krull--Schmidt category.

\begin{thm}
\label{thm:proj-Hecke-ungr}
The category $\modf_{(I^A_\unip,\cX_A)}(\cR)$ coincides with $\Mod_{(I^A_\unip,\cX_A)}(\cR)^\flen$, and this abelian category has enough projectives and enough injectives, and these classes of objects coincide.
\end{thm}

For $c \in {}^A W^\res_\ext/\sim$ we will denote by $\cQ^A_c$ the injective hull of $\cL^A_c$ in the category $\modf_{(I^A_\unip,\cX_A)}(\cR)$.

The following property is more specific to the ungraded setting.

\begin{lem}
\label{lem:G-action}
For any $\cF,\cG$ in $\Perv_{(I^A_\unip,\cX_A)}(\Gr,\bk)$, the finite-dimensional vector space
\[
\Hom_{\modf_{(I^A_\unip,\cX_A)}(\cR)}(\Phi^A(\cF),\Phi^A(\cG))
\]
carries a canonical structure of an algebraic $G^\vee_\bk$-module, which is functorial in $\cF$ and $\cG$ and compatible (in the natural way) with composition, and such that the map
\[
\Hom_{\Perv_{(I^A_\unip,\cX_A)}(\Gr,\bk)}(\cF,\cG) \to \Hom_{\modf_{(I^A_\unip,\cX_A)}(\cR)}(\Phi^A(\cF),\Phi^A(\cG))
\]
identifies the left-hand side with the $G^\vee_\bk$-invariants in the right-hand side.
\end{lem}

\begin{proof}
Recall that if $k$ is a field and $\mathsf{C}$ is a $k$-linear additive category which admits arbitrary coproducts, then given a $k$-vector space $V$ and an object $X$ in $\mathsf{C}$ one defines the tensor product $V \otimes_k X$ as the object representing the functor
\[
Y \mapsto \Hom_k(V, \Hom_{\mathsf{C}}(X,Y));
\]
any choice of basis $(e_i : i \in I)$ in $V$ provides an isomorphism $V \otimes_k X \cong X^{\oplus I}$. This construction is functorial, in the sense that if $\mathsf{D}$ is another $k$-linear additive category which admits arbitrary coproducts and $F : \mathsf{C} \to \mathsf{D}$ is a $k$-linear additive functor commuting with coproducts, then for any $V$ and $X$ as above there exists a canonical isomorphism $F(V \otimes_k X) \cong V \otimes_k F(X)$.

We apply this construction first in the case where $\mathsf{C}$ is the category of all $G^\vee_\bk$-modules, i.e.~the category of ind-objects in $\Rep(G^\vee_\bk)$. Here the comultiplication in the Hopf algebra $\scO(G^\vee_\bk)$ composed with switching the factors provides a canonical morphism
\begin{equation}
\label{eqn:comult-OG}
\scO(G^\vee_\bk) \to \scO(G^\vee_\bk) \otimes_\bk \scO(G^\vee_\bk),
\end{equation}
where the domain and the right-hand copy of $\scO(G^\vee_\bk)$ in the codomain are equipped with the left regular $G^\vee_\bk$-module structure, while the left-hand copy of $\scO(G^\vee_\bk)$ in the codomain is regarded just as a vector space. Next we apply the functor induced by $\Satake$ on ind-objects; the properties recalled above imply that we have $\Satake(\scO(G^\vee_\bk) \otimes_\bk \scO(G^\vee_\bk)) \cong \scO(G^\vee_\bk) \otimes_\bk \cR$, so that we obtain a canonical morphism
\begin{equation}
\label{eqn:comult-R}
\cR \to \scO(G^\vee_\bk) \otimes_\bk \cR.
\end{equation}
(Here the functor on ind-objects induced by $\Satake$ commutes with coproducts by the description of coproducts in~\cite[Theorem~8.6.5(v)]{ks}.)
Note that we have an exact sequence
\[
0 \to \bk \to \scO(G^\vee_\bk) \to \scO(G^\vee_\bk) \otimes_\bk \scO(G^\vee_\bk),
\]
where the rightmost arrow is the difference of~\eqref{eqn:comult-OG} and the map $f \mapsto 1 \otimes f$. We deduce an exact sequence
\begin{equation}
\label{eqn:inv-comult-R}
0 \to \IC^0 \to \cR \to \scO(G^\vee_\bk) \otimes_\bk \cR,
\end{equation}
where the rightmost arrow is the difference of~\eqref{eqn:comult-R} and the map corresponding to $1 \otimes \id_\cR$ under the canonical identification
\[
\Hom(\cR, \scO(G^\vee_\bk) \otimes_\bk \cR) \cong \scO(G^\vee_\bk) \otimes_\bk \Hom(\cR,\cR).
\]

We can at last use these constructions in the setting of the lemma. For $\cF,\cG$ as in the statement, by~\eqref{eqn:Rfree-hom-ungr} we have
\[
\Hom_{\modf_{(I^A_\unip,\cX_A)}(\cR)}(\Phi^A(\cF),\Phi^A(\cG)) \cong \Hom_{\Perv_{(I^A_\unip,\cX_A)}(\Gr,\bk)}(\cF,\cG \star^{\cL^+ G} \cR).
\]
Applying the considerations above to the functor $\Hom_{\Perv_{(I^A_\unip,\cX_A)}(\Gr,\bk)}(\cF,\cG \star^{\cL^+ G} (-))$ (where here $\cG \star^{\cL^+ G} (-)$ means the canonical extension of the functor $\cH \mapsto \cG \star^{\cL^+ G} \cH$ to ind-objects, and we consider morphisms of ind-perverse sheaves; the compatibility with coproducts is guaranteed by~\cite[Comments in Notation~8.6.1]{ks}) we obtain a canonical isomorphism
\begin{multline*}
\Hom_{\Perv_{(I^A_\unip,\cX_A)}(\Gr,\bk)} \bigl(\cF,\cG \star^{\cL^+ G} (\scO(G^\vee_\bk) \otimes_\bk \cR) \bigr) \\
\cong \scO(G^\vee_\bk) \otimes_\bk \Hom_{\Perv_{(I^A_\unip,\cX_A)}(\Gr,\bk)}(\cF,\cG \star^{\cL^+ G} \cR).
\end{multline*}
Hence, applying this functor to~\eqref{eqn:comult-R} we obtain a canonical morphism
\[
\Hom_{\modf_{(I^A_\unip,\cX_A)}(\cR)}(\Phi^A(\cF),\Phi^A(\cG)) \to \scO(G^\vee_\bk) \otimes_\bk \Hom_{\modf_{(I^A_\unip,\cX_A)}(\cR)}(\Phi^A(\cF),\Phi^A(\cG))
\]
which defines a structure of a (right) $\scO(G^\vee_\bk)$-comodule, i.e.~of a $G^\vee_\bk$-module, on the $\bk$-vector space $\Hom_{\modf_{(I^A_\unip,\cX_A)}(\cR)}(\Phi^A(\cF),\Phi^A(\cG))$. Moreover, from the exact sequence~\eqref{eqn:inv-comult-R} we deduce an exact sequence
\begin{multline*}
0 \to \Hom_{\Perv_{(I^A_\unip,\cX_A)}(\Gr,\bk)}(\cF,\cG) \to \Hom_{\modf_{(I^A_\unip,\cX_A)}(\cR)}(\Phi^A(\cF),\Phi^A(\cG)) \\
\to \scO(G^\vee_\bk) \otimes_\bk \Hom_{\modf_{(I^A_\unip,\cX_A)}(\cR)}(\Phi^A(\cF),\Phi^A(\cG))
\end{multline*}
which shows that $\Hom_{\Perv_{(I^A_\unip,\cX_A)}(\Gr,\bk)}(\cF,\cG)$ identifies the $G^\vee_\bk$-invariants in the $G^\vee_\bk$-module $\Hom_{\modf_{(I^A_\unip,\cX_A)}(\cR)}(\Phi^A(\cF),\Phi^A(\cG))$.
\end{proof}

\subsection{Forgetting the grading}
\label{ss:For}

We continue with the setting of~\S\ref{ss:indecomp-Rmod}.
There is an obvious exact forget-the-grading functor
\begin{equation}
\label{eqn:For-mod-R}
\For: \Mod^\bY_{(I^A_\unip,\cX_A)}(\cR) \to \Mod_{(I^A_\unip,\cX_A)}(\cR),
\end{equation}
which sends the ``formal'' direct sum $\bigoplus_\mu \cF_\mu$ to the ``true'' direct sum $\bigoplus_\mu \cF_\mu$ in the category of ind-objects in $\Perv_{(I^A_\unip,\cX_A)}(\Gr,\bk)$. This functor is exact; and it
satisfies $\For \circ \langle \lambda \rangle \cong \For$ for any $\lambda \in \bY$, commutes with the functors $\Phi^A$, and sends simple modules to simple modules. (In particular, it sends finitely generated modules to finitely generated modules.)  More specifically, for any $w \in {}^A W_\ext$, we have
\begin{equation}
\label{eqn:For-LGT}
\For(\LGT^A_w) \cong \cL^A_{[x]}
\qquad\text{if $w = xt_\lambda$ with $x \in {}^A W^\res_\ext$ and $\lambda \in \bY$.}
\end{equation}

\begin{lem}
\label{lem:R-degrade}
For $\cF \in \modf^\bY_{(I^A_\unip,\cX_A)}(\cR)$ and $\cG \in \Mod^\bY_{(I^A_\unip,\cX_A)}(\cR)$, the functor $\For$ induces an isomorphism
\[
\bigoplus_{\lambda \in \bY} \Hom_{\Mod^\bY_{(I^A_\unip,\cX_A)}(\cR)}(\cF,\cG\la \lambda\ra) \simto \Hom_{\Mod_{(I^A_\unip,\cX_A)}(\cR)}(\For(\cF), \For(\cG)).
\]
\end{lem}

\begin{proof}
By definition,
every finitely generated graded $\cR$-module is 
a quotient of a free graded $\cR$-module of finite type; using this, a routine five-lemma argument shows that it is enough to prove the lemma in the case that $\cF$ 
is free.  
In fact, we may even assume that $\cF = \Phi^A(\cF')$ 
for some $\cF' \in \Perv_{(I^A_\unip,\cX_A)}(\Gr,\bk)$.  In this case, using Lemma~\ref{lem:Rfree-hom} and its ungraded analogue~\eqref{eqn:Rfree-hom-ungr}, we see that the left-hand side is given by
\[
\bigoplus_{\lambda \in \bY} \Hom_{\Perv_{(I^A_\unip,\cX_A)}(\Gr,\bk)}(\cF', \cG_\lambda),
\]
while the right-hand side is
\[
\Hom_{\Perv_{(I^A_\unip,\cX_A)}( \Gr,\bk)} \left(\cF', \bigoplus_{\lambda \in \bY} \cG_\lambda \right).
\]
The fact that these spaces coincide follows from~\cite[Theorem~8.6.5(v)]{ks} and the comments in~\cite[Notation~8.6.1]{ks}.
\end{proof}

\begin{rmk}
In case $\cG$ belongs to $\modf^\bY_{(I^A_\unip,\cX_A)}(\cR)$, in view of~\eqref{eqn:R-findim-ungr} the direct sum appearing in Lemma~\ref{lem:R-degrade} only has finitely many nonzero terms.
\end{rmk}

The following statement follows from the construction of projective and injective objects in the proof of Theorem~\ref{thm:proj-Hecke-Whit}, and the parallel construction that proves Theorem~\ref{thm:proj-Hecke-ungr}.

\begin{prop}
\label{prop:For-proj-inj}
The functor~\eqref{eqn:For-mod-R} sends projective (i.e.~injective) objects in $\modf^\bY_{(I^A_\unip,\cX_A)}(\cR)$ to projective (i.e.~injective) objects in $\modf_{(I^A_\unip,\cX_A)}(\cR)$.
\end{prop}

Using~\eqref{eqn:For-LGT} one can make the statement of Proposition~\ref{prop:For-proj-inj} more precise: for any $w \in {}^A W_\ext$, writing $w=xt_\lambda$ with $x \in {}^A W_\ext^\res$ and $\lambda \in \bY$ we have
\[
\For(\injh^A_w) \cong \cQ^A_{[x]}.
\]
In particular the functor $\For$ sends these indecomposable objects to indecomposable objects. In fact this property holds for general indecomposable objects, as shown in the next statement.

\begin{cor}
\label{cor:For-indec}
Let $\cF \in \modf^\bY_{(I^A_\unip,\cX_A)}(\cR)$. Then $\cF$ is indecomposable (in the category $\modf^\bY_{(I^A_\unip,\cX_A)}(\cR)$) iff $\For(\cF)$ is indecomposable (in $\modf_{(I^A_\unip,\cX_A)}(\cR)$).
\end{cor}

\begin{proof}
Since the categories $\modf^\bY_{(I^A_\unip,\cX_A)}(\cR)$ and $\modf_{(I^A_\unip,\cX_A)}(\cR)$ are Krull--Schmidt (see Lemma~\ref{lem:R-findim} and the comments after~\eqref{eqn:R-findim-ungr}), $\cF$, resp.~$\For(\cF)$, is indecomposable iff the ring
\[
\End_{\modf^\bY_{(I^A_\unip,\cX_A)}(\cR)}(\cF), \quad \text{resp.} \quad \End_{\modf_{(I^A_\unip,\cX_A)}(\cR)}(\For(\cF)),
\]
is local. Then the claim follows from Lemma~\ref{lem:R-degrade} and the standard fact that a finite-dimensional $\bY$-graded $\bk$-algebra is local (as a nongraded ring) iff its degree-$0$ component is local, see e.g.~\cite{gg}.
\end{proof}


We conclude this subsection with a lemma relating indecomposable objects in $\Perv_{(I^A_\unip,\cX_A)}(\Gr,\bk)$ to indecomposable ungraded $\cR$-modules.  The proof is based on arguments found in~\cite[\S2]{donkin}.  

\begin{lem}
\label{lem:indecomp-crit}
Assume that $\bk$ is algebraically closed.  Let $\cF \in \Perv_{(I^A_\unip,\cX_A)}(\Gr,\bk)$ be an indecomposable perverse sheaf such that $\Phi^A(\cF)$ remains indecomposable in $\modf_{(I^A_\unip,\cX_A)}(\cR)$. If $\cG \in \Perv_{\cL^+G}(\Gr,\bk)$ is also indecomposable, then $\cF \star^{\cL^+G} \cG \in \Perv_{(I^A_\unip,\cX_A)}(\Gr,\bk)$ is indecomposable.
\end{lem}

\begin{proof}
By assumption, the object $\Phi^A(\cF)$ is indecomposable in the Krull--Schmidt category $\modf_{(I^A_\unip,\cX_A)}(\cR)$; the algebra $\End_{\modf_{(I^A_\unip,\cX_A)}(\cR)}(\Phi^A(\cF))$ is therefore local. Since this algebra is also finite-dimensional (see~\eqref{eqn:R-findim-ungr}), its unique maximal ideal consists of nilpotent elements. Moreover, since $\bk$ is algebraically closed, the quotient of $\End_{\modf_{(I^A_\unip,\cX_A)}(\cR)}(\Phi^A(\cF))$ by its unique maximal ideal is identified with $\bk$. We will denote by
\[
q: \End_{\modf_{(I^A_\unip,\cX_A)}(\cR)}(\Phi^A(\cF)) \to \bk
\]
the quotient map. By Lemma~\ref{lem:G-action} there exists a canonical (algebraic) action of $G^\vee_\bk$ on $\End_{\modf_{(I^A_\unip,\cX_A)}(\cR)}(\Phi^A(\cF))$ by algebra automorphisms; the unique maximal ideal is necessarily stable under this action, so that $q$ is $G^\vee_\bk$-equivariant (for the trivial action on $\bk$).

Next, let $V=\Satake(\cG)$. Then $V$ is a finite-dimensional algebraic $G^\vee_\bk$-module, and using~\eqref{eqn:frobtwist-free} we see that
\begin{align*}
\End_{\modf_{(I^A_\unip,\cX_A)}(\cR)}(\Phi^A(\cF \star^{\cL^+G} \cG)) &\cong \End_{\modf_{(I^A_\unip,\cX_A)}(\cR)}(\Phi^A(\cF)) \otimes \End_\bk(V), \\
\End_{\modf_{(I^A_\unip,\cX_A)}(\cR)}(\Phi^A(\cG)) &\cong \End_\bk(V), \\
\End_{\Perv_{\cL^+ G}(\Gr,\bk)}(\cG) &\cong \End_{G^\vee_\bk}(V).
\end{align*}
Here the first two isomorphisms are $G^\vee_\bk$-equivariant for the actions provided by Lemma~\ref{lem:G-action} and the action on $V$.

Define a ring homomorphism
\[
\tilde a: \End_{\modf_{(I^A_\unip,\cX_A)}(\cR)}(\Phi^A(\cF \star^{\cL^+G} \cG)) \to \End_{\modf_{(I^A_\unip,\cX_A)}(\cR)}(\Phi^A(\cG))
\]
to be the map which corresponds to 
\[
q \otimes \id: \End_{\modf_{(I^A_\unip,\cX_A)}(\cR)}(\Phi^A(\cF)) \otimes \End_\bk(V) \to \End_\bk(V)
\]
under the isomorphisms above.  Since the kernel of $q$ is finite-dimensional and consists of nilpotent elements, the kernel of $\tilde a$ does as well. Since $q$ is $G^\vee_\bk$-equivariant, $\tilde a$ is also equivariant.

We claim that there is a unique ring homomorphism $a$ that makes the following diagram (of ring homomorphisms) commute:
\begin{equation}\label{eqn:donkin-argument}
\begin{tikzcd}
\End_{\Perv_{(I^A_\unip,\cX_A)}(\Gr,\bk)}(\cF \star^{\cL^+G} \cG) \ar[r, dashed, "a"] \ar[d, "\Phi^A"'] &
  \End_{\Perv_{\cL^+G}(\Gr,\bk)}(\cG) \ar[d, "\Phi^A"] \\
\End_{\modf_{(I^A_\unip,\cX_A)}(\cR)}(\Phi^A(\cF \star^{\cL^+G} \cG)) \ar[r, "\tilde a"] &   \End_{\modf_{(I^A_\unip,\cX_A)}(\cR)}(\Phi^A(\cG)).
\end{tikzcd}
\end{equation}
In fact, this follows from the fact that $\tilde a$ is $G^\vee_\bk$-equivariant, and that in each column the domain of the map identifies with the space of $G^\vee_\bk$-invariants in its target (see Lemma~\ref{lem:G-action}).
Since the kernel of $\tilde a$ consists of nilpotent elements, the same holds for the kernel of $a$.  
In view of Lemma~\ref{lem:local-rings} below and the indecomposability of $\cG$, this implies that the algebra $\End_{\Perv_{(I^A_\unip,\cX_A)}(\Gr,\bk)}(\cF \star^{\cL^+G} \cG)$ is local, and hence that the object $\cF \star^{\cL^+G} \cG$ is indecomposable.
%
\end{proof}

\begin{lem}
\label{lem:local-rings}
Let $\bk$ be an algebraically closed field, let $A,A'$ be finite-dimensional $\bk$-algebras, and let $a : A \to A'$ be an algebra homomorphism. Assume that
\begin{enumerate}
\item $\ker(a)$ consists of nilpotent elements;
\item $A'$ is local.
\end{enumerate}
Then $A$ is local.
\end{lem}

\begin{proof}
Let $\mathfrak{m}' \subset A'$ be the unique maximal ideal in $A'$, and set $\mathfrak{m}=a^{-1}(\mathfrak{m}')$. Since $\ker(a)$ and $\mathfrak{m}'$ consist of nilpotent elements (in $A$ and $A'$ respectively), so does $\mathfrak{m}$. On the other hand, $A/\mathfrak{m}$ is a $\bk$-algebra which embeds in $A' / \mathfrak{m}'$, which is isomorphic to $\bk$ since this field is algebraically closed; it follows that $\mathfrak{m}$ is a maximal ideal and that $A=\bk \cdot 1 \oplus \mathfrak{m}$. Since $\mathfrak{m}$ consists of nilpotent elements this shows that any element in $A \smallsetminus \mathfrak{m}$ is invertible, and hence that any ideal of $A$ is contained in $\mathfrak{m}$, which finishes the proof.
\end{proof}

\subsection{A geometric version of Donkin's conjecture}
\label{ss:Donkin-conj}

A celebrated conjecture of Donkin~\cite{donkin-tilting} asserts that certain indecomposable tilting modules for reductive groups should remain indecomposable upon restriction to the Frobenius kernel. (For recent developments on this question, see~\cite{bnps}.) In this subsection we study the analogue of this property in our geometric setting. We show in particular that the geometric variant of this conjecture implies a ``Steinberg-type'' formula for tilting perverse sheaves (as in the representation-theoretic context, see~\cite[\S II.E.9]{jantzen}).

\begin{thm}
\label{thm:indecomposable}
Let $w \in {}^A W^\res_\ext$.  The following conditions are equivalent:
\begin{enumerate}
\item
\label{it:head-simple}
The head of $\TGr^A_{w_A w^\triangle}$ is simple.
\item
\label{it:head-simple-2}
The socle of $\TGr^A_{w_A w^\triangle}$ is simple.
\item 
\label{it:phi-indecomp}
The object $\Phi^A(\TGr^A_{w_A w^\triangle})$ is indecomposable (in the category $\modf^\bY_{(I_\unip^A,\cX_A)}(\cR)$ or in $\modf_{(I_\unip^A, \cX_A)}(\cR)$).
\item 
\label{it:phi-projcov}
For all $\lambda \in \bY$, the object $\Phi^A(\TGr_{w_A w^\triangle})\la -\lambda\ra$ is both the projective cover and injective hull of $\LGT^A_{wt_\lambda}$ in $\modf^\bY_{(I_\unip, \cX_A)}(\cR)$.
\end{enumerate}
If the conditions above hold, then we also have the following:
\begin{enumerate}
\setcounter{enumi}{4}
\item 
\label{it:donkin-perv}
(Donkin formula for tilting sheaves) For all $\mu \in \bY_+$, we have
\[
\TGr^A_{w_A w^\triangle} \star^{\cL^+G} \cT^\mu \cong \TGr^A_{w_A w^\triangle t_{w_\circ(\mu)}}.
\]
\end{enumerate}
\end{thm}

Note that if the head or socle of $\TGr^A_{w_A w^\triangle}$ is simple, they must be isomorphic to $\LGr^A_w$ by Corollary~\ref{cor:tilt-mult}. In~\eqref{it:phi-indecomp}, indecomposability in $\modf^\bY_{I_\unip}(\cR)$ or in $\modf_{I_\unip}(\cR)$ are equivalent in view of Corollary~\ref{cor:For-indec}.

\begin{proof}
The equivalence of~\eqref{it:head-simple} and~\eqref{it:head-simple-2} follows from Verdier self-duality of $\TGr^A_{w_A w^\triangle}$, see Remark~\ref{rmk:Verdier-duality}.

Let us now prove that
\eqref{it:head-simple} $\Rightarrow$ \eqref{it:phi-indecomp}. Suppose that $\TGr^A_{w_A w^\triangle}$ has a simple head, but that $\Phi^A(\TGr^A_{w_A w^\triangle})$ is decomposable.  
Proposition~\ref{prop:tilt-proj-inj-R}\eqref{it:Phi-tilting-injective-2} and Lemma~\ref{lem:R-degrade}
imply that
\[
\dim \Hom_{\modf_{(I^A_\unip, \cX_A)}(\cR)}(\Phi^A(\TGr^A_{w_A w^\triangle}), \Phi^A(\LGr^A_w)) = 1,
\]
so there must be some element $y \in {}^A W^\res_\ext$ with $y \not\sim w$ (under the equivalence relation considered in Proposition~\ref{prop:R-simple-ungr}) such that
\[
\Hom_{\modf_{(I^A_\unip, \cX_A)}(\cR)}(\Phi^A(\TGr^A_{w_A w^\triangle}), \Phi^A(\LGr^A_y)) \ne 0.
\]
By~\eqref{eqn:Rfree-hom-ungr}, both sides of the following isomorphism are then nonzero:
\begin{multline*}
\Hom_{\Perv_{(I^A_\unip, \cX_A)}(\Gr,\bk)}(\TGr^A_{w_A w^\triangle}, \LGr^A_y \star^{\cL^+G} \cR)
\\
\cong \bigoplus_{\mu \in \bY} \varinjlim_\lambda \Hom(\TGr^A_{w_A w^\triangle}, \LGr^A_y \star^{\cL^+G} \cI_*^{w_\circ(\mu) +\lambda} \star^{\cL^+G} \cI_*^{-w_\circ(\lambda)}).
\end{multline*}
However, by Theorem~\ref{thm:geometric-Steinberg} every composition factor of $\LGr^A_y \star^{\cL^+G} \cI_*^{w_\circ(\mu) +\lambda} \star^{\cL^+G} \cI_*^{-w_\circ(\lambda)}$ is of the form $\LGr^A_{y t_\nu}$ for some $\nu \in -\bY_+$.  The unique simple quotient of $\TGr^A_{w_A w^\triangle}$, namely $\LGr^A_w$ (see Corollary~\ref{cor:tilt-mult}), is not of this form, so
\[
\Hom(\TGr^A_{w_A w^\triangle}, \LGr^A_y \star^{\cL^+G} \cI_*^{w_\circ(\mu) +\lambda} \star^{\cL^+G} \cI_*^{-w_\circ(\lambda)}) = 0
\]
for any $\lambda,\mu$, a contradiction.


We now show that
\eqref{it:phi-indecomp} $\Rightarrow$ \eqref{it:phi-projcov}. 
If $\Phi^A(\TGr^A_{w_A w^\triangle})$ is indecomposable, then Proposition~\ref{prop:tilt-proj-inj-R}\eqref{it:Phi-tilting-injective-2} shows that
\[
\injh^A_w \cong \Phi^A(\TGr^A_{w_A w^\triangle}) \cong \injh^A_{\iota_A(w)}.
\]
The claim in~\eqref{it:phi-projcov} follows using~\eqref{eqn:injh-translation}.

We next show that
\eqref{it:phi-projcov} $\Rightarrow$ \eqref{it:head-simple}. Assume that~\eqref{it:phi-projcov} holds, and that the head of $\TGr^A_{w_A w^\triangle}$ has more than one summand. By Corollary~\ref{cor:tilt-mult}, there exists $y \in {}^A W_\ext^S$ with $y \neq x$ such that $\LGr^A_y$ is a quotient of $\TGr^A_{w_A w^\triangle}$. Applying $\Phi^A$ we deduce a surjection $\Phi^A(\TGr^A_{w_A w^\triangle}) \twoheadrightarrow \Phi^A(\LGr^A_y)$. Theorem~\ref{thm:geometric-Steinberg}, Lemma~\ref{lem:frobtwist-free} and~\eqref{eqn:simple-translation} show that $\Phi^A(\LGr^A_y)$ surjects to $\LGT_y^A$, so that $\Phi^A(\TGr^A_{w_A w^\triangle})$ has two different simple quotients. This contradicts the fact that this object is a projective cover.


Finally we show that
\eqref{it:phi-projcov} $\Rightarrow$ \eqref{it:donkin-perv}. First, note that the perverse sheaf $\TGr^A_{w_A w^\triangle} \star^{\cL^+G} \cT^\mu$ is tilting. In fact $\TGr_{w^\triangle} \star^{\cL^+ G} \cT^\mu$ is tilting by
Lemma~\ref{lem:twtri-standard}, hence
\[
\Av^A_\psi(\TGr_{w^\triangle} \star^{\cL^+ G} \cT^\mu) \cong \Av^A_\psi(\TGr_{w^\triangle}) \star^{\cL^+ G} \cT^\mu
\]
is tilting too, see~\S\ref{ss:Av-Iw}. By
Proposition~\ref{prop:Av-tilting-indec} this object is a direct sum of copies of $\TGr^A_{w_A w^\triangle} \star^{\cL^+G} \cT^\mu$, so that the latter object is tilting.

Now that we know that this object is tilting, support considerations show that $\TGr^A_{w_A w^\triangle} \star^{\cL^+G} \cT^\mu$ admits $\TGr^A_{w_A w^\triangle t_{w_\circ(\mu)}}$ as a direct summand; we therefore only need to show that this object is indecomposable.
A routine argument (cf.~\cite[Proposition~B.3]{modrap1}) shows that this tilting perverse sheaf is indecomposable if and only if the object obtained by extension of scalars to the algebraic closure of $\bk$ is indecomposable.  Thus, we may assume without loss of generality that $\bk$ is algebraically closed. Then the claim follows from Lemma~\ref{lem:indecomp-crit}. 
\end{proof}

\begin{rmk}
\phantomsection\label{rmk:indecomposable}
\begin{enumerate}
\item
\label{it:indecomposable-char0}
In case $\bk$ has characteristic $0$, Theorem~\ref{thm:proj-char0-Whit} says in particular that condition~\eqref{it:head-simple} in Theorem~\ref{thm:indecomposable} holds for any $w \in {}^A W_\ext^\res$. Hence in this case the injective hulls $\injh^A_y$ ($y \in {}^A W_\ext$) can be described explicitly: if $y=xt_\lambda$ with $x \in {}^A W_\ext^\res$ and $\lambda \in \bY$ then
\[
\injh^A_y \cong \Phi^A(\TGr^A_{w_A x^\triangle}) \langle -\lambda \rangle.
\]
\item
\label{it:indecomposable-varsigma}
When $\bk$ has positive characteristic,
one instance in which the conditions of Theorem~\ref{thm:indecomposable} hold is for the element $w = t_\varsigma w_\circ \in W^\res_\ext$ in case $A=\varnothing$.  Indeed, by Proposition~\ref{prop:avS-calc}, the object $j^\varsigma_!\cS_\varsigma \cong j^\varsigma_*\cS_\varsigma \cong \TGr_{t_{w_\circ(\varsigma)}}$ has a simple head and socle.  (Note that $w^\triangle = t_{w_\circ(\varsigma)}$.)  Thus, by Theorem~\ref{thm:indecomposable}, for any $\lambda \in \bY$ the object $\Phi(\TGr_{t_{w_\circ(\varsigma)}})\la -\lambda \ra$ is both the injective hull and the projective cover of $\LGT_{t_\varsigma w_\circ t_\lambda}$; we thus have
\[
\injh_{t_\varsigma w_\circ t_\lambda} = \Phi(\TGr_{t_{w_\circ(\varsigma)}})\la -\lambda \ra.
\]
\end{enumerate}
\end{rmk}

\section{Baby Verma and co-Verma modules}
\label{sec:baby-verma}

In this section we introduce 
objects of $\modf^\bY_{(I^A_\unip,\cX_A)}(\cR)$ which are geometric analogues of baby Verma 
modules  (i.e.~the objects denoted $\widehat{Z}(\lambda)$ in~\cite[Chap.~II.9]{jantzen}). 
These objects will be obtained from the baby co-Verma modules of~\S\ref{ss:cbV} using a ``Verdier duality" autoequivalence.
(In the representation-theoretic context, such a relation is well known, see~\cite[Equation~(5) in~\S 9.3]{jantzen}.) 

\subsection{Verdier duality}
\label{ss:verdier}

We now explain how to define Verdier duality in the categories $\modf^\bY_{(I^A_\unip,\cX_A)}(\cR)$. 

Recall the Verdier duality functor
\[
\D : \Db_{(I^A_\unip,\cX_A)}(\Gr) \to \Db_{(I^A_\unip,\cX^{-1}_A)}(\Gr)
\]
considered in Remark~\ref{rmk:Verdier-duality}.
What we now want to do is to ``extend" this functor to the category $\modf^\bY_{(I^A_\unip,\cX_A)}(\cR)$, i.e.~to define an exact anti-equivalence
\[
\D : \modf^\bY_{(I^A_\unip,\cX_A)}(\cR) \simto \modf^\bY_{(I^A_\unip,\cX^{-1}_A)}(\cR)
\]
 which satisfies
\begin{equation}
\label{eqn:D-free}
\D(\Phi^A(\cF) \langle \lambda \rangle) \cong \Phi^A(\D(\cF)) \langle \lambda \rangle
\end{equation}
for any $\cF$ in $\Perv_{(I^A_\unip,\cX_A)}(\Gr)$ and $\lambda \in \bY$. (Here we are making an abuse of notation similar to that of Remark~\ref{rmk:Verdier-duality}, in the sense that the notation $\Phi^A$ on either side of this equation is used for two different functors: on the left-hand side the functor is defined using the local system $\cX_A$, while on the right-hand side it is defined using $\cX_A^{-1}$.)
Note that a ``naive'' extension of $\D$ to ind-objects would send ind-objects to pro-objects, and thus not give an endofunctor of $\cR$-modules. Instead, we will use the fact that~\eqref{eqn:D-free} prescribes the definition of $\D$ on free $\cR$-modules of finite type, and that the category $\modf^\bY_{(I^A_\unip,\cX_A)}(\cR)$ can be described in terms of these free modules.

We start by making formal sense of this latter idea. For this, we define the additive $\bk$-linear category $\mathsf{Free}^\bY_{(I^A_\unip,\cX_A)}(\cR)$ whose objects are formal direct sums
\[
\bigoplus_{j \in J} (\cF_j,\lambda_j)
\]
where $J$ is a finite set, each $\cF_j$ is in $\Perv_{(I^A_\unip,\cX_A)}(\Gr,\bk)$, and each $\lambda_j$ is in $\bY$, and such that the space of morphisms from $\bigoplus_{j \in J} (\cF_j,\lambda_j)$ to $\bigoplus_{k \in K} (\cG_k,\mu_k)$ is
\begin{multline*}
\Hom_{\modf^\bY_{(I^A_\unip,\cX_A)}(\cR)}\left( \bigoplus_{j \in J} \Phi^A(\cF_j) \langle \lambda_j \rangle, \bigoplus_{k \in K} \Phi^A(\cG_k) \langle \mu_k \rangle \right) \\
= \bigoplus_{\substack{j \in J \\ k \in K}} \Hom_{\modf^\bY_{(I^A_\unip,\cX_A)}(\cR)}(\Phi^A(\cF_j),\Phi^A(\cG_k)\la \mu_k - \lambda_j\ra).
\end{multline*}
By definition $\Phi^A$ factors as a composition of additive $\bk$-linear functor
\[
\Perv_{(I^A_\unip,\cX_A)}(\Gr,\bk) \xrightarrow{\Phi^A_1} \mathsf{Free}^\bY_{(I^A_\unip,\cX_A)}(\cR) \xrightarrow{\Phi^A_2} \modf^\bY_{(I^A_\unip,\cX_A)}(\cR),
\]
where $\Phi^A_2$ is fully faithful. Moreover, the objects in the essential image of $\Phi^A_2$ are exactly the free $\cR$-modules of finite type.

We now consider the homotopy category $K(\mathsf{Free}^\bY_{(I^A_\unip,\cX_A)}(\cR))$, and the triangulated subcategory $K(\mathsf{Free}^\bY_{(I^A_\unip,\cX_A)}(\cR))^{\mathrm{b}}$ of complexes whose image under $K(\Phi^A_2)$ has bounded cohomology. (This subcategory is \emph{not} the bounded homotopy category of $\mathsf{Free}^\bY_{(I^A_\unip,\cX_A)}(\cR)$.) Since $\Db \modf^\bY_{(I^A_\unip,\cX_A)}(\cR)$ identifies with the full subcategory of the unbounded derived category $D(\modf^\bY_{(I^A_\unip,\cX_A)}(\cR))$ whose objects have bounded cohomology, the composition of $K(\Phi^A_2)$ with the canonical functor $K(\modf^\bY_{(I^A_\unip,\cX_A)}(\cR)) \to D(\modf^\bY_{(I^A_\unip,\cX_A)}(\cR))$ restricts to a functor
\begin{equation}
\label{eqn:functor-KPhi}
K(\mathsf{Free}^\bY_{(I^A_\unip,\cX_A)}(\cR))^{\mathrm{b}} \to \Db(\modf^\bY_{(I^A_\unip,\cX_A)}(\cR)).
\end{equation}

Next, let $D(\mathsf{Free}^\bY_{(I^A_\unip,\cX_A)}(\cR))^{\mathrm{b}}$ be the Verdier quotient of $K(\mathsf{Free}^\bY_{(I^A_\unip,\cX_A)}(\cR))^{\mathrm{b}}$ by the kernel of \eqref{eqn:functor-KPhi}. Then, by the universal property of the Verdier quotient, the functor~\eqref{eqn:functor-KPhi} factors through a triangulated functor
\begin{equation}
\label{eqn:D-Psi}
D(\Phi^A_2) : D(\mathsf{Free}^\bY_{(I^A_\unip,\cX_A)}(\cR))^{\mathrm{b}} \to \Db(\modf^\bY_{(I^A_\unip,\cX_A)}(\cR)).
\end{equation}

\begin{lem}
The functor~\eqref{eqn:D-Psi} is an equivalence of categories.
\end{lem}

\begin{proof}
Essential surjectivity follows from Remark~\ref{rmk:free-proj-inj}\eqref{it:free-proj-inj}. More precisely, any object in $\Db(\modf^\bY_{(I^A_\unip,\cX_A)}(\cR))$ is isomorphic to a bounded complex of objects in $\modf^\bY_{(I^A_\unip,\cX_A)}(\cR)$. Given such a complex $\cF$, this remark implies that there exists a bounded above complex $\cG$ of free $\cR$-modules of finite type which are projective and a quasi-isomorphism $\cG \to \cF$. Then $\cG$ belongs to the essential image of our functor, and is isomorphic to $\cF$ in $\Db(\modf^\bY_{(I^A_\unip,\cX_A)}(\cR))$.

Next we prove that the functor is full. Fix objects $\cF,\cG$ in $D(\mathsf{Free}^\bY_{(I^A_\unip,\cX_A)}(\cR))^{\mathrm{b}}$. A morphism $f : D(\Phi^A_2)(\cF) \to D(\Phi^A_2)(\cG)$ is represented by a diagram
\[
K(\Phi^A_2)(\cF) \xleftarrow{g} \cH \xrightarrow{h} K(\Phi^A_2)(\cG)
\]
where $\cH$ is a complex of objects in $\modf^\bY_{(I^A_\unip,\cX_A)}(\cR)$, and $g,h$ are morphisms of complexes with $g$ a quasi-isomorphism (which implies that $\cH$ has bounded cohomology). Using a truncation functor we can assume that $\cH$ is bounded above. Then, as above there exists a complex $\mathcal{K}$ of objects of $\mathsf{Free}^\bY_{(I^A_\unip,\cX_A)}(\cR)$ and a quasi-isomorphism $k : K(\Phi^A_2)(\mathcal{K}) \to \cH$. The object $\mathcal{K}$ belongs to $K(\mathsf{Free}^\bY_{(I^A_\unip,\cX_A)}(\cR))^{\mathrm{b}}$, and $f$ is also represented by the diagram
\[
K(\Phi^A_2)(\cF) \xleftarrow{g \circ k} K(\Phi^A_2)(\mathcal{K}) \xrightarrow{h \circ k} K(\Phi^A_2)(\cG)
\]
where $g \circ k$ is a quasi-isomorphism,
and thus is the image of a morphism from $\cF$ to $\cG$ in $D(\mathsf{Free}^\bY_{(I^A_\unip,\cX_A)}(\cR))^{\mathrm{b}}$.

Finally we prove faithfulness. Fix again objects $\cF,\cG$ in $D(\mathsf{Free}^\bY_{(I^A_\unip,\cX_A)}(\cR))^{\mathrm{b}}$, and consider a morphism $f : \cF \to \cG$ such that $D(\Phi^A_2)(f)=0$. Here $f$ is represented by a diagram
\[
\cF \xleftarrow{g} \cH \xrightarrow{h} \cG
\]
where $\cH$ is in $K(\mathsf{Free}^\bY_{(I^A_\unip,\cX_A)}(\cR))^{\mathrm{b}}$, $g$ and $h$ are morphisms of complexes, and $K(\Phi^A_2)(g)$ is a quasi-isomorphism. Once again there exists a bounded above complex $\mathcal{K}$ of projective free $\cR$-modules of finite type and a quasi-isomorphism $\mathcal{K} \to K(\Phi^A_2)(\cH)$. Then there exists $\mathcal{L}$ in $K(\mathsf{Free}^\bY_{(I^A_\unip,\cX_A)}(\cR))^{\mathrm{b}}$ and an isomorphism of complexes $K(\Phi^A_2)(\mathcal{L}) \simto \mathcal{K}$, and $f$ is represented by a diagram
\[
\cF \xleftarrow{g'} \mathcal{L} \xrightarrow{h'} \cG
\]
where $K(\Phi^A_2)(g')$ is a quasi-isomorphism, i.e.~$D(\Phi^A_2)(g')$ is an isomorphism. Then $D(\Phi^A_2)(h')=0$. Since $\mathcal{K}$ is a bounded above complex of projective objects, this implies that $K(\Phi^A_2)(h')=0$, and hence that $h'$ is homotopic to $0$, and finally that $f=0$.
\end{proof}

Now we address the question of defining $\mathbb{D}$ on free $\cR$-modules of finite type. More precisely, we define an additive contravariant equivalence
\[
\D_\fr : \mathsf{Free}^\bY_{(I^A_\unip,\cX_A)}(\cR) \simto \mathsf{Free}^\bY_{(I^A_\unip,\cX_A^{-1})}(\cR)
\]
as follows. On objects, this functors sends $\bigoplus_{j \in J} (\cF_j,\lambda_j)$ to $\bigoplus_{j \in J} (\D(\cF_j),\lambda_j)$. By additivity, to define this functor on morphisms it suffices to consider the case of objects of the form $(\cF,\lambda)$. We therefore consider $\cF,\cG$ in $\Perv_{(I^A_\unip,\cX_A)}(\Gr,\bk)$, and $\lambda,\mu \in \bY$. Then by definition
\[
\Hom_{\mathsf{Free}^\bY_{(I^A_\unip,\cX_A)}(\cR)}((\cF,\lambda),(\cG,\mu)) = \Hom_{\modf_{(I^A_\unip,\cX_A)}^\bY(\cR)}(\Phi^A(\cF), \Phi^A(\cG) \langle \mu-\lambda \rangle ).
\]
By Lemma~\ref{lem:Rfree-hom} and the definition of $\Phi^A$, the right-hand side identifies with
\begin{multline*}
\Hom_{\Perv_{(I^A_\unip,\cX_A)}(\Gr,\bk)}(\cF, \cG \star^{\cL^+G} \cR_{\lambda-\mu}) \\
= \varinjlim_\nu \Hom_{\Perv_{(I^A_\unip,\cX_A)}(\Gr,\bk)}(\cF, \cG \star^{\cL^+G} \cI_*^{w_\circ(\lambda-\mu)+\nu} \star^{\cL^+G} \cI_*^{-w_\circ(\nu)}).
\end{multline*}
By~\eqref{eqn:duals}, for any $\nu \in \bY_+ \cap (-w_\circ(\lambda-\mu) + \bY_+)$ we have
\begin{multline*}
\Hom_{\Perv_{(I^A_\unip,\cX_A)}(\Gr,\bk)}(\cF,\cG \star^{\cL^+G} \cI_*^{w_\circ(\lambda-\mu)+\nu} \star^{\cL^+G} \cI_*^{-w_\circ(\nu)}) \\ \cong
\Hom_{\Perv_{(I^A_\unip,\cX_A)}(\Gr,\bk)}(\cF \star^{\cL^+G} \cI_!^{\mu-\lambda - w_\circ(\nu)} \star^{\cL^+G} \cI_!^{\nu}, \cG).
\end{multline*}
We next use the fact that $\D$ commutes with convolution and sends $\cI_!^\eta$ to $\cI_*^\eta$ for any $\eta \in \bY_+$ to identify this space with
\[
\Hom_{\Perv_{(I^A_\unip,\cX_A^{-1})}(\Gr,\bk)}(\D(\cG), \D(\cF) \star^{\cL^+G} \cI_*^{\mu-\lambda - w_\circ(\nu)} \star^{\cL^+G} \cI_*^{\nu}).
\]
Following this series of identifications we have constructed a natural isomorphism from $\Hom_{\mathsf{Free}^\bY_{(I^A_\unip,\cX_A)}(\cR)}((\cF,\lambda),(\cG,\mu))$ to
\[
\varinjlim_\nu \Hom_{\Perv_{(I^A_\unip,\cX_A^{-1})}(\Gr,\bk)}(\D(\cG), \D(\cF) \star^{\cL^+G} \cI_*^{\mu-\lambda - w_\circ(\nu)} \star^{\cL^+G} \cI_*^{\nu}).
\]
Setting $\nu'=-w_\circ(\mu-\lambda)+\nu$ we see that this inductive limit identifies with
\begin{multline*}
\Hom_{\Perv_{(I^A_\unip,\cX_A^{-1})}(\Gr,\bk)}(\D(\cG), \D(\cF) \star^{\cL^+G} \cR_{\mu-\lambda}) \cong \\
\ \Hom_{\mathsf{Free}^\bY_{(I^A_\unip,\cX_A^{-1})}(\cR)}((\D(\cG),\mu),(\D(\cF),\lambda));
\end{multline*}
we have therefore completed the definition of the functor $\D_\fr$. 

In order to ``extend'' this functor to the appropriate derived categories we will use the following lemma.

\begin{lem}
\label{lem:duality-exactness}
Let 
\[
\cM_1 \xrightarrow{f} \cM_2 \xrightarrow{g} \cM_3
\]
be a sequence of objects and morphisms of $\mathsf{Free}^\bY_{(I^A_\unip,\cX_A)}(\cR)$ such that the sequence
\[
\Phi^A_2(\cM_1) \xrightarrow{\Phi^A_2(f)} \Phi^A_2(\cM_2) \xrightarrow{\Phi^A_2(g)} \Phi^A_2(\cM_3)
\]
is exact at $\Phi^A_2(\cM_2)$.  Then the sequence
\[
\Phi^A_2(\D_\fr(\cM_3)) \xrightarrow{\Phi^A_2(\D_\fr(g))} \Phi^A_2(\D_\fr(\cM_2)) \xrightarrow{\Phi^A_2(\D_\fr(f))} \Phi^A_2(\D_\fr(\cM_1))
\]
is exact at $\Phi^A_2(\D_\fr(\cM_2))$.
\end{lem}

\begin{proof}
Suppose the latter sequence is not exact, or in other words that
\[
\mathrm{im}(\Phi^A_2(\D_\fr(g))) \subsetneq \ker(\Phi^A_2(\D_\fr(f))).
\]
By Remark~\ref{rmk:free-proj-inj}\eqref{it:free-proj-inj}, there exists an object $\cP$ in $\mathsf{Free}^\bY_{(I^A_\unip,\cX_A)}(\cR)$ which is a direct sum of objects of the form $(\cF,\lambda)$ with $\lambda \in \bY$ and $\cF$ such that $\D(\cF)\cong\cF$ and $\Phi^A(\cF)$ projective and injective, and a morphism $q: \cP \to \D_\fr(\cM_2)$ such that the image of $\Phi^A_2(q)$ is $\ker (\Phi^A_2(\D_\fr(f)))$. Then $q$ does not factor through $\D_\fr(g)$:
\[
\begin{tikzcd}
& \cP \ar[d, "q"] \ar[dr, "\D_\fr(f) \circ q = 0"] \ar[dl, dashed, "\text{does not exist}"'] \\
\D_\fr(\cM_3) \ar[r, "\D_\fr(g)"] & \D_\fr(\cM_2) \ar[r, "\D_\fr(f)"] & \D_\fr(\cM_1).
\end{tikzcd}
\]
The object $\Phi^A_2(\D_\fr(\cP))$ is
injective.  Applying $\D_\fr$ to the diagram above we obtain a diagram 
\[
\begin{tikzcd}
\cM_1 \ar[r, "f"] \ar[dr, "\D_\fr(q) \circ f = 0"'] & \cM_2 \ar[r, "g"] \ar[d, "\D_\fr(q)"] & \cM_3 \ar[dl, dashed, "\text{does not exist}"] \\
& \D_\fr(\cP).
\end{tikzcd}
\]
This diagram implies that the image under $\Phi^A_2$ of the top row is not exact at $\cM_2$, a contradiction.
\end{proof}

Lemma~\ref{lem:duality-exactness} implies that the functor
\[
K(\D_\fr) : K(\mathsf{Free}^\bY_{(I^A_\unip,\cX_A)}(\cR)) \to K(\mathsf{Free}^\bY_{(I^A_\unip,\cX^{-1}_A)}(\cR))
\]
restricts to an endofunctor of $K(\mathsf{Free}^\bY_{(I^A_\unip,\cX_A)}(\cR))^{\mathrm{b}}$ that preserves the kernel of~\eqref{eqn:functor-KPhi} (in the sense that it sends the kernel of the version for $\cX_A$ to the kernel of the version for $\cX_A^{-1}$).  It therefore induces a contravariant triangulated functor 
\[
D(\mathsf{Free}^\bY_{(I^A_\unip,\cX_A)}(\cR))^{\mathrm{b}} \to D(\mathsf{Free}^\bY_{(I^A_\unip,\cX^{-1}_A)}(\cR))^{\mathrm{b}}.
\]
Conjugating this functor with the equivalence~\eqref{eqn:D-Psi} (more precisely, the version for $\cX_A$ and that for $\cX_A^{-1}$) we obtain a contravariant triangulated functor
\[
\D : \Db \modf^\bY_{(I^A_\unip,\cX_A)}(\cR) \to \Db \modf^\bY_{(I^A_\unip,\cX_A^{-1})}(\cR).
\]
It is clear from this construction that $\D$ is involutive, in the sense that the composition of the version for $\cX_A$ and that for $\cX_A^{-1}$ is the identity and vice versa (in particular, $\D$ is an equivalence of categories), and that~\eqref{eqn:D-free} holds for any $\cF$ in $\Perv_{(I^A_\unip,\cX_A)}(\Gr,\bk)$ and $\lambda \in \bY$. In particular, 
for any $w \in {}^A W_\ext$ we have
\begin{equation}
\label{eqn:D-simples}
\D(\LGT^A_w) \cong \LGT^A_{w}.
\end{equation}
Since any simple object in $\modf^\bY_{(I^A_\unip,\cX_A)}(\cR)$ is of this form (see Theorem~\ref{thm:R-simple}), we deduce that $\D$ is exact for the natural t-structures on the categories $\Db \modf^\bY_{(I^A_\unip,\cX_A)}(\cR)$ and $\Db \modf^\bY_{(I^A_\unip,\cX_A^{-1})}(\cR)$. The restriction of this functor to $\cR$-modules therefore possesses all the required properties.

\begin{rmk}
\phantomsection
\label{rmk:D}
\begin{enumerate}
\item
The construction of the anti-autoequivalence $\D$ given here follows the one suggested in~\cite[Top of p.~297]{abbgm}. Note however that in \emph{loc}.~\emph{cit}.~the authors do not justify that this construction does not depend on the choices one has to make, and defines a functor. Our study of projective objects exactly fills this gap.
\item
\label{it:Groth-gp-duality}
Note that the isomorphisms~\eqref{eqn:D-simples} show that for any $\cF$ in $\modf^\bY_{(I^A_\unip,\cX_A)}(\cR)$ and $w \in {}^A W_\ext$ we have
\[
[\cF : \LGT^A_w] = [\D(\cF) : \LGT^A_w].
\]
In particular, the endomorphism induced by $\D$ on the Grothendieck group $[\modf^\bY_{(I^A_\unip,\cX_A)}(\cR)]$  is trivial.
\end{enumerate}
\end{rmk}

Recall the map $\iota_A$ from~\eqref{eqn:def-iota}, which is defined so that $\injh^A_{\iota_A(w)}$ is the projective cover of $\LGT^A_w$. The properties of $\D$ stated above imply that $\D(\injh^A_{w})$ is also the projective cover of $\LGT^A_w$; it follows that
\begin{equation}
\label{eqn:D-injh}
\D(\injh^A_{w}) \cong \injh^A_{\iota_A(w)} \quad \text{for any $w \in W_\ext$.}
\end{equation}

\subsection{Baby Verma modules}
\label{ss:baby-Verma}

Given $w \in {}^A W_\ext$, 
we define the \emph{baby Verma module} with label $w$ to be the $\cR$-module $\bV^A_w$ given by
\[
\bV^A_w := \D(\cbV^A_{w})
\qquad \in \modf^\bY_{(I^A_\unip,\cX_A)}(\cR).
\]
We will say that an object $\cF$ in $\modf^\bY_{(I^A_\unip,\cX_A)}(\cR)$ \emph{admits a baby Verma filtration} if it admits a finite filtration whose subquotients are isomorphic to baby Verma modules. It is clear from the definition that $\cF$ admits a baby Verma filtration iff $\D(\cF)$ admits a baby co-Verma filtration.

The following properties of baby Verma modules are immediate consequences of the corresponding facts for baby co-Verma modules (see~\S\ref{ss:cbV} and~\S\ref{ss:proj-inj}), together with 
the properties of Verdier duality stated in~\S\ref{ss:verdier}.

\begin{lem}
\phantomsection
\label{lem:verma-basic}
\begin{enumerate}
\item
For $w \in {}^A W_\ext$ and $\lambda \in \bY$ we have
\[
\bV^A_{wt_\lambda} \cong \bV^A_w \langle -\lambda \rangle.
\]
\item For $w \in {}^A W_\ext$, $\bV^A_w$ has a simple head, isomorphic to $\LGT^A_w$.
\item For $w, y \in {}^A W_\ext$, we have $[\bV^A_y : \LGT^A_w] = [\cbV^A_y : \LGT^A_w]$.
\end{enumerate}
\end{lem}

\begin{prop}
\phantomsection
\label{prop:bV-projective}
\begin{enumerate}
\item If $w,x \in {}^A W_\ext$ satisfy $[\bV^A_w : \LGT^A_x] \neq 0$, we have $x \preceq w \preceq w_A x^\triangle$. Moreover, we have $[\bV^A_x : \LGT^A_x]=1$ and $[\bV^A_{w_A x^\triangle} : \LGT_x] \le 1$.
\item For any $w \in {}^A W_\ext$, $\bV^A_w$ is the projective cover of $\LGT^A_w$ in the Serre subcategory of $\modf^{\bY}_{(I^A_\unip,\cX_A)}(\cR)$ generated by the simple objects $\LGT^A_y$ with $y \not\succ w$.
\end{enumerate}
\end{prop}

\begin{rmk}
\label{rmk:bV-Groth-gp}
As in Remark~\ref{rmk:cbV-Groth-gp}, Proposition~\ref{prop:bV-projective} implies that if $\cM \in \modf^{\bY}_{(I^A_\unip,\cX_A)}(\cR)$ admits a baby Verma filtration, then the number $(\cM : \bV^A_w)$ of occurrences of a given baby Verma module $\bV^A_w$ in such a filtration is independent of the choice of filtration; in fact these numbers are determined by the equality
\[
[\cM] = \sum_{w \in {}^A W_\ext} (\cM : \bV^A_w) \cdot [\bV^A_w]
\]
in $[\modf^{\bY}_{(I^A_\unip,\cX_A)}(\cR)]$.

Suppose now that $\cM$ admits both a baby Verma filtration and baby co-Verma filtration.  By Remark~\ref{rmk:D}\eqref{it:Groth-gp-duality}, we have $[\bV^A_w] = [\cbV^A_w]$ for all $w \in {}^A W_\ext$, so comparing the equation above with that in Remark~\ref{rmk:cbV-Groth-gp}, we deduce that
\begin{equation}
\label{eqn:equality-multiplicites}
(\cM: \bV^A_w) = (\cM:\cbV^A_w)
\qquad\text{for all $w \in {}^A W_\ext$.}
\end{equation}
\end{rmk}

Recall that in the proof of Theorem~\ref{thm:proj-Hecke-Whit}, we showed that $\modf^{\bY}_{(I^A_\unip,\cX_A)}(\cR)$ is an ``essentially finite highest weight category'' in the sense of~\cite[Definition~3.7]{bs}.  In that setting, comparing Proposition~\ref{prop:bV-projective} with~\cite[Lemma~3.1]{bs}, we see that for any $w \in {}^A W_\ext$ the baby Verma module $\bV^A_w$ is isomorphic to the objects denoted $\Delta(w)$ and $\bar{\Delta}(w)$ in~\cite{bs}. Now that these objects are identified, we can state property ``$(P\Delta_\varepsilon)$'' from~\cite{bs} in our setting, which does hold by~\cite[Theorem~3.5]{bs}.

\begin{prop}
\label{prop:inj-Verma-filt}
For any $w \in {}^A W_\ext$, the object $\injh^A_w$ admits a baby Verma filtration.
\end{prop}

The following lemma is a restatment of~\cite[Theorems~3.11 and 3.14]{bs} in our context.

\begin{lem}
\phantomsection
\label{lem:verma-hom}
Let $\cM \in \modf^{\bY}_{(I^A_\unip,\cX_A)}(\cR)$.
\begin{enumerate}
\item
The object $\cM$ admits a baby co-Verma filtration if and only if it satisfies $\Ext^1_{\modf^{\bY}_{(I^A_\unip,\cX_A)}(\cR)}(\bV^A_w, \cM)=0$ for any $w \in {}^A W_\ext$. Moreover, in this case, for any $w \in {}^A W_\ext$ we have
\[
\Ext^n_{\modf^{\bY}_{(I^A_\unip,\cX_A)}(\cR)}(\bV^A_w, \cM)=0 \quad \text{for any $n \geq 1$}
\]
and
\[
\dim \Hom_{\modf^{\bY}_{(I^A_\unip,\cX_A)}(\cR)}(\bV^A_w, \cM) = ( \cM : \cbV^A_w).
\]
\item 
The object $\cM$ admits a baby Verma filtration if and only if it satisfies $\Ext^1_{\modf^{\bY}_{(I^A_\unip,\cX_A)}(\cR)}(\cM, \cbV^A_w)=0$ for any $w \in {}^A W_\ext$. Moreover, in this case, for any $w \in {}^A W_\ext$ we have
\[
\Ext^n_{\modf^{\bY}_{(I^A_\unip,\cX_A)}(\cR)}(\cM, \cbV^A_w) \quad \text{for any $n \geq 1$}
\]
and
\[
\dim \Hom_{\modf^{\bY}_{(I^A_\unip,\cX_A)}(\cR)}(\cM,\cbV^A_w) = ( \cM : \bV^A_w).
\]
\end{enumerate}
\end{lem}

Lemma~\ref{lem:verma-hom} implies in particular that the property of admitting a baby Verma filtration is stable under direct summands, and similarly for baby co-Verma filtrations. It also implies that we have a ``reciprocity formula'' (see~\cite[Corollaries~3.12 and 3.15]{bs}): for $w,y \in {}^A W_\ext$ we have
\begin{equation}
\label{eqn:reciprocity}
(\injh^A_w : \cbV^A_y) = [\bV^A_y : \LGT^A_w], \qquad (\injh^A_{\iota_A(w)} : \bV^A_y) = [\cbV^A_y : \LGT^A_w],
\end{equation}
where $\iota_A(w)$ is defined in~\eqref{eqn:def-iota}.
(Here, the numbers on both sides of the first equation are equal to $\dim \Hom(\bV^A_y, \injh^A_w)$, and those in the second equation are equal to $\dim \Hom(\injh^A_{\iota_A(w)}, \cbV^A_y)$.)

\begin{rmk}
\label{rmk:labels-injh}
The first equality in~\eqref{eqn:reciprocity} and Proposition~\ref{prop:bV-projective} show that $(\injh^A_w : \cbV^A_y)$ vanishes unless $w \preceq y$, and is equal to $1$ if $y=w$. In particular, given an indecomposable injective object $\injh$ in $\modf_{(I^A_\unip,\cX_A)}(\cR)$, to determine the socle of $\injh$ (i.e.~its label) it suffices to determine the smallest element $w$ (for the order $\preceq$) such that $(\injh : \cbV^A_w) \neq 0$. Since the classes $([\cbV^A_w] : w \in {}^A W_\ext)$ in $[\modf_{(I^A_\unip,\cX_A)}(\cR)]$ are linearly independent (see Remark~\ref{rmk:cbV-Groth-gp}), this also implies that the classes $([\injh^A_w] : w \in {}^A W_\ext)$ are linearly independent.
\end{rmk}

\begin{cor}
\label{cor:dimhom-verma}
Let $\cM, \cN \in \modf^{\bY}_{(I^A_\unip,\cX_A)}(\cR)$.  If $\cM$ admits a baby Verma filtration and $\cN$ admits a baby co-Verma filtration, then
\[
\dim \Hom_{ \modf^{\bY}_{(I^A_\unip,\cX_A)}(\cR)}(\cM,\cN) = \sum_{y \in {}^A W_\ext} (\cM : \bV^A_y) (\cN: \cbV^A_y).
\]
\end{cor}
\begin{proof}
We proceed by induction on the number of steps in a baby co-Verma filtration of $\cN$.  If $\cN$ is itself a baby co-Verma module, say $\cN = \cbV^A_w$, then the lemma reduces to the claim that $\dim \Hom(\cM, \cbV^A_w) = (\cM: \bV^A_w)$, which is part of Lemma~\ref{lem:verma-hom}.  Otherwise, choose a short exact sequence $0 \to \cbV^A_w \to \cN \to \cN' \to 0$, where $\cN'$ has a baby co-Verma filtration with fewer steps.  Since $\Ext^1(\cM, \cbV^A_w) = 0$ by Lemma~\ref{lem:verma-hom}, we get a short exact sequence
\[
0 \to \Hom(\cM, \cbV^A_w) \to \Hom(\cM, \cN) \to \Hom(\cM, \cN') \to 0,
\]
and hence
\[
\dim \Hom(\cM, \cN) = (\cM: \bV^A_w) + \dim \Hom(\cM,\cN').
\]
The lemma follows by induction.
\end{proof}

\begin{cor}
\label{cor:endtilt-indep}
Let $\omega \in \Omega$, and let $s_1, \ldots, s_r \in S_\aff$.  The integer
\[
\dim \End(\Phi(\xi_{s_r} \cdots \xi_{s_1} \xi_{\omega^{-1}}(j^\varsigma_{!*} \cS_\varsigma)))
\]
is independent of the field $\bk$.
\end{cor}

\begin{proof}
As in the proof of Proposition~\ref{prop:inj-hull-bV}
the object $\Phi(\xi_{s_r} \cdots \xi_{s_1} \xi_{\omega^{-1}}(j^\varsigma_{!*} \cS_\varsigma))$ is both injective and projective, so it admits both a baby Verma filtration and a baby co-Verma filtration (see Theorem~\ref{thm:proj-Hecke-Whit}\eqref{it:proj-filt} and Proposition~\ref{prop:inj-Verma-filt}). By Corollary~\ref{cor:dimhom-verma}, it is therefore enough to show that the baby (co-)Verma multiplicities are independent of the field of coefficients $\bk$.  The baby co-Verma multiplicities in $\Phi(j^\varsigma_{!*}\cS_\varsigma)$ are given explicitly in Lemma~\ref{lem:Pi-sigma-bV-filt} (and are manifestly independent of $\bk$).  Then, the baby co-Verma multiplicities in $\Phi(\xi_{s_r} \cdots \xi_{s_1} \xi_{\omega^{-1}}(j^\varsigma_{!*} \cS_\varsigma))$ can be computed by the combinatorial rules from Lemma~\ref{lem:bV-xi}, from which we see that these multiplicities are again independent of $\bk$. We deduce the same property for baby Verma multiplicities using~\eqref{eqn:equality-multiplicites}.
\end{proof}

\subsection{Integral \texorpdfstring{$\cR$}{R}-modules}
\label{ss:integral}

In order to prove some further properties of baby Verma and co-Verma modules, we need to make a detour through a version of our categories over the ring of integers $\bO$ of a finite extension of $\Q_\ell$. (Our understanding of this theory is quite limited, and the results obtained below are clearly not fully satisfactory. They will still be sufficient for the applications we have in mind.)
For clarity, in this subsection we will sometimes add the coefficients in parentheses to various notations introduced above.

Let $\bO$ be the ring of integers in a finite extension $\bK$ of $\Q_\ell$ containing a nontrivial $p$-th root of unity, and let $\bF$ be its residue field. Then we can consider the categories $\Db_{(I^A_\unip, \cX_A)}(\Gr,\bE)$ and $\Db_{\cL^+G}(\Gr,\bE)$ for any $\bE \in \{\bK, \bO, \bF\}$, and we have change-of-scalars functors
\begin{align*}
\bK := \bK \lotimes_\bO (-) : \Db_{(I^A_\unip, \cX_A)}(\Gr,\bO) \to \Db_{(I^A_\unip, \cX_A)}(\Gr,\bK), \\
\bF := \bF \lotimes_\bO (-) : \Db_{(I^A_\unip, \cX_A)}(\Gr,\bO) \to \Db_{(I^A_\unip, \cX_A)}(\Gr,\bF),
\end{align*}
and similarly for the $\cL^+G$-equivariant categories. (Here the symbol $(I_\unip^A,\cX_A)$ refers to the Whittaker condition over $\bO$, $\bF$ or $\bK$ depending on the coefficients appearing elsewhere in the notation.) 


The definition of the ind-perverse sheaf $\cR$ with coefficients in $\bO$ (denoted $\cR(\bO)$ below) can be copied verbatim from~\S\ref{ss:reg-perv-sheaf}.
Note that the geometric Satake equivalence is also known over $\bO$; however, in that setting
the definition of the convolution product on $\Perv_{\cL^+G}(\Gr,\bO)$ involves a perverse truncation, see~\cite[Equation~(4.2)]{mv}. Here we continue to define the bifunctors
\begin{align*}
\star^{\cL^+ G} &: \Db_{\cL^+G}(\Gr,\bO) \times \Db_{\cL^+G}(\Gr,\bO) \to \Db_{\cL^+G}(\Gr,\bO), \\
\star^{\cL^+ G} &: \Db_{(I^A_\unip, \cX_A)}(\Gr,\bO) \times \Db_{\cL^+G}(\Gr,\bO) \to \Db_{(I^A_\unip, \cX_A)}(\Gr,\bO)
\end{align*}
as in~\eqref{eqn:convol-L+G} or~\S\ref{ss:Whittaker}, i.e.~without incorporating the perverse truncation; with this notation, the convolution product on $\Perv_{\cL^+G}(\Gr,\bO)$ used in the construction of the geometric Satake equivalence is therefore given by
\[
(\cF, \cG) \mapsto \pH^0 \bigl( \cF \star^{\cL^+ G} \cG \bigr).
\]

\begin{lem}
\label{lem:exactness-conv-O}
For any $\cF$ in $\Perv_{(I^A_\unip, \cX_A)}(\Gr,\bO)$ and any $\cG \in \Perv_{\cL^+G}(\Gr,\bO)$ such that $\bF(\cG)$ is perverse, the complex $\cF \star^{\cL^+ \cG} \cG$ is perverse. In particular, for any $\cF$ in $\Perv_{(I^A_\unip, \cX_A)}(\Gr,\bO)$ and $\lambda \in \bY_+$, the complex
$\cF \star^{\cL^+G} \cI^\lambda_*$ is perverse.
\end{lem}

\begin{proof}
To prove this lemma we will use the description of the product $\star^{\cL^+ G}$ in terms of nearby cycles first made explicit by Gaitsgory~\cite{gaitsgory}; see~\cite[Proposition~3.3.1]{ar-book} for more details. By t-exactness of nearby cycles, this description shows that to prove the first statement in the lemma it suffices to prove that the complex
$\cF \lboxtimes_\bO \cG$
on $\Gr \times \Gr$ is perverse.
By right exactness of the derived tensor product, this complex belongs to the nonpositive part of the perverse t-structure. To check that it belongs to the nonnegative part, we have to check that for any $I^A_\unip$-orbit $X \subset \Gr$ and any $\cL^+ G$-orbit $Y \subset \Gr$, the corestriction of our complex $\cF \lboxtimes_\bO \cG$ to $X \times Y$ is concentrated in nonnegative perverse degrees. Now the embedding of $X \times Y$ in $\Gr \times \Gr$ can be written as a composition
\[
X \times Y \hookrightarrow X \times \Gr \xrightarrow{i \times \id} \Gr \times \Gr,
\]
where $i : X \to \Gr$ is the embedding. It therefore suffices to show that the complex
\[
(i \times \id)^! (\cF \lboxtimes_\bO \cG) = (i^! \cF) \lboxtimes_\bO \cG
\]
is concentrated in nonnegative perverse degrees. Now $i^! \cF$ is an extension of constant sheaves $\underline{M}_X[n]$ with $M$ a finitely generated $\bO$-module and $n \leq \dim(X)$. Our assumption on $\cG$ ensures that each $\underline{M}_X[n] \lboxtimes_\bO \cG$ is in nonnegative perverse degrees, which implies our claim.

The second assertion of the lemma follows from the first one, since $\cI^\lambda_*$ satisfies the required assumption by~\cite[Proposition~8.1]{mv}.
\end{proof}

Lemma~\ref{lem:exactness-conv-O} shows that for any $\cF$ in $\Perv_{(I^A_\unip, \cX_A)}(\Gr,\bO)$ and any $\lambda,\mu \in \bY$ appearing in~\eqref{eqn:description-Rmu} the convolution $\cF \star^{\cL^+ G} \cI^{w_\circ(\mu) + \lambda}_* \star^{\cL^+G} \cI^{-w_\circ(\lambda)}_*$ is perverse. For any ind-object $\cF$ in $\Perv_{(I^A_\unip, \cX_A)}(\Gr,\bO)$ and any $\mu \in \bY$ we can therefore consider the ind-object $\cF \star^{\cL^+ G} \cR_\mu$ in $\Perv_{(I^A_\unip, \cX_A)}(\Gr,\bO)$, and then make sense
of the abelian category $\Mod^\bY_{(I^A_\unip, \cX_A)}(\cR(\bO))$ as in~\S\ref{ss:graded-R-modules}. The functor
\[
\Phi^A: \Perv_{(I^A_\unip, \cX_A)}(\Gr,\bO) \to \Mod^\bY_{(I^A_\unip, \cX_A)}(\cR(\bO)),
\]
and the notions of \emph{free $\cR(\bO)$-module of finite type} and \emph{finitely generated $\cR(\bO)$-module} can also be copied. Lemma~\ref{lem:Rfree-hom} holds unchanged.  The definitions of the various averaging functors from~\S\ref{ss:av-functors} carry over to this setting, and one can show that these functors are still exact;
in particular, there is an exact functor
\[
\xi_A: \Mod^\bY_{I_\unip}(\cR(\bO)) \to \Mod^\bY_{I_\unip}(\cR(\bO))
\]
that sends finitely generated $\cR(\bO)$-modules to finitely generated $\cR(\bO)$-modules.

\begin{rmk}\label{rmk:fingen-O}
As with field coefficients, it is not immediately obvious that the category of finitely generated $\cR(\bO)$-modules, which one may denote $\modf^\bY_{(I^A_\unip, \cX_A)}(\cR(\bO))$, is abelian.  We will not address this question in this paper, as working in the category $\Mod^\bY_{(I^A_\unip, \cX_A)}(\cR(\bO))$ will be sufficient for our purposes.
\end{rmk}

The change-of-scalars functor $\bF$ induces a right-exact functor
\[
\bF^0 : \Mod^\bY_{(I^A_\unip, \cX_A)}(\cR(\bO)) \to \Mod^\bY_{(I^A_\unip, \cX_A)}(\cR(\bF))
\]
defined as follows. If $\cM \in \Mod^\bY_{(I^A_\unip, \cX_A)}(\cR(\bO))$, the underlying graded ind-perverse sheaf of $\bF^0(\cM)$ is $\pH^0(\bF(\cM))$, with the obvious grading. The morphisms 
\[
\pH^0(\bF(\cM_\lambda)) \star^{\cL^+G} \cR_\mu(\bF) \to \pH^0(\bF(\cM_{\lambda+\mu}))
\]
are obtained from the morphisms $\cM_\lambda \star^{\cL^+ G} \cR_\mu(\bO) \to \cM_{\lambda+\mu}$ by application of the functor $\pH^0(\bF(-))$, using the fact that for any $\cF$ in $\Perv_{(I^A_\unip, \cX_A)}(\Gr,\bO)$ we have
\begin{multline*}
\pH^0(\bF(\cF \star^{\cL^+G} \cI^{w_\circ(\mu) + \lambda}_*(\bO) \star^{\cL^+G} \cI^{-w_\circ(\lambda)}_*(\bO))) \cong \\
\pH^0(\bF(\cF)) \star^{\cL^+G} \cI^{w_\circ(\mu) + \lambda}_*(\bF) \star^{\cL^+G} \cI^{-w_\circ(\lambda)}_*(\bF)
\end{multline*}
by commutation of $\bF$ with convolution, the fact that $\bF(\cI^\nu_*(\bO)) \cong \cI^\nu_*(\bF)$ for any $\nu$ (see the proof of Lemma~\ref{lem:exactness-conv-O}), and exactness of convolution over $\bF$. Similar considerations show that the functor $\bK$ induces in the natural way an exact functor
\[
\bK : \Mod^\bY_{(I^A_\unip, \cX_A)}(\cR(\bO)) \to \Mod^\bY_{(I^A_\unip, \cX_A)}(\cR(\bK)).
\]
(Here there is no perverse cohomology involved in the construction since $\bK$ is t-exact; we therefore do not add any superscript in the notation.) It is clear from this definition that for $\cF$ in $\Perv_{(I^A_\unip, \cX_A)}(\Gr,\bO)$ we have canonical isomorphisms
\[
\bF^0 \circ \Phi^A(\cF) \cong \Phi^A(\pH^0(\bF(\cF))), \quad \bK \circ \Phi^A(\cF) \cong \Phi^A(\bK(\cF)).
\]
The functors $\bK$ and $\bF^0$ also commute in the natural way with the averaging functors.

We will now analyze the effect of the functors $\bK$ and $\bF^0$ on morphisms. We start with the much easier case of $\bK$.

\begin{lem}
\label{lem:scalars-R-mod-K}
Let $\cF, \cG \in \Perv_{(I^A_\unip, \cX_A)}(\Gr,\bO)$. 
For any $\nu \in \bY$, the functor $\bK$ induces an isomorphism
\begin{multline*}
\bK \otimes_\bO \Hom_{\Mod^\bY_{(I^A_\unip, \cX_A)}(\cR(\bO))}(\Phi(\cF), \Phi(\cG)\la \nu\ra) \simto \\
\Hom_{\modf^\bY_{(I^A_\unip, \cX_A)}(\cR(\bK))}(\Phi(\bK(\cF)), \Phi(\bK(\cG))\la \nu \ra).
\end{multline*}
\end{lem}

\begin{proof}
By (the analogue of) Lemma~\ref{lem:Rfree-hom} we have
\begin{multline*}
\Hom_{\Mod^\bY_{(I^A_\unip, \cX_A)}(\cR(\bO))}(\Phi(\cF), \Phi(\cG)\la \nu\ra) \cong \\
\varinjlim_\lambda \Hom_{\Perv_{(I^A_\unip, \cX_A)}(\Gr,\bO)}(\cF, \cG \star^{\cL^+G} \cI_*^{w_\circ(\nu)+\lambda}(\bO) \star^{\cL^+G} \cI_*^{-w_\circ(\lambda)}(\bO)),
\end{multline*}
and similarly over $\bK$. The claim follows, using the fact that inductive limits commute with tensor product, and that for $\cH_1, \cH_2$ in $\Perv_{(I^A_\unip, \cX_A)}(\Gr,\bO)$ the morphism
\begin{equation}
\label{eqn:scalars-perv-O-K}
\bK \otimes_\bO \Hom_{\Perv_{(I^A_\unip, \cX_A)}(\Gr,\bO)}(\cH_1, \cH_2) \to \Hom_{\Perv_{(I^A_\unip, \cX_A)}(\Gr,\bK)}(\bK(\cH_1), \bK(\cH_2))
\end{equation}
induced by $\bK$ is an isomorphism.
\end{proof}

We now consider the more delicate case of $\bF^0$.
Recall that a perverse sheaf $\cF$ in $\Perv_{(I^A_\unip, \cX_A)}(\Gr,\bO)$ or $\Perv_{\cL^+ G}(\Gr,\bO)$ is called \emph{torsion-free} if multiplication by a uniformizer of $\bO$ is injective on $\cF$, or equivalently if $\bF(\cF)$ is perverse. Using this characterization and t-exactness of convolution over $\bF$, one easily sees that if $\cF$ in $\Perv_{(I^A_\unip, \cX_A)}(\Gr,\bO)$ or $\Perv_{\cL^+ G}(\Gr,\bO)$ and $\cG$ in $\Perv_{\cL^+ G}(\Gr,\bO)$ are torsion-free, then $\cF \star^{\cL^+G} \cG$ is torsion-free. Recall also that the $\bO$-module of morphisms between two torsion-free perverse sheaves is finitely generated and torsion-free, hence free of finite rank. 
The following lemma comprises an analogue of Lemma~\ref{lem:hom-stabilize} over $\bO$.

\begin{lem}
\label{lem:scalars-R-mod-F}
Let $\cF, \cG \in \Perv_{(I^A_\unip, \cX_A)}(\Gr,\bO)$ be torsion-free, and let $\nu \in \bY$.

\begin{enumerate}
\item
\label{it:scalars-R-mod-F-1}
If $\lambda \in \bY_+ \cap (w_\circ(\nu) + \bY_+)$
is sufficiently large, the natural map
\[
\Hom(\cF, \cG \star^{\cL^+G} \cI^{-w_\circ(\nu)+\lambda}_*(\bO) \star^{\cL^+G} \cI^{-w_\circ(\lambda)}_*(\bO)) \to \Hom(\Phi(\cF), \Phi(\cG)\la \nu\ra)
\]
is an isomorphism.  In particular, $\Hom(\Phi(\cF), \Phi(\cG)\la \nu\ra)$ is a free $\bO$-module of finite rank.
\item
\label{it:scalars-R-mod-F-2}
The functor $\bF^0$ induces an injection
\[
\bF \otimes_\bO \Hom_{\modf^\bY_{I_\unip}(\cR(\bO))}(\Phi(\cF), \Phi(\cG)\la \nu \ra) \hookrightarrow \Hom_{\modf^\bY_{I_\unip}(\cR(\bF))}(\Phi(\bF(\cF)), \Phi(\bF(\cG))\la \nu \ra),
\]
which is an isomorphism if and only if
\begin{multline*}
\dim_\bK \Hom_{\modf^\bY_{I_\unip}(\cR(\bK))}(\Phi(\bK(\cF)), \Phi(\bK(\cG))\la \nu \ra)
= \\
 \dim_\bF \Hom_{\modf^\bY_{I_\unip}(\cR(\bF))}(\Phi(\bF(\cF)), \Phi(\bF(\cG))\la \nu \ra).
\end{multline*}
\end{enumerate}
\end{lem}

\begin{proof}
\eqref{it:scalars-R-mod-F-1}
Recall that if $\cH_1$ and $\cH_2$ are torsion-free perverse sheaves, then there is a natural injective map
\begin{equation}
\label{eqn:scalars-perv-O}
\bF \otimes_\bO \Hom(\cH_1,\cH_2) \hookrightarrow \Hom(\bF \lotimes_\bO \cH_1, \bF \lotimes_\bO \cH_2),
\end{equation}
whose cokernel identifies with $\mathrm{Tor}^\bO_1(\bF, \Ext^1(\cH_1,\cH_2))$.

Let $\lambda \in \bY_+ \cap (w_\circ(\nu) + \bY_+)$, let $\mu \in \bY_+$, and set $\lambda':=\lambda+\mu$.  In the following diagram, the vertical maps come from the morphisms~\eqref{eqn:scalars-perv-O-K} and~\eqref{eqn:scalars-perv-O}, and the horizontal maps are defined by the considerations in~\S\ref{ss:reg-perv-sheaf}.  Note that the second row consists of free $\bO$-modules, while the top and bottom rows consist of $\bK$- and $\bF$-vector spaces, respectively.
{\small
\[
\begin{tikzcd}[column sep=small]
\Hom(\bK(\cF), \bK(\cG) \star \cI^{-w_\circ(\nu)+\lambda}_* \star \cI^{-w_\circ(\lambda)}_*) \ar[r, "\text{(i)}"] &
\Hom(\bK (\cF), \bK(\cG) \star \cI^{-w_\circ(\nu)+\lambda'}_* \star \cI^{-w_\circ(\lambda')}_*) \\
\Hom(\cF, \cG \star \cI^{-w_\circ(\nu)+\lambda}_* \star \cI^{-w_\circ(\lambda)}_*) \ar[r, "\text{(ii)}"] \ar[u, hook]\ar[d, two heads] &
\Hom(\cF, \cG \star \cI^{-w_\circ(\nu)+\lambda'}_* \star \cI^{-w_\circ(\lambda')}_*) \ar[u, hook]\ar[d, two heads] \\
\bF \otimes \Hom(\cF, \cG \star \cI^{-w_\circ(\nu)+\lambda}_* \star \cI^{-w_\circ(\lambda)}_*) \ar[r, "\text{(iii)}"] \ar[d, hook] &
\bF \otimes \Hom(\cF, \cG \star \cI^{-w_\circ(\nu)+\lambda'}_* \star \cI^{-w_\circ(\lambda')}_*) \ar[d, hook] \\
\Hom(\bF(\cF), \bF(\cG) \star \cI^{-w_\circ(\nu)+\lambda}_* \star \cI^{-w_\circ(\lambda)}_*) \ar[r, "\text{(iv)}"] &
\Hom(\bF (\cF), \bF(\cG) \star \cI^{-w_\circ(\nu)+\lambda'}_* \star \cI^{-w_\circ(\lambda')}_*).
\end{tikzcd}
\]
}Our goal is to prove that when $\lambda$ is large enough, arrow~(ii) is an isomorphism.  
Lemma~\ref{lem:hom-stabilize} and Remark~\ref{rmk:hom-stabilize} show that when $\lambda$ is large enough, arrows~(i) and~(iv) are both isomorphisms.  The topmost and bottommost commutative squares then show that arrows~(ii) and~(iii) are both injective.

The free $\bO$-modules in the second row must have equal (finite) ranks, because the $\bK$-vector spaces they give rise to in the first row have equal dimensions.  It then follows that the $\bF$-vector spaces in the third row also have equal dimensions.  Since arrow~(iii) is an injective map between $\bF$-vector spaces of equal dimension, it is actually an isomorphism.  In view of this, Nakayama's lemma implies that the middle one is surjective, which finishes the proof.

\eqref{it:scalars-R-mod-F-2}
If $\lambda$ is large enough, by~\eqref{it:scalars-R-mod-F-1} and Lemma~\ref{lem:hom-stabilize} we have identifications
\[
\Hom(\cF, \cG \star^{\cL^+G} \cI^{-w_\circ(\nu)+\lambda}_*(\bO) \star^{\cL^+G} \cI^{-w_\circ(\lambda)}_*(\bO)) \simto \Hom(\Phi(\cF), \Phi(\cG)\la \nu\ra)
\]
and
\begin{multline*}
\Hom(\bE(\cF), \bE(\cG) \star^{\cL^+G} \cI^{-w_\circ(\nu)+\lambda}_*(\bE) \star^{\cL^+G} \cI^{-w_\circ(\lambda)}_*(\bE)) \\
\simto \Hom(\Phi(\bE(\cF)), \Phi(\bE(\cG))\la \nu\ra)
\end{multline*}
for $\bE=\bK$ or $\bF$. Via these identifications, the morphism under consideration identifies with the morphism
\begin{multline*}
\Hom(\cF, \cG \star^{\cL^+G} \cI^{-w_\circ(\nu)+\lambda}_*(\bO) \star^{\cL^+G} \cI^{-w_\circ(\lambda)}_*(\bO)) \to \\
\Hom(\bF(\cF), \bF(\cG) \star^{\cL^+G} \cI^{-w_\circ(\nu)+\lambda}_*(\bF) \star^{\cL^+G} \cI^{-w_\circ(\lambda)}_*(\bF)) 
\end{multline*}
induced by $\bF$. The latter morphism is injective, as noted in~\eqref{eqn:scalars-perv-O}. It is an isomorphism if and only if these vector spaces have equal dimension, which in view of the identifications above and the isomorphism~\eqref{eqn:scalars-perv-O-K} is equivalent to the condition in the statement.
\end{proof}

\begin{rmk}
\label{rmk:morph-O-P-tilting}
For later use, let us record a special case in which the condition in Lemma~\ref{lem:scalars-R-mod-F}\eqref{it:scalars-R-mod-F-2} is satisfied. We assume that $A=\varnothing$, and fix some simple reflections $s_1, \ldots, s_r \in S_\aff$ and some $\omega \in \Omega$. For $\bE \in \{\bK, \bO, \bF\}$ we set
\[
\cP(\bE) := \xi_{s_1} \cdots \xi_{s_r} \xi_{\omega}(\TGr_{t_{w_\circ(\varsigma)}}(\bE)).
\]
(See~\S\ref{ss:kl} for generalities on tilting $\bO$-perverse sheaves.) Here we have $\TGr_{t_{w_\circ(\varsigma)}}(\bO) \cong j^{\varsigma}_!\cS_{\varsigma}(\bO) \cong j^{\varsigma}_*\cS_{\varsigma}(\bO)$. In particular we have
\[
\bF(\TGr_{t_{w_\circ(\varsigma)}}(\bO)) \cong \TGr_{t_{w_\circ(\varsigma)}}(\bF), \quad \bK(\TGr_{t_{w_\circ(\varsigma)}}(\bO)) \cong \TGr_{t_{w_\circ(\varsigma)}}(\bK),
\]
which implies that
\[
\bK(\cP(\bO)) \cong \cP(\bK)
\qquad\text{and}\qquad
\bF(\cP(\bO)) \cong \cP(\bF).
\]
In particular, this shows that $\bF(\cP(\bO))$ is tilting, which by standard arguments implies that $\cP(\bO)$ is tilting.

If $\cF \in \Perv_{I_\unip}(\Gr, \bO)$ is tilting, then the condition in Lemma~\ref{lem:scalars-R-mod-F}\eqref{it:scalars-R-mod-F-2} is satisfied for the pairs of objects $(\cF, \cP(\bO))$ and $(\cP(\bO),\cF)$ and any $\nu \in \bY$.
Indeed, for $\lambda \in \bY_+ \cap (w_\circ(\nu) + \bY_+)$ the object $\cI^{-w_\circ(\nu)+ \lambda}_*(\bF) \star^{\cL^+G} \cI^{-w_\circ(\lambda)}_*(\bF)$ has a costandard filtration (in $\Perv_{\cL^+ G}(\Gr,\bF)$) by~\cite[Proposition~4.8]{projGr1}, hence
\begin{multline*}
\cP(\bF) \star^{\cL^+G} \cI^{-w_\circ(\nu)+ \lambda}_*(\bF) \star^{\cL^+G} \cI^{-w_\circ(\lambda)}_*(\bF) \cong \\
\bF(\cP(\bO) \star^{\cL^+G} \cI^{-w_\circ(\nu)+ \lambda}_*(\bO) \star^{\cL^+G} \cI^{-w_\circ(\lambda)}_*(\bO))
\end{multline*}
admits a costandard filtration (in $\Perv_{I_\unip}(\Gr,\bF)$) by Lemma~\ref{lem:twtri-standard} and the considerations in~\S\ref{ss:wc-functors}. By standard arguments, this implies that the perverse sheaf $\cP(\bO) \star^{\cL^+G} \cI^{-w_\circ(\nu)+ \lambda}_*(\bO) \star^{\cL^+G} \cI^{-w_\circ(\lambda)}_*(\bO)$ admits a costandard filtration in the category $\Perv_{I_\unip}(\Gr,\bO)$. We deduce the equality
\begin{multline*}
\dim \Hom(\bF(\cF), \cP(\bF) \star^{\cL^+G} \cI^{-w_\circ(\nu)+ \lambda}_*(\bF) \star^{\cL^+G} \cI^{-w_\circ(\lambda)}_*(\bF)) = \\
\dim \Hom(\bK(\cF), \cP(\bK) \star^{\cL^+G} \cI^{-w_\circ(\nu)+ \lambda}_*(\bK) \star^{\cL^+G} \cI^{-w_\circ(\lambda)}_*(\bK))
\end{multline*}
by the usual formula calculating dimensions of morphism spaces from a standardly-filtered object to a costandardly-filtered object in terms of multiplicities. Similar arguments apply in the second setting, using the isomorphism
\begin{multline*}
\Hom(\cP(\bE), \bE(\cF) \star^{\cL^+G} \cI^{-w_\circ(\nu)+ \lambda}_*(\bE) \star^{\cL^+G} \cI^{-w_\circ(\lambda)}_*(\bE)) \cong \\
\Hom(\cP(\bE) \star^{\cL^+G} \cI^{\nu- w_\circ(\lambda)}_!(\bE) \star^{\cL^+G} \cI^{\lambda}_!(\bE), \bE(\cF))
\end{multline*}
for $\bE=\bK$ or $\bF$.
\end{rmk}

We finally come to the main result of this subsection.

\begin{thm}
\label{thm:injh-integral}
For each $w \in W_\ext$, there exists an indecomposable finitely generated $\cR(\bO)$-module $\injh_w(\bO)$ with the following properties:
\begin{enumerate}
\item 
\label{it:qO-free}
it is a direct summand of an object $\Phi(\xi_{s_1} \cdots \xi_{s_r} \xi_{\omega}(\TGr_{t_{w_\circ(\varsigma)}}(\bO))) \langle \nu \rangle$ for some $s_1, \cdots, s_r \in S_\aff$, $\omega \in \Omega$, $\nu \in \bY$;
\item 
\label{it:qO-reduce}
we have $\bF^0( \injh_w(\bO)) \cong \injh_w(\bF)$;
\item 
\label{it:qO-frac}
the $\cR(\bK)$-module $\bK ( \injh_w(\bO))$ is projective and injective, and contains the indecomposable object $\injh_w(\bK)$ as a direct summand with multiplicity $1$.
\end{enumerate}
\end{thm}

\begin{proof}
By periodicity, it is enough to prove this claim when $w \in W^\res_\ext$.  Assume this from now on.  Choose $s_1, \ldots, s_r \in S_\aff$ and $\omega \in \Omega$ as in Lemma~\ref{lem:proj-prep}, and for $\bE \in \{\bK, \bO, \bF\}$, let
\[
\cP(\bE) := \xi_{s_r} \cdots \xi_{s_1} \xi_{\omega^{-1}}(\TGr_{t_{w_\circ(\varsigma)}}(\bE)).
\]
(See Remark~\ref{rmk:morph-O-P-tilting} for comments on this definition.)

When $\bE=\bK$ or $\bF$, the proof of Proposition~\ref{prop:inj-hull-bV} shows that $\Phi(\cP(\bE))$ is injective and projective, and that it contains the injective envelope $\injh_w(\bE)$ of $\LGT_w(\bE)$ as a direct summand.  More precisely, considering the baby co-Verma multiplicities described in that proposition and comparing with Theorem~\ref{thm:proj-Hecke-Whit}, we see that in this case:
\begin{itemize}
\item $\Phi(\cP(\bE))$ contains $\injh_w(\bE)$ as a direct summand with multiplicity $1$;
\item for any other indecomposable injective object $\injh_y(\bE)$ occuring as a direct summand of $\Phi(\cP(\bE))$, we have $y \succ w$.
\end{itemize}
The first item above implies that we then have
\begin{equation}\label{eqn:scp-mult-field}
\dim \Hom(\Phi(\LGr_w(\bE)), \Phi(\cP(\bE))) = 1.
\end{equation}

We now consider the case $\bE=\bO$.
Let $\LGr_w(\bO)$ be the intersection cohomology complex associated with the constant $\bO$-local system on $\Gr_w$. It is well known that $\bK(\LGr_w(\bO)) \cong \LGr_w(\bK)$, so that~\eqref{eqn:scp-mult-field} implies that
\begin{equation}
\label{eqn:scp-mult-char0}
\dim \Hom(\Phi(\bK(\LGr_w(\bO))), \Phi(\cP(\bK))) = 1.
\end{equation}
Next, let us study the analogous problem over $\F$.  The modular reduction $\F(\LGr_w(\bO))$ is a perverse sheaf which is not simple in general; instead, there is a short exact sequence
\[
0 \to \cK \to \bF(\LGr_w(\bO)) \to \LGr_w(\bF) \to 0
\]
where $\cK$ is a perverse sheaf with composition factors of the form $\LGr_z(\bF)$ with $z \in W_\ext^S$ such that $z < w$ (and hence $z \prec w$ by Lemma~\ref{lem:per-order-properties}\eqref{it:per-order-3}).  By Lemma~\ref{lem:Phi-order}, the $\cR(\bF)$-module $\Phi(\cK)$ has all of its composition factors of the form $\LGT_{z'}(\bF)$ with $z' \prec w$.


Apply the exact functor $\Hom(\Phi({-}), \Phi(\cP(\bF)))$ to the exact sequence above to get a short exact sequence
\begin{multline*}
0 \to \Hom (\Phi(\LGr_w(\bF)), \Phi(\cP(\bF))) \to \Hom(\bF(\Phi(\LGr_w(\bO))), \Phi(\cP(\bF))) \\
\to \Hom(\Phi(\cK), \Phi(\cP(\bF))) \to 0.
\end{multline*}
By~\eqref{eqn:scp-mult-field}, the first term has dimension~$1$.  The claim above on the composition factors of $\Phi(\cK)$ 
and the description of the summands $\injh_y(\bF)$ that can occur in $\Phi(\cP(\bF))$ imply that $\Hom(\Phi(\cK), \Phi(\cP(\bF))) = 0$.  We conclude that
\begin{equation}
\label{eqn:scp-mult-charp}
\dim \Hom(\Phi(\bF(\LGr_w(\bO))), \cP(\bF)) = 1.
\end{equation}

Combining~\eqref{eqn:scp-mult-char0}, \eqref{eqn:scp-mult-charp}, and Lemma~\ref{lem:scalars-R-mod-F}, we see that the functor $\bF^0$ induces an isomorphism
\begin{equation}\label{eqn:homlp-reduce}
\bF \otimes_\bO \Hom(\Phi(\LGr_w(\bO)), \Phi(\cP(\bO))) \simto \Hom(\Phi(\bF(\LGr_w(\bO))), \Phi(\cP(\bF))).
\end{equation}
Similarly, according to Corollary~\ref{cor:endtilt-indep}, the rings $\End(\Phi(\cP(\bK)))$ and $\End(\Phi(\cP(\bF)))$ have equal dimensions, so by Lemma~\ref{lem:scalars-R-mod-F} again, the functor $\bF^0$ induces an isomorphism
\begin{equation}
\label{eqn:endp-reduce}
\bF \otimes_\bO \End(\Phi(\cP(\bO))) \simto \End(\Phi(\cP(\bF))).
\end{equation}
Of course, analogues of both of these isomorphisms hold if $\Phi(\cP(\bO))$ is replaced by some direct summand.

Let us now study the direct summands of $\Phi(\cP(\bO))$.  Because $\bO$ is a complete noetherian local ring, the finite $\bO$-algebra $\End(\Phi(\cP(\bO)))$ is a semiperfect ring by~\cite[Example~23.3]{lam}, and similarly for direct sums of copies of $\Phi(\cP(\bO))$ or direct summands of such objects. Then, by~\cite[Corollary~4.4]{krause} applied to the full subcategory of $\Mod^\bY_{I_\unip}(\cR(\bO))$ generated by $\Phi(\cP(\bO))$ under direct sums and direct summands, the Krull--Schmidt theorem applies to $\Phi(\cP(\bO))$: this object has a unique (up to isomorphism and reordering) decomposition into indecomposable summands
\[
\Phi(\cP(\bO)) \cong \cM_1 \oplus \cdots \oplus \cM_n
\]
where each $\cM_i$ has a local endomorphism ring. 
As explained above, for any $i$ we have a natural isomorphism
\[
\bF \otimes_\bO \End(\cM_i) \simto \End(\bF(\cM_i)).
\]
In particular, $\End(\bF(\cM_i))$ is again a local ring.  In other words, each $\cR(\bF)$-module $\bF(\cM_i)$ is indecomposable, and these are all the indecomposable summands of $\Phi(\cP(\bF))$.  Exactly one of these summands is isomorphic to $\injh_w(\bF)$; we can therefore assume without loss of generality that $\injh_w(\bF) \cong \bF(\cM_1)$.  We define
\[
\injh_w(\bO) := \cM_1.
\]

Parts~\eqref{it:qO-free} and~\eqref{it:qO-reduce} of the theorem hold by construction.  It remains to prove part~\eqref{it:qO-frac}.  This is equivalent to showing that
\[
\dim \Hom(\LGT_w(\bK), \bK(\injh_w(\bO))) = 1.
\]
To see this, observe first that since $\bK(\LGr_w(\bO)) \cong \LGr_w(\bK)$, by Lemmas~\ref{lem:scalars-R-mod-K} and~\ref{lem:scalars-R-mod-F} and the construction of $\injh_w(\bO)$ we have
\[
\dim \Hom(\LGT_w(\bK), \bK(\injh_w(\bO))) = \dim \bF \otimes_\bO \Hom(\Phi(\LGr_w(\bO)), \injh_w(\bO)).
\]
Next, it follows from~\eqref{eqn:homlp-reduce} that
\[
\dim \bF \otimes_\bO \Hom(\Phi(\LGr_w(\bO)), \injh_w(\bO))
= \dim \Hom(\Phi(\bF(\LGr_w(\bO))), \injh_w(\bF)),
\]
and the right-hand side is equal to $1$ by~\eqref{eqn:scp-mult-charp}.
\end{proof}

\begin{rmk}
\phantomsection
\label{rmk:Hom-Q-O}
\begin{enumerate}
\item
It is possible to adapt the reasoning of Propositions~\ref{prop:IW-R-semisimple}, \ref{prop:projectives}, and~\ref{prop:inj-hull-bV} to the setting of $\cR(\bO)$-modules, to show that the object $\cP(\bO)$ appearing in the preceding proof is a projective object in $\Mod_{I_\unip}^\bY(\cR(\bO))$.  Therefore, its direct summand $\injh_w(\bO)$ is also projective.  (Of course, it would also be projective in $\modf_{I_\unip}^\bY(\cR(\bO))$ if one knew that the latter category was abelian: see Remark~\ref{rmk:fingen-O}.) Note, however, that $\injh_w(\bO)$ is \emph{not} injective, unlike its field counterparts: this essentially comes down to the fact that $\bO$ is not an injective $\bO$-module.
\item
\label{it:Hom-Q-O}
For any $y,w \in W_\ext$, the $\bO$-module $\Hom(\injh_y(\bO), \injh_w(\bO))$ is free of finite rank, and the functors $\bK$ and $\bF^0$ induce isomorphisms
\[
\bK \otimes_{\bO} \Hom(\injh_y(\bO), \injh_w(\bO)) \simto \Hom(\bK(\injh_y(\bO)), \bK(\injh_w(\bO)))
\]
and
\[
\bF \otimes_{\bO} \Hom(\injh_y(\bO), \injh_w(\bO)) \simto \Hom(\injh_y(\bF), \injh_w(\bF)).
\]
In fact this follows from part~\eqref{it:qO-free} in Theorem~\ref{thm:injh-integral},
combined with Lemmas~\ref{lem:scalars-R-mod-K} and~\ref{lem:scalars-R-mod-F} and Remark~\ref{rmk:morph-O-P-tilting}.
\end{enumerate}
\end{rmk}


\begin{cor}
\label{cor:vanishing-F-K}
Let $\cF$ be a tilting object in $\Perv_{I_\unip}(\Gr,\bO)$. If $\cM$ is a direct summand of $\Phi(\cF)$ which satisfies $\bF^0(\cM)=0$, we have $\bK(\cM)=0$.
\end{cor}

\begin{proof}
Let $y \in W_\ext$. By part~\eqref{it:qO-free} in Theorem~\ref{thm:injh-integral} and by Lemma~\ref{lem:scalars-R-mod-F} and Remark~\ref{rmk:morph-O-P-tilting}, the $\bO$-module $\Hom(\cM, \injh_y(\bO))$ is free of finite rank, and the functor $\bF^0$ induces an isomorphism
\[
\bF \otimes_\bO \Hom(\cM, \injh_y(\bO)) \to \Hom(\bF^0(\cM), \injh_y(\bF)).
\]
Since the right-hand side vanishes, this implies that we have $\Hom(\cM, \injh_y(\bO))=0$. Using Lemma~\ref{lem:scalars-R-mod-K} we deduce that
$\Hom(\bK(\cM), \bK(\injh_y(\bO)))=0$, and hence that $\Hom(\bK(\cM), \injh_y(\bK))=0$ by part~\eqref{it:qO-frac} in Theorem~\ref{thm:injh-integral}.

Since $\bK(\cM)$ is finitely generated, what we have shown implies that this object has no composition factor, i.e.~is trivial.
\end{proof}

The following application of Theorem~\ref{thm:injh-integral} will be needed in~\S\ref{ss:app-mult-proj} (where we will show that, in fact, the dimension in question is equal to $1$).

\begin{prop}
\label{prop:injh-tri-mult}
If $\bk$ is any field satisfying the assumptions of~\S\ref{ss:intro-main}, then for any $w \in W_\ext$ the dimension of the $\bk$-vector space $\Hom_{\modf^\bY_{I_\unip}(\cR(\bk))}(\injh_w(\bk), \injh_{w^\triangle}(\bk))$ is at least $1$.
\end{prop}


\begin{proof}
By periodicity (see~\eqref{eqn:triangle-translation}), it is enough to prove this property when $w \in W^\res_\ext$.  We assume this from now on.

Suppose first that $\bk$ is a field of characteristic $0$.  In this case, according to Remark~\ref{rmk:indecomposable}\eqref{it:indecomposable-char0} we have $\injh_w \cong \Phi(\TGr_{w^\triangle})$.  The object $\LGr_{w^\triangle}$ occurs as a composition factor of $\TGr_{w^\triangle}$, so $\Phi(\LGr_{w^\triangle})$ is a subquotient of $\injh_w$.  It follows from Theorem~\ref{thm:geometric-Steinberg} and Lemma~\ref{lem:frobtwist-free} that $\LGT_{w^\triangle}$ is a direct summand of $\Phi(\LGr_{w^\triangle})$, and hence a composition factor of $\injh_w$.  The result follows.

We now consider the case where $\bk$ has positive characteristic. We can assume that $\bk$ is finite. Let $\bK$ be a finite extension of $\Q_\ell$ whose ring of integers $\bO$ has $\bk$ as residue field.  
Since $\bK ( \injh_w(\bO))$ contains $\injh_w(\bK)$ as a direct summand, and likewise for $w^\triangle$ (part~\eqref{it:qO-frac} in Theorem~\ref{thm:injh-integral}), the previous paragraph and Remark~\ref{rmk:Hom-Q-O}\eqref{it:Hom-Q-O} imply that the free $\bO$-module $\Hom(\injh_w(\bO), \injh_{w^\triangle}(\bO))$ has rank${}\ge 1$. Using again Remark~\ref{rmk:Hom-Q-O}\eqref{it:Hom-Q-O}, we deduce that $\dim \Hom(\injh_w(\bk), \injh_{w^\triangle}(\bk)) \ge 1$ as well.
\end{proof}

\subsection{Application: multiplicities and projective covers}
\label{ss:app-mult-proj}

We now return to the setting of field coefficients.  In the following pro\-position we gather the main properties of the objects $\injh^A_w$, some of which have already appeared in earlier statements.

\begin{prop}
\label{prop:reciprocity}
Let $w, x \in {}^A W_\ext$.
\begin{enumerate}
\item
\label{it:recip-proj} The object $\injh^A_x$ is both the injective envelope and the projective cover of $\LGT^A_x$.
\item
\label{it:recip-dual} We have $\D(\injh^A_x) \cong \injh^A_x$.
\item
\label{it:recip-Av} We have $\Av^A_*(\injh^A_x) \cong \injh_x$ and $\Av^A_!(\injh^A_x) \cong \injh_x$.
\item
\label{it:recip-mult} 
The multiplicities in baby Verma and baby co-Verma filtrations of $\injh^A_x$ satisfy the following ``reciprocity laws'':
\[
(\injh^A_x : \cbV^A_w) = [\bV^A_w : \LGT^A_x ] = [\cbV^A_w : \LGT^A_x ] = (\injh^A_x : \bV^A_w) .
\]
Moreover, these numbers are zero unless $x \preceq w \preceq w_Ax^\triangle$, and they are equal to $1$ when $w = x$ or $w = w_Ax^\triangle$.
\end{enumerate}
\end{prop}

Recall the map $\iota_A$ defined in~\eqref{eqn:def-iota}, characterized by the property that $\injh^A_{\iota_A(x)}$ is the projective cover of $\LGT^A_x$.  In view of~\eqref{eqn:D-injh}, parts~\eqref{it:recip-proj} and~\eqref{it:recip-dual} are both equivalent to the claim that
\begin{equation}\label{eqn:iota-id}
\iota_A(x) = x \qquad\text{for all $x \in {}^A W_\ext$.}
\end{equation}

\begin{proof}
Let us start with part~\eqref{it:recip-mult}. The first equality has already been noted in~\eqref{eqn:reciprocity}. The second one is part of Lemma~\ref{lem:verma-basic}. Finally, since $\injh^A_x$ admits both a baby Verma filtration and a baby co-Verma filtration (see Theorem~\ref{thm:proj-Hecke-Whit}\eqref{it:proj-filt} and Proposition~\ref{prop:inj-Verma-filt}), the equality $(\injh^A_x : \cbV^A_w) = (\injh^A_x: \bV^A_w)$ holds by Remark~\ref{rmk:bV-Groth-gp}.
%
%
Most of the last assertion of this part
has already been established in Corollary~\ref{cor:multiplicites-bV}; it only remains to show that
\begin{equation}\label{eqn:recip-mult-top}
[ \cbV^A_{w_Ax^\triangle} : \LGT^A_x ] = 1 \qquad\text{for all $x \in {}^A W_\ext$.}
\end{equation}
We will return to this later in the proof.  

We will now prove parts~\eqref{it:recip-proj} and~\eqref{it:recip-dual}.  Observe that by applying $\D$, we have
\[
(\D(\injh^A_x) : \cbV^A_z) = (\injh^A_x : \bV^A_z)
\]
for all $z \in {}^A W_\ext$. Now Remark~\ref{rmk:labels-injh} and the equalities from part~\eqref{it:recip-mult} show that there exists a unique minimal element $z$ (with respect to $\preceq$) such that the right-hand side is nonzero, namely $x$. Since $\D(\injh^A_x)$ is an indecomposable injective object, in view of Remark~\ref{rmk:labels-injh} again this implies that $\D(\injh^A_x)\cong \injh^A_x$, proving part~\eqref{it:recip-dual}. As noted above, this claim is equivalent to that in~\eqref{it:recip-proj}.


For part~\eqref{it:recip-Av}, the first isomorphism was already established in Theorem~\ref{thm:proj-Hecke-Whit}\eqref{it:Av-injh}.  The second one follows by Verdier duality; alternatively, since we now know that $\injh_x$ is the projective cover of $\LGT_x$, one can repeat the argument from the proof of Theorem~\ref{thm:proj-Hecke-Whit}\eqref{it:Av-injh} to show that $\Av^A_!(\injh^A_x) \cong \injh_x$.

It remains to prove~\eqref{eqn:recip-mult-top}.  By Corollary~\ref{cor:multiplicites-bV}, we at least know that $[ \cbV^A_{w_Ax^\triangle} : \LGT^A_x ] \le 1$.  Suppose in fact that this multiplicity is $0$ for some $x$. By the portion of part~\eqref{it:recip-mult} that is already proved, we see that the baby co-Verma modules $\cbV^A_y$ appearing in a baby co-Verma filtration of $\injh_x^A$ all satisfy $x \preceq y \prec w_Ax^\triangle$.  Using Corollary~\ref{cor:multiplicites-bV} again, we see that the composition factors of $\injh^A_x$ are of the form
$\LGT^A_z $ with $z \prec w_Ax^\triangle$;
in particular, $[\injh^A_x : \LGT^A_{w_Ax^\triangle}] = 0$, and hence $\Hom(\injh^A_x, \injh^A_{w_Ax^\triangle}) = 0$.  If $A = \varnothing$, this contradicts Proposition~\ref{prop:injh-tri-mult}, and we are done.
Before passing to the case where $A \ne \varnothing$, let us note that once we know that $(\injh_x : \cbV_{x^\triangle}) = 1$, then by Remark~\ref{rmk:order-cbV-filtration} we know that $\cbV_{x^\triangle}$ is actually a quotient of $\injh_x$.

Let us now finish the proof of~\eqref{eqn:recip-mult-top} in the case where $A \ne \varnothing$.  Since $\LGT^A_{w_Ax^\triangle}$ is the socle of $\cbV^A_{w_Ax^\triangle}$, the equation $[\injh^A_x : \LGT^A_{w_Ax^\triangle}] = 0$ implies that $\Hom(\injh^A_x, \cbV^A_{w_Ax^\triangle}) = 0$.  We have
\begin{multline*}
0 = \Hom(\injh^A_x, \cbV^A_{w_Ax^\triangle}) = \Hom(\injh^A_x, \Av^A_\psi(\cbV_{x^\triangle}))  \\
= \Hom(\Av^A_!(\injh^A_x), \cbV_{x^\triangle})
= \Hom(\injh_x, \cbV_{x^\triangle}),
\end{multline*}
where the second equality uses Lemma~\ref{lem:bV-Av}\eqref{it:bV-Av-psi}, the third one follows from adjunction, and the last one uses part~\eqref{it:recip-Av}.
This contradicts the previous paragraph, and thus finishes the proof.
\end{proof}

\subsection{Further properties of baby Verma and baby co-Verma modules}

In this subsection, we prove a number of statements exhibiting a symmetry between baby Verma modules and baby co-Verma modules.


We start with the following corollary of Proposition~\ref{prop:reciprocity}.

\begin{lem}
\label{lem:cbV-head-pre}
For any $w \in {}^A W_\ext$, the object $\cbV^A_{w_A w^\triangle}$ has a simple head, and the object $\bV^A_{w_Aw^\triangle}$ has a simple socle, both isomorphic to $\LGT^A_w$.
\end{lem}

\begin{proof}
Proposition~\ref{prop:reciprocity}\eqref{it:recip-mult} and Remark~\ref{rmk:order-cbV-filtration} imply that $\cbV^A_{w_A w^\triangle}$ is a quotient of $\injh^A_w$, so like $\injh^A_w$, it has a simple head, isomorphic to $\LGT^A_w$.  The claim for $\bV^A_{w_A w^\triangle}$ follows by Verdier duality.
\end{proof}

Below we will require the following combinatorial lemma.

\begin{lem}
\label{lem:tri-bijection}
The map ${}^A W_\ext \to {}^A W_\ext$ given by $w \mapsto w_A w^\triangle$ is a bijection.
\end{lem}

Note that $w_A w^\triangle$ at least lies in ${}^A W_\ext$ by Remark~\ref{rmk:triangle-AWext}\eqref{it:triangle-AWext}.  This lemma may have a purely combinatorial or alcove-geometric proof, but we will give an argument that relies on properties of $\cR$-modules and is intertwined with the proofs of the next two statements in the following way:
\[
\hbox{\scriptsize$
\begin{array}{c}
\text{Lemma~\ref{lem:tri-bijection}} \\ \text{for $A = \varnothing$}
\end{array}
\Rightarrow
\begin{array}{c}
\text{Prop.~\ref{prop:cbV-head}} \\ \text{for $A = \varnothing$}
\end{array}
\Rightarrow
\begin{array}{c}
\text{Prop.~\ref{prop:opposite-hwc}} \\ \text{for $A = \varnothing$}
\end{array}
\Rightarrow
\begin{array}{c}
\text{Lemma~\ref{lem:tri-bijection}} \\ \text{for any $A$}
\end{array}
\Rightarrow
\begin{array}{c}
\text{Prop.~\ref{prop:cbV-head}} \\ \text{for any $A$}
\end{array}
\Rightarrow
\begin{array}{c}
\text{Prop.~\ref{prop:opposite-hwc}} \\ \text{for any $A$}
\end{array}$}.
\]

\begin{proof}[Proof of Lemma~\ref{lem:tri-bijection} for $A = \varnothing$]
The inverse map is given by~\eqref{eqn:triangle-inverse}.
\end{proof}

\begin{prop}
\label{prop:cbV-head}
Let $w \in {}^A W_\ext$.
\begin{enumerate}
\item The object 
$\cbV^A_{w_A w^\triangle}$ is the projective cover of $\LGT^A_w$ in the Serre subcategory of $\modf^{\bY}_{(I^A_\unip,\cX_A)}(\cR)$ generated by the simple objects $\LGT^A_y$ with $w_A y^\triangle \not\prec w_A w^\triangle$.\label{it:cbV-head}

\item The object 
$\bV_{w_A w^\triangle}$ is the injective hull of $\LGT^A_w$ in the Serre subcategory of $\modf^{\bY}_{(I^A_\unip,\cX_A)}(\cR)$ generated by the simple objects $\LGT^A_y$ with $w_A y^\triangle \not\prec w_A w^\triangle$.
\end{enumerate}
\end{prop}

\begin{proof}[Proof assuming that Lemma~\ref{lem:tri-bijection} holds for $A$]
We will prove the claims for $\cbV^A_{w_A w^\triangle}$.  Those for $\bV^A_{w_A w^\triangle}$ follow by Verdier duality.  

We have already seen in Lemma~\ref{lem:cbV-head-pre} that $\cbV^A_{w_A w^\triangle}$ has a unique simple quotient, isomorphic to $\LGT^A_w$.
Next, we must show that $\cbV^A_{w_A w^\triangle}$ lies in the Serre subcategory described in the statement.  This is a consequence of Corollary~\ref{cor:multiplicites-bV}, which tells us that
\[
[ \cbV^A_{w_A w^\triangle} : \LGT^A_v ] \ne 0
\qquad\text{implies}\qquad
w_A v^\triangle \succeq w_A w^\triangle.
\]

Finally, to show that $\cbV^A_{w_A w^\triangle}$ is the projective cover of $\LGT^A_w$ in this Serre subcategory, we must show that
\[
\Ext^1(\cbV^A_{w_A w^\triangle}, \LGT^A_y) = 0
\qquad\text{if $w_A y^\triangle \not\prec w_A w^\triangle$.}
\]
As noted in the proof of Lemma~\ref{lem:cbV-head-pre} there is a surjective map $\injh^A_w \twoheadrightarrow \cbV^A_{w_A w^\triangle}$ whose kernel $\mathcal{K}$ admits a baby co-Verma filtration.
Since $\injh^A_w$ is projective, we have a surjective map $\Hom(\mathcal{K},\LGT^A_y) \to \Ext^1(\cbV^A_{w_A w^\triangle}, \LGT^A_y)$.  So it is enough to show that
\begin{equation}\label{eqn:kernel-injh}
\Hom(\mathcal{K}, \LGT^A_y) = 0
\qquad\text{if $w_A y^\triangle \not\prec w_A w^\triangle$.}
\end{equation}
By Proposition~\ref{prop:reciprocity}\eqref{it:recip-mult}, the baby co-Verma modules which occur in a baby co-Verma filtration of $\mathcal{K}$ have their labels that lie between $w$ and $w_A w^\triangle$ in the period order.  More precisely, by Lemma~\ref{lem:tri-bijection}, the terms of the baby co-Verma filtration can be written as
$\cbV^A_{w_A z^\triangle}$ with $w \preceq w_A z^\triangle \prec w_A w^\triangle$. By Lemma~\ref{lem:cbV-head-pre}, the unique simple quotient of $\cbV^A_{w_A z^\triangle}$ is $\LGT^A_z$, so
$\Hom(\cbV^A_{w_A z^\triangle}, \LGT^A_y) = 0$ if $y \neq z$, i.e.~if $w_A y^\triangle \ne w_A z^\triangle$.
This implies~\eqref{eqn:kernel-injh}.
\end{proof}

\begin{prop}
\label{prop:opposite-hwc}
For $w, w' \in {}^A W_\ext$, we have
\begin{multline*}
\Ext^n_{\modf^{\bY}_{(I^A_\unip,\cX_A)}(\cR)}(\bV^A_w, \cbV^A_{w'}) \cong
\Ext^n_{\modf^{\bY}_{(I^A_\unip,\cX_A)}(\cR)}(\cbV^A_{w'}, \bV^A_w)\\
 \cong
\begin{cases}
\bk & \text{if $w = w'$ and $n = 0$,} \\
0 & \text{otherwise.}
\end{cases}
\end{multline*}
\end{prop}

\begin{proof}[Proof assuming that Proposition~\ref{prop:cbV-head} holds for $A$]
The claims for $\Ext^n(\bV^A_w, \cbV^A_{w'})$ are contained in Lemma~\ref{lem:verma-hom}.  Recall that this lemma comes from the general theory developed in~\cite{bs}; it is available here because, as shown in the proof of Theorem~\ref{thm:proj-Hecke-Whit}, $\modf^{\bY}_{(I^A_\unip,\cX_A)}(\cR)$ is an essentially finite highest weight category with respect to the poset $({}^A W_\ext, \preceq)$.

Define a new order $\unlhd$ on ${}^A W_\ext$ by declaring that $w \unlhd y$ if $w_Aw^\triangle \succeq w_Ay^\triangle$.  Using Proposition~\ref{prop:cbV-head}, one can show that $\modf^{\bY}_{(I^A_\unip,\cX_A)}(\cR)$ admits a second structure as an essentially finite highest weight category, this time with respect to the poset $({}^A W_\ext, \unlhd)$.  In this context, baby co-Verma modules are standard objects, while baby Verma modules are costandard objects.  We omit the details, as they are very similar to those in the proof of Theorem~\ref{thm:proj-Hecke-Whit}.

The claims in the proposition for $\Ext^n(\cbV^A_{w'}, \bV^A_w)$ then follow from the analogue of Lemma~\ref{lem:verma-hom} for this new highest weight structure.
\end{proof}

\begin{proof}[Proof of Lemma~\ref{lem:tri-bijection} for $A \ne \varnothing$]
In view of~\eqref{eqn:triangle-inverse}, the rule $w \mapsto w_A w^\triangle$ gives a bijection $W_\ext \to W_\ext$, so the induced map ${}^A W_\ext \to {}^A W_\ext$ (see Remark~\ref{rmk:triangle-AWext}\eqref{it:triangle-AWext}) is at least injective.  We must prove that it is also surjective.

Let $v \in {}^A W_\ext$, and choose a simple quotient of $\cbV^A_v$, say $\LGT^A_w$.  Since $\LGT^A_w$ is the socle of $\bV^A_{w_A w^\triangle}$ by Lemma~\ref{lem:cbV-head-pre}, the following groups are all nonzero:
\begin{multline*}
\Hom(\cbV^A_v, \bV^A_{w_A w^\triangle}) \cong \Hom(\Av^A_\psi(\cbV_v), \Av^A_\psi(\bV_{w_A w^\triangle})) \\
\cong \Hom(\xi_A(\cbV_v), \bV_{w_A w^\triangle})
\cong \Hom(\Av^A_*(\cbV^A_v), \bV_{w_A w^\triangle}),
\end{multline*}
where the isomorphisms follow from Lemma~\ref{lem:bV-Av}\eqref{it:bV-Av-psi} and adjunction.
The description of the baby co-Verma filtration of $\Av^A_*(\cbV^A_v)$ in Lemma~\ref{lem:bV-Av}\eqref{it:bV-Av-*}, together with the ``$\varnothing$'' case of Proposition~\ref{prop:opposite-hwc}, imply that
\[
\dim \Hom (\Av^A_*(\cbV^A_v), \bV_{w_A w^\triangle}) =
\begin{cases}
1 & \text{if $w_A w^\triangle \in W_A v$,} \\
0 & \text{otherwise.}
\end{cases}
\]
So we must have $W_A v = W_A w^\triangle$.  Since both $v$ and $w_A w^\triangle$ belong to ${}^A W_\ext$, using Lemma~\ref{lem:AWext-representatives} we conclude that $v = w_A w^\triangle$, as desired.
\end{proof}

\subsection{More on injective \texorpdfstring{$\cR$}{R}-modules and tilting perverse sheaves}
\label{ss:more-on-inj}

We conclude with a complement to Proposition~\ref{prop:tilt-proj-inj-R}\eqref{it:Phi-tilting-injective}.  The following statement can be regarded as a geometric counterpart of~\cite[Lemma~E.8]{jantzen}.

\begin{prop}
\label{prop:injective-tilting}
Let $w \in {}^A W^S_\ext$.  The object $\Phi^A(\TGr^A_w)$ is a projective (equivalently, injective) $\cR$-module if and only if $w = w_A x^\triangle$ for some $x \in {}^A W^S_\ext$.
\end{prop}
\begin{proof}
The ``if'' direction is Proposition~\ref{prop:tilt-proj-inj-R}\eqref{it:Phi-tilting-injective}; we need only prove the ``only if'' direction.

First we treat the case when $\bk$ has characteristic $0$.  Assume that $\Phi^A(\TGr^A_w)$ is injective.  Then the functor
\[
\Hom_{\modf^\bY_{(I_\unip^A, \cX_A)}(\cR)}(\Phi^A({-}), \Phi^A(\TGr^A_w)) \cong \Hom_{\Perv_{(I^A_\unip, \cX_A)}(\Gr,\bk)} ({-}, \TGr^A_w \star^{\cL^+G} \cR_0)
\]
(where the identification follows from Lemma~\ref{lem:Rfree-hom})
is exact.  As explained in~\S\ref{ss:splitting-unit}, since $\bk$ has characteristic $0$, the skyscraper sheaf $\IC^0$ is a direct summand of $\cR_0$, so $\Hom({-}, \TGr^A_w \star^{\cL^+G} \IC^0) \cong \Hom({-}, \TGr^A_w)$ is an exact functor.  That is, $\TGr^A_w$ is an injective object in $\Perv_{(I^A_\unip, \cX_A)}(\Gr,\bk)$. This object is also indecomposable. From the classification of indecomposable injective objects in Theorem~\ref{thm:proj-char0-Whit}, we see that $\TGr^A_w \cong \TGr^A_{w_A x^\triangle}$ for some $x \in {}^A W^S_\ext$; we then have $w = w_A x^\triangle$.

Now suppose $\bk$ has positive characteristic, but that $A = \varnothing$.  We may assume that $\bk$ is finite.  Choose a ring $\bO$ as in~\S\ref{ss:integral} that has $\bk$ as its residue field, and let $\bK$ be its fraction field.  Let $w \in W^S_\ext$, and assume that $\Phi(\TGr_w(\bk)) \cong \Phi(\bk(\TGr_w(\bO)))$ is injective. Then there exist $w_1, \ldots, w_k \in W_\ext$ and an isomorphism
\begin{equation}
\label{eqn:Phi-T-inj}
\Phi(\TGr_w(\bk)) \cong \injh_{w_1}(\bk) \oplus \injh_{w_2}(\bk) \oplus \cdots \oplus \injh_{w_k}(\bk).
\end{equation}
In view of part~\eqref{it:qO-free} in Theorem~\ref{thm:injh-integral}, Lemma~\ref{lem:scalars-R-mod-F} and Remark~\ref{rmk:morph-O-P-tilting} imply that the morphisms
\begin{multline*}
\bk \otimes_\bO \Hom(\Phi(\TGr_w(\bO)), \injh_{w_1}(\bO) \oplus  \cdots \oplus \injh_{w_k}(\bO)) \to \\
\Hom(\Phi(\TGr_w(\bk)), \injh_{w_1}(\bk) \oplus  \cdots \oplus \injh_{w_k}(\bk))
\end{multline*}
and
\begin{multline*}
\bk \otimes_\bO \Hom(\injh_{w_1}(\bO) \oplus  \cdots \oplus \injh_{w_k}(\bO), \Phi(\TGr_w(\bO))) \to \\
\Hom(\injh_{w_1}(\bk) \oplus  \cdots \oplus \injh_{w_k}(\bk), \Phi(\TGr_w(\bk)))
\end{multline*}
induced by $\bk^0$
are isomorphisms. Using~\eqref{eqn:Phi-T-inj}, we deduce that there exist morphisms
\[
f : \Phi(\TGr_w(\bO)) \to \injh_{w_1}(\bO) \oplus  \cdots \oplus \injh_{w_k}(\bO)
\]
and
\[
g : \injh_{w_1}(\bO) \oplus  \cdots \oplus \injh_{w_k}(\bO) \to \Phi(\TGr_w(\bO))
\]
such that $\bk^0(f)$ and $\bk^0(g)$ are mutually inverse isomorphisms. Similarly, by Remark~\ref{rmk:Hom-Q-O}\eqref{it:Hom-Q-O}, $\End(\injh_{w_1}(\bO) \oplus  \cdots \oplus \injh_{w_k}(\bO))$ has finite rank over $\bO$, and the functor $\bk^0$ induces an isomorphism
\[
\bk \otimes_\bO \End(\injh_{w_1}(\bO) \oplus  \cdots \oplus \injh_{w_k}(\bO)) \to \End(\injh_{w_1}(\bk) \oplus  \cdots \oplus \injh_{w_k}(\bk)).
\]
In view of Lemma~\ref{lem:invertible-F} below, this implies that $f \circ g$ is an isomorphism, and hence that $\injh_{w_1}(\bO) \oplus  \cdots \oplus \injh_{w_k}(\bO)$ is a direct summand in $\Phi(\TGr_w(\bO))$. In other words, there exists $\cM$ in $\Mod^{\bY}_{I_\unip}(\cR(\bO))$ and an isomorphism
\[
\Phi(\TGr_w(\bO)) \cong \bigl( \injh_{w_1}(\bO) \oplus  \cdots \oplus \injh_{w_k}(\bO) \bigr) \oplus \cM.
\]
We have $\bk^0(\cM)=0$, so by Corollary~\ref{cor:vanishing-F-K} we have
$\bK(\cM)=0$ as well.
This implies that
the object $\Phi(\bK (\TGr_w(\bO)))$ is also injective.  This object contains $\Phi(\TGr_w(\bK))$ as a direct summand, so $\Phi(\TGr_w(\bK))$ is injective.  The field $\bK$ has characteristic $0$, so by the previous paragraph, we conclude that $w = x^\triangle$ for some $x \in W^S_\ext$, which completes the proof in this case.


Finally, we consider the case where $\bk$ has positive characteristic, but $A \ne \varnothing$.  Let $w \in {}^A W^S_\ext$, and suppose that $\Phi^A(\TGr^A_w)$ is injective.  Then $\Av^A_*(\Phi^A(\TGr^A_w)) \cong \Phi(\Av^A_*(\TGr^A_w))$ is also injective.  By Proposition~\ref{prop:Av-tilting-indec}, the latter is isomorphic to $\Phi(\TGr_{w_A w})$.  By the previous paragraph, we must have $w_A w = x^\triangle$, or $w = w_A x^\triangle$, for some $x \in W^S_\ext$.  On the other hand, Lemma~\ref{lem:tri-bijection} implies that there exists $y \in {}^A W_\ext$ such that $w=w_A y^\triangle$. The injectivity of the map $z \mapsto z^\triangle$ implies that $x=y$, so that this element belongs to $W_\ext^S \cap {}^A W_\ext = {}^A W_\ext^S$, which finishes the proof.
\end{proof}

\begin{lem}
\label{lem:invertible-F}
Let $\bO$ be the ring of integers in a finite extension of $\Q_\ell$, let $\bF$ be its residue field and let $A$ be a finite $\bO$-algebra. If $a \in A$ is such that its image in $\bF \otimes_\bO A$ is invertible, then $a$ is invertible.
\end{lem}

\begin{proof}
Let $\varpi$ be a uniformizer in $\bO$. Then by completeness it suffices to prove that the image of $a$ in each $A / \varpi^n A$ ($n \geq 1$) is invertible. This is checked by induction, the case $n=1$ being true by assumption. If we know that $a$ is invertible in $A / \varpi^n A$, and if $b \in A$ has image in $A / \varpi^n A$ the inverse of $a$, then $ab = 1 + \varpi^n c$ for some $c \in A$. If $d \in A$ has image in $A/\varpi A$ the inverse of $a$, then $ad \in 1 + \varpi A$, hence $a(b-\varpi^n dc) \in 1 + \varpi^{n+1} A$, which shows that $a$ is invertible in $A / \varpi^{n+1} A$, as desired.
\end{proof}

\end{document}